\documentclass[11pt]{amsart}

\usepackage{amsthm, amssymb, amsmath, stmaryrd, color, marvosym, rotating, graphicx, mathrsfs}

\input xy
\xyoption{all}

\newcommand{\quash}[1]{}  

\makeatletter
\newcommand*\leftdash{\rotatebox[origin=c]{-45}{$\dabar@\dabar@\dabar@$}}
\newcommand*\rightdash{\rotatebox[origin=c]{45}{$\dabar@\dabar@\dabar@$}}
\makeatother

\newcommand{\quot}[3]{{#1}\backslash{#2}/{#3}}
\newcommand{\wqw}[3]{{#1}\leftdash{#2}\rightdash{#3}}

\newcommand{\qw}[2]{{#1}\rightdash{#2}}
\newcommand{\leftexp}[2]{{\vphantom{#2}}^{#1}{#2}}
\newcommand{\geom}[1]{{#1}\otimes_k\bar{k}}
\newcommand{\conv}[1]{\stackrel{#1}{\ast}}
\newcommand{\twtimes}[1]{\stackrel{#1}{\times}}
\newcommand{\isom}{\xrightarrow{\sim}}

\newcommand{\Ltimes}{\stackrel{\bL}{\otimes}}

\DeclareMathOperator{\SL}{SL}
\DeclareMathOperator{\Hom}{Hom}
\DeclareMathOperator{\End}{End}
\DeclareMathOperator{\Ext}{Ext}
\DeclareMathOperator{\ext}{ext}

\DeclareMathOperator{\Aut}{Aut}
\DeclareMathOperator{\Sym}{Sym}
\DeclareMathOperator{\Spec}{Spec}

\DeclareMathOperator{\Rep}{Rep}

\DeclareMathOperator{\Vect}{Vect}
\DeclareMathOperator{\Gal}{Gal}
\DeclareMathOperator{\id}{id}

\DeclareMathOperator{\Pic}{Pic}

\DeclareMathOperator{\Lie}{Lie}

\DeclareMathOperator{\Gr}{Gr}
\DeclareMathOperator{\RHom}{\bR\Hom}

\DeclareMathOperator{\Forg}{Forg}
\DeclareMathOperator{\Frac}{Frac}

\DeclareMathOperator{\opp}{opp}
\DeclareMathOperator{\coker}{coker}

\def\AA{\mathbb{A}}
\newcommand\BB{\mathbb{B}}
\newcommand\CC{\mathbb{C}}

\newcommand\FF{\mathbb{F}}
\newcommand\GG{\mathbb{G}}
\newcommand\HH{\mathbb{H}}
\newcommand\MM{\mathbb{M}}
\newcommand\PP{\mathbb{P}}
\newcommand\QQ{\mathbb{Q}}

\newcommand\VV{\mathbb{V}}
\newcommand\ZZ{\mathbb{Z}}

\newcommand\bR{\mathbf{R}}
\newcommand\bL{\mathbf{L}}

\newcommand\frg{\mathfrak{g}}

\newcommand\frh{\mathfrak{h}}
\newcommand\frb{\mathfrak{b}}

\newcommand\frsl{\mathfrak{sl}}

\newcommand\calA{\mathcal{A}}
\newcommand\calB{\mathcal{B}}
\newcommand\calC{\mathcal{C}}
\newcommand\calD{\mathcal{D}}

\newcommand\calF{\mathcal{F}}
\newcommand\calG{\mathcal{G}}
\newcommand\calH{\mathcal{H}}
\newcommand\calK{\mathcal{K}}
\newcommand\calL{\mathcal{L}}
\newcommand\calM{\mathcal{M}}
\newcommand\calO{\mathcal{O}}
\newcommand\calP{\mathcal{P}}

\newcommand\calT{\mathcal{T}}

\newcommand\tilC{\widetilde{C}}

\newcommand\tilK{\widetilde{K}}

\newcommand\tilf{\widetilde{f}}

\newcommand\tilY{\widetilde{Y}}
\newcommand\tili{\widetilde\imath}
\newcommand\tilj{\widetilde\jmath}
\newcommand\tilm{\widetilde{m}}

\newcommand\tilalpha{\widetilde{\alpha}}

\newcommand\tilpi{\widetilde{\pi}}

\newcommand\tilrho{\widetilde{\rho}}

\newcommand\tcF{\widetilde{\calF}}
\newcommand\tcK{\widetilde{\calK}}
\newcommand\tcT{\widetilde{\calT}}
\newcommand\tcP{\widetilde{\calP}}

\newcommand\tcL{\widetilde{\calL}}
\newcommand\tDel{\widetilde{\Delta}}
\newcommand\tnab{\widetilde{\nabla}}
\newcommand\tdel{\widetilde{\delta}}

\newcommand\hatH{\widehat{H}}
\newcommand\hatD{\widehat{D}}
\newcommand\hatP{\widehat{P}}

\newcommand\hatS{\widehat{S}}
\newcommand\hatI{\widehat{I}}

\newcommand\hatm{\widehat{m}}
\newcommand\hata{\widehat{a}}

\newcommand\hatphi{\widehat{\phi}}
\newcommand\hatpsi{\widehat{\psi}}
\newcommand\hatalpha{\widehat{\alpha}}
\newcommand\hatbeta{\widehat{\beta}}
\newcommand\hattau{\widehat{\tau}}
\newcommand\hatPhi{\widehat{\Phi}}

\newcommand{\hatnu}{\widehat{\nu}}

\newcommand\barw{\overline{w}}
\newcommand\baru{\overline{u}}
\newcommand\barv{\overline{v}}
\newcommand\barf{\overline{f}}

\newcommand\barPhi{\overline{\Phi}}

\newcommand\unw{\underline{w}}
\newcommand\unu{\underline{u}}
\newcommand\unv{\underline{v}}
\newcommand\unE{\underline{E}}
\newcommand\unM{\underline{M}}

\newcommand\scD{\mathscr{D}}
\newcommand\scM{\mathscr{M}}
\newcommand\scP{\mathscr{P}}
\newcommand\scQ{\mathscr{Q}}
\newcommand\scE{\mathscr{E}}
\newcommand\scT{\mathscr{T}}
\newcommand{\scC}{\mathscr{C}}
\newcommand{\scV}{\mathscr{V}}
\newcommand\scTr{{\mathscr{T}^{\dagger}}}

\newcommand{\naC}{\leftexp{\natural}{\calC}}

\newcommand\hsM{\widehat{\mathscr{M}}}

\newcommand\hsP{\widehat{\mathscr{P}}}

\newcommand\epdag{\epsilon^{\dagger}}

\newcommand{\pidag}{\pi^{\dagger}}
\newcommand{\piddag}{\pi_{\dagger}}

\newcommand\Dleft{\leftexp{\dagger}{\scD}}
\newcommand\Dright{{\scD^\dagger}}

\newcommand\Grass{\mathcal{G}r}
\newcommand\Fl{\mathcal{F}\ell}
\newcommand\tilFl{\widetilde{\mathcal{F}\ell}}
\newcommand\hatFl{\widehat{\mathcal{F}\ell}}
\newcommand\PFl{\vphantom{\Fl}_{\Theta}{\Fl}}
\newcommand\LFl{\vphantom{\Fl}_{\varnothing}{\Fl}}

\newcommand\midd{{\rm mid}}

\newcommand\hatLG{\widehat{G_0((t))}}
\newcommand\gcen{\GG^{\textup{cen}}_m}
\newcommand\grot{\GG^{\textup{rot}}_m}

\newcommand\Ql{\overline{\QQ}_{\ell}}
\newcommand\IC{\mathcal{IC}}
\newcommand\Mod{\textup{Mod}}
\newcommand\Frob{\textup{Fr}}
\newcommand{\Funi}{\textup{Fr-unip}}
\newcommand\perf{\textup{perf}}
\newcommand\AS{\textup{AS}}
\newcommand\Av{\textup{Av}}
\newcommand\av{\textup{av}}
\newcommand{\bav}{\overline{\av}}
\newcommand\coW{W_\Theta\backslash W}
\newcommand{\pt}{\textup{pt}}
\newcommand{\triv}{\textup{triv}}
\newcommand{\pro}{{\rm pro}}
\newcommand{\mult}{{\rm mult}}
\newcommand{\adj}{\textup{adj}}
\newcommand\nil{\textup{nil}}

\newcommand{\Free}{\textup{Free}}
\newcommand{\free}{\textup{free}}
\newcommand{\acyc}{{\rm acyc}}
\newcommand{\pure}{{\rm pur}}
\newcommand{\fm}{free-monodromic }
\newcommand{\inv}{{\rm inv}}
\newcommand{\aug}{{\rm aug}}
\newcommand{\Tel}{\textup{Tel}}
\newcommand{\Gm}{\GG_m}
\newcommand{\Wa}{W_{\textup{aff}}}

\newcommand\pH{\leftexp{p}{H}}
\newcommand\ptau{\leftexp{p}{\tau}}
\newcommand\pD{\leftexp{p}{D}}

\newcommand\Tleq{\stackrel{\Theta}{\leq}}
\newcommand\lTh{\ell_\Theta}

\newcommand\xch{\mathbb{X}^*}
\newcommand\xcoch{\mathbb{X}_*}

\newcommand{\prolim}{\textup{\textquotedblleft}\varprojlim\textup{\textquotedblright}}
\newcommand\plim[1]{\textup{\textquotedblleft}\varprojlim_{#1}\textup{\textquotedblright}}

\newcommand\dG{{G^\vee}}
\newcommand\dH{{H^\vee}}
\newcommand\dB{{B^\vee}}
\newcommand\dU{{U^\vee}}
\newcommand\dS{\check{S}}
\newcommand\Sb{\check{S}^\bullet}

\newcommand\mon[1]{D'(\qw{#1}{A})}
\newcommand\hmon[1]{\hatD'(\qw{#1}{A})}
\newcommand\monm[1]{D'_m(\qw{#1}{A})}
\newcommand\hmonm[1]{\hatD'_m(\qw{#1}{A})}

\newcommand\per[1]{P'(\qw{#1}{A})}
\newcommand\hper[1]{\hatP'(\qw{#1}{A})}
\newcommand\perm[1]{P'_m(\qw{#1}{A})}
\newcommand\hperm[1]{\hatP'_m(\qw{#1}{A})}

\newcommand{\Oct}{{\bf Oct}}
\newcommand{\Tri}{{\bf Tri}}
\newcommand{\tTri}{\widetilde{\Tri}}
\newcommand{\Set}{{\bf Set}}
\newcommand{\preOct}{{\bf Oct}^{\rm pre}}
\newcommand{\Fun}{{\bf Fun}}
\newcommand{\upH}{\textup{H}}

\swapnumbers
\theoremstyle{plain}
\newtheorem{theorem}[subsubsection]{Theorem}
\newtheorem{lemma}[subsubsection]{Lemma}
\newtheorem{cor}[subsubsection]{Corollary}
\newtheorem{prop}[subsubsection]{Proposition}
\newtheorem*{claim}{Claim}

\newtheorem{alemma}[subsection]{Lemma}

\newtheorem*{mainthm}{Main Theorem}
\newtheorem*{introthm}{Theorem}
\newtheorem*{introprop}{Proposition}
\newtheorem*{var}{Variant}

\theoremstyle{definition}
\newtheorem{defn}[subsubsection]{Definition}
\newtheorem{exam}[subsubsection]{Example}

\newtheorem{remark}[subsubsection]{Remark}

\numberwithin{equation}{section}

\title{On Koszul duality for Kac-Moody groups}
\author{Roman Bezrukavnikov}
\address{Department of Mathematics, Massachusetts Institute of Technology, Cambridge,
MA 02139, USA}
\email{bezrukav@math.mit.edu}
\thanks{R.B. is partly supported by the NSF grant DMS-0854764.}
\author[Zhiwei Yun]{Zhiwei Yun \\ with appendices by Zhiwei Yun}
\thanks{Z.Y. was supported by the NSF grant DMS-0635607 and Zurich Financial Services as a member at the Institute for Advanced Study; currently Z.Y. is supported by the NSF grant DMS-0969470.} 
\address{School of Mathematics, Institute for Advanced Study, Princeton, NJ 08540, USA}
\curraddr{Department of Mathematics, Massachusetts Institute of Technology, Cambridge,
MA 02139, USA}
\email{zyun@math.mit.edu}
\subjclass[2000]{20G44, 14M15, 14F05}

\date{October 2008; Revised December 2009}

\begin{document}

\begin{abstract}
For any Kac-Moody group $G$ with Borel $B$, we give a monoidal equivalence between the derived category of $B$-equivariant mixed complexes on the flag variety $G/B$ and (a certain completion of) the derived category of $\dB$-monodromic mixed complexes on the enhanced flag variety $\dG/\dU$, here $\dG$ is the Langlands dual of $G$. We also prove variants of this equivalence, one of which is the equivalence between the derived category of $U$-equivariant mixed complexes on the partial flag variety $G/P$ and certain ``Whittaker model'' category of mixed complexes on $\dG/\dB$. In all these equivalences, intersection cohomology sheaves correspond to (free-monodromic) tilting sheaves. Our results generalize the Koszul duality patterns for reductive groups in \cite{BGS}.
\end{abstract}

\maketitle

\tableofcontents

\setcounter{section}{-1}
\setcounter{secnumdepth}{3}

\section{Introduction}

\subsection{History}

The formalism of Koszul duality in representation theory goes back to the work of Beilinson, Ginzburg, Schechtman \cite{BGSch} and Soergel \cite{Soe} from 1980's, and was developed later by these and other authors in \cite{BGS}, \cite{BGi} etc. The formalism uncovers some intriguing phenomena. On the one hand, it shows that some categories of representations (such as Bernstein-Gel'fand-Gel'fand category $\calO$) are ``controlled'' by Koszul quadratic algebras; this fact, closely related to Kazhdan-Lusztig conjectures, is proven using purity theorem about Frobenius (or Hodge) weights on Ext's between irreducible perverse sheaves. On the other hand, the duality (or rather equivalence) between derived categories of representations has some interesting geometric properties. In particular, it interchanges the Lefschetz $\frsl(2)$  (i.e. the $\frsl(2)$  containing multiplication by the first Chern class of an ample line bundle acting on cohomology of a smooth projective variety) with the Picard-Lefschetz $\frsl(2)$ (i.e. $\frsl(2)$ containing the logarithm of monodromy acting on cohomology of nearby cycles).\footnote{We mention in passing that this property is at least formally similar
to a key property of mirror symmetry; perhaps a better understanding of this similarity can lead to an insight into the nature of Koszul duality.} 

In this paper we extend the result of \cite{Soe} and \cite{BGS} to a much more general setting: we replace a semi-simple algebraic group considered in {\em loc. cit.} by an arbitrary Kac-Moody group. A comment is required on the precise relation between the two settings. First, \cite{Soe} works with a regular integral block in category $\calO$  of highest weight modules over the semi-simple Lie algebra. By Beilinson-Bernstein Localization Theorem this category is identified with a category of perverse sheaves on the flag variety. In this paper we work directly with the geometric category of sheaves and its generalizations. (A generalization of Localization Theorem to a general Kac-Moody group is not known, so one can not restate our result in terms of modules in this more general setting). The parabolic-singular variant of Koszul duality developed in \cite{BGS} involves singular category $\calO$. By \cite{MS} the latter is equivalent to the category of "generalized Whittaker" perverse sheaves on the flag variety; hence the appearance of Whittaker sheaves in the present paper.

Finally, we would like to point out that equivalences below generalize the variant of Koszul duality equivalence suggested in \cite{BGi} rather than the original equivalences of \cite{Soe} and \cite{BGS}. While the latter send irreducible objects to projective ones, the former sends irreducible objects to tilting ones. The advantage of the "tilting" version of the equivalence is that it turns out to be a monoidal functor (in the cases when the categories in question are monoidal); in the finite dimensional group case this verifies a conjecture in \cite[Conjecture 5.18]{BGi}. For a finite dimensional semi-simple group, the two functors differ by a long intertwining functor (Radon transform). In the Kac-Moody setting there is a more essential difference between the two formulations; in fact, the categories we consider do not have enough projectives, so the requirement for the functor to send irreducibles to projectives does not apply here. So we work out a generalization of the ``tilting'' version of the formalism, and show that the resulting equivalences are monoidal (when applicable). 
 
The price to pay for including monoidal categories into consideration is additional technical difficulties of foundational nature (appearing already in the finite dimensional semi-simple group case). As pointed out in \cite{BGi}, the dual to the Borel equivariant derived category of the flag variety is a completion of the category of unipotently monodromic sheaves on the base affine space to a category of pro-objects. The formal definition of such a completion and extension of the convolution monoidal structure to it requires additional work, done in the Appendix to the paper. See \cite{BGi} for a discussion of the relation of convolution with such pro-objects to projective functors on category $\calO$.

We should also mention that although in this article we work with mixed $\ell$-adic sheaves on varieties over a finite field, there should be a parallel story for mixed Hodge modules on the complex analogs of the relevant varieties.


\subsection{Main results}\label{ss:main}
Fix a finite field $k=\FF_q$. Let $G$ be a Kac-Moody group defined over $k$. For the purpose of the introduction, the reader is welcomed to take $G$ to be a split reductive group over $k$. Let $B=UH$ be a Borel subgroup of $G$ with unipotent radical $U$ and Cartan subgroup $H$. The ind-scheme $G/B$ is called the flag variety of $G$ and $G/U$ is called the enhanced flag variety of $G$. For other notations associated to $G$, we refer the readers to \S\ref{ss:KM} and the ``List of Symbols'' at the end of the paper.

Let $\dG$ be the Langlands dual Kac-Moody group of $G$. This is a Kac-Moody group with root system dual to that of $G$, with Borel subgroup $\dB=\dU\dH$. Let $W$ be the Weyl group of $G$ and $\dG$, which is a Coxeter group with simple reflections $\Sigma$ (in bijection with simple roots of $G$). Let $\Theta\subset\Sigma$ be such that the subgroup $W_\Theta$ generated by $\Theta$ is finite, hence determining a parabolic subgroup $P_\Theta$ of $G$. The main results of the paper consist of four equivalences of derived categories in the spirit of Koszul duality:

\begin{mainthm} There are equivalences of triangulated categories:
\begin{itemize}
\item Equivariant-monodromic duality (Theorem \ref{th:emduality}) which is a monoidal equivalence:
\begin{equation*}
\Phi: D^b_{m}(\quot{B}{G}{B})\isom\hatD^b_{m}(\wqw{\dB}{\dG}{\dB});
\end{equation*}
\item ``Self-duality'' (Theorem \ref{th:selfduality}):
\begin{equation*}
\Psi: D^b_{m}(\quot{\dB}{\dG}{\dU})\isom D^b_{m}(\quot{U}{G}{B});
\end{equation*}
\item Parabolic-Whittaker duality (Theorem \ref{th:pwduality}):
\begin{equation*}
\Phi_\Theta: D^b_{m}(\quot{P_\Theta}{G}{B})\isom\hatD^b_{m}((U^{\vee,\Theta}U^{\vee,-}_\Theta,\chi)\backslash\qw{\dG}{\dB});
\end{equation*}
\item ``Paradromic-Whittavariant'' duality (Theorem \ref{th:pwselfduality}):
\begin{equation*}
\Psi_\Theta: D^b_{m}(\quot{P^\vee_\Theta}{\dG}{\dU})\isom D^b_{m}(\quot{(U^\Theta U^-_\Theta,\chi)}{G}{B});
\end{equation*}
\end{itemize}
\end{mainthm}

We need to explain some notations. For a scheme $X$ over $k$ with a smooth group scheme $A$ over $k$ acting from the left, we denote by $D^{b}_{m,A}(X)$ or $D^b_m(A\backslash X)$ the derived category of $A$-equivariant mixed $\Ql$-complexes on $X$ (using either an $\ell$-adic analog of \cite{BL}, or the formalism of \cite{LO} if we view $A\backslash X$ as a stack). Therefore, $D^b_m(\quot{B}{G}{B})$ is understood as the derived category of left-$B$-equivariant mixed complexes on the flag variety $G/B$, etc.


The category $\hatD^b_{m}(\wqw{\dB}{\dG}{\dB})$ is a {\em completion} of the category $D^b_m(\wqw{\dB}{\dG}{\dB})$, the latter being the derived category of left $\dU$-equivariant mixed complexes on the enhanced flag variety $\dG/\dU$, which,  along the $\dH$-orbits (under the action given by either left or right multiplication), have unipotent monodromy. The completion procedure adds objects with free unipotent monodromy (called {\em \fm} sheaves) to the monodromic category. For details about the completion procedure, see the discussion in Appendix \ref{a:compmono}.

In the target of the last equivalence $\Psi_{\Theta}$, $U^{\Theta}$ is the unipotent radical of $P_\Theta$, and $U^{-}_\Theta$ is the unipotent radical of a Borel subgroup of $L_\Theta$ (the Levi subgroup of $P_\Theta$), which is opposite to the standard Borel. The left quotient by $(U^{\Theta}U^{-}_\Theta,\chi)$ means taking mixed complexes which are left equivariant under $U^{\Theta}U^{-}_\Theta$ against a generic character $\chi:U^{-}_\Theta\to\GG_a$. Such a construction is called the {\em geometric Whittaker model} (cf. \cite{BBM2}). The meaning of $(U^{\vee,\Theta}U^{\vee,-}_{\Theta},\chi)$ in the target of $\Phi_{\Theta}$ is similar, with $G$ replaced by $\dG$.

The equivalences in the Main Theorem enjoy the following properties:
\begin{itemize}
\item They respect the relevant monoidal structures. For example, both sides of the equivariant-monodromic duality carry monoidal structures given by convolution of sheaves, and $\Phi$ is a monoidal functor. Similarly, both sides of the parabolic-Whittaker duality are module categories under the respective monoidal categories in the equivariant-monodromic duality (given by convolution on the right), and $\Phi_{\Theta}$ respects these module category structures.
\item They send standard (resp. costandard) sheaves to standard (resp. costandard) sheaves. The spaces in question have Schubert stratifications indexed by (cosets of) the Weyl group. The standard and costandard sheaves are $!$ and $*$-extensions of constant sheaves (or \fm sheaves) on the strata.
\item They send intersection cohomology (IC-)sheaves to indecomposable (\fm) tilting sheaves (Def.\ref{def:fmt}). For example, under the equivalence $\Phi$, the intersection cohomology sheaf $\IC_{w}$ ($w\in W$, the Weyl group of $G$) of the closure of the Schubert stratum $BwB/B\subset G/B$ is sent to the \fm tilting sheaf $\tcT_{w}$ supported on the closure of $\dB w\dB/\dU\subset \dG/\dU$. In the case of $\Psi$ and $\Psi_\Theta$, they also send indecomposable tilting sheaves to IC-sheaves. More generally, all these equivalences send very pure complexes of weight 0 (Def.\ref{def:vpure}) to (\fm) tilting sheaves.
\item They are exact functors between triangulated categories, but {\em not} t-exact with respect to the perverse t-structures. Under all these equivalences, the Tate twist $(1)$ becomes the functor $[-2](-1)$.
\end{itemize}

\subsection{A baby case}
We look at the simplest case $G=\Gm$. Let $\Ql[T]=\upH^*_{\Gm}(\pt)$ be the $\Gm$-equivariant cohomology ring of a point, where $T$ is a generator in degree 2 with (geometric) Frobenius acting by $q$. The analog of \cite[Main Theorem 12.7.2(i)]{BL} in the mixed $\ell$-adic setting gives an equivalence
\begin{equation*}
D^b_m(\quot{\Gm}{\Gm}{\Gm})=D^b_{m,\Gm}(\pt)\cong D^{fg}(\Ql[T],\Frob),
\end{equation*}
the RHS being the derived category of finitely generated differential graded $\Ql[T]$-modules $L=[\cdots L^{-1}\to L^0\to\cdots]$ with a Frobenius action on each $L^i$, compatible with the Frobenius action on $\Ql[T]$, and with integer weights (see \S\ref{not:alg} for the definition of weights).

The Langlands dual group $\dG$ is again $\Gm$. We consider the non-mixed situation first (i.e., passing to $\overline{k}$, the non-mixed derived categories will be denoted by $D^b_c$ instead of $D^b_m$). A local system  on $\Gm$ with unipotent monodromy is given by a representation of the pro-$\ell$ quotient of $\pi_1(\Gm\otimes_{k}\bar{k})$. Taking the logarithm of the unipotent monodromy, such a sheaf corresponds to a finite dimensional $\Ql[[t]]$-module on which $t$ acts nilpotently. Denote the category of such $\Ql[[t]]$-modules by $\Mod^{\nil}(\Ql[[t]])$, then
\begin{equation*}
D^b_c(\wqw{\Gm}{\Gm}{\Gm})\cong D^b(\Mod^{\nil}(\Ql[[t]])).
\end{equation*}
The completion procedure will give
\begin{equation*}
\hatD^b_c(\wqw{\Gm}{\Gm}{\Gm})\cong D^b(\Ql[[t]]),
\end{equation*}
the RHS being the bounded derived category of all finitely generated $\Ql[[t]]$-modules. The object $\tcL$ in the completed category that corresponds to $\Ql[[t]]\in D^b(\Mod(\Ql[[t]]))$ is a \fm sheaf. The mixed version reads:
\begin{equation*}
\hatD^b_m(\wqw{\Gm}{\Gm}{\Gm})\cong D^{b}(\Ql[[t]],\Frob).
\end{equation*}
Here the RHS is the bounded derived category of finitely generated $\Ql[[t]]$-modules with a compatible Frobenius action (Frobenius acts on $t$ by $q^{-1}$). One can even replace $\Ql[[t]]$ by $\Ql[t]$ to get an equivalent derived category on the RHS (see Remark \ref{r:Ff}).

The equivariant-monodromic equivalence $\Phi$ for $G=\Gm$ and $\dG=\Gm$ is given by the following {\em regrading functor}
\begin{equation*}
\phi:D^{fg}(\Ql[T],\Frob)\isom D^{b}(\Ql[t],\Frob).
\end{equation*}
For a differential graded $\Ql[T]$-module $L=\oplus L^i$ with each $L^i$ a Frobenius module, write each $L^i=\oplus_j L^i_j$ according to the weights of the Frobenius action. Then $\phi(L)$ is a complex with $i$-th degree $\phi(L)^i=\oplus_j(L^{i-j}_{-j})^{\Finv}$. Here $(-)^{\Finv}$ denotes the same vector space with the inverse Frobenius action. Each term $\phi(L)^i$ then carries a $\Ql[t]$-module structure, with  $t$-action induced from that of $T$ on $L$.

\subsection{Other results}
Along the way of proving the Main Theorem, we also show
\begin{var}
\begin{enumerate}
\item The various categories involving $G$ in the Main Theorem can be combinatorially reconstructed from the pair $(V_H,W)$ alone ($V_H$ is the $\Ql$-Tate module of the maximal torus $H$ in $G$).
\item If $\Lie G$ is symmetrizable (e.g., $G$ is a reductive group or an affine Kac-Moody group), then we can replace all the $\dG$'s by $G$ in the various equivalences in the Main Theorem.
\end{enumerate}
\end{var}

In fact, for $\Lie G$ symmetrizable one can choose a $W$-equivariant isomorphism $V_H\isom V_\dH$. Hence by (1) above, the various categories for $\dG$ in the Main Theorem can be combinatorially identified with the corresponding categories for $G$.

Recall that the categories $D^b_{m}(\quot{B}{G}{B})$ and $\hatD^b_{m}(\wqw{B}{G}{B})$ carry convolution products, which we denote by $\conv{B}$ and $\conv{U}$. In proving the Main Theorem, we also get some results on IC and \fm tilting sheaves regarding the Frobenius semisimplicity of their convolutions:
\begin{introprop}[see Proposition \ref{p:icdecomp} and Corollary \ref{c:tdecomp}]
\begin{enumerate}
\item []
\item For $w_1,w_2\in W$, the convolution $\IC_{w_1}\conv{B}\IC_{w_2}$, as a {\em mixed} complex, is a direct sum of $\IC_w[n](n/2)$ for $n\equiv\ell(w_1)+\ell(w_2)-\ell(w)(\textup{mod }2)$;
\item For $w_1,w_2\in W$, the convolution $\tcT_{w_1}\conv{U}\tcT_{w_2}$, as a {\em mixed} complex, is a direct sum of $\tcT_w(n/2)$ for $n\equiv\ell(w_1)+\ell(w_2)-\ell(w)(\textup{mod }2)$.
\end{enumerate}
\end{introprop}

\subsection{The case of loop groups}

Among all Kac-Moody groups the affine ones are of particular interest in representation theory. These are modifications of the loop groups of reductive groups. Below we spell out our Main Theorem in the case of loop groups. We should mention that the results listed below are not literally special cases of the Main Theorem; nevertheless, only minor modifications are needed to prove them from the argument given in the main body of the paper.

Let $G$ be the affine Kac-Moody group associated to the loop group of a split simply-connected almost simple group $G_0$ over $k$. In other words, $G=\hatLG\rtimes\grot$ where $\hatLG$ is a nontrivial central extension of $G_0((t))$ by a one-dimensional torus $\gcen$ and the one-dimensional torus $\grot$ acts on $G_0((t))$ by ``rotating the loops''. Fix a split maximal torus $H_0$ in $G_0$ and a Borel subgroup $B_0\subset G_0$ containing $H_0$. We get an Iwahori subgroup $I\subset G_0((t))$ as the preimage of $B_0$ under the evaluation map $G_0[[t]]\to G_0$. We put a hat on top of $I$ or $H_0$ to denote their preimage in $\hatLG$.  The unipotent radical $I^u$ of $I$ admits a canonical lifting into $\hatLG$. The affine Cartan subgroup of $G$ is $H=\gcen\times H_0\times\grot$ and $B=HI^u$ is the Borel subgroup of $G$ with unipotent radical $U=I^u$.

The ind-scheme $\Fl=G_0((t))/I=G/B$ is the affine flag variety of $G_0$; the ind-scheme $\hatFl=\hatLG/I^u$ is the enhanced affine flag variety of $G_0$, which is a right $\hatH_0$-torsor over $\Fl$. Note that this is different from the enhanced flag variety $\tilFl=G/I^u$ (which is a right $H$-torsor over $\Fl$).

The group $I\rtimes\grot$ acts on the $G_0((t))/I$ (where $I$ acts by left translation and $\grot$ acts by rotating the loop). Let $\scE_0$ be the derived category of $I\rtimes\grot$-equivariant mixed complexes on $\Fl$. On the other hand, let $\scM_0$ be the derived category of left $I^u$-equivariant and right $\hatH_0$-monodromic complexes (with unipotent monodromy) on $\hatFl$. We can also define a completion $\hsM_0$ of $\scM_0$ by adding objects with free unipotent monodromy (see Appendix \ref{ss:compmono}). There are monoidal structures on $\scE_0$ and $\hsM_0$ defined by convolutions (similar as the convolutions in \S\ref{ss:Econv} and \S\ref{ss:Mconv}).

The various duality theorems for loop groups take the form:
\begin{introthm}
\begin{itemize}
\item []
\item Equivariant-monodromic duality (quantized version):
\begin{equation*}
\Phi:\scE_0=D^b_{m,\grot}(\quot{I}{G_0((t))}{I})\isom \hatD^b_{m}(\wqw{\hatI}{\hatLG}{\hatI})=\hsM_0.
\end{equation*}
This is a monoidal equivalence. The ``quantization parameter'' is given by a generator of $\upH^2_{\grot}(\pt)$ on the LHS, and is given by the logarithmic monodromy along $\gcen$-orbits on the RHS.
\item Equivariant-monodromic duality (non-quantized version):
\begin{equation*}
D^b_{m}(\quot{I}{G_0((t))}{I})\isom\hatD^b_{m}(\wqw{I}{G_0((t))}{I}).
\end{equation*}
This is obtained from the above quantized version by specializing the ``quantization parameters'' to zero.
\item Self-duality:
\begin{equation*}
\Psi:D^b_{m}(\quot{I^u}{G_0((t))}{I})\isom D^b_{m}(\quot{I}{G_0((t))}{I^u}).
\end{equation*}
which exchanges IC-sheaves and tilting sheaves. Moreover, the functor $\inv\circ\Psi$ is involutive (where $\inv:D^b_{m}(\quot{I^u}{G_0((t))}{I})\isom D^b_{m}(\quot{I}{G_0((t))}{I^u})$ is induced by the inversion map of $G_0((t))$).
\end{itemize}
\end{introthm}

For parabolic-Whittaker duality, we need to fix a {\it parahoric}\footnote{Parabolic subgroups of a loop group are usually called parahoric subgroups.} subgroup of $G_0((t))$. Here, to simplify notations, we only spell out the case when this parahoric subgroup is $G_0[[t]]$. Let $\Grass_{G_0}$ be the affine Grassmannian $G_0[[t]]\backslash G_0((t))$. Let $D^b_{m,\grot}(\Grass_{G_0}/I)$ be the derived category of mixed complexes on $\Grass_{G_0}$ equivariant under the right $I$-action and the loop rotation. Let $V$ be the preimage of $U^-_0$ (unipotent radical of the Borel $B^-_0$ opposite to $B_0$) under the evaluation map $G_0[[t]]\to G_0$. Let $\chi:V\to U^-_0\to\GG_a$ be a generic additive character. We can consider the category $D^b_{m}((V,\chi)\backslash\qw{\hatFl}{\hatH_0})$ of $(V,\chi)$-equivariant complexes on $\hatFl$ which are monodromic under the right $\hatH_0$-action with unipotent monodromy.
{\it
\begin{itemize}
\item []
\item Parabolic-Whittaker duality (for the affine Grassmannian):
\begin{equation*}
\Phi_{\Grass}: D^b_{m,\grot}(\Grass_{G_0}/I)\isom \hatD^b_{m}((V,\chi)\backslash\qw{\hatFl}{\hatH_0}).
\end{equation*}
\item Paradromic-Whittavariant duality (for the affine Grassmannian):
\begin{equation*}
\Psi_{\Grass}:D^b_{m}(\Grass_{G_0}/I^u)\isom D^b_{m}((V,\chi)\backslash\Fl).
\end{equation*}
\end{itemize}
}

\subsection{Main steps of the proof} To motivate the main idea of the proof of the equivariant-monodromic duality (Theorem \ref{th:emduality}), we briefly indicate the main steps of the proof of the quantized equivariant-monodromic duality for loop groups.

\noindent\textbf{Step I} (\S\ref{s:eq}). Taking global sections (or equivariant cohomology) of an object $\calF\in\scE_0$ gives a module over the equivariant cohomology ring $\upH^*_{\grot}(\quot{I}{G((t))}{I})$. This equivariant cohomology ring has been studied by Kostant-Kumar \cite{KK}. We can identify $\upH^2_{\grot}(\quot{I}{G((t))}{I})$ with $V^{\vee}_{H'}$, the dual of the $\Ql$-Tate module of the following torus (see \cite[\S3.7]{Yun1})
\begin{equation*}
H'=\ker(H\times H\xrightarrow{p_1/p_2}\grot)/\Delta(\gcen).
\end{equation*}
Here $p_1$ and $p_2$ are the canonical projections $H\to\grot$ applied to the first and second copy of $H$, and $\Delta$ means the diagonal embedding. Therefore we get a global section functor
\begin{equation*}
\HH:\scE_0\to D^b(\Sym(V_{H'}^\vee),\Frob)
\end{equation*}
(Here the grading on $\HH\calF=\upH^{*}_{I}(\Fl,\calF)$ is modified: it is given by a mixture of cohomological grading and Frobenius weights) In \S\ref{ss:Econv}, we show that $\HH$ has a natural monoidal structure (Proposition \ref{p:hcomp}). In \S\ref{ss:proofH}, we prove that $\HH$ is fully faithful on very pure complexes, using essentially the argument of Ginzburg \cite{Ginz}.

\noindent\textbf{Step II} (\S\ref{s:mon}). Each object of $\hsM_0$ carries unipotent monodromy coming from the action of $H\times H$ on $\quot{I^u}{G}{I^u}$ by $(h_1,h_2)\cdot x=h_1xh_2$. More precisely, let $H''$ be the torus
\begin{equation*}
\ker(H\times H\xrightarrow{p_1p_2}\grot)/\Delta^-(\gcen).
\end{equation*}
where $\Delta^-$ is the anti-diagonal embedding. Then $\Sym(V_{H''})$ acts as logarithmic monodromy operators on each object of $\hsM_0$. In \S\ref{ss:vv}, we define an exact functor:
\begin{equation*}
\VV:\hsM_0\to D^b(\Sym(V_{H''}),\Frob).
\end{equation*}
The functor $\VV$ can be thought of as an averaging functor. In \S\ref{ss:av} we define the usual averaging functors relating $\hsM_{0}$ and its Whittaker versions. However, extending this definition to $\VV$ involves averaging along infinite dimensional orbits. This technical complication is worked out in \S\ref{ss:vv}. In \S\ref{ss:tcP}, we show that $\VV$ has a natural monoidal structure (Proposition \ref{p:vcomp}). In \S\ref{ss:proofV}, we prove that $\VV$ is fully faithful on \fm tilting sheaves, generalizing \cite[Proposition  in \S2.1]{BBM}.

There are other technical complications in dealing with the completed category $\hsM_{0}$, e.g., the construction of the convolution structure on $\hsM_{0}$ in \S\ref{ss:Mconv}. 

\noindent\textbf{Step III} (\S\ref{ss:em}). Let $\{\IC_w|w\in\Wa\}$ be the IC-sheaves in $\scE_0$ and $\{\tcT_w|w\in \Wa\}$ be the indecomposable free-monodromic tilting sheaves in $\hsM_0$ (both indexed by the affine Weyl group $\Wa$). We define two algebras
\begin{eqnarray*}
E_0:=\bigoplus_{u,v\in \Wa}\Ext^\bullet_{\scE_0}(\IC_u,\IC_v);\\
M_0:=\bigoplus_{u,v\in \Wa}\Hom_{\hsM_0}(\tcT_u,\tcT_v)^f.
\end{eqnarray*}
where $(-)^f$ means taking the Frobenius locally finite part (here the Hom and Ext spaces are taken in the non-mixed categories, hence carrying Frobenius actions). Applying the general result in Appendix \ref{a:dgmodel}, we get equivalences
\begin{equation*}
\scE_0\cong D_{\perf}(E_0,\Frob);\hspace{1cm}\hsM_0\cong D_{\perf}(M_0,\Frob).
\end{equation*}
In other words, $E_0$ and $M_0$ serve as differential graded models for the triangulated categories $\scE_0$ and $\hsM_0$. By the discussion in the previous two steps, we can compute $E_0$ by the endomorphism algebra of $\oplus_w\HH(\IC_w)$, and compute $M_0$ by the endomorphism algebra of $\oplus_w\VV(\tcT_w)$. Therefore to prove the equivalence, we first need to identify $V_{H'}^\vee$ with $V_{H''}$ (up to an inversion of the Frobenius action) using the Killing form, and then identify $\HH(\IC_w)$ with $\VV(\tcT_w)$, which can be done in an explicit way. In fact, our strategy will be slightly different: instead of using $\IC_{w}$ and $\tcT_{w}$ to produce the algebras $E_{0}$ and $M_{0}$, we use the iterated convolutions $\IC_{s_{1}}\conv{I}\cdots\conv{I}\IC_{s_{m}}$ and $\tcT_{s_1}\conv{I^{u}}\cdots\conv{I^{u}}\tcT_{s_m}$ for reduced words $s_{1}\cdots s_{m}$. This strategy only requires explicit knowledge of the case $\SL(2)$ (which is done in Appendix \ref{a:SL2}).

The above discussion also shows why $\scE_0$ and $\hsM_0$ only depends on the combinatorial data $(V_H,W)$: the algebras $E_0$ (or $M_0$) can be identified with the endomorphism algebra of certain explicit $\Sym(V^\vee_{H'})$-modules. These are the so-called {\em Soergel bimodules} in the case $G$ is a reductive group.

\subsection{Organization of the paper} Above we reviewed the contents of \S\ref{s:eq} through \S\ref{ss:em}. The rest of \S\ref{s:dual} is devoted to the proof of the other three dualities mentioned in the Main Theorem. The self-duality is derived from the equivariant-monodromic duality by killing part of the equivariance/monodromy. The parabolic-Whittaker duality is derived from the equivariant-monodromic duality by a Barr-Beck type argument.

This paper has three appendices, written by Z.Yun. Appendix \ref{a:compmono} constructs the completions of the various monodromic categories by adding objects with free unipotent monodromy. To this end, we need to set up the framework for working with pro-objects in triangulated categories. Appendix \ref{a:dgmodel} constructs the differential graded models for the equivariant categories and completed monodromic categories. We treat these two cases in a uniform way. Appendix \ref{a:SL2} collects some simple results in the case of $G=\SL(2)$ which are proved by direct calculations.

{\bf Acknowledgement}
The work started during the special year of IAS on ``New Connections of Representation Theory to Algebraic Geometry and Physics''. The authors would like to thank IAS for its hospitality. Z.Y. would also like to thank the hospitality of the Kalvi Institute of Theoretic Physics where part of the paper was written.

R.B. would like to thank A.Beilinson and V.Ginzburg for comments on the history of the subject, and I.Mirkovi\'c for generously sharing his ideas which contributed to this work. Z.Y. would like to thank P.Deligne for helping him through the difficulties of defining completions of monodromic categories. The authors thank G.Williamson and the referee for helping improve the presentation of the paper.


\section{Notation and conventions}

\subsection{Notation concerning categories}
Given an adjoint pair of functors $(L,R)$ (i.e., $L$ is the left adjoint of $R$), we usually write the arrow representing $L$ above the arrow representing $R$. For example, the diagram
\begin{equation*}
\xymatrix{\scD_1\ar@<.7ex>[r]^{L} & \scD_2\ar@<.7ex>[l]^{R}}
\end{equation*}
means that $L$ is the left adjoint of $R$.

We adopt the following notation: let $\calF_i$ be objects in a triangulated category $\scD$, then $\langle\calF_1,\calF_2\cdots\rangle$ (denoted by $\langle\calF_1\conv{}\calF_2\conv{}\cdots\rangle$ in \cite{BBD}) means the class of objects in $\scD$ which are successive extensions of $\calF_i$.

\subsection{Notation concerning algebra}\label{not:alg}
Let $k=\FF_q$ be a finite field. Let $\Frob$ be the geometric Frobenius element in $\Gal(\overline{k}/k)$. Let $\ell$ be a prime different from char($k$). Fix an isomorphism $\Ql\cong\CC$ so that we have an archimedean norm $|-|$ on $\Ql$. Fix a square root of $q$ in $\Ql$ so that the half Tate-twist $(1/2)$ makes sense.

A $\Frob$-module is a $\Ql$-vector space equipped with an automorphism $\Frob_{M}:M\to M$. A $\Frob$-module is called {\em locally finite} if it is a union of finite-dimensional $\Frob$-submodules. We will use
\begin{equation}\label{Frobfin}
(-)^f: \{\Frob\textup{-modules}\}\to\{\textup{locally finite }\Frob\textup{-modules}\}
\end{equation}
to denote the functor which sends a $\Frob$-module $M$ to the union of its finite-dimensional $\Frob$-submodules. 

A locally finite $\Frob$-module is called {\em continuous} if the eigenvalues of $\Frob_{M}$ on $M$ are $\ell$-adic units. If $M$ is finite-dimensional, this is equivalent to saying that the assignment $\Frob\mapsto\Frob_{M}$ extends to a continuous homomorphism $\Gal(\overline{k}/k)\to\Aut_{\Ql}(M)$ (the target being under the $\ell$-adic topology). A general $\Frob$-module $M$ is called continuous if $M^{f}$ is.

For a locally finite $\Frob$-module $M$, the {\em weights} of $M$ are the real numbers $2\log(|\lambda|)/\log(q)$ where $\lambda$ runs over the eigenvalues of $\Frob_{M}$ on $M$. The weights of a general $\Frob$-module $M$ are those of $M^{f}$.

For a $\Frob$-module $M$, we use $M^{\Funi}$ to denote the $\Frob$-submodule of $M^{f}$ on which $\Frob$ acts unipotently. 

For a $\Frob$-module $M$, we use $M^\Finv$ to denote the same vector space $M$, but the action of $\Frob$ is the inverse of the original one.

For a $\Ql$-algebra $E$, we denote by $\Mod(E)$ the abelian category of finitely generated $E$-modules. If $E$ carries a continuous $\Frob$-module structure which is compatible with its algebra structure, let $\Mod(E,\Frob)$ denote the abelian category consisting of $E$-modules $M$ together with a compatible $\Frob$-action, which can be written as a quotient of $E\otimes V$ where $V$ is a finite-dimensional continuous $\Frob$-module with integer weights. We have the bounded derived categories $D^{b}(E)$ (resp. $D^{b}(E,\Frob)$) of $\Mod(E)$ (resp. $\Mod(E,\Frob)$).

{\em Unless otherwise claimed, all $\Frob$-modules in the sequel are understood to be continuous with integer weights}.

\subsection{Notation concerning geometry.}\label{not:geom} All stacks in this paper on which we talk about $\Ql$-sheaves will be the quotient stack $X=[G\backslash Y]$ where $Y$ is a scheme of finite type over $k$ and $G$ a smooth group scheme over $k$ acting on $Y$. We will encounter ind-schemes such as the flag variety $\Fl$ for a Kac-Moody group; however, when talking about sheaves on them, we actually mean sheaves on their finite-type subschemes $Y\subset\Fl$ (with the only exception of the so-called *-complexes, see \S\ref{sss:starsheaf}).

For a global quotient stack $X=[G\backslash Y]$ over $\overline{\FF}_{q}$, we will need the notion of the bounded derived category $D^{b}_{c}(X)$ of constructible $\Ql$-complexes on $X$. Following \cite{BL}, we may define this as the derived category of Cartesian and constructible $\Ql$-complexes on the simplicial scheme 
\begin{equation}\label{simplicial}
\xymatrix{\cdots G\times G\times Y\ar@<.7ex>[r]\ar[r]\ar@<-.7ex>[r] & G\times Y\ar@<.5ex>[r]\ar@<-.5ex>[r]& Y}
\end{equation}
In a series of papers \cite{Olsson},\cite{LO}, Laszlo and Olsson show that the usual sheaf-theoretic operations work also for such stacks.

When $X=[G\backslash Y]$ is a global quotient stack over $k=\FF_{q}$, we also need the notion of mixed $\Ql$-complexes on $X$. We first recall the definition of the mixed derived category $D^{b}_{m}(Y)$ for a scheme $Y$ over $k$. This is the bounded derived category of $\Ql$-complexes on $Y$ whose cohomology sheaves are mixed with integer punctual weights (cf. \cite[\S 5.1.5]{BBD}). Now for a stack $X=[G\backslash Y]$, we define $D^{b}_{m}(X)$ to be the derived category of Cartesian $\Ql$-complexes on the simplicial scheme \eqref{simplicial} (based changed to $\bar{k}$), whose value on $Y$ (and hence on each $G^{n}\times Y$) belongs to $D^{b}_{m}(Y)$.

In particular, $D^b_{m}(\textup{pt})\cong D^b(\Frob)$. When we talk about a ``twist'' of an object $\calF\in D^b_m(X)$, we mean $\calF\otimes M$ for some one dimensional $\Frob$-module (continuous with integer weights). The notation $\calF(?)$ means any such twist.

Let $\omega: D^b_{m}(X)\to D^b_{c}(\geom{X})$ be the pull-back along $\geom{X}\to X$. For a subcategory $\scD\subset D^b_{m}(X)$, we use $\omega\scD$ to denote its essential image in $D^b_{c}(\geom{X})$ under the functor $\omega$. We use the notation $\langle n\rangle$ to mean any combination of shifts and twists which increases the weight by $n$ (note that $[1]$ increases the weight by 1).

We think of $D^b_{m}(X)$ as enriched over $D^b(\Frob)$: for any two objects $\calF,\calF'$, we have $\Frob$-modules:
\begin{eqnarray*}
\RHom_X(\calF,\calF')=\RHom_{\geom{X}}(\omega\calF,\omega\calF')\in D^b(\Frob),\\
\Ext^i_X(\calF,\calF')=\Ext^i_{D^b_{c}(\geom{X})}(\omega\calF,\omega\calF')\in\Mod(\Frob),
\end{eqnarray*}
which are the $\RHom$-complex and Ext-groups in $D^b_{c}(\geom{X})$, rather than in $D^b_m(X)$. The actual $\RHom$-complex in $D^b_{m}(X)$ is
\begin{equation*}
\bR\hom_X(\calF,\calF')=\bR\Gamma(\ZZ\Frob,\RHom_X(\calF,\calF')).
\end{equation*}
where $\bR\Gamma(\ZZ\Frob,-)$ means the derived functor of taking $\Frob$-invariants on $D^b(\Frob)$. The actual Ext-groups (the cohomology groups of $\bR\hom_X(\calF,\calF')$) in $D^b_m(X)$ are denoted by $\ext^i_X(\calF,\calF')$, and they fit into short exact sequences (see \cite[Eq. 5.1.2.5]{BBD})
\begin{equation}\label{intromix}
0\to\Ext^{i-1}_X(\calF,\calF')_{\Frob}\to\ext^i_X(\calF,\calF')\to\Ext^i_X(\calF,\calF')^{\Frob}\to0.
\end{equation}
In summary, the ``Hom'' and ``Ext'' groups are Frobenius modules, while ``hom'' and ``ext'' groups are plain vector spaces.

We use $\Ext^\bullet$ to mean the sum of all $\Ext^i$.

The notation $\upH^*(X)$ or $\upH^*_c(X)$ is understood to be the \'etale cohomology (with compact support) of $\geom{X}$ with constant coefficients $\Ql$.

If $Y$ is a scheme over $k$, the triangulated category $D^b_m(Y)$ carries the perverse t-structure with middle perversity $(\pD^{\leq0}_m(Y),\pD^{\geq0}_m(Y))$ (cf. \cite[\S 2.2]{BBD}). The heart of this t-structure is denoted $P_m(Y)$, the mixed perverse sheaves. For a subcategory $\scD\subset D^b_m(Y)$, we usually omit the left exponent $^p$ and write $\scD^{\leq0}=\scD\cap\pD^{\leq0}_m(Y)$, etc.

For a torus $A$ over $k$, let $T_\ell(A)$ be its $\ell$-adic Tate module and $V_A=T_\ell(A)\otimes_{\ZZ_{\ell}}\Ql\cong H_1(A,\Ql)$. This is a $\Frob$-module of weight -2.

\subsection{Sheaves on ind-schemes}\label{sss:starsheaf} Let $X=\bigcup_{\alpha\in I}X_\alpha$ be an ind-scheme with prescribed closed subschemes $X_{\leq\alpha}$ indexed by a partially ordered set $I$. For $\alpha\leq\beta\in I$, let $i_{\alpha,\beta}:X_{\leq\alpha}\hookrightarrow X_{\leq\beta}$ be the closed embedding.

The categories $\{D^b_{m}(X_{\leq\alpha})\}_{\alpha\in I}$ together with the functors $i_{\alpha,\beta,*}$ form an inductive system of triangulated categories. Let
\begin{equation*}
D^b_{m}(X)=2-\varinjlim_{\alpha\in I}D^b_{m}(X_{\leq\alpha}).
\end{equation*}
be the inductive 2-limit of $D^b_{m}(X_{\leq\alpha})$.

On the other hand, the categories $\{D^b_{m}(X_{\leq\alpha})\}_{\alpha\in I}$ together with the pullback functors $i^*_{\alpha,\beta}$ form a projective system of triangulated categories. Let
\begin{equation*}
\underleftarrow{D}^b_{m}(X)=2-\varprojlim_{\alpha\in I}D^b_{m}(X_{\leq\alpha})
\end{equation*}
be the projective 2-limit of $D^b_{m}(X_{\leq\alpha})$. Objects of $\underleftarrow{D}^b_{m}(X)$ are called {\em *-complexes}, and are usually denoted by $\calF=(\calF_{\leq\alpha})$ with $\calF_{\leq\alpha}\in D^b_{m}(X_{\leq\alpha})$.

There is an obvious fully faithful embedding
\begin{equation*}
D^b_{m}(X)\hookrightarrow\underleftarrow{D}^b_{m}(X).
\end{equation*}

A morphism of ind-schemes $f:X=\bigcup_{\alpha\in I} X_\alpha\to Y=\bigcup_{\beta\in J}Y_\beta$ is said to be {\em bounded} if for every $\beta\in J$, the preimage $f^{-1}(Y_\beta)$ is contained in $X_\alpha$ for some $\alpha\in I$, and the restriction of $f$ to $Y_\beta$ is of finite type. For a bounded morphism $f$, we can define the functor
\begin{equation*}
f_!:\underleftarrow{D}^b_{m}(X)\to\underleftarrow{D}^b_{m}(Y).
\end{equation*}
In fact, for $\calF=(\calF_\alpha)\in D^b_{m}(X)$, let $(f_!\calF)_{\beta}=(f|_{f^{-1}(Y_\beta)})_!(j^*\calF_{\alpha})$ (where $j:f^{-1}(Y_\beta)\hookrightarrow X_\alpha$ is the inclusion). The fact that this family of objects is compatible with the pullback functors $i_{\beta,\beta'}$ for $\beta\leq\beta'\in J$ follows from the proper base change theorem. Therefore the functor $f_!$ sends $\underleftarrow{D}^b_{m}(X)$ to $\underleftarrow{D}^b_{m}(Y)$.

For a morphism of ind-schemes $f:X=\bigcup_{\alpha\in I} X_{\leq\alpha}\to Y=\bigcup_{\beta\in J}Y_{\leq\beta}$, the functor
\begin{equation*}
f^*:\underleftarrow{D}^b_{m}(Y)\to\underleftarrow{D}^b_{m}(X).
\end{equation*}
is always defined. In fact, for a complex $\calF=(\calF_{\leq\beta})\in\underleftarrow{D}^b_{m}(X)$, we let $(f^*\calF)_{\leq\alpha}=j^*(f|_{f^{-1}(Y_{\leq\beta})})^*(\calF_{\leq\beta})$ where $j:X_{\leq\alpha}\to f^{-1}(Y_{\leq\beta})$ is the inclusion. If, in addition, $f$ is bounded, then $f^*$ sends $D^b_{m}(Y)$ to $D^b_{m}(X)$.

We also need a variant of the notion of $*$-complexes in the case of completed monodromic categories (cf. Appendix \ref{ss:compmono}). In the case where $X=\bigcup_{\alpha\in I}X_{\alpha}$ is an $A$-torsor over an ind-scheme $Y=\bigcup_{\alpha\in I}Y_{\alpha}$ (with the induced ind-scheme structure: $X_{\alpha}$ is the preimage of $Y_\alpha$), where $A$ is a torus, we similarly define
\begin{equation*}
\underleftarrow{\hatD}^b_m(\qw{X}{A})=2-\varprojlim_{\alpha\in I}\hatD^b_m(\qw{X_{\leq\alpha}}{A})
\end{equation*}
with the transition functors given by $\tili^*_{\alpha,\beta}$.


\section{Kac-Moody groups and their flag varieties}

\subsection{Kac-Moody groups}\label{ss:KM}

We briefly review the notations concerning Kac-Moody groups that we will use in this paper, following \cite{Mat}. Let $A$ be a generalized Cartan matrix of either finite or affine type, together with a realization over $\QQ$. Let $\frg=\frg(A)$ be the Kac-Moody algebra associated to $A$, which is a Lie algebra over $\QQ$. It has a root decomposition:
\begin{equation}\label{eq:kmlie}
\frg=\frh\oplus\bigoplus_{\alpha\in R}\frg_\alpha
\end{equation}
where $\frh$ is the Cartan subalgebra and $R\subset\frh^*$ is the set of roots. By construction, we have a set of simple roots $\Sigma\subset R$, hence also the positive roots $R^+\subset R$. Let $W$ be the Weyl group associated to $\frh$. This is a Coxeter group with simple reflections in bijection with the set of simple roots $\Sigma$. Let $\ell:W\to\ZZ_{\geq 0}$ be the length function of $W$ in terms of the simple reflections $\Sigma$.

The universal enveloping algebras $U(\frg),U(\frh)$ as well as the integrable highest weight representations $L(\lambda)$ of $\frg$ admit $\ZZ$-forms. Let $k$ be any field. Using these $\ZZ$-forms, one can construct a Kac-Moody group $G$ over $k$. This is a group ind-scheme over $k$. A construction of this group ind-scheme is given in \cite[\S II]{Mat}. We also have the Borel subgroup $B\subset G$ (an affine group scheme), its pro-unipotent radical $U$, and the Cartan subgroup $H$ (a finite dimensional split torus over $k$), such that $B=UH$. The Lie algebras of $H$ and $U$ are $k$-forms of $\frh$ and $\oplus_{\alpha\in R^+}\frg_\alpha$.

\subsection{Flag varieties and Schubert varieties}\label{ss:Flag}
The flag variety $\Fl=\Fl_G$ associated to $G$ is the ind-scheme $G/B$ over $k$. For finite type $G$, $\Fl$ is the usual flag variety parametrizing Borel subgroups of $G$. In general, the ind-scheme structure on $\Fl$ is defined by a family of closed projective subschemes $\Fl_{\leq w}$ called {\em Schubert varieties} (denoted by $S_w$ in \cite{Mat}). Here $\Fl_{\leq w}$ is the closure of the $B$-orbit (also the $U$-orbit) $\Fl_w\subset\Fl$ under left translation. The orbit $\Fl_w$ is isomorphic to an affine space $\AA^{\ell(w)}$. We have $\Fl_{\leq w_1}\subset\Fl_{\leq w_2}$ if and only if $w_1\leq w_2$ in the Bruhat order of $W$. Let $\Fl_{<w}=\Fl_{\leq w}-\Fl_w$. Let $i_w,i_{\leq w}$ and $i_{<w}$ be the embeddings of $\Fl_w,\Fl_{\leq w}$ and $\Fl_{<w}$ into $\Fl$.

For each subset $\Theta$ of $\Sigma$, let $W_\Theta\subset W$ be the subgroup generated by $\Theta$. We say $\Theta$ is of {\em finite type} if $W_\Theta$ is finite. Associated to such a  $\Theta$ of finite type we have a standard parabolic subgroup $P_\Theta$ containing $B$ with Levi decomposition $P_\Theta=U^\Theta L_\Theta$ (where $L_\Theta$ contains $H$). Let $U_\Theta=U\cap L_\Theta$. Let $U^-_\Theta\subset L_\Theta$ be the radical of the Borel of $L_\Theta$ which is opposite to $B\cap L_\Theta$; i.e., $U^-_\Theta$ is the group generated by $U^-_s$ for $s\in\Theta$. We can identify $W_\Theta$ with the Weyl group of $L_\Theta$. Let $w_\Theta\in W_\Theta$ be the element with maximal length, whose length we denote by $\lTh$. Let $[\coW]\subset W$ (resp. $\{\coW\}$) be the minimal (resp. maximal) length representatives of cosets in $\coW$. We also have a length function $\ell:\coW\cong[\coW]\xrightarrow{\ell}\ZZ_{\geq 0}$ and a partial order on $\coW$ inherited from the Bruhat order on $W$.

Let $\PFl=P_\Theta\backslash G$ be the partial flag variety associated to the parabolic subgroup $P_\Theta$. Let $\pi^\Theta:\LFl=B\backslash G\to\PFl$ be the natural projection. The orbits of the right $B$ (or $U$) action on $\PFl$ are indexed by $\coW$. For each $\barw\in\coW$, the orbit $\PFl_{\barw}=P_\Theta\backslash P_\Theta wB$ is isomorphic to $\AA^{\ell(\barw)}$. As in the case of $\Fl$, the notations $\PFl_{\leq\barw},\PFl_{<\barw},i_{\barw},i_{\leq\barw}$ have the obvious meanings.

Fix $\Theta\subseteq\Sigma$. The $U^\Theta U^-_{\Theta}$-orbits on $\Fl$ are still indexed by the Weyl group $W$. The closure relation of $U^\Theta U^-_{\Theta}$-orbits define another partial ordering $\Tleq$ on $W$: we have $w\Tleq w'\Leftrightarrow w_\Theta w\leq w_\Theta w'$. For each $w\in W$, let $\Fl^\Theta_w=U^\Theta U^-_{\Theta} wB/B$ and $\Fl^\Theta_{\leq w}$ be its closure in $\Fl$.

Let $\tilFl:=G/U$ be the enhanced affine flag variety of $G$. The natural projection $\pi:\tilFl\to\Fl$ is a right-$H$-torsor. The ind-scheme $\tilFl$ is stratified by $B$-orbits which are also indexed by $W$. Let $\tilFl_w,\tilFl_{\leq w},\tilFl_{<w},\tilFl^\Theta_w,\tilFl^\Theta_{\leq w}$ and $\tilFl^{\leq w}$ be the preimages of their counterparts in $\Fl$ under $\pi$. Let $\tili_w$ (resp. $\tili_{\leq w}$) be the inclusion of $\tilFl_{w}$ (resp. $\tilFl_{\leq w}$) into $\tilFl$.

The following fact is well-known:
\begin{lemma}\label{l:fact}
Fix $\Theta\subset\Sigma$. For each $\barw\in\coW$, there exists a normal subgroup $J_{\barw}$ of $U$ of finite codimension such that the right translation action of $J_{\barw}$ on $\PFl_{\leq w}$ is trivial.
\end{lemma}

\subsection{The big cell}
In \cite[Remarks following Lemma 8]{Mat}, the big open cell $C=C(G/B)\subset\Fl$ is defined as follows. Recall from \cite[\S I]{Mat} that $\Fl_{\leq w}$ admits a projective embedding $\Fl_{\leq w}\hookrightarrow\PP(E_w(\lambda))$, where $E_w(\lambda)=U(\frb)L(\lambda)_{w\lambda}$, and $L(\lambda)_{w\lambda}$ is the $w\lambda$-weight line in the highest weight representation $L(\lambda)$ (the highest weight $\lambda$ is regular dominant). Let $L(\lambda)^*$ be the contragredient of $L(\lambda)$ with lowest weight vector $\sigma_{-\lambda}$ of weight $-\lambda$. Then $\sigma_{-\lambda}=0$ defines a hyperplane in $\PP(E_w(\lambda))$ and we let $C_{\leq w}:=C\cap\Fl_{\leq w}$ be the complement of this hyperplane in $\Fl_{\leq w}\hookrightarrow\PP(E_w(\lambda))$.

For any simple reflection $s\in\Sigma$ corresponding to the simple root $\alpha_s$, pick a nonzero vector $e_s\in\frg_{\alpha_s}$. Consider the vector $e_s\sigma_{-\lambda}\in L(\lambda)^*$, which has weight $-\lambda+\alpha_s$. The rational function $e_s\sigma_{-\lambda}/\sigma_{-\lambda}$ on $\PP(E_w(\lambda))$, pulled back to $\Fl_{\leq w}$, gives a rational function $\rho_{s,\lambda}$ on $\Fl_{\leq w}$. It is easy to check

\begin{lemma}\label{l:maptoUs}
The rational function $\rho_{s,\lambda}$ is independent of the regular dominant weight $\lambda$ and compatible with the embeddings $\Fl_{\leq w}\hookrightarrow\Fl_{\leq w'}$. Therefore it defines a rational function $\rho_s$ on $\Fl$ which is regular on $C$.
\end{lemma}

Let $\tilC\subset\tilFl$ be the preimage of $C$.
\begin{lemma}\label{l:Ctriv}
The $H$-torsor $\pi^C:\tilC\to C$ is trivializable.
\end{lemma}
\begin{proof}
This follows from \cite[Remark before Lemma 9]{Mat}: ``Le morphism $P\to\Pic C(G/B)$ est nul''.
\end{proof}

\begin{lemma}\label{l:fiberA1}
Let $\pi^C_s:C\hookrightarrow G/B\to G/P_s$ be the projection to the minimal partial flag variety corresponding to a simple reflection $s\in\Sigma$. For any geometric point $x\in G/P_s$, let $C_x:=\pi_s^{C,-1}(x)\subset C$ be the fiber. Then the function $\rho_{s'}$ is constant on $C_x$ if $s'\neq s$ and the function $\rho_s$ gives an isomorphism $\rho_s:C_x\isom\AA^1$.
\end{lemma}
\begin{proof}
Fix a regular dominant weight $\lambda$, and the embedding $\iota:\Fl_{\leq w}\hookrightarrow\PP(E_w(\lambda))$. Choose bases $v_0$ and $v_{s'}$ of the one dimensional weight spaces $L(\lambda)_\lambda$ and $L(\lambda)_{\lambda-\alpha_{s'}}$ ($s'$ is any simple root). Let $c\in C_{\leq w}$ be any geometric point. Then $\iota(c)$ is a line in $E_w(\lambda)$ which contains a vector $v_0+$(lower weight vectors). We may write $\iota(c)=[gv_0]$ ($[v]$ stands for the line containing $v$) for some $g\in G$ equal to a product of elements in $U_{-\beta}$ for negative simple roots $-\beta$ (this follows by looking at the Bott-Samelson resolution of $\Fl_{\leq w}$). For $g$ of this form, we have
\begin{equation}\label{coef}
gv_0=v_0+\textup{(lower weight terms)};\hspace{1cm} gv_s=v_s+\textup{(lower weight terms)}.
\end{equation}
Let $x=\pi^C_s(c)$. The fiber $\pi_s^{-1}(x)\cong\PP^1$, under the embedding $\iota$, can be identified with the pencil of lines $[g(t_0v_0+t_sv_s)]$ for $[t_0,t_s]\in\PP^1$. The fiber $C_x\subset\PP^1$ is the set of lines $[g(v_0+tv_s)]$ for $t\in\AA^1$. By \eqref{coef}, for $s'\neq s$, the coefficient of $v_{s'}$ in $gv_s$ is zero; hence the coefficient of $v_{s'}$ in $g(v_0+tv_s)$ is independent of $t$. This implies that $\rho_{s'}|C_x$ is constant. On the other hand, the coefficient of $v_s$ in $g(v_0+tv_s)$ is a non-constant linear function in $t$, which implies that $\rho_s|C_x$ induces an isomorphism $\rho_s:C_x\isom\AA^1$.
\end{proof}

It is also easy to see:
\begin{lemma}\label{l:contract}
Let $\mu:\GG_m\subset H$ be given by any anti-dominant regular coweight. Then for any $w\in W$, $C_{\leq w}$ contracts to the base point $B/B\in\Fl$ under the left action of $\mu(\GG_m)$. More generally, for any $v,w\in W$, $vC\cap\Fl_{\leq w}$ contracts to the point $vB/B\in\Fl$ under the action of $(v\mu)(\GG_m)$.
\end{lemma}


\section{Equivariant categories}\label{s:eq}

In this section, we define and study the category of $B$-equivariant complexes on the flag variety $\Fl=G/B$ of the Kac-Moody group $G$, as well as its parabolic version. We will study functors between these categories and the convolution product on the equivariant category. Of particular importance is the global section functor $\HH$. We will also give emphasis on the behavior of very pure complexes (such as IC-sheaves) under these operations.

\subsection{The equivariant category and its parabolic version}\label{ss:eqpar}

For each $\Theta\subseteq\Sigma$, consider the right $B$-action on $\PFl_G$. For each $\barw\in\coW$, choose $J_{\barw}\lhd U$ as in Lemma \ref{l:fact}(1), and we define $\scE_{\Theta,\leq w}$ to be the derived category of (right) $B/J_{\barw}$-equivariant mixed complexes on $\PFl_{G,\leq w}$. It is easy to see that this category is canonically independent of the choice of $J_{\barw}$. These form an inductive system under the fully faithful functors $i_{\barw,*}:\scE_{\Theta,\leq \barw}\to\scE_{\Theta,\leq \barw'}$ induced by the closed embeddings $i_{\barw}:\PFl_{G,\leq \barw}\hookrightarrow\PFl_{\leq \barw'}$ for $\barw\leq \barw'\in\coW$. Let $\scE_\Theta$ be the inductive 2-limit of $\scE_{\Theta,\leq \barw}$.

Recall that $V_H$ is the $\Ql$-Tate module of $H$. Then the graded algebra $\Sb:=\Sym(V_H^\vee[-2])$ is the $H$-equivariant cohomology ring of a point.

For $\Theta=\varnothing$, we also write $\scE$ for $\scE_\varnothing=D^b_m(\quot{B}{G}{B})$. Consider the action of $H\times H$ on the stack $\quot{U}{G}{U}$ given by $(h_1,h_2)\cdot x=h_1xh_2^{-1}$. We may view $\scE$ as the derived category of $H\times H$-equivariant complexes on $\quot{U}{G}{U}$, hence $\scE$ has the structure of an $\Sb\otimes \Sb$-linear category: $\Sb\otimes \Sb$ acts on $\Ext^\bullet_{\scE}(\calF,\calF)$ for all $\calF\in\scE$ functorially. Each $\scE_\Theta$ is naturally an $\Sb$-linear category for the right copy of $\Sb$.

For each $\barw\in\coW$, the standard, costandard and IC-complexes indexed by $\barw$ are
\begin{eqnarray*}
\Delta_{\barw}=i_{\barw,!}\Ql[\ell(\barw)](\ell(\barw)/2);\\
\nabla_{\barw}=i_{\barw,*}\Ql[\ell(\barw)](\ell(\barw)/2);\\
\IC_{\barw}=i_{\barw,!*}\Ql[\ell(\barw)](\ell(\barw)/2).
\end{eqnarray*}

The projection $\pi^{\Theta}:\LFl_G\to\PFl_G$ gives adjunctions
\begin{equation}\label{eq:paradj}
\xymatrix{\scE\ar[r] &
\scE_\Theta\ar@<1ex>[l]^{\pi^{\Theta,!}}\ar@<-2ex>[l]_{\pi^{\Theta,*}}^{\pi^\Theta_*}}.
\end{equation}

Consider the $H\times H$-equivariant global sections functor
\begin{equation*}
\bR\Gamma_{H\times H}(\quot{U}{G}{U},-):\scE\to D^b_m(\BB(H\times H))
\end{equation*}
By Corollary \ref{c:ba}, we have an equivalence $D^b_m(\BB(H\times H))\cong D_{\perf}(\dS\otimes\dS,\Frob)$. Here $\dS=\Sym(V^\vee_H)$ is viewed as a non-graded algebra. We can thus consider the $H\times H$-equivariant global section functor as a functor:
\begin{equation*}
\HH: \scE\to D_{\perf}(\dS\otimes\dS,\Frob)
\end{equation*}

For $w\in W$, let $\Gamma(w)=\{(w\cdot v,v)|v\in V_H\}\subset V_H\times V_H$ be the graph of the $w$-action on $V_H$. We view $V_H\times V_H$ as the spectrum of $\dS\otimes \dS$ and denote by $\calO(\Gamma(w))$ the coordinate ring of the closed subscheme $\Gamma(w)\subset V_H\times V_H$, which carries a grading and a $\Frob$-action.

\begin{lemma}\label{l:hst}
For each $w\in W$, we have isomorphisms of graded $(\dS\otimes\dS,\Frob)$-modules
\begin{eqnarray*}
\HH(\Delta_w)&\cong&\calO_{\Gamma(w)}[-\ell(w)](-\ell(w)/2),\\
\HH(\nabla_w)&\cong&\calO_{\Gamma(w)}[\ell(w)](\ell(w)/2).
\end{eqnarray*}
\end{lemma}
\begin{proof}
Consider the left $H$-equivariant embedding $\iota:HwH/H=wB/B\hookrightarrow\Fl_w$. The restriction map on cohomology $\iota^*:\upH^*(\Fl_w)\to \upH^*(wB/B)$ is an isomorphism because $\Fl_w$ is isomorphic to an affine space. Since both $wB/B$ and $\Fl_w$ are equivariantly formal with respect to the left $H$-action, the restriction map is also an isomorphism on equivariant cohomology, i.e.,
\begin{equation*}
\HH(\nabla_w)[-\ell(w)](-\ell(w)/2)\cong \upH^*(H\backslash\Fl_{w})\cong \upH^*(\quot{H}{HwH}{H}).
\end{equation*}
Here the stabilizer of the $H\times H$-action on $HwH$ (recall that the action is given by $(h_1,h_2)\cdot x\mapsto h_1xh_2^{-1}$) is the subtorus $\{(whw^{-1},h)|h\in H\}\subset H\times H$. Therefore $\upH^*(\quot{H}{HwH}{H})$ is isomorphic to $\calO_{\Gamma(w)}$. The second identity follows.

The proof of the first identity is similar, except we use the natural isomorphisms
\begin{eqnarray*}
&&\HH(\Delta_w)[-\ell(w)](-\ell(w)/2)\cong \upH_c^*(\quot{B}{BwB}{B})\\ &\cong& \upH^*(\quot{H}{HwH}{H},i^!\Ql)\cong \upH^*(\quot{H}{HwH}{H})[-2\ell(w)](-\ell(w)).
\end{eqnarray*}
\end{proof}

Recall from \cite[Definition 1.2.2(i)]{Del} that a local system $\calL$ on a scheme $X$ over $k$ is {\em pointwise pure} of weight $n$ (with respect to the chosen isomorphism $\Ql\isom\CC$), if for any closed point $x\in X$ with residue field $k(x)$, all the eigenvalues of the geometric Frobenius $\Frob_x$ acting on the stalk $\calF_{\bar{x}}$ has norm $\#k(x)^{n}$ under the chosen isomorphism $\Ql\isom\CC$. If $\calL$ is pointwise pure of weight $n$, then we say $\calL[m]$ is pure of weight $m+n$.

\begin{defn}[compare {\cite[\S5.2]{BB}}]\label{def:vpure}
Let $X=\bigsqcup X_\alpha$ be a stratified scheme and $\calF\in D^b_{m}(X)$ is constructible with respect to the stratification. Let $i_\alpha:X_\alpha\hookrightarrow X$ be the embeddings. Then $\calF$ is said to be {\em $*$-pure} (resp. {\em $!$-pure}) of weight $n$ if for each $\alpha$ and $m\in\ZZ$, the local system $\calH^{m}i_{\alpha}^*\calF$ (resp. $\calH^{m}i_{\alpha}^!\calF$) is pointwise pure of weight $n+m$. It is said to be {\em very pure} of weight $n$ if it is both $*$-pure and $!$-pure of weight $n$.
\end{defn}

We use $\scV\subset\scE$ (resp. $\scV_\Theta\subset\scE_\Theta$) to denote the full subcategory of very pure complexes of weight 0. The notion of very purity is stronger than the notion of purity of complexes (cf. \cite[\S5.1]{BBD}). However, in the situation of flag varieties, they are equivalent.

\begin{lemma}\label{l:samepure}
Suppose $\calF\in\scE_{\Theta}$ is pure of weight 0 in the sense of \cite[\S5.1]{BBD}, then it is very pure of weight 0.
\end{lemma}
\begin{proof}
We only need to consider the case $\calF\in\scE$, the parabolic case $\calF\in\scE_{\Theta}$ can be deduced from the case of $\pi^{\Theta,*}\calF\in\scE$.

Assume $\calF\in\scE_{\leq w}$. By Lemma \ref{l:contract}, for $v\in W$, the open subset $vC\cap\Fl_{\leq w}$ contracts to the point $v$ under a one-parameter subgroup of $H$. We denote the inclusion of $v$ into $\Fl_{G}$ still by $v$. Therefore, by \cite[Corollary 1]{Spr}, we have
\begin{equation*}
v^{*}\calF=\upH^*(vC\cap\Fl_{\leq w},\calF)
\end{equation*}
which has weight $\geq0$ as a complex because the open restriction $\calF|_{vC\cap\Fl_w}$ is pure of weight 0 and $\upH^{*}(-)$ does not decrease weight. On the other hand, since $\calF$ is pure of weight 0, $v^{*}\calF$ has weights $\leq0$. Therefore $v^{*}\calF$ has weight 0, hence so is $i^{*}_{v}\calF$ for any $v$, i.e., $\calF$ is $*$-pure of weight 0. A dual argument shows that $\calF$ is also $!$-pure of weight 0. Hence $\calF$ is very pure of weight 0.
\end{proof}

\begin{exam}\label{ex:vpure}
The IC-complex $\IC_{\barw}$ is pure of weight zero in the sense of \cite{BBD}, hence very pure by the above lemma. One can alternatively show this by the argument of Lemma \ref{l:Hss} below (essentially using Bott-Samelson resolution). By Example \ref{ex:icc}, the subcategory $\scV\subset\scE$ (resp. $\scV_\Theta\subset\scE_\Theta$) satisfies all the assumptions in Appendix \ref{a:dgmodel}.
\end{exam}

Here are some easy consequences of purity.
\begin{lemma}\label{l:purity}
\begin{enumerate}
\item []
\item If $\calF\in\scE$ is either $*$-pure or $!$-pure of weight 0, $\HH^i(\calF)$ is a $\Frob$-module of weight $i$ and $\HH(\calF)$ is free over each of the left and the right copy of $\dS$ (note that we are not claiming the freeness as $(\dS,\Frob)$-modules, but only the freeness as $\dS$-modules).
\item If $\calF_1\in\scE$ is $*$-pure of weight 0 and $\calF_2\in\scE$ is $!$-pure of weight 0, then $\Ext^i_{\scE}(\calF_1,\calF_2)$ is a $\Frob$-module of weight $i$ and $\Ext^\bullet_{\scE}(\calF_1,\calF_2)$ is free over each of the left and the right copy of the graded algebra $\Sb$.
\end{enumerate}
\end{lemma}
\begin{proof}
Note that (1) is a special case of (2) when $\calF_{1}$ is the constant sheaf. Therefore we only give the proof of (2). We use induction on the support of $\calF_1,\calF_2$. Suppose the statement is true for $\calF_i\in\scE_{<w}$. Now consider $\calF_i\in\scE_{\leq w}$, then we have a long exact sequence
\begin{equation*}
\cdots\to\Ext^i(i^*_{<w}\calF_1,i^!_{<w}\calF_2)\to \Ext^i(\calF_1,\calF_2)\to\Ext^i(i^*_w\calF_1,i^*_w\calF_2)\to\cdots
\end{equation*}
By assumption $i^*_w\calF_1$ and $i^!_w\calF_2$ are pure of weight 0, hence $\Ext^i(i^*_w\calF_1,i^*_w\calF_2)$ has weight $i$. Also $i^*_{<w}\calF_1$ (resp. $i^!_{<w}\calF_2$) is $*$-pure (resp. $!$-pure) of weight 0, by induction hypothesis we know that $\Ext^i(i^*_{<w}\calF_1,i^!_{<w}\calF_2)$ has weight $i$. By weight reasons, the above long exact sequence splits into short exact sequences:
\begin{equation*}
0\to\Ext^\bullet(i^*_{<w}\calF_1,i^!_{<w}\calF_2)\to \Ext^\bullet(\calF_1,\calF_2)\to\Ext^\bullet(i^*_w\calF_1,i^*_w\calF_2)\to0.
\end{equation*}
The two ends of the short sequences are free over each copy of $\Sb$, hence so is the middle one.
\end{proof}

An important property of the global sections functor $\HH$ is the following, whose proof (essentially borrowed from the argument of Ginzburg in \cite{Ginz}) will be postponed to section \ref{ss:proofH}.
\begin{prop}\label{p:hff}
Suppose $\calF_1,\calF_2\in\scV$, then the natural map
\begin{equation}\label{eq:Hff}
\Ext^\bullet_{\scE}(\calF_1,\calF_2)\to\Hom_{\dS\otimes\dS}(\HH(\calF_1),\HH(\calF_2))
\end{equation}
is an isomorphism of $\Frob$-modules.
\end{prop}

\begin{lemma}\label{l:pist}
For any $w\in W$, write $w=uv$ with $u\in W_\Theta$ and $v\in[\coW]$. We have
\begin{eqnarray*}
\pi^\Theta_*\Delta_w&\cong&\Delta_{\barw}[-\ell(u)](-\ell(u)/2);\\
\pi^\Theta_*\nabla_w&\cong&\nabla_{\barw}[\ell(u)](\ell(u)/2)
\end{eqnarray*}
\end{lemma}
\begin{proof}
We only need to observe that the projection $\LFl_{G,\leq w}\to\PFl_{G,\barw}$ is a trivial fibration with fibers isomorphic to $\AA^{\ell(u)}$.
\end{proof}

\begin{cor}
The functor $\pi^\Theta_*$ sends very pure (resp. $*$-pure, $!$-pure) complexes of weight 0 to very pure (resp. $*$-pure, $!$-pure) complexes of weight 0.
\end{cor}
\begin{proof}
By Lemma \ref{l:pist}, $\pi^\Theta_*$ sends $\langle\Delta_{w}\langle0\rangle|w\in W\rangle$ to
$\langle\Delta_{\barw}\langle0\rangle|\barw\in\coW\rangle$ and sends $\langle\nabla_w\langle0\rangle|w\in W\rangle$ to $\langle\nabla_{\barw}\langle0\rangle|\barw\in\coW\rangle$. But these two classes consist precisely of $*$-pure and $!$-pure complexes of weight 0. Moreover very pure complexes are precisely those objects in the intersection of the two classes.
\end{proof}


\subsection{Convolution}\label{ss:Econv}
Consider the convolution diagram
\begin{equation}\label{d:convI}
\xymatrix{ & G\twtimes{B}\Fl\ar[dl]_{p_1}\ar[dr]^{p_2}\ar[rr]^{m} & & \Fl\\
\Fl & & B\backslash\Fl}
\end{equation}
where $p_1,p_2$ are projections to the left and right factors and $m$ is induced by the multiplication map of $G$. The convolution diagram induces a convolution product
\begin{eqnarray*}
\conv{B}:\scE\times\scE&\to& \scE\\
(\calF_1,\calF_2)&\mapsto&m_!(\calF_1\stackrel{B}{\boxtimes}\calF_2).
\end{eqnarray*}
Note that $\calF_1\boxtimes\calF_2$ is an $B$-equivariant complex on $G\times\Fl$ with respect to the action $i\cdot(g,x)=(gi^{-1},ix)$ and hence descends to a complex $\calF_1\stackrel{B}{\boxtimes}\calF_2$ on $G\twtimes{B}\Fl$. There is an obvious associativity constraint which makes $\conv{B}$ into a monoidal structure on $\scE$. More generally, the convolution gives a right action of the monoidal category $\scE$ on $\scE_\Theta$ given by the same formula.

\begin{prop}\label{p:hcomp}
The functor $\HH$ has a natural monoidal structure which intertwines the convolution $\conv{B}$ on $\scE$ and the tensor product $(N_1,N_2)\mapsto N_1\Ltimes_{\dS}N_2$ (with respect to the right $\dS$-action on $N_1$ and left $\dS$-action on $N_2$) on $D_{\perf}(\dS\otimes\dS,\Frob)$.
\end{prop}
\begin{proof}
Consider the group $H\times H$ acting on $\tilFl\times\Fl$ by $(h_1,h_2)\cdot(x,y)=(xh_1^{-1},h_2y)$. The quotient map by $H\times H$ can be factorized into two steps:
\begin{equation*}
\tilFl\times\Fl\xrightarrow{p_0}\tilFl\twtimes{H}\Fl\xrightarrow{(p_1,p_2)}\Fl\times H\backslash\Fl.
\end{equation*}
where $p_0$ is the quotient by the diagonal copy of $\Delta(H)\subset H\times H$.

Applying Corollary \ref{c:eqkillS} (the isomorphism \eqref{killeqcoho}) to the $H\times H/\Delta(H)$-torsor $\tilFl\twtimes{H}\Fl\xrightarrow{(p_1,p_2)}\Fl\times H\backslash\Fl$, we get a functorial quasi-isomorphism
\begin{equation}\label{eq:convtensor}
(\HH(\calF_1)\Ltimes\HH(\calF_2))\Ltimes_{\dS\otimes \dS}\dS\cong\HH(\tilFl\twtimes{H}\Fl,\calF_1\stackrel{H}{\boxtimes}\calF_2).
\end{equation}
In the tensor product on LHS, the $\dS\otimes\dS$-module structure on $\HH(\calF_1)\Ltimes\HH(\calF_2)$ comes from the right $\dS$-action on $\HH(\calF_1)$ and the left $\dS$-action on $\HH(\calF_2)$; the $\dS\otimes\dS$-module structure on $\dS$ comes from left and right multiplication. Now (\ref{eq:convtensor}) is obviously the same as
\begin{equation}\label{eq:con1}
\HH(\calF_1)\Ltimes_{\dS}\HH(\calF_2)\cong\HH(\tilFl\twtimes{H}\Fl,\calF_1\stackrel{H}{\boxtimes}\calF_2).
\end{equation}
Since both projections $G\twtimes{H}\Fl\to\tilFl\twtimes{H}\Fl$ and $G\twtimes{H}\Fl\to G\twtimes{B}\Fl$ are fibrations with fibers isomorphic to pro-affine spaces, we have
\begin{equation}\label{eq:con2}
\HH(\tilFl\twtimes{H}\Fl,\calF_1\stackrel{H}{\boxtimes}\calF_2)\cong\HH(G\twtimes{B}\Fl,\calF_1\stackrel{B}{\boxtimes}\calF_2)=\HH(\calF_1\conv{B}\calF_2).
\end{equation}
Combining (\ref{eq:con1}) and (\ref{eq:con2}), we get a functorial quasi-isomorphism
\begin{equation*}
\HH(\calF_1)\Ltimes_{\dS}\HH(\calF_2)\cong\HH(\calF_1\conv{B}\calF_2).
\end{equation*}
\end{proof}

\begin{lemma}\label{l:convIst}
Suppose $w_1,w_2\in W$ and $\ell(w_1w_2)=\ell(w_1)+\ell(w_2)$, then
\begin{equation}\label{w1w2}
\Delta_{w_1}\conv{B}\Delta_{w_2}\cong\Delta_{w_1w_2};\hspace{1cm} \nabla_{w_1}\conv{B}\nabla_{w_2}\cong\nabla_{w_1w_2}.
\end{equation}
Moreover, $\delta:=\Delta_{e}=\nabla_{e}\in\scE$ ($e$ is the identity in $W$) is the unit object under the convolution $\conv{B}$.
\end{lemma}
\begin{proof}
The morphism $m$ in the convolution diagram \eqref{d:convI} restricts to the following $B$-equivariant {\em isomorphism} 
\begin{equation}\label{mw1w2}
Bw_{1}B\twtimes{B}\Fl_{w_{2}}\xrightarrow{m_{w_{1},w_{2}}}\Fl_{w_{1}w_{2}}.
\end{equation}
The isomorphisms in \eqref{w1w2} follows easily from the isomorphism \eqref{mw1w2}. The second statement about $\delta$ is obvious.
\end{proof}

\begin{prop}\label{p:vpure}
If $\calF_1,\calF_2\in\scV$, then so is $\calF_1\conv{B}\calF_2$.
\end{prop}
\begin{proof}
By definition,
\begin{equation*}
\scV=\langle\Delta_w\langle0\rangle|w\in W\rangle\cap\langle\nabla_w\langle0\rangle|w\in W\rangle
\end{equation*}
We observe that
\begin{eqnarray}\label{eq:cdel}
\Delta_{w_1}\conv{B}\Delta_{w_2}&\in&\langle\Delta_w\langle\leq0\rangle|w\in W\rangle\\
\label{eq:cnab}
\nabla_{w_1}\conv{B}\nabla_{w_2}&\in&\langle\nabla_w\langle\geq0\rangle|w\in W\rangle.
\end{eqnarray}
In fact, to prove (\ref{eq:cdel}), we can write each $\Delta_w=\Delta_{s_1}\conv{B}\cdots\conv{B}\Delta_{s_m}$ by Lemma \ref{l:convIst} (for a reduced word expression $w=s_1\cdots s_m$) and we reduce to the computation of $\Delta_{s}\conv{B}\Delta_{s'}$ for two simple reflections $s,s'$. If $s\neq s'$, then $\Delta_s\conv{B}\Delta_{s'}\cong\Delta_{ss'}$ by Lemma \ref{l:convIst}. For $s=s'$, this follows by Lemma \ref{l:convSL2}. The proof of (\ref{eq:cnab}) is similar.

Therefore, for $\calF_1,\calF_2\in\scV$, we have
\begin{equation*}
\calF:=\calF_1\conv{B}\calF_2\in\langle\Delta_w\langle\leq0\rangle|w\in W\rangle\cap\langle\nabla_w\langle\geq0\rangle|w\in W\rangle.
\end{equation*}
But this is the same as saying that $\calF$ is pure of weight 0 in the sense of \cite{BBD}. By Lemma \ref{l:samepure}, $\calF$ is very pure of weight 0.
\end{proof}

\begin{lemma}\label{l:Hss}
For each $w\in W$, $\HH_w:=\HH(\IC_w)$ is a direct sum of $\Frob$-modules $\Ql[n](n/2)$ for $n\equiv\ell(w)(\textup{mod }2)$.
\end{lemma}
\begin{proof}
For $w=s$ a simple reflection, by Lemma \ref{l:Hsimple}, we have $\HH(\IC_s)\cong\calO(\Gamma(e)\cup\Gamma(s))[1](1/2)$, for which this statement is true. In general, write $w$ as a reduced word $w=s_1\cdots s_m$ where $s_j$ are simple reflections. Let $\IC_{\unw}:=\IC_{s_1}\conv{B}\cdots\conv{B}\IC_{s_m}$, which is a very pure of weight 0 by Proposition \ref{p:vpure}.

By Proposition \ref{p:hff} and Proposition \ref{p:hcomp}, $\End(\IC_{\unw})$ is a direct summand of
\begin{equation*}
\End_{\dS\otimes\dS}(\HH(\IC_{\unw}))=\End_{\dS\otimes\dS}(\HH_{s_1}\otimes_{\dS}\cdots\otimes_{\dS}\HH_{s_m}),
\end{equation*}
which is in particular $\Frob$-semisimple. By Corollary \ref{c:Cdecomp}, $\IC_{\unw}$ decomposes as a sum of shifted and twisted IC-sheaves, among which $\IC_w$ necessarily appears with multiplicity one. Then, as a $\Frob$-module, $\HH_w$ is a direct summand of $\HH(\IC_{\unw})=\HH_{s_1}\otimes_{\dS}\cdots\otimes_{\dS}\HH_{s_m}$, which is a direct sum of $\Ql[n](n/2)$ for $n\equiv\ell(w)(\textup{mod }2)$.
\end{proof}

\begin{prop}\label{p:icdecomp} For $w_1,w_2\in W$, the convolution $\IC_{w_1}\conv{B}\IC_{w_2}$, as a {\em mixed} complex, is a direct sum of $\IC_w[n](n/2)$ for $n\equiv\ell(w_1)+\ell(w_2)-\ell(w)(\textup{mod }2)$. In particular, if $\ell(w_1w_2)=\ell(w_1)+\ell(w_2)$, then $\IC_{w_1w_2}$ is a direct summand of $\IC_{w_1}\conv{B}\IC_{w_2}$ with multiplicity one.
\end{prop}
\begin{proof}
Let $\calF=\IC_{w_1}\conv{B}\IC_{w_2}$. By Proposition \ref{p:hff} and Proposition \ref{p:hcomp}, $\End_{\scE}(\calF)$ is a direct summand of $\End_{\dS\otimes\dS}(\HH(\calF))=\End_{\dS\otimes\dS}(\HH_{w_1}\otimes_{\dS}\HH_{w_2})$, which is $\Frob$-semisimple. By Corollary \ref{c:Cdecomp}, $\calF$ then decomposes as a direct sum of $\IC_w\otimes M_w$ for some complexes $M_w$ of semisimple $\Frob$-modules. Apply $\HH$, we see $\HH(\calF)=\HH_{w_1}\otimes_{\dS}\HH_{w_2}$ is a direct sum of $\HH_w\otimes M_w$. Apply Lemma \ref{l:Hss} again, we conclude that $M_w$ is a direct sum of $\Ql[n](n/2)$ for $n\equiv\ell(w_1)+\ell(w_2)-\ell(w)(\textup{mod }2)$.

If $\ell(w_1w_2)=\ell(w_1)+\ell(w_2)$, then the multiplication $\overline{Bw_1B}\twtimes{B}\overline{Bw_2B}/B\to\overline{Bw_1w_2B/B}$ is birational. Therefore $\omega\IC_{w_1w_2}$ is a direct summand of $\omega\calF$ with multiplicity one. By the above discussion, $\IC_{w_1w_2}$ is also a direct summand of $\calF$ with multiplicity one.
\end{proof}

\begin{remark}
The new content of this proposition is the semisimplicity of the $\Frob$-action on $\IC_{w_1}\conv{B}\IC_{w_2}$, which does not seem to be known before.
\end{remark}

For $\Theta\subset\Sigma$ of finite type, let $\calC_\Theta$ be the constant sheaf on $\LFl_{\leq w_\Theta}$. Then $\calC_\Theta\cong\IC_{w_\Theta}[-\lTh](-\lTh/2)$.

\begin{lemma}\label{l:convic}
We have a natural isomorphism of functors
\begin{equation*}
\pi^{\Theta,*}\pi^\Theta_*(-)\cong\calC_\Theta\conv{B}(-).
\end{equation*}
\end{lemma}
\begin{proof}
We have a Cartesian diagram
\begin{equation*}
\xymatrix{N_\Theta\backslash L_\Theta\twtimes{N_\Theta}P^u_\Theta\backslash G\ar[r]^(.6){m}\ar[d]^{p_2} & B\backslash G\ar[d]^{\pi^\Theta}\\
B\backslash G\ar[r]^{\pi^\Theta} & P_\Theta\backslash G}
\end{equation*}
where $m,p_2$ are as in the convolution diagram (\ref{d:convI}). By proper base change we get for any $\calF\in\scE$,
\begin{equation*}
\pi^{\Theta,*}\pi^\Theta_*\calF\cong m_*p_2^*\calF\cong m_*(\calC_\Theta\stackrel{N_\Theta}{\boxtimes}\calF)\cong\calC_\Theta\conv{B}\calF.
\end{equation*}
\end{proof}

\begin{remark}\label{rm:Ccoalg}
By the adjunction \eqref{eq:paradj}, the functor $\pi^{\Theta,*}\pi^\Theta_*$ has a comonad structure. By Lemma \ref{l:convic}, the object $\calC_\Theta$ is hence a coalgebra object in the monoidal category $\scE$. In other words, there are comultiplication map $\mu:\calC_\Theta\to\calC_\Theta\conv{B}\calC_\Theta$ and counit map $\epsilon:\calC_\Theta\to\delta$ satisfying obvious associativity and compatibility conditions.
\end{remark}

\begin{lemma}\label{l:pidecomp}
For $w\in W$, the complex $\pi^\Theta_*\IC_w$ is a direct sum of $\IC_{\barv}[n](n/2)$ for $n\equiv\ell(w)-\ell(\barv)(\textup{mod }2)$. In particular, for $w\in[\coW]$, $\IC_{\barw}$ is a direct summand of $\pi^\Theta_*\IC_w$ with multiplicity one.
\end{lemma}
\begin{proof}
By the Decomposition Theorem, $\omega\pi^\Theta_*\IC_w$ is a direct sum of $\omega\IC_{\barv}[n]$. By adjunction and Lemma \ref{l:convic},
\begin{eqnarray*}
\End_{\scE_\Theta}(\pi^\Theta_*\IC_w)&\cong&\Hom_{\scE_\Theta}(\pi^{\Theta,*}\pi^\Theta_*\IC_w,\IC_w)\\
&\cong&\Hom_{\scE}(\calC_\Theta\conv{B}\IC_w,\IC_w)
\end{eqnarray*}
But the latter is a direct summand of $\Hom_{\dS\otimes\dS}(\HH_{w_\Theta}[-\lTh](-\lTh/2)\otimes_{\dS}\HH_w,\HH_w)$ by Proposition \ref{p:hff}. By Lemma \ref{l:Hss}, $\Frob$-acts semisimply on $\Hom_{\dS\otimes\dS}(\HH_{w_\Theta}[-\lTh](-\lTh/2)\otimes_{\dS}\HH_w,\HH_w)$, hence on $\End_{\scE_\Theta}(\pi^\Theta_*\IC_w)$. By Corollary \ref{c:Cdecomp}, $\pi^\Theta_*\IC_w$ is a direct sum of $\IC_{\barv}\otimes M_{\barv}$ for some complexes $M_{\barv}$ of semisimple $\Frob$-modules. Then, $\pi^{\Theta,*}\pi^\Theta_*\IC_w=\calC_\Theta\conv{B}\IC_w$ is a direct sum of $\pi^{\Theta,*}\IC_{\barv}\otimes M_{\barv}\cong\IC_{v}[-\ell_\Theta](-\ell_\Theta/2)\otimes M_{\barv}$ where $v\in\{\coW\}$ lifting $\barv$. Applying $\HH$ to this decomposition and using Lemma \ref{l:Hss} again, we conclude that each $M_{\barv}$ is a direct sum of $\Ql[n](n/2)$ for $n\equiv\ell(w)-\ell(\barv)(\textup{mod }2)$ .

If $w\in[\coW]$, then $\Fl_{\leq w}\to\PFl_{\leq\barw}$ is birational, therefore $\omega\IC_{\barw}$ is a direct summand of $\omega\pi^\Theta_*\IC_w$ of multiplicity one. By the above discussion, $\IC_{\barw}$ is a direct summand of $\pi^\Theta_*\IC_w$ with multiplicity one.
\end{proof}


\subsection{Proof of Proposition \ref{p:hff}}\label{ss:proofH}
This section is devoted to the proof of Proposition \ref{p:hff}. Let $A_{\leq w}$ be the equivariant cohomology ring $\upH^*(B\backslash\Fl_{\leq w})$. We first show

\begin{lemma}\label{l:hff} Let $\calF_1,\calF_2\in\scV_{\leq w}$, then there is an isomorphism of $\Frob$-modules
\begin{equation}\label{eq:hff}
\Ext^\bullet_{\scE}(\calF_1,\calF_2)\isom\Hom_{A_{\leq w}}(\HH(\calF_1),\HH(\calF_2))
\end{equation}
\end{lemma}
\begin{proof}
The proof is essentially borrowed from \cite{Ginz}. Since Ginzburg's proof in {\em loc.cit.} was carried out for varieties over $\CC$ using mixed Hodge modules, and we are working with mixed complexes on varieties over $\FF_q$, we decide to include a self-contained proof here. We use induction on the set $\{v\in W|v\leq w\}$ (for this we extend the partial ordering to a total ordering) to show that
\begin{equation*}
\Ext^\bullet_{\scE}(i^*_{\leq v}\calF_1,i^!_{\leq v}\calF_2)\isom\Hom_{A_{\leq v}}(\HH(i_{\leq v,*}i^*_{\leq v}\calF_1),\HH(i_{\leq v,*}i^!_{\leq v}\calF_2)).
\end{equation*}

For $v=e$ this follows from the equivalence $\scE_{\leq e}\cong D^b_{m}(\BB H)\cong D_{\perf}(\dS,\Frob)$ established in Corollary \ref{c:ba}. Suppose this is proved for all elements $v'<v$. Let
\begin{equation*}
Z:=B\backslash\Fl_{<v}\stackrel{i}{\hookrightarrow}X:=B\backslash\Fl_{\leq v}\stackrel{j}{\hookleftarrow}U:=B\backslash\Fl_{v}
\end{equation*}
be the inclusions. Now let
\begin{equation}
\calK_{1}=i^{*}_{\leq v}\calF_{1}; \hspace{.5cm}\calK_{2}=i^{!}_{\leq v}\calF_{2}.
\end{equation}
Note that $\calK_{1}$ is now only $*$-pure and $\calK_{2}$ is only $!$-pure.

\begin{lemma}\label{l:noconnecting}
For $\calK=\calK_1$ or $\calK_2$, we have exact sequences
\begin{eqnarray}\label{eq:hc}
0\to\HH(j_!j^*\calK)\to\HH(\calK)\to\HH(i_*i^*\calK)\to0;\\
\label{eq:h}
0\to\HH(i_*i^!\calK)\to\HH(\calK)\to\HH(j_*j^*\calK)\to0.
\end{eqnarray}
\end{lemma}
\begin{proof}
The exactness of \eqref{eq:hc} for $\calK_{1}$ and the exactness of \eqref{eq:h} for $\calK_{2}$ easily follows from the same argument as Lemma \ref{l:purity}. We prove the exactness of \eqref{eq:h} for $\calK_{1}$, and the exactness of \eqref{eq:hc} for $\calK_{2}$ follows by duality.

Now let $\calK=\calK_{1}=i^{*}_{\leq v}\calF_{1}$. It suffices to show that the restriction map $j^{*}:\upH^{k}(X,\calK)\to \upH^{k}(U,\calK)$ is surjective for all $k$. Let $v_{H}$ denotes the inclusion of the stack $H\backslash\{v\}$ into $H\backslash\Fl_{G}$. We can factor the restriction of $\upH^{k}(B\backslash\Fl_{G},\calF_{1})\cong\upH^{k}(H\backslash\Fl_{G},\calF_{1})$ to its stalk cohomology $\upH^{k}v^{*}_{H}\calF_{1}$ in two ways:
\begin{equation}
\xymatrix{\upH^{k}(B\backslash\Fl_{G},\calF_{1})\ar[r]\ar@{=}[d] & \upH^{k}(X,\calK)\ar[r]^{j^{*}} & \upH^{k}(U,\calK)\ar[r]^{v^{*}_{H}} & \upH^{k}v^{*}_{H}\calK\ar@{=}[d]\\
\upH^{k}(H\backslash\Fl_{G},\calF_{1})\ar[rr]^{\alpha^{*}} & & \upH^{k}(H\backslash vC,\calF_{1})\ar[r]^{v^{*}} & \upH^{k}v^{*}_{H}\calF_{1}}
\end{equation}
Here $\alpha:H\backslash vC\subset H\backslash \Fl_{G}$ is an open substack and $vC$ contracts to the point $v$ under the left action of some $\GG_{m}\subset H$ as in Lemma \ref{l:contract}. The two maps labelled by $v^{*}_{H}$ are isomorphisms by \cite[Corollary 1]{Spr} (because of the contracting $\GG_{m}$-action). Therefore in order to show that $j^{*}$ is surjective, it suffices to show that $\alpha^{*}$ is. Let $\beta:H\backslash(\Fl_{G}-vC)\hookrightarrow H\backslash\Fl_{G}$ be the closed embedding. Consider the long exact sequence associated with the triangle $\beta_{*}\beta^{!}\calF_{1}\to\calF_{1}\to\alpha_{*}\alpha^{*}\calF_{1}$, we get
\begin{equation}\label{vC}
\cdots\to \upH^{k}(H\backslash\Fl_{G},\calF_{1})\to \upH^{k}(H\backslash vC,\calF_{1})\to \upH^{k+1}(H\backslash (\Fl_{G}-vC),\beta^{!}\calF_{1})\to\cdots
\end{equation}
Since $\calF_{1}$ is pure of weight $0$, $\beta^{!}\calF_{1}$ is of weight $\geq0$ and hence $\upH^{k+1}(H\backslash(\Fl_{G}-vC),\beta^{!}\calF_{1})$ has weight $\geq k+1$ because $\upH^{*}(-)$ does not decrease weights. One the other hand, since $vC$ contracts to $v$, $\upH^{k}(H\backslash vC,\calF_{1})=v^{*}_{H}\calF_{1}$ has weight 0, therefore the connecting homomorphism in \eqref{vC} is zero. This shows $\alpha^{*}$ is surjective, hence $j^{*}$ is also surjective.
\end{proof}

We continue the proof of Lemma \ref{l:hff}. We have the following commutative diagram from the functoriality of $\HH$
\begin{equation}\label{threebytwo}
\xymatrix{\Ext^\bullet_Z(i^*\calK_1,i^!\calK_2)\ar[d]\ar[r]^(.4){a} & \Hom_{A_{<v}}(\HH(i_*i^*\calK_1),\HH(i_*i^!\calK_2))\ar[d]\\ \Ext^\bullet_X(\calK_1,\calK_2)\ar[d]\ar[r]^(.4){b} & \Hom_{A_{\leq v}}(\HH(\calK_1),\HH(\calK_2))\ar[d]\\
\Ext^\bullet_U(j^*\calK_1,j^*\calK_2)\ar[r]^(.4){c} & \Hom_{\upH^*(U)}(\HH(j_*j^*\calK_1),\HH(j_*j^*\calK_2))}
\end{equation}
Now $a$ is an isomorphism by inductive hypothesis; $c$ is an isomorphism because $\scE_v=D^b_{m}(U)\cong D_{\perf}(\dS,\Frob)$ by Corollary \ref{c:ba}. The left side sequence is exact by Lemma \ref{l:purity}(2) (for weight reasons). The right side sequence is exact on the top by the exact sequences (\ref{eq:hc}) and (\ref{eq:h}). We claim that the right side sequence is also exact in the middle. Admitting this fact, then $b$ is also an isomorphism and the induction is complete.

We have an exact sequence
\begin{equation*}
0\to \upH^*_c(U)\to A_{\leq v}\to A_{<v}\to 0.
\end{equation*}
Now $\upH^*_c(U)$ is a free $\dS$-module of rank one. Choose a generator $[U]\in\upH^{2\ell(v)}_{c}(U)$(corresponding to a lifting of a fundamental class into equivariant cohomology). By Lemma \ref{l:noconnecting}, we see that the action of $[U]$ on $\HH(\calK)$ ($\calK=\calK_{1}$ or $\calK_{2}$) factors as:
\begin{equation}\label{eq:u}
\xymatrix{\HH(\calK)\ar@{->>}[r] & \HH(j_*j^*\calK)\ar[d]^{u}_{\wr}\\
& \HH(j_!j^*\calK)[2\ell(v)](\ell(v))\ar@{^{(}->}[r] & \HH(\calK)[2\ell(v)](\ell(v))}
\end{equation}
where $u$ is an isomorphism. Now we show that the right side sequence in \eqref{threebytwo} is exact in the middle. If $\phi:\HH(\calK_1)\to\HH(\calK_2)$ is an $A_{\leq v}$-linear homomorphism which induces the zero map $\HH(j_*j^*\calK_1)\to\HH(j_*j^*\calK_2)$, then the image of $\phi$ lies in $\HH(i_*i^!\calK_2)$. Moreover, $[U]\circ\phi=0$ because $[U]$ factors through $\HH(j_*j^*\calK_1)$. Therefore $\phi\circ[U]=0$ (because $[U]\in A_{\leq v}$ commutes with $\phi$), hence $\phi$ is zero on the image of $[U]$, which is $\HH(j_!j^*\calK_1)$. Therefore $\phi$ comes from an $A_{<v}$-linear homomorphism $\HH(i_*i^*\calK_1)\to\HH(i_*i^!\calK_2)$. This completes the proof of the claim.
\end{proof}

Now we show that Lemma \ref{l:hff} implies Proposition \ref{p:hff}. We use the following simple observation
\begin{lemma}\label{l:samehom}
Let $S$ be a ring and $B\to C$ be a homomorphism of $S$-algebras that induces a surjection after base change to the ring of total fractions 	$\Frac(S)$. Let $M_1,M_2$ be two $C$-modules with $M_2$ torsion-free over $S$. Then the natural homomorphism
\begin{equation*}
\Hom_C(M_1,M_2)\to \Hom_B(M_1,M_2)
\end{equation*}
is an isomorphism.
\end{lemma}

We want to apply this Lemma to the situation $B=\dS\otimes\dS, C=A_{\leq w}$ and $S$ the right copy of $\dS$ in $\dS\otimes\dS$. For this we need
\begin{lemma}\label{l:eqloc}
The homomorphism of $(\dS\otimes\dS,\Frob)$-modules given by restrictions
\begin{equation*}
A_{\leq w}\to\prod_{v\leq w}\upH^*(B\backslash\Fl_{w})
\end{equation*}
is an isomorphism after tensoring by $\Frac(\dS)$ over the right $\dS$-module structures.
\end{lemma}
\begin{proof}
We do induction on $w$. We have a commutative diagram
\begin{equation*}
\xymatrix{\upH^*_c(B\backslash\Fl_{w})\ar[r]\ar[d]^{a} & A_{\leq w}\ar[r]\ar[d]^{b} & A_{<w}\ar[d]^{c}\\
\upH^*(B\backslash\Fl_{w})\ar[r] & \prod_{v\leq w}\HH(\nabla_v)\ar[r] & \prod_{v<w}\HH(\nabla_v)}
\end{equation*}
where $a$ is the ``forgetting the support'' map. To show that $b\otimes_{\dS}\Frac(\dS)$ is an isomorphism, it suffices to show $a\otimes_{\dS}\Frac(\dS)$ and $c\otimes_{\dS}\Frac(\dS)$ are. For $c$ we can use inductive hypothesis. The map $a$ factors as
\begin{equation*}
\upH^*_c(B\backslash\Fl_{w})\xrightarrow{a_1}\upH^*(B\backslash\Fl_{\leq w})\xrightarrow{a_2} \upH^*(B\backslash\Fl_{w}).
\end{equation*}
where $a_1$ is ``forgetting the support'' and $a_2$ is the restriction map. The cones of $a_1$ and $a_2$ are successive extensions of shifts and twists of $\HH(\Delta_v)$ and $\HH(\nabla_v)$ for $v<w$. As $\dS\otimes\dS$-modules, the supports of the cones of $a_1$ and $a_2$ are contained in the union of $\Gamma(v)$ for $v<w$ by Lemma \ref{l:hst}. Since the source and target of $a$ are supported on $\Gamma(w)$, $a\otimes_{\dS}\Frac(\dS)$ is the same thing as the localization of $a$ at the generic point of $\Gamma(w)$, where the cones of $a_1$ and $a_2$ become zero (because $\Gamma(w)\cap\Gamma(v)$ is a proper subscheme of $\Gamma(w)$). Therefore $a\otimes_{\dS}\Frac(\dS)$ is an isomorphism. The proof is complete.
\end{proof}

Consider the composition
\begin{equation*}
\dS\otimes\dS\to A_{\leq w}\to\prod_{v\leq w}\upH^*(\quot{H}{HvH}{H})\cong\prod_{v\leq w}\calO(\Gamma(v)).
\end{equation*}
After tensoring these maps by $\Frac(\dS)$ over the right copy of $\dS$, we get
\begin{equation*}
\dS\otimes\Frac(\dS)\to A_{\leq w}\otimes_{\dS}\Frac(\dS)\isom\prod_{v\leq w}\calO(\Gamma(v))\otimes_{\dS}\Frac(\dS)
\end{equation*}
which is obviously surjective (on the level of spectra, this corresponds to the closed embedding of the generic points of the graphs $\Gamma(v)$ into $V_H\otimes_k\Frac(\dS)$). Also notice that $\HH(\calF_2)$ is free (hence torsion-free) over either copy of $\dS$ by Lemma \ref{l:purity}(1). Therefore we can apply Lemma \ref{l:samehom} to conclude
\begin{eqnarray*}
\Ext^\bullet_{\scE}(\calF_1,\calF_2)&\cong&\Hom_{A_{\leq w}}(\HH(\calF_1),\HH(\calF_2))\\
&\cong&\Hom_{\dS\otimes\dS}(\HH(\calF_1),\HH(\calF_2))
\end{eqnarray*}
as graded $\Frob$-modules.


\section{Monodromic categories}\label{s:mon}
In this section, we define and study the category of $U$-equivariant and $H$-monodromic complexes on the enhanced flag variety $\tilFl=G/U$ of the Kac-Moody group $G$, as well as its Whittaker version. We will study averaging functors relating these categories and the convolution product on the monodromic category. We will give emphasis to the behavior of (free-monodromic) tilting objects under these operations. We have tried to arrange the materials in parallel with that of \S\ref{s:eq}, with the exception of the functor $\VV$ (the counterpart of $\HH$), whose definition requires extra work when $\tilFl$ is infinite dimensional.

This section relies on the foundational material on the completed monodromic categories in Appendix \ref{a:compmono}. We suggest reading \S\ref{as:mono} before getting into this section, leaving however the rest of Appendix \ref{a:compmono} as references. 

\subsection{The monodromic category}\label{ss:monocat}

Recall $\Fl=G/B$ is the flag ind-variety for $G$ and $\tilFl=G/U$ is the enhanced flag ind-variety.
Consider the right $H$-torsor $\pi:\tilFl\to\Fl$. Let $\scD_{\leq w}=D^b_{m}(U\backslash\Fl_{\leq w})$ be the derived category of $U$-equivariant mixed complexes on $\Fl_{\leq w}$. It is easy to check that, as a full subcategory of $D^b_{m}(\Fl_{\leq w})$, $\scD_{\leq w}$ satisfies the assumptions in Appendix \ref{ss:stra}, so that we can define the monodromic categories
\begin{equation*}
\scM_{\leq w}:=D^b_{m}(U\backslash\qw{\tilFl_{\leq w}}{H}).
\end{equation*}
and its completion $\hsM_{\leq w}$ following the construction in Appendix \ref{ss:compmono} and \ref{ss:stra}. Let $\scM$ (resp. $\hsM$, $\scD$) be the inductive 2-limit of $\scM_{\leq w}$ (resp. $\hsM_{\leq w}$, $\scD_{\leq w}$).

The triangulated category $\scM$ carries the perverse t-structure with heart $\scP$. By Remark \ref{l:prot}, this t-structure extends to $\hsM$, and we denote its heart by $\hsP$. Recall from \S \ref{as:mono} that $\pidag=\pi^![-r]:\scD\to\hsM$ is t-exact and its left adjoint $\piddag=\pi_![r]:\hsM\to\scD$ is right t-exact.

Let $\hsP_{\leq w}=\hsP\cap\hsM_{\leq w}$. The irreducible objects in $\hsP$ are twists of $\pidag\IC_w$. When there is no confusion, we will also write $\IC_w$ for $\pidag\IC_w$. The basic \fm perverse local system $\tcL_w$ on $\tilFl_{w}$ (see Def.\ref{def:fm}) is normalized so that $\piddag\tcL_w=\Ql[\ell(w)](\ell(w)/2)$ on $\Fl_{w}$. For $w=e$, we also write $\tdel$ for $\tcL_e$, a \fm perverse local system on $H=\tilFl_e$. In comparison with the \fm local system $\tcL$ on $A=H$ in Example \ref{ex:fm}, we have $\tdel=\tcL[r](r)$. The \fm standard and costandard sheaves are denoted by $\tDel_w$ and $\tnab_w$.

The group $H\times H$ acts on the stack $\quot{U}{G}{U}$ via $(h_1,h_2)\cdot x=h_1xh_2$. Note that this action differs from the action defined in \S\ref{ss:eqpar} by an inversion of the right copy of $H$. We will see later (in the proofs of Lemma \ref{l:vst} and Proposition \ref{p:vcomp}) that this modification makes the equivariant and monodromic categories match perfectly. It is easy to see that the full subcategory $\scM\subset D^b_{c}(\quot{U}{G}{U})$ consists exactly of $H\times H$-monodromic objects (because the generating objects $\IC_w$ are). Let $S=\Sym(V_H)$ and $\hatS=\varprojlim S/(V_H^n)$. The left and right actions of $H$ give logarithmic monodromy operators by the algebra $S\otimes S$ (see discussions in Appendix \ref{as:mono}), so that $\omega\hsM$ is naturally an $\hatS\otimes\hatS$-linear category.

\begin{remark} We have defined $\hsM$ as the completion with respect to the monodromy of the right copy of $H$. We could have defined another completion of $\scM$ using the left copy of $H$. It turns out that these two completions are canonically equivalent. Therefore we sometimes prefer to use the more symmetric notation $\hatD^b_{m}(\wqw{B}{G}{B})$ to denote $\hsM$.
\end{remark}

\subsection{The Whittaker category}\label{ss:Whit}
Let $\Theta\subset\Sigma$ be a subset of finite type. For each simple reflection $s\in W$, recall that $U^-_s$ denotes the 1-dimensional unipotent subgroup of $G$ whose Lie algebra is the root space corresponding to $-\alpha_s$. Then we have a canonical isomorphism:
\begin{equation}\label{eq:Nab}
\prod_{s\in\Theta}U^-_s\isom U^-_\Theta/[U^-_\Theta,U^-_\Theta].
\end{equation}
Fix an isomorphism $U^-_s\isom\GG_a$ for each $s\in\Sigma$. Let
\begin{equation*}
\chi:\prod_{s\in\Theta}U^-_s\isom\prod_{s\in\Theta}\GG_a\xrightarrow{+}\GG_a
\end{equation*}
be the sum of the isomorphisms $U^-_s\isom\GG_a$ . We can view the map $\chi$ as an additive character of $U^-_\Theta$, or even of the pro-unipotent group $U^\Theta U^-_\Theta$.

Fix a non-trivial additive character $\psi:k\to\Ql^{\times}$. This determines an Artin-Schreier local system $\AS_\psi$ on $\GG_a$ and hence the local system $\chi^*\AS_\psi$ on $U^-_\Theta$ or $U^\Theta U^-_\Theta$. We want to define the category $\scD_\Theta$ of complexes on $\Fl$ which are $(U^\Theta U^-_\Theta,\chi)$-equivariant, i.e., equivariant under $U^\Theta U^-_{\Theta}$ against the character sheaf $\chi^*\AS_\psi$.

We first recall some definitions in the finite-dimensional setting. Suppose $V$ is a group scheme with a one-dimensional local system $\calA$ on it which is a character sheaf. This means there is an isomorphism $m^*\calA\cong\calA\boxtimes\calA$ on $V\times V$ with which is compatible with the identity section and the associativity of $V$ in the obvious sense. Let $X$ be a scheme with a $V$-action $a:V\times X\to X$. A perverse $\calF$ on $X$ is said to be $(V,\calA)$-equivariant if it is equipped with an isomorphism $a^*\calF\cong\calA\boxtimes\calF$ with obvious compatibility conditions. When $V$ is connected, the category of $(V,\calA)$-equivariant perverse sheaves is a full subcategory of perverse sheaves on $X$.

In our situation, each orbit $\Fl^\Theta_{w}$ of $U^\Theta U^-_\Theta$ is finite dimensional whose closure $\Fl^\Theta_{\leq w}$ is a projective variety contained in some Schubert variety of $\Fl$. The closure relation among the orbits defines a partial order $\Tleq$ on $W$: $w_{1}\Tleq w_{2}$ if $\Fl^{\Theta}_{w_{1}}$ is in the closure of $\Fl^{\Theta}_{w_{2}}$. By Lemma \ref{l:fact}, we can choose $J_w\lhd U^{\Theta}$ of finite codimension which acts trivially on $\Fl^\Theta_{\leq w}$. We can define $\scQ_{\Theta,\Tleq w}$ to be the category of $(J_w\backslash U^\Theta\cdot U^-_\Theta,\chi)$-equivariant perverse sheaves on $\Fl^\Theta_{\leq w}$. This notion is obviously independent of $J_w$. Let $\scQ_{\Theta}$ be the inductive 2-limit of $\{\scQ_{\Theta,\Tleq w}\}$. Let $\scD_{\Theta}$ be the triangulated subcategory of $D^b_{m}(\Fl)$ generated by $\scQ_{\Theta}$.

Recall $[\coW]$ is set of minimal length representatives in the left $W_{\Theta}$-cosets of $W$.

\begin{lemma}\label{l:mw} The subquotient categories $\scD_{\Theta,w}=\scD_{\Theta,\Tleq w}/\scD_{\Theta,\stackrel{\Theta}{<}w}$ admit t-exact equivalences
\begin{equation*}
\scD_{\Theta,w}\isom\begin{cases}D^b(\Frob) & w\in[\coW]\\0 & w\notin[\coW]\end{cases}
\end{equation*}
\end{lemma}
\begin{proof}
Let $\calF\in\scD_{\Theta,w}$. For $w\notin[\coW]$, we can find some simple reflection $s\in\Theta$ such that $\ell(sw)<\ell(w)$. Then the stabilizer of the point $wB/B$ under $U^-_\Theta$ contains $U^-_s$, on which $\chi$ is nontrivial. Therefore the stalk cohomology of $\calF$ at $wB/B$ is zero, hence $\calF$ has to be zero along $\Fl^\Theta_w$ since its cohomology sheaves are locally constant along $\Fl^\Theta_w$. This implies $\calF=0\in\scD_{\Theta,w}$.

If $w\in[\coW]$, then the action of $U^-_\Theta$ on $\Fl^\Theta_{w}$ is free with quotient isomorphic to an affine space $\AA^{\ell(w)}$. We may choose a section of the quotient map $\Fl^\Theta_w\to\AA^{\ell(w)}$ and identify $\Fl^\Theta_w$ with $U^-_\Theta\times\AA^{\ell(w)}$. Any $(U^-_\Theta,\chi)$-equivariant perverse sheaf $\calF$ on $\Fl^\Theta_w$ has the form $\chi^*\AS_\psi\boxtimes\calF[\ell_\Theta]$ for some perverse sheaf $\calF$ on $\AA^{\ell(w)}$, and vice versa. The equivariance under $U^\Theta$ forces $\calF$ to be constant. Therefore $\scD_{\Theta,w}$ is equivalent to the full triangulated subcategory of $D^b_{m}(U^-_\Theta\times\AA^{\ell(w)})$ generated by twists of the local system $\chi^*\AS_\psi\boxtimes\Ql$. Hence $\scD_{\Theta,w}\cong D^b(\Frob)$.
\end{proof}

The above lemma implies that $\scD_{\Theta}$ satisfies Assumption S in Appendix \ref{ss:stra}, therefore we can define the Whittaker-monodromic category $\scM_\Theta$ of $(U^\Theta U^-_\Theta,\chi)$-equivariant and right $H$-monodromic complexes on $\tilFl$. Note that $\scM_{\Theta}$ in fact depends on the character $\chi$ (which in turn depends on the choice of the isomorphism $U^{-}_{s}\isom\GG_{a}$). However, to alleviate notation, we omit $\chi$ systematically. 

We can also define the completion $\hsM_\Theta$ of $\scM_{\Theta}$. According to Lemma \ref{l:mw}, we can index the subquotient categories of $\scD_\Theta,\scM_\Theta$ by elements or subsets of $\coW$, for example $\scM_{\Theta,\leq\barw}$ for $\barw\in\coW$. The categories $\scM_\Theta$ and $\hsM_\Theta$ carry the perverse t-structure with hearts $\scP_\Theta$ and $\hsP_\Theta$ (see Lemma \ref{l:prot}).

For each $\barw\in\coW$, we have a $(U^\Theta U^-_\Theta,\chi)$-equivariant perverse sheaf $\calL_{\barw,\chi}$ of rank one and weight 0 in $\scQ_{\Theta,w}$ ($w\in[\coW]$ representing $\barw$). This is the sheaf $\chi^*\AS_\psi\boxtimes\Ql[\ell(\barw)+\ell_\Theta](\frac{\ell(\barw)+\ell_\Theta}{2})$that appear in the proof of Lemma \ref{l:mw}. We also have the basic \fm $(U^\Theta U^-_\Theta,\chi)$-equivariant perverse local system $\tcL_{\barw,\chi}\in\hsP_{\Theta,w}$ (cf. Def.\ref{def:fm}), which we normalize so that $\piddag\tcL_{\barw,\chi}\cong\calL_{\barw,\chi}$. We also have the standard and costandard sheaves $\Delta_{\barw,\chi}=i_{\barw,!}\calL_{\barw,\chi}$ and $\tnab_{\barw,\chi}=i_{\barw,*}\calL_{\barw,\chi}$ in $\scQ_\Theta$. We have standard and costandard \fm sheaves $\tDel_{\barw,\chi}$ and $\tnab_{\barw,\chi}$ in $\hsP_{\Theta}$.

Since $\overline{e}$ is the minimal element in $\coW$, we immediately conclude with the cleanness of the local system $\tcL_{e,\chi}$:
\begin{cor}\label{c:clean}
The natural maps
\begin{equation*}
\Delta_{\overline{e},\chi}\to\nabla_{\overline{e},\chi};\hspace{1cm}\tDel_{\overline{e},\chi}\to\tnab_{\overline{e},\chi}
\end{equation*}
are isomorphisms. We denote these objects by $\delta^\Theta_\chi\in\scQ_\Theta$ and $\tdel^\Theta_\chi\in\hsP_\Theta$ respectively.
\end{cor}


\subsection{Convolution}\label{ss:Mconv}
Consider the convolution diagram
\begin{equation}\label{d:convIu}
\xymatrix{ & G\twtimes{U}\tilFl\ar[dl]_{p_1}\ar[dr]^{p_2}\ar[rr]^{\tilm} & & \tilFl\\
\tilFl & & U\backslash\tilFl}
\end{equation}
where $p_1,p_2$ are projections to the left and right factors and $\tilm$ is induced by the multiplication map of $G$. The convolution diagram induces a convolution product on $\scM$:
\begin{eqnarray*}
\conv{U}:\scM\times\scM&\to&\scM\\
(\calF_1,\calF_2)&\mapsto&\tilm_{!}(\calF_1\stackrel{U}{\boxtimes}\calF_2)[r].
\end{eqnarray*}
Note that $\calF_1\boxtimes\calF_2$ is an $U$-equivariant complex on $G\times\tilFl$ with respect to the action $i\cdot(g,x)=(gi^{-1},ix)$ and hence descends to a complex $\calF_1\stackrel{U}{\boxtimes}\calF_2$ on $G\twtimes{U}\tilFl$. There is an obvious associativity constraint which makes $\conv{U}$ into a monoidal structure on $\scM$. More generally, the convolution gives a right action of the monoidal category $\scM$ on $\scM_\Theta$ given by the same formula.

\begin{lemma}
The monoidal structure $\conv{U}$ naturally extends to the completed category $\hsM$.
\end{lemma}
\begin{proof}
We can decompose $\conv{U}$ into two steps: the first step is
\begin{eqnarray*}
\phi(-,-):\scM\times\scM&\to& D^b_m(U\backslash\qw{G\twtimes{U}\tilFl}{H_{\midd}})\to D^b_m(U\backslash G\twtimes{B}\qw{\tilFl}{H}).\\
(\calF,\calF')&\mapsto&\calF\stackrel{U}{\boxtimes}\calF'\mapsto\pi_{\midd,!}(\calF\stackrel{U}{\boxtimes}\calF')[r].
\end{eqnarray*}
where $H_{\midd}$ means the torus acts on $G\twtimes{U}\tilFl$ by $(g_1,g_2)\cdot h=(g_1h,h^{-1}g_2)$, and $\pi_{\midd}$ denotes the quotient map by $H_{\midd}$.

Fix $\calF'\in\scM$. For any pro-object $\prolim\calF_n\in\hsM$, the pro-object $\prolim\phi(\calF_n,\calF')$ is in fact isomorphic to an object in $D^b_m(G\twtimes{B}\tilFl)$. In fact, since $\calF'$ is a successive extension of $\pi^{!}\calF''$ for $\calF''\in\scD$, it suffice to check with $\calF'=\pi^{!}\calF''$. In this case one easily sees
\begin{equation}\label{piconvU}
\phi(\calF_n,\pi^{!}\calF'')=\pi^{!}\left((\pi_{!}\calF_n)\stackrel{B}{\boxtimes}\calF''\right).
\end{equation}
Therefore $\prolim\phi(\calF_n,\calF')=\phi(\prolim(\pi_{!}\calF_n),\calF'')$ is essentially constant because $\prolim(\pi_{!}\calF_n)$ is essentially constant (i.e., belongs to $\scD$). This shows that $\phi$ extends to
\begin{equation*}
\hatphi:\hsM\times\scM\to D^b_m(U\backslash G\twtimes{B}\qw{\tilFl}{H}).
\end{equation*}

Similarly, we may define
\begin{equation*}
\hatpsi:\hsM\times\scD\to D^b_m(U\backslash G\twtimes{B}\Fl)
\end{equation*}
so that the following diagram commutes
\begin{equation*}
\xymatrix{\hsM\times\scM\ar[r]^{\hatphi}\ar[d]^{\id\times\pi_!} & D^b_m(U\backslash G\twtimes{B}\qw{\tilFl}{H})\ar[d]^{\Pi_!}\\
\hsM\times\scD\ar[r]^{\hatpsi} & D^b_m(U\backslash G\twtimes{B}\Fl)}
\end{equation*}
where $\Pi:U\backslash G\twtimes{B}\tilFl\to U\backslash G\twtimes{B}\Fl$ is the projection. Proposition \ref{compfun} then implies that $\hatphi$ further extends to
\begin{equation*}
\hatphi:\hsM\times\hsM\to\hatD^b_m(U\backslash G\twtimes{B}\qw{\tilFl}{H})
\end{equation*}

The second step is given by the multiplication $m:G\twtimes{B}\tilFl\to\tilFl$
\begin{equation*}
m_!:D^b_m(U\backslash G\twtimes{B}\qw{\tilFl}{H})\to\scM.
\end{equation*}
which, by Corollary \ref{c:functor}, extends to completed categories
\begin{equation*}
\hatm_!:\hatD^b_m(U\backslash G\twtimes{B}\qw{\tilFl}{H})\to\hsM.
\end{equation*}

Now for $\calF=\prolim\calF_m, \calF'=\prolim\calF'_n\in\hsM$, we define
\begin{eqnarray*}
\calF\conv{U}\calF'&=&\hatm_!\hatphi(\calF,\calF')=\plim{n}m_!(\hatphi(\calF,\calF_n'))\\
&=&\plim{n}\plim{m}m_!\phi(\calF_m,\calF_n)=\plim{n}\plim{m}\calF_m\conv{U}\calF_n'.
\end{eqnarray*}

We construct the associativity constraint for the extended $\conv{U}$. Let $\calF=\prolim\calF_j,\calF'=\prolim\calF'_m,\calF''=\prolim\calF''_n\in\hsM$. On one hand,
\begin{eqnarray*}
(\calF\conv{U}\calF')\conv{U}\calF''&=&\plim{n} m_!\hatphi(\calF\conv{U}\calF',\calF''_n)\\
&=&\plim{n} m_!\hatphi(\plim{m} m_!\hatphi(\calF,\calF'_m),\calF''_n)\\
&=&\plim{n}\plim{m} m_!\phi(m_!\hatphi(\calF,\calF'_m),\calF''_n)\\
&=&\plim{n}\plim{m}\plim{j} m_!\phi(m_!\phi(\calF_j,\calF'_m),\calF''_n)\\
&=&\plim{n}\plim{m}\plim{j}(\calF_j\conv{U}\calF'_m)\conv{U}\calF''_n.
\end{eqnarray*}
Here the order in which the $\prolim$ is taken is important. Similarly, one verifies
\begin{equation*}
\calF\conv{U}(\calF'\conv{U}\calF'')=\plim{n}\plim{m}\plim{j}\calF_j\conv{U}(\calF'_m\conv{U}\calF''_n).
\end{equation*}
Let $a(\calG,\calG',\calG''):(\calG\conv{U}\calG')\conv{U}\calG''\isom\calG\conv{U}(\calG'\conv{U}\calG'')$ be the associativity constraint in $(\scM,\conv{U})$, then we define the associativity constraint $\hata$ for $(\hsM,\conv{U})$ by
\begin{equation*}
\hata(\calF,\calF',\calF'')=\plim{n}\plim{m}\plim{j}a(\calF_j,\calF'_m,\calF''_n).
\end{equation*}
To check the pentagon relation for $\hata$, we only need to notice that the two ways of getting from $((\calF\conv{U}\calF')\conv{U}\calF'')\conv{U}\calF'''$ to $\calF\conv{U}(\calF'\conv{U}(\calF''\conv{U}\calF'''))$ is obtained by taking $\plim{n}\plim{m}\plim{j}\plim{i}$ of the two ways of getting from $((\calF_i\conv{U}\calF'_j)\conv{U}\calF''_m)\conv{U}\calF'''_n$ to $\calF_i\conv{U}(\calF'_j\conv{U}(\calF''_m\conv{U}\calF'''_n))$. This completes the proof.

\end{proof}

\begin{remark}\label{r:midS}
Recall the $H_{\midd}$-action on $G\twtimes{U}\tilFl$ defined in the proof of the above Lemma. The monodromy action of $V_{H_{\midd}}$ on $\calF_1\stackrel{U}{\boxtimes}\calF_2$ corresponds to the difference of the right $V_H$-action on $\calF_1$ and the left $V_H$-action on $\calF_2$. Since the multiplication map $\tilm$ factors as $\tilm=m\circ\pi_{\midd}$, the $V_\ell(H_{\midd})$ acts trivially on $\pi_{\midd}(\calF_1\stackrel{U}{\boxtimes}\calF_2)$, hence on $\calF_1\conv{U}\calF_2$. Therefore, the following two $\hatS$-actions on $\calF_1\conv{U}\calF_2$ are the same: one is the right $\hatS$-action on $\calF_1$; the other is the left $\hatS$-action on $\calF_2$. Here we are making use of the convention of the $H\times H$-action fixed in \S\ref{ss:monocat}.
\end{remark}

Similarly, the convolution action of $\scM$ on $\scM_\Theta$ extends to an action of the monoidal category $(\hsM,\conv{U})$ on $\hsM_\Theta$. Using similar convolution diagrams, we can define a right convolution $\conv{B}$ of $\scE$ on $\scD=D^b_{m}(U\backslash\Fl)$; we can also define a left convolution $\conv{U}$ of $\hsM$ on $\scD$.

\begin{lemma}\label{l:convst}
Suppose $w_1,w_2\in W$ and $\ell(w_1w_2)=\ell(w_1)+\ell(w_2)$, then
\begin{equation*}
\tDel_{w_1}\conv{U}\tDel_{w_2}\cong\tDel_{w_1w_2};\hspace{1cm}\tnab_{w_1}\conv{U}\tnab_{w_2}\cong\tnab_{w_1w_2}.
\end{equation*}
Moreover, the object $\tdel$ is the unit object in the monoidal category $\hsM$.
\end{lemma}
\begin{proof}
Here we only give the proof of $\tdel\conv{U}\calF\cong\calF$, the rest is either similar (and parallel to Lemma \ref{l:convIst}). The relevant convolution diagram becomes simply the left action map $a:H\times\tilFl\to\tilFl$. Therefore
\begin{equation*}
\tdel\conv{U}\calF\cong a_!(\tdel\boxtimes\calF)[r]=a_!(\tcL\boxtimes\calF)[2r](r).
\end{equation*}
By Lemma \ref{l:actionF}, we have $a_!(\tcL\boxtimes\calF)[2r](r)\cong\calF$.
\end{proof}


\begin{prop}\label{p:tiltconv}
For \fm tilting sheaves $\tcT_1,\tcT_2\in\hsP$, the convolution $\tcT_1\conv{U}\tcT_2$ is also a \fm tilting sheaf.
\end{prop}
\begin{proof}
By Lemma \ref{l:pitilt}, it is enough to check that $\calT:=\piddag(\tcT_1\conv{U}\tcT_2)\cong\tcT_1\conv{U}\piddag\tcT_2$ is a tilting sheaf on $\Fl$. Observe that
\begin{eqnarray}\label{eq:convdel}
\tDel_{w_1}\conv{U}\Delta_{w_2}\cong\Delta_{w_1}\conv{B}\Delta_{w_2}\in\langle\Delta_w(?)[\leq0]|w\in W\rangle\subset\scD\\
\label{eq:convnab}
\tnab_{w_1}\conv{U}\nabla_{w_2}\cong\nabla_{w_1}\conv{B}\nabla_{w_2}\in\langle\nabla_w(?)[\geq0]|w\in W\rangle\subset\scD.
\end{eqnarray}
In fact, to prove (\ref{eq:convdel}), we can write each $\Delta_w=\Delta_{s_1}\conv{B}\cdots\conv{B}\Delta_{s_m}$ by Lemma \ref{l:convIst} (for a reduced word expression $w=s_1\cdots s_m$) and we reduce to the computation of $\Delta_{s}\conv{B}\Delta_{s'}$ for two simple reflections $s,s'$. If $s\neq s'$, then $\Delta_s\conv{B}\Delta_{s'}\cong\Delta_{ss'}$. For $s=s'$, this follows by Lemma \ref{l:convSL2}. The proof of (\ref{eq:convnab}) is similar.

Therefore, since $\piddag\tcT_2$ admits a $\Delta$-flag and a $\nabla$-flag, the convolution $\calT=\tcT_1\conv{U}\piddag\tcT_2$ satisfies
\begin{equation}\label{eq:estT}
\omega\calT\in\langle\Delta_w[\leq0]|w\in W\rangle\cap\langle\nabla_w[\geq0]|w\in W\rangle.
\end{equation}
We show that the above condition already implies that $\calT$ is a tilting sheaf. In fact, we know that $\nabla_w$ is perverse (since $i_w$ is affine), hence $\calT\in\langle\nabla_w[\geq0]|w\in W\rangle\subset\scD^{\leq0}$, i.e., $i_w^*\calT\in\scD_w^{\leq0}$. On the other
hand, $\calT\in\langle\Delta_w[\leq0]|w\in W\rangle$ implies that $i_w^*\calT\in\scD_w^{\geq0}$. Hence $i^*_w\calT$ is perverse. Similarly, we can argue that $i_w^!\calT$ is also perverse. Therefore $\calT$ is a tilting sheaf, and $\tcT_1\conv{U}\tcT_2$ is a \fm tilting sheaf.
\end{proof}


\subsection{Averaging functors}\label{ss:av} In this section, we fix a subset $\Theta\subset\Sigma$ of finite type.
\subsubsection{Averaging along $U_\Theta$} Consider the left action:
\begin{equation*}
a^+:U_\Theta\times\tilFl\to\tilFl.
\end{equation*}
For $?=!$ or $*$, define the functors
\begin{eqnarray*}
\av^\Theta_?:D^b_{m}(\tilFl)&\to& D^b_{m}(U_\Theta\backslash\tilFl)\\
\calF&\mapsto& a^+_?(\Ql[\lTh](\lTh/2)\boxtimes\calF).
\end{eqnarray*}
The functor $\av^\Theta_?$ obviously preserves right $H$-monodromic subcategories. Moreover, since $U^{\Theta}$ is normal in $U$ with quotient $U_\Theta$, the functor $\av^\Theta_?$ also preserves left $U^{\Theta}$-equivariant structures. Therefore, we get a functor
\begin{equation*}
\scM_\Theta\xrightarrow{\Forg} D^b_{m}(U^\Theta\backslash\qw{\tilFl}{H})\xrightarrow{\av^\Theta_?}\scM
\end{equation*}
which passes to the completions (cf. Proposition \ref{compfun})
\begin{equation*}
\Av^\Theta_?:\hsM_\Theta\to\hsM.
\end{equation*}

\subsubsection{Averaging along $(U^-_\Theta,\chi)$} Similarly, consider the action:
\begin{equation}\label{eq:a-}
a^-:U^-_\Theta\times\tilFl\to\tilFl.
\end{equation}
For $?=!$ or $*$, define the functors
\begin{equation*}
\av^\Theta_{\chi,?}(\calF):=a^-_?(\chi^*\AS_\psi[\lTh](\lTh/2)\boxtimes\calF).
\end{equation*}
As in the case of $\av^\Theta_{?}$, the functor $\av^\Theta_{\chi,?}$ preserves right $H$-monodromicity and left $U^\Theta$-equivariance. Therefore we get a functor
\begin{equation*}
\scM\xrightarrow{\Forg}D^b_{m}(U^\Theta\backslash\qw{\tilFl}{H})\xrightarrow{\av^\Theta_{\chi,?}}\scM_{\Theta}.
\end{equation*}
which passes to the completions
\begin{equation*}
\Av^\Theta_{\chi,?}:\hsM\to\hsM_\Theta.
\end{equation*}

Using the convolution, we give an alternative description for $\Av^\Theta_{\chi,?}$.
\begin{lemma}\label{l:samedef}
We have a natural isomorphism
\begin{equation}\label{eq:samedef}
\Av^\Theta_{\chi,?}(-)\cong\tdel^\Theta_\chi\conv{U}(-).
\end{equation}
In particular, there is a natural isomorphism of functors $\Av^\Theta_{\chi,!}\isom\Av^\Theta_{\chi,*}$. From now on, we denote these functors by $\Av^\Theta_\chi$.
\end{lemma}
\begin{proof} The argument is essentially the same as \cite[Theorem 1.5(1),Theorem 2.2]{BBM}. We only need to exhibit such a natural isomorphism between the restriction of the functors to $\scM$. Let $j$ be the open immersion of the big Bruhat cell in the flag variety of $L_\Theta$:
\begin{equation*}
j:U^-_\Theta\hookrightarrow L_\Theta U/B=L_\Theta/L_\Theta\cap B,
\end{equation*}
and let
\begin{equation*}
\tilj:U^-_\Theta\times H\hookrightarrow L_\Theta U/U\cong L_\Theta/U_\Theta.
\end{equation*}
By Corollary \ref{c:clean}, we can view the \fm perverse local system $\tdel^\Theta_{\chi}$ as either $\tilj_!$ or $\tilj_*$ of the perverse local system $\chi^*\AS_\psi[\lTh](\lTh/2)\boxtimes\tdel$ on $U^-_\Theta\times H$.

Consider the diagram
\begin{equation*}
\xymatrix{U^-_\Theta\times H\times\tilFl\ar[d]^{\id\times a_H}\ar[r]^{\tilj\times\id}& L_\Theta U\twtimes{U}\tilFl\ar[d]^{q_H}\ar[dr]^{\tilm}\\
U^-_\Theta\times\tilFl\ar[r]^{j\times\id} & L_\Theta U\twtimes{B}\tilFl\ar[r]^{m} & \tilFl}
\end{equation*}
where $a_H:H\times\tilFl\to\tilFl$ is the left action map. For $?=!$ or $*$, we have
\begin{eqnarray*}
\notag\tdel_\chi\conv{U}\calF &=& \tilm_!(\tilj\times\id)_?(\chi^*\AS_\psi[\lTh](\lTh/2)\boxtimes\tdel\boxtimes\calF)[r]\\
&=&m_!(j\times\id)_?(\chi^*\AS_\psi[\lTh](\lTh/2)\boxtimes a_{H,!}(\tdel\boxtimes\calF))[r]
\end{eqnarray*}

By Lemma \ref{l:actionF}, we have $a_{H,!}(\tdel\boxtimes\calF)\cong\calF[-r]$ (note that $\tdel$ is normalized to be $\tcL[r](r)$ on $H$). Hence
\begin{eqnarray*}
\tdel_\chi\conv{U}\calF&\cong& m_!(j\times\id)_!(\chi^*\AS_\psi[\lTh](\lTh/2)\boxtimes\calF)\\
&=& a^-_!(\chi^*\AS_\psi[\lTh](\lTh/2)\boxtimes\calF)=\Av^\Theta_{\chi,!}(\calF)
\end{eqnarray*}
which proves the (\ref{eq:samedef}) for $?=!$.

To prove the case $?=*$, we note that by Corollary \ref{c:clean}, $j_!(\chi^*\AS_\psi)\cong j_*(\chi^*\AS_\psi)=\delta^\Theta_\chi$. Notice also that $m$ is proper (hence $m_!=m_*$), therefore
\begin{eqnarray*}
\tdel_\chi\conv{U}\calF&\cong& m_!(j\times\id)_!(\chi^*\AS_\psi[\lTh](\lTh/2)\boxtimes\calF)\\
&\cong& m_!(\delta^\Theta_\chi\stackrel{B}{\boxtimes}\calF)=m_*(\delta^\Theta_\chi\stackrel{B}{\boxtimes}\calF)\\
&\cong& m_*(j\times\id)_*(\chi^*\AS_\psi[\lTh](\lTh/2)\boxtimes\calF)\\
&=& a^-_*(\chi^*\AS_\psi[\lTh](\lTh/2)\boxtimes\calF)=\Av^\Theta_{\chi,*}(\calF).
\end{eqnarray*}
\end{proof}

\begin{cor}\label{c:avexact}
The functor $\Av^\Theta_\chi$ is t-exact with respect to the perverse t-structures on $\hsM$ and $\hsM_\Theta$.
\end{cor}
\begin{proof}
Since the action map $a^-$ in (\ref{eq:a-}) are affine, we conclude that $\Av^\Theta_{\chi,!}$ is right exact and $\Av^\Theta_{\chi,*}$ is left exact, by \cite[Th\'{e}or\`{e}me 4.1.1 and Corollaire 4.1.2]{BBD}. By Lemma \ref{l:samedef}, $\Av^\Theta_\chi$ is exact.
\end{proof}

\begin{lemma}\label{l:adj}
We have adjunctions
\begin{equation}\label{d:whadj}
\xymatrix{\hsM\ar[r] & \hsM_\Theta\ar@<1ex>[l]^{\Av^\Theta_*}\ar@<-2ex>[l]_{\Av^\Theta_!}^{\Av^\Theta_\chi}}
\end{equation}
\end{lemma}
\begin{proof}
By Proposition \ref{compfun}, it suffices to check the adjunctions for functors before completion. There we have the adjunctions
\begin{equation*}
\xymatrix{\scM\ar[r] & D^b_{m}(U^\Theta\backslash\qw{\tilFl}{H})\ar@<1ex>[l]^(.7){\av^\Theta_*}\ar@<-1ex>[l]_(.7){\av^\Theta_!}\ar@<1ex>[r]^(.7){\av^\Theta_{\chi,!}}\ar@<-1ex>[r]_(.7){\av^\Theta_{\chi,*}} & \scM_\Theta\ar[l]}
\end{equation*}
where the unlabeled functors are forgetful functors $\Forg$. The compositions give adjunction pairs $(\av^\Theta_!\circ\Forg,\av^\Theta_{\chi,*}\circ\Forg)$ and
$(\av^\Theta_{\chi,!}\circ\Forg,\av^\Theta_*\circ\Forg)$, i.e., $(\Av^\Theta_!,\Av^\Theta_\chi)$ and $(\Av^\Theta_\chi,\Av^\Theta_*)$.
\end{proof}

\begin{lemma}\label{l:killnonmin}
For $w\notin[\coW]$, we have $\Av^\Theta_\chi(\IC_w)=0$.
\end{lemma}
\begin{proof}
If $w\notin[\coW]$, then there exists a simple reflection $s\in\Theta$ such that $\ell(w)=\ell(sw)+1$. Therefore $\IC_w$ is $P_s$-equivariant with respect to the left action of $P_s$ on $\tilFl$. Let $\pi_s: L_\Theta/U_\Theta=L_\Theta U/U\to L_\Theta U/P_s=L_\Theta/L_\Theta\cap P_s$ be
the natural projection. Then by lemma \ref{l:samedef},
\begin{equation*}
\Av^{\Theta}_\chi(\IC_w)=\tdel^\Theta_\chi\conv{U}\IC_w=\pi_{s,!}(\tdel^\Theta_\chi)\conv{P_s}\IC_w
\end{equation*}
where the convolution $\conv{P_s}: D^b_{m}(G/P_s)\times D^b_{m}(P_s\backslash\tilFl)\to D^b_{m}(\tilFl)$ is defined in a similar way as $\conv{B}$. Now $\pi_{s,!}(\tdel^\Theta_\chi)\in D^b_{m}((U^-_\Theta,\chi)\backslash L_\Theta/P_s\cap L_\Theta)$ and we claim this category is zero. Just as in the proof of Lemma \ref{l:mw}, it suffices to show that the stabilizer of any $v(P_s\cap L_\Theta)/(P_s\cap L_\Theta)$ ($v\in W_\Theta$) under $U^-_\Theta$ contains $U^-_t$ for some $t\in\Theta$ (on which $\chi$ is nontrivial). In fact, if $v\neq e$, then $\ell(tv)<\ell(v)$ for some $t\in\Theta$ and the stabilizer of $v(P_s\cap L_\Theta)/(P_s\cap L_\Theta)$ contains $N^-_t$; if $v=e$, the stabilizer of $(P_s\cap L_\Theta)/(P_s\cap L_\Theta)$ contains $U^-_s$. Therefore $\pi_{s,!}(\tdel^\Theta_\chi)=0$ and $\Av^{\Theta}_\chi(\IC_w)=0$. This completes the proof.
\end{proof}

Let $\scQ$ be the category of left $U$-equivariant mixed perverse sheaves on $\Fl$. Let $\scQ_{+}$ be the Serre subcategory of $\scQ$ generated by twists of $\IC_w, w>e$.

\begin{lemma}\label{l:socle} For each $w\in W$,
\begin{enumerate}
\item There is an injection $\delta(\ell(w)/2)\hookrightarrow\Delta_w$ in $\scQ$ whose cokernel is contained in $\scQ_{+}$, and $\omega\delta$ is the only semisimple sub-object of $\omega\Delta_w$;
\item Dually, there is a surjection $\nabla_w\twoheadrightarrow\delta(-\ell(w)/2)$ in $\scQ$ whose kernel is contained in $\scQ_{+}$, and $\omega\delta$ is the only semisimple quotient object of $\omega\nabla_w$.
\end{enumerate}
\end{lemma}
\begin{proof} The proof is essentially borrowed from the proof \cite[\S2.1]{BBM}, where the finite flag variety was treated. We prove (1) and (2) follows by Verdier duality.

We do induction on $\ell(w)$. For $w=e$ this is clear. Suppose $\ell(w)>0$ then $\ell(w)=\ell(ws)+1$ for some simple reflection $s$. Consider the $\PP^1$-fibration $\pi_s:\Fl\to G/P_s$. Then we have an exact sequence in $\scQ$
\begin{equation*}
0\to\Delta_{ws}(1/2)\to\Delta_w\to \pi_s^*\Delta_{\barw}[1](1/2)\to 0
\end{equation*}
where $\Delta_{\barw}$ is the standard sheaf on $G/P_s$ corresponding to the $B$-orbit $BwP_s/P_s$. By inductive hypothesis, we have an injection $\delta(\ell(ws)/2)\hookrightarrow\Delta_{ws}$ whose cokernel is in $\scQ_{+}$. Note that the simple constituents of $\pi_s^*\Delta_{\barw}[1](1/2)$ are twists of $\IC_v=\pi_s^*\IC_{\barv}[1](1/2)$ for some $v\in\{\coW\}$, hence $\pi_s^*\Delta_{\barw}[1](1/2)\in\scQ_{+}$. This proves the first statement of (1).

Let $\omega\IC_v\hookrightarrow\omega\Delta_w$ be a simple sub-object. Consider the image of $\omega\IC_v$ in $\omega \pi_s^*\Delta_{\barw}[1]$. If this image is nonzero, then $v\in\{\coW\}$ and $\IC_v=\pi_s^*\IC_{\barv}[1](1/2)$. We have
\begin{eqnarray*}
\Hom_{\scQ}(\IC_v,\Delta_w)&=&\Hom_{\scQ}(\pi_s^*\IC_{\barv}[1](1/2),\Delta_w)\\
&\cong&\Hom_{G/P_s}(\IC_{\barv}[1](1/2),\pi_{s,*}\Delta_w)\\
&\cong&\Hom_{G/P_s}(\IC_{\barv}[1](1/2),\Delta_{\barw}[-1](-1/2))=0
\end{eqnarray*}
Here we use the fact that $\pi_s$ is proper and $\Fl_{w}\to UwP_s/P_s$ is a trivial $\AA^1$-bundle to conclude $\pi_{s,*}\Delta_w=\pi_{s,!}\Delta_w\cong\Delta_{\barw}[-1](-1/2)$. The above vanishing means that $\omega\IC_v$ has zero image in $\omega\pi_s^*\Delta_{\barw}[1]$ and hence lies in $\omega\Delta_{ws}$. We then use inductive hypothesis for $\Delta_{ws}$ to conclude that $v$ must be $e$. Similarly, any semisimple sub-object of $\omega\Delta_w$ must also lie in $\omega\Delta_{ws}$. Hence such a semisimple sub-object can only be $\omega\delta$, by inductive hypothesis.
\end{proof}

\begin{lemma}\label{l:Avst}
\begin{enumerate}
\item []
\item For $u\in W_\Theta$, we have
\begin{eqnarray}\label{eq:avdelu}
\Av^\Theta_\chi(\tDel_u)\cong\tdel^\Theta_{\chi}(\ell(u)/2);\\
\label{eq:avnabu}
\Av^\Theta_\chi(\tnab_u)\cong\tdel^\Theta_{\chi}(-\ell(u)/2).
\end{eqnarray}
\item For $w\in W$, write $w=uv$ where $u\in W_\Theta$ and $v\in[\coW]$, then
\begin{eqnarray}\label{eq:avdel}
\Av^\Theta_\chi(\tDel_w)\cong\tDel_{\barv,\chi}(\ell(u)/2);\\
\label{eq:avnab}
\Av^\Theta_\chi(\tnab_w)\cong\tnab_{\barv,\chi}(-\ell(u)/2).
\end{eqnarray}
\end{enumerate}
\end{lemma}

\begin{proof}
We prove the statements about $\tDel_w$; the argument for $\tnab_w$ is similar. We first show that (1) implies (2). In fact, by Lemma \ref{l:convst}, $\tDel_w\cong\tDel_u\conv{U}\tDel_v$, therefore
\begin{equation*}
\Av^\Theta_\chi(\tDel_w)\cong\tdel^\Theta_\chi\conv{U}\tDel_u\conv{U}\tDel_v\cong\Av^\Theta_\chi(\tDel_u)\conv{U}\tDel_v.
\end{equation*}

Assuming (\ref{eq:avdelu}), we get
\begin{equation}\label{eq:wuv}
\Av^\Theta_\chi(\tDel_w)\cong\Av^\Theta_\chi(\tDel_u)\conv{U}\tDel_v\cong\tdel^\Theta_\chi(\ell(u)/2)\conv{U}\tDel_v=\Av^\Theta_\chi(\tDel_v)(\ell(u)/2).
\end{equation}
Since $v\in[\coW]$, the action map $a^-$ gives an isomorphism
\begin{equation*}
a^-:U^-_\Theta\times BvB/U\cong\tilFl^\Theta_{\barv}.
\end{equation*}
Therefore $\Av^\Theta_\chi(\tDel_v)\cong\tDel_{\barv,\chi}$ follows from the definition of $\Av^\Theta_\chi$. This, combined with (\ref{eq:wuv}), proves the isomorphism (\ref{eq:avdel}).

It remains to prove (1). By the last sentence in Remark \ref{r:projresmix}, it suffices to show that $\piddag\Av^\Theta_\chi(\tDel_u)\cong\delta^\Theta_\chi(\ell(u)/2)$. We have
\begin{equation*}
\piddag\Av^\Theta_\chi(\tDel_u)=\tdel^\Theta_\chi\conv{U}\piddag\tDel_u=\tdel^\Theta_\chi\conv{U}\Delta_u.
\end{equation*}
By Lemma \ref{l:socle}(1), there is an injection $\delta(\ell(u)/2)\hookrightarrow\Delta_w$ whose cokernel is in $\scQ_+$. By the argument of Lemma \ref{l:killnonmin}, $\tdel^\Theta_\chi\conv{U}(-)$ is zero on $\scQ_+$, hence
\begin{equation*}
\tdel^\Theta_\chi\conv{U}\Delta_u\xleftarrow{\sim}\tdel^\Theta_\chi\conv{U}\delta(\ell(u)/2)\cong\piddag\tdel^\Theta_\chi(\ell(u)/2)\cong\delta^\Theta_\chi(\ell(u)/2).
\end{equation*}
This completes the proof of the lemma.
\end{proof}

The following is an immediate consequence of Proposition \ref{l:Avst}.
\begin{cor}
If $\tcF\in\hsP$ is a free-monodromic tilting sheaf, then $\Av^\Theta_\chi\tcT$ is also a free-monodromic tilting sheaf.
\end{cor}

\subsubsection{The object $\tcP_\Theta$} Define the object
\begin{eqnarray}\label{eq:Ptheta}
\tcP_\Theta&:=&\Av^\Theta_!(\tdel^\Theta_\chi).
\end{eqnarray}
Since $\tdel^\Theta_\chi$ is supported on $\tilFl_{\leq w_\Theta}=L_\Theta B/B$, $\tcP_\Theta$ is also supported on $\tilFl_{\leq w_\Theta}$; i.e., $\tcP_\Theta\in\hsM_{\leq w_\Theta}$.

\begin{lemma}\label{l:PTheta}
\begin{enumerate}
\item []
\item The object $\omega\tcP_\Theta$ is a projective cover of $\omega\delta$ in $\omega\hsP_{\leq w_\Theta}$.
\item The object $\tcP_\Theta$ is a successive extension of $\tDel_u(\ell(u)/2)$ for $u\in W_\Theta$, each appearing exactly once.
\item There is a natural isomorphism of functors $\hsM\to\hsM$
\begin{equation*}
\Av^\Theta_!\Av^\Theta_\chi(-)\cong\tcP_\Theta\conv{U}(-).
\end{equation*}
\end{enumerate}
\end{lemma}
\begin{proof}
(1) Note that we have an equivalence $\iota:\hsM_{\Theta,\leq\overline{e}}\cong D^b(\hatS,\Frob)$ with $\tdel^\Theta_\chi$ corresponding to $\hatS$. For any $\calF\in\omega\hsM_{\leq w_\Theta}$, we have
\begin{eqnarray}\label{eq:prep}
\bR\Hom_{\hsM_{\leq w_\Theta}}(\tcP_\Theta,\calF)&\cong&\bR\Hom_{\hsM_{\Theta,\leq e}}(\tdel^\Theta_\chi,\Av^\Theta_\chi(\calF))\\
&\cong&\bR\Hom_{\hatS}(\hatS,\iota\Av^\Theta_\chi(\calF))=\iota\Av^\Theta_\chi(\calF).
\end{eqnarray}
Therefore $\omega\tcP_\Theta$ represents the exact functor $\iota\circ\Av^\Theta_\chi:\omega\hsM_{\leq w_\Theta}\to D^b(\hatS,\Frob)$. The exactness implies $\omega\tcP_\Theta$ is a projective object in $\omega\hsP_{\leq w_\Theta}$. By Lemma \ref{l:killnonmin} and Proposition \ref{l:Avst}, we have
\begin{equation*}
\Hom_{\hsP}(\tcP_\Theta,\IC_w)=\iota\Av^\Theta_\chi(\IC_w)=\begin{cases}\Ql, & w=e\\0, & w\in W_\Theta-\{e\}\end{cases}
\end{equation*}
Therefore $\omega\tcP_\Theta$ is a projective cover $\omega\delta$ in $\omega\hsP_{\leq w_\Theta}$.

(2) By (\ref{eq:prep}) and the isomorphism (\ref{eq:avnabu}), we have for any $u\in W_\Theta$
\begin{equation*}
\Hom_{\hsM}(\tcP_\Theta,\tnab_{u})=\iota\Av^\Theta_\chi(\tnab_u)=\hatS(-\ell(u)/2).
\end{equation*}
This implies that in the $\tDel$-flag of $\tcP_\Theta$, each $\tDel_u(\ell(u)/2)$ appears exactly once.

(3) For any object $\calF\in\hsM$, we have functorial isomorphisms
\begin{equation*}
\Av^\Theta_!\Av^\Theta_\chi(\calF)\cong\Av^\Theta_!(\tdel^\Theta_\chi\conv{U}\calF)\cong(\Av^\Theta_!\tdel^\Theta_{\chi})\conv{U}\calF=\tcP_\Theta\conv{U}\calF.
\end{equation*}
Here we used the obvious fact that $\Av^\Theta_!$ commutes with right convolution.
\end{proof}

\begin{remark}\label{rm:Pcoalg}
As in Remark \ref{rm:Ccoalg}, by Lemma \ref{l:PTheta}(3), the comonad structure on $\Av^\Theta_!\Av^\Theta_\chi$ gives a coalgebra structure on $\tcP_\Theta$ with respect to the convolution $\conv{U}$; we will see a similar phenomenon in Proposition \ref{p:vcomp}.
\end{remark}


\subsection{The functor $\VV$}\label{ss:vv}
In this section, we will define a functor $\VV:\hsM\to D^b(S\otimes S,\Frob)$. In the case $G$ is of finite type, this is essentially the averaging functor $\Av^\Sigma_\chi$. However, when $G$ is infinite-dimensional, the averaging procedure involves the infinite-dimensional big cell $C\subset\Fl$, which causes some technical complication.

\subsubsection{The functor $\av_\chi$}
Recall from Lemma \ref{l:maptoUs} that we have a regular function $\rho_s$ on $C$ for every simple reflection $s$. Let $\chi$ be the sum of these functions:
\begin{equation*}
\chi:C\xrightarrow{\prod_s\rho_s}\prod_{s\in\Sigma}\AA^1\xrightarrow{+}\AA^1.
\end{equation*}
We define $\calL_\chi$ to be the *-complex $\chi^*\AS_\psi$ on the ind-scheme $C$. Then $\calL_{\chi}$ is the projective limit of local systems on $C_{\leq w}$.

By Lemma \ref{l:Ctriv}, the $H$-torsor $\pi^C:\tilC\to C$ is trivializable. Let us fix a section $\sigma:C\to\tilC$ whose image is denoted by $C^\sigma$. Let $C^\sigma_G\subset G$ be the preimage of $C^\sigma\subset\tilFl$, which admits a right $U$-action, and the quotient $C^\sigma_G/U\cong C$. We will also view $\calL_\chi$ as a *-complex on $C^\sigma$ or $C^\sigma_G$ by pull-back.

For each $w\in W$, consider convolution:
\begin{equation*}
a^-_{\leq w}:C^\sigma_G\twtimes{U}\tilFl_{\leq w}\subset G\twtimes{U}\tilFl\xrightarrow{m}\tilFl.
\end{equation*}
Both the source and the target of the morphism $a^-_{\leq w}$ are ind-schemes (the ind-scheme structure of the source are given by $\bigcup_{w'\in W}C^\sigma_{G,\leq w'}\twtimes{U}\tilFl_{\leq w}$), and $a^-_{\leq w}$ is clearly of finite type, therefore we can define the functor
\begin{eqnarray*}
\av_{\chi,\leq w,!}:\scM_{\leq w}&\to&\underleftarrow{D}^b_{m}(\qw{\tilFl}{H})\\
\calF&\mapsto&a^-_{\leq w,!}(\calL_\chi\stackrel{U}{\boxtimes}\calF).
\end{eqnarray*}
By Proposition \ref{compfun}, this functor extends to
\begin{equation*}
\av_{\chi,\leq w,!}:\hsM_{\leq w}\to\underleftarrow{\hatD}^b_{m}(\qw{\tilFl}{H}).
\end{equation*}
Passing to the inductive 2-limit, we get
\begin{equation*}
\av_{\chi,!}:\hsM\to\underleftarrow{\hatD}^b_{m}(\qw{\tilFl}{H}).
\end{equation*}

Recall the projection $\pi^C_s:C\subset\Fl\to G/P_s$ for any simple reflection $s$.
\begin{lemma}\label{l:projsL}
The $*$-complex $\pi^C_{s,!}\calL_\chi$ is zero.
\end{lemma}
\begin{proof}
It is enough to check that the stalk of $\pi^C_{s,!}\calL_\chi$ at any geometric point $x\in G/P_s$ is zero. By Lemma \ref{l:fiberA1}, the restriction of $\calL_\chi$ to the fiber $C_x=\pi^{C,-1}_s(x)$ can be identified with the Artin-Schreier sheaf $\AS_\psi$ on $\AA^1$ via $\rho_s:C_x\isom\AA^1$. Therefore the stalk of $\pi^\sigma_{s,!}\calL_\chi$ at $x$ is
\begin{equation*}
\upH^*_c(C_x,\calL_\chi|_{C_x})\cong \upH^*(\AA^1,\AS_\psi)=0.
\end{equation*}
\end{proof}

\begin{lemma}\label{l:Vkillnomin} For $w\neq e$, $\av_{\chi,!}(\IC_w)=0$.
\end{lemma}
\begin{proof}
For $\calF_1\in D^b_{m}(U\backslash\Fl),\calF_2\in D^b_{m}(B\backslash\Fl)$, $\av_!(\calF_1\conv{B}\calF_2)\cong\av_{\chi,!}(\calF_1)\conv{B}\calF_2$ because $\av_{\chi,!}$ is itself defines by convolution. Since each $\IC_w$ ($w\neq e$) is a direct summand of $\IC_{s_1}\conv{B}\cdots\conv{B}\IC_{s_m}$ for a reduced word $w=s_1\cdots s_m$, it suffices to show that $\av_{\chi,!}(\IC_s)=0$ for any simple reflection $s\in\Sigma$.

Let $\tilpi_s:\tilFl\to G/P_s$ be the projection. Let $\delta_s$ be the skyscraper sheaf at $P_s/P_s\in G/P_s$. Then $\IC_s$ can be identified with $\tilpi_s^*\delta_s$ up to shift and twist. We have
\begin{equation*}
\av_{\chi,!}(\tilpi_s^*\delta_s)=\tilpi_s^*a^-_{!}(\calL_\chi\stackrel{U}{\boxtimes}\delta_s)=\tilpi_s^*\pi^C_{s,!}\calL_\chi
\end{equation*}
which is zero by Lemma \ref{l:projsL}. Hence $\av_{\chi,!}(\IC_s)=0$ and the lemma is proved.
\end{proof}

If we further take stalks along the stratum $\tilFl_{e}$, we get
\begin{equation*}
\VV':=\tili_e^*\av_{\chi,!}:\hsM\to\underleftarrow{\hatD}^b_{m}(\qw{\tilFl}{H})\to \hatD^b_{m}(\qw{\tilFl_e}{H})\isom D^b(\hatS,\Frob)
\end{equation*}

\begin{cor}\label{c:Vexact}
The functor $\VV'$ is t-exact.
\end{cor}
\begin{proof} By Lemma \ref{l:prot}, in order to show that $\VV'$ is t-exact, it suffices to show that it is t-exact when restricted to $\scM$.

By Lemma \ref{l:Vkillnomin}, we see that $\VV'(\IC_w)=0$ for $w\neq e$. For $w=e$, $\VV'(\pidag\delta)=\tili^*_e\pidag\calL_\chi=\Ql[r](r)=\pidag\delta\in D^b_m(\qw{\tilFl_e}{H})$ corresponds to the trivial module $\Ql\in D^b(\hatS,\Frob)$ placed at degree 0. Therefore, $\VV'$ sends simple objects $\IC_w\in\scP$ to the heart of $D^b(\hatS,\Frob)$, hence t-exact on $\scP$.
\end{proof}

\subsubsection{The functor $\VV$}
Let $(\VV')^f$ be the composition of $\VV'$ with the equivalence (cf. \eqref{Frobfin})
\begin{equation*}
(-)^f:D^b(\hatS,\Frob)\cong D^b(S,\Frob).
\end{equation*}
By Corollary \ref{c:Vexact}, $(\VV')^f$ restricts to an exact functor $\hsP\to\Mod(S,\Frob)$ with the $S$-action on $\VV'(\calF)^f$ coming from the right $H$-monodromy. We also have the left $H$-monodromy acting on each object $\calF\in\scP$ functorially, hence acting as natural transformations on the functor $(\VV')^f$. Therefore, we can lift $(\VV')^f$ uniquely into an exact functor
\begin{equation}\label{Vnonder}
\VV=(\VV')^f:\hsP\to\Mod(S\otimes S,\Frob).
\end{equation}
We also write
\begin{equation*}
\VV:\hsM\to D^b(S\otimes S,\Frob)
\end{equation*}
for the derived functor of \eqref{Vnonder}. It is easy to see that $\VV$ is a lifting of $(\VV')^f$ as functors on $\hsM$.


For $w\in W$, let the $\Gamma^*(w)=\{(w\cdot v^*,v^*)|v^*\in V_H^\vee\}\subset V_H^\vee\times V_H^\vee$ be the graph of the $w$-action on $V_H^\vee$. Let $\calO_{\Gamma^*(w)}$ be the coordinate ring of $\Gamma^*(w)\subset V_H^\vee\times V^\vee_H$.

\begin{lemma}\label{l:vst}
For each $w\in W$, we have
\begin{eqnarray*}
\VV(\tDel_w)&\cong&\calO_{\Gamma^*(w)}(\ell(w)/2),\\
\VV(\tnab_w)&\cong&\calO_{\Gamma^*(w)}(-\ell(w)/2).
\end{eqnarray*}
\end{lemma}
\begin{proof}
We prove the first identity; the proof of the second one is similar. We first claim that $\VV(\tDel_w)$ as a right $S$ is isomorphic to $S(\ell(w)/2)$. For this, it suffices to show that $\VV'(\tDel_w)=\tili^*_e\av_{\chi,!}(\tDel_w)[r](r)\cong\tcL[r](r+\ell(w)/2)$. By the last sentence in Remark \ref{r:projresmix}, it suffices to show that $\piddag\tili^*_e\av_{\chi,!}(\tDel_w)\cong\Ql(\ell(w)/2)\in D^b_{c}(\Fl_e)$.

We can similarly define
\begin{equation*}
\bav_{\chi,!}:\scD\to\underleftarrow{D}^b_{m}(\Fl)
\end{equation*}
which kills all $\IC_w$ except $\delta$ (see Lemma \ref{l:Vkillnomin}). By the definition of $\VV'$, we have
\begin{equation*}
\piddag\tili^*_e\av_{\chi,!}(\tDel_w)=i^*_e\piddag\av_{\chi,!}(\tDel_w)=i^*_e\bav_{\chi,!}(\Delta_w).
\end{equation*}
By Lemma \ref{l:socle}(1), we have an injection $\delta(\ell(w)/2)\hookrightarrow\Delta_w$ whose cokernel is in $\scQ_+$ (hence killed by $\bav_{\chi,!}$), therefore
\begin{equation*}
i^*_e\bav_{\chi,!}(\Delta_w)\cong i^*_e\bav_{\chi,!}(\delta_w)(\ell(w)/2)\cong\Ql(\ell(w)/2).
\end{equation*}
This shows that $\VV(\tDel_w)\cong S(\ell(w)/2)$ as right $S$-modules.

Secondly, we show that the $S\otimes S$-action on $\VV(\tDel_w)$ factors through $\calO_{\Gamma^*(w)}$. Note that the $S\otimes S$-structure on $\VV(\tDel_w)$ comes from the action of $S\otimes S$ on $\tDel_w$. Since $\tilFl_{w}\cong\Fl_{w}\times H$ and $\Fl_{w}$ is isomorphic to an affine space, we have
\begin{equation}\label{eq:ssfactor}
\End(\tDel_w)=\End_{\Fl_{w}\times H}(\Ql\boxtimes\tdel)\cong\End_{wB/U}(\tdel)
\end{equation}
Note that the $H\times H$-action on $HwH\cong wB/U$ factors through $(H\times H)/H_w$ where $H_w=\{(whw^{-1},h^{-1})|h\in H\}$, therefore the $S\otimes S$-action on $\End_{wB/U}(\tdel)$ factors through $\Sym((V_H\oplus V_H)/V_{H_w})$, which is $\calO_{\Gamma^*(w)}$.

Combining the two steps, we see that $\VV(\tDel_w)\cong\calO_{\Gamma^*(w)}(\ell(w)/2)$.
\end{proof}


The following result is parallel to Proposition \ref{p:hff}. We postpone its proof to \S\ref{ss:proofV}.

\begin{prop}\label{p:Vff}
Suppose $\tcT_1,\tcT_2\in\hsP$ are free-monodromic tilting sheaves, then the natural map
\begin{equation}\label{eq:Vff}
\Hom_{\hsP}(\tcT_1,\tcT_2)^f\to\Hom_{S\otimes S}(\VV(\tcT_1),\VV(\tcT_2))
\end{equation}
is an isomorphism of $\Frob$-modules.
\end{prop}

\subsection{The pro-sheaf $\tcP$}\label{ss:tcP}
We first define a shifted version of $\Av_!$, averaging along $U$-orbits. For $w\in W$, pick a normal subgroup  $J_w\lhd U$ of finite codimension $d(J_w)$ which acts trivially on $\tilFl_{\leq w}$. Let $a^+_{\leq w}:J_w\backslash U\times\tilFl_{\leq w}\to\tilFl_{\leq w}$ be the action morphism. Define
\begin{eqnarray*}
\av_{\leq w,!}:D^b_{m}(\tilFl_{\leq w})&\to& D^b_{c}(\tilFl_{\leq w})\\
\calF&\mapsto& a^+_{\leq w,!}(\Ql[2d(J_w)](d(J_w))\boxtimes\calF).
\end{eqnarray*}
It is easy to see that $\av_{\leq w,!}$ is independent of the choice of $J_w$ and compatible with the restriction functors $\tili_{w,w'}^*$ for the inclusions $\tili_{w,w'}:\tilFl_{\leq w}\hookrightarrow\tilFl_{\leq w'}$, hence it defines a functor
\begin{equation*}
\av_!:\underleftarrow{\hatD}^b_{m}(\qw{\tilFl}{H})\to\underleftarrow{\hsM}
\end{equation*}
which is left adjoint to the forgetful functor $\underleftarrow{\hsM}\to\underleftarrow{\hatD}^b_{m}(\qw{\tilFl}{H})$.

Recall that we have a trivialization $\tilC=C^\sigma\times H$. Let $\tcL_\chi$ be the pro-object $\calL_\chi\boxtimes\tdel$ in $\underleftarrow{D}^b_{m}(\qw{\tilC}{H})$ (where $\tdel$ is the basic free-monodromic local system on $H$). Let $\tilj$ be the open embedding $\tilC\hookrightarrow\tilFl$. We define
\begin{equation*}
\tcP:=\av_!\tilj_!\tcL_\chi\in\underleftarrow{\hsM}.
\end{equation*}

\begin{lemma}\label{l:bigP}
There is a nonzero morphism $\tcP\to\delta$ making $\omega\tcP$ a projective cover of $\omega\delta$ in $\omega\scP$. In particular, we can view $\tcP$ as an object in $\underleftarrow{\hsP}=2-\varprojlim_{w\in W}\hsP_{\leq w}$.
\end{lemma}
\begin{proof}
We first show that
\begin{equation}\label{homPIC}
\Hom_{\underleftarrow{\scM}}(\tcP,\IC_w)=\begin{cases}0 & w\neq e;\\ \Ql & w=e \end{cases}.
\end{equation}
Since $\av_!$ is adjoint to the forgetful functor, we have $\Hom(\tcP,\IC_w)=\Hom_{\Fl}(\tilj_!\tcL_\chi,\IC_w)$. If $w\neq e$, then $\IC_w$ has the form $\tilpi^!_s\calF$ for some simple reflection $s$ and some complex $\calF\in D^b_{m}(G/P_s)$ ($\tilpi_s:\tilFl\to G/P_s$ is the projection). Hence
\begin{eqnarray*}
\Hom_{\tilFl}(\tilj_!\tcL_\chi,\IC_w)&=&\Hom_{\tilFl}(\tilj_!\tcL_\chi,\tilpi_s^!\calF)\\
&=&\Hom_{G/P_s}(\tilpi_{s,!}\tilj_!\tcL_\chi,\tilpi_s^!\calF)\\
&=&\Hom_{G/P_s}(\pi^C_{s,!}\pi^C_!\tcL_\chi,\calF)=\Hom_{G/P_s}(\pi^C_{s,!}\calL_\chi,\calF).
\end{eqnarray*}
Here we used the fact that $\pi^C_!\tcL_\chi=\calL_\chi$. By Lemma \ref{l:projsL}, $\pi^C_{s,!}\calL_\chi=0$. Hence $\Hom(\tcP,\IC_w)=0$ for $w\neq e$.

For $w=e$,
\begin{equation*}
\Hom(\tcP,\IC_e)=\Hom_{\tilFl}(\tilj_!\tcL_\chi,\pi^!\delta)=i_e^*\calL_\chi=\Ql.
\end{equation*}
This proves \eqref{homPIC}.

We then prove $\RHom(\tcP,-):\hsM\to D^b(\Vect)$ is an exact functor: i.e., $\Ext^{<0}(\tcP,\hsM^{\geq0})=0$ and $\Ext^{>0}(\tcP,\hsM^{\leq0})=0$. By Lemma \ref{l:prot}, it suffices to show that $\Ext^{<0}(\tcP,\scM^{\geq0})=0$ and $\Ext^{>0}(\tcP,\scM^{\leq0})=0$. But this follows from \eqref{homPIC}, because every object in $\omega\scM^{\geq0}$ (resp. $\omega\scM^{\leq0}$) is a successive extension of $\omega\IC_w[\leq0]$ (resp. $\omega\IC_w[\geq0]$). This finishes the proof.
\end{proof}

\begin{cor}\label{c:tDelflagP}
The object $\tcP\in\underleftarrow{\hsM}$ is a successive extension of $\tDel_w(\ell(w)/2)$ for $w\in W$, each appearing exactly once.
\end{cor}
\begin{proof}
By Lemma \ref{l:socle}(2), $\delta(-\ell(w)/2)$ is the only simple constituent of $\nabla_w$ whose underlying complex is $\omega\delta$. By Lemma \ref{l:bigP}, we have
\begin{equation*}
\Hom(\piddag\tcP,\nabla_w)=\Hom(\tcP,\pidag\nabla_w)=\Ql(-\ell(w)/2).
\end{equation*}
This means $\piddag\tili^*_w\tcP=i^*_w\piddag\tcP\cong\Ql(\ell(w)/2)$. By Remark \ref{r:projresmix}, $\tili^*_w\tcP\cong\tcL_w(\ell(w)/2)$, which proves the corollary.
\end{proof}


Composing with the exact functor $(-)^f$, the functor $\Hom(\tcP,-)^f$ on $\hsP$ is still exact. Since $\Hom(\tcP,-)^f$ carries an action of $S\otimes S$ coming from the left and right $H$-monodromy, it can be lifted to an exact functor
\begin{equation}\label{hPnonder}
\Hom(\tcP,-)^f:\hsP\to \Mod(S\otimes S,\Frob).
\end{equation}
We define
\begin{equation*}
\RHom(\tcP,-)^f:\hsP\to D^b(S\otimes S,\Frob)
\end{equation*}
to be the derived functor of \eqref{hPnonder}. It is easy to see that, the $i$-th cohomology of $\RHom(\tcP,\calF)^f$ is nothing but the $\Frob$-locally finite part of Hom-space between $\omega\tcP$ and $\omega\calF[i]$ as pro-objects in $\scM$.

\begin{lemma}\label{l:bigprep}
There is a natural isomorphism of functors
\begin{equation*}
\RHom(\tcP,-)^f\cong\VV:\hsM\to D^b(S\otimes S,\Frob).
\end{equation*}
Moreover, such an isomorphism is unique up to a scalar.
\end{lemma}
\begin{proof}
First we claim that $\VV(\tcP)^{\Funi}=\Ql$. In fact, by Corollary \ref{c:tDelflagP}, $\tcP$ is a successive extension of $\tDel(\ell(w)/2)$. By Lemma \ref{l:vst}, $\VV(\tDel(\ell(w)/2))\cong\calO_{\Gamma^*(w)}(\ell(w))$ has negative $\Frob$-weights except when $w=e$, in which case $\VV(\tdel)^{\Funi}=\Ql$.

The identity $\VV(\tcP)^{\Funi}=\Ql$ gives a map, functorial in $\calF\in\hsM$:
\begin{equation*}
\beta(\calF):\RHom(\tcP,\calF)^f=\RHom(\tcP,\calF)^f\otimes\VV(\tcP)^{\Frob}\hookrightarrow\RHom(\tcP,\calF)^f\otimes\VV(\tcP)\to\VV'(\calF).
\end{equation*}
We claim $\beta(\calF)$ is a quasi-isomorphism for any $\calF\in\hsM$. By our remarks following the definitions of the derived $\VV(-)$ and $\RHom(\tcP,-)$, for a general object $\calF=\prolim\calF_n$, the $i$-th cohomology groups of $\VV(\calF)$ and $\RHom(\tcP,\calF)$ are computed as the projective limits of $i$-th cohomology groups of $\VV(\calF_n)$ and $\RHom(\tcP,\calF_n)$, hence it suffices to show that $\beta(\calF)$ is an isomorphism for any $\calF\in\scM$, or even for the generating objects $\{\IC_w\}$. Using Lemma \ref{l:Vkillnomin}, $\beta(\IC_w)$ is trivially an isomorphism for $w\neq e$; for $w=e$, $\beta(\delta):\Ql\to\Ql$ is also an isomorphism by construction. Hence $\beta(\calF)$ is an isomorphism for all $\calF\in\scM$, hence also for all $\calF\in\hsM$.

The uniqueness (up to scalar) of $\beta$ follows from the fact that the $\Frob$-equivariant endomorphisms of the functor $\RHom(\tcP,-)^f$ reduce to $\VV(\tcP)^{\Frob}=\Ql$.
\end{proof}


The following result is the counterpart of Proposition \ref{p:hcomp}. In the statement, we need to consider the convolution $\tcP\conv{U}\tcP$, which we understand as the pro-object $``\varprojlim_{v,w\in W}''\tili^{*}_{\leq v}\tcP\conv{U}\tili^{*}_{\leq w}\tcP$ in $\pro\hsM$. Note that this object does not have finite dimension stalks, and hence is not an object in $\underleftarrow{\hsM}$.

\begin{prop}\label{p:vcomp}
\begin{enumerate}
\item []
\item The pro-object $\tcP$ has a coalgebra structure with respect to the convolution $\conv{U}$; i.e., there is a comultiplication map $\mu:\tcP\to\tcP\conv{U}\tcP$ and a counit map $\epsilon:\tcP\to\tdel$ satisfying obvious co-associativity and compatibility conditions. Moreover, this coalgebra structure is unique once we fix the counit map $\epsilon$, which is unique up to a scalar.
\item The functor $\VV$ has a monoidal structure which intertwines the convolution $\conv{U}$ on $\hsM$ and the tensor product $(N_1,N_2)\mapsto N_1\Ltimes_{S}N_2$ (with respect to the right $S$-action on $N_1$ and the left $S$-action on $N_2$) on $D^b(S\otimes S,\Frob)$.
\end{enumerate}
\end{prop}
\begin{proof}
(1) By Lemma \ref{l:vst}, we have
\begin{equation*}
\hom_{\underleftarrow{\hsP}}(\tcP,\tdel)=\Hom_{\underleftarrow{\hsP}}(\tcP,\tdel)^{\Frob}\cong \hatS^{\Frob}=\Ql,
\end{equation*}
Hence we have a map $\epsilon:\tcP\to\tdel$ in $\hsM$, unique up to a scalar. We fix such an $\epsilon$.

Using the argument in the proof of Lemma \ref{l:bigprep}, we see that the only simple constituent of $\tcP\conv{U}\tcP$ isomorphic to $\delta$ (and not just a twist of it) is the quotient $\tcP\conv{U}\tcP\to\tdel\conv{U}\tdel=\tdel\to\delta$. In other words,
\begin{equation}\label{hommu}
\hom_{\pro\hsM}(\tcP,\tcP\conv{U}\tcP)\cong\hom_{\underleftarrow{\hsM}}(\tcP,\delta)=\Ql,
\end{equation}
which gives a map $\mu:\tcP\to\tcP\conv{U}\tcP$ in $\pro\hsM$, unique up to a scalar. If we require that $\tcP\xrightarrow{\mu}\tcP\conv{U}\tcP\xrightarrow{\epsilon\conv{U}\epsilon}\tdel\conv{U}\tdel=\tdel$ be the same as $\epsilon$, then $\mu$ is uniquely determined. The co-associativity of $\mu$ follows essentially from the fact that
\begin{equation*}
\hom_{\pro\hsM}(\tcP,\tcP\conv{U}\tcP\conv{U}\tcP)\cong\hom_{\underleftarrow{\hsM}}(\tcP,\delta)=\Ql,
\end{equation*}
which is proved using the same argument as \eqref{hommu}.

(2) Using Lemma \ref{l:bigprep} and the coalgebra structure $\mu$ defined in (1), we have a map functorial in $\calF_1,\calF_2\in\hsM$:
\begin{eqnarray*}
\VV(\calF_1)\otimes\VV(\calF_2)&=&\Hom(\tcP,\calF_1)^f\otimes\Hom(\tcP,\calF_2)^f\\
&\xrightarrow{\alpha}&\Hom_{\pro\hsM}(\tcP\conv{U}\tcP,\calF_1\conv{U}\calF_2)^f\\
&\xrightarrow{\mu^*}&\Hom(\tcP,\calF_1\conv{U}\calF_2)^f=\VV(\calF_1\conv{U}\calF_2).
\end{eqnarray*}
By Remark \ref{r:midS}, the map $\alpha$ above factors through $\Hom(\tcP,\calF_1)^f\otimes_{S}\Hom(\tcP,\calF_2)^f$ because the right $S$-action on the first term and the left $S$-action on the second term coincide after applying $\alpha$. Therefore, we get a bifunctorial map
\begin{equation}\label{eq:betabifun}
\beta(\calF_1,\calF_2):\VV(\calF_1)\Ltimes_{S}\VV(\calF_2)\to\VV(\calF_1\conv{U}\calF_2).
\end{equation}
The compatibility of $\beta$ with the monoidal structures $\conv{U}$ and $\otimes_{S}$ follows from the coalgebra structure of $\tcP$ given in (1).

It remains to show that $\beta(\calF_1,\calF_2)$ is an isomorphism for any $\calF_1,\calF_2\in\hsM$. Clearly, it suffices to show that $\beta|_{\scM\times\scM}$ is an isomorphism. Since $\beta(-,-)$ is a natural transformation between bi-exact bifunctors, it suffices to show that $\beta(\calF_1,\calF_2)$ is an isomorphism for generating objects of $\scM$, say $\calF_1=\pidag\IC_{w}$ and $\calF_2=\pidag\IC_{w'}$ for $w,w'\in W$. If $w$ and $w'$ are not both equal $e$, the convolution $\calF_1\conv{U}\calF_2=\pidag(\IC_w\conv{B}\IC_{w'})$ does not have simple constituent isomorphic to $\delta$ (for example, if $w'\neq e$, then $\IC_{w'}$ is the pull-back of a complex on $G/P_s$ for some simple reflection $s$, hence so is $\IC_w\conv{B}\IC_{w'}$). Therefore, in this case, both sides of \eqref{eq:betabifun} are zero, hence $\beta$ is trivially an isomorphism.

In the case $w=w'=e$, both sides of \eqref{eq:betabifun} are isomorphic to $S$ viewed as an $S$-bimodule, and the map $\beta(\delta,\delta)$ is also easily seen to be an isomorphism. This completes the proof.
\end{proof}


\subsection{Proof of Proposition \ref{p:Vff}}\label{ss:proofV}
We partly follow the strategy of the proof of \cite[Proposition  in \S2.1]{BBM}. Fix $w\in W$ and let $\tcP_{\leq w}=\tili^*_{\leq w}\tcP\in\hsP_{\leq w}$, whose underlying complex is a projective cover of $\omega\delta$ in $\omega\hsP_{\leq w}$, by Lemma \ref{l:bigP}. By Lemma \ref{l:bigprep}, $\VV|\hsP_{\leq w}$, factors as:
\begin{equation*}
\VV:\hsP_{\leq w}\xrightarrow{\alpha=\Hom(\tcP_{\leq w},-)}\Mod(A_{\leq w},\Frob)\to\Mod(S\otimes S,\Frob)
\end{equation*}
Here $\alpha(-)=\Hom(\tcP_{\leq w},-)$ and $A^{\opp}_{\leq w}:=\End_{\hsP_{\leq w}}(\tcP_{\leq w})$ and the second functor above is the restriction of scalars via the central homomorphism $S\otimes S\to A_{\leq w}$ given by left and right logarithmic monodromy operators.

The functor $\alpha$ admits a left adjoint
\begin{eqnarray*}
\beta:\Mod(A_{\leq w},\Frob)&\to&\hsP_{\leq w}\\
M&\mapsto&\tcP_{\leq w}\otimes_{A_{\leq w}}M.
\end{eqnarray*}
Concretely, if we write $M$ as the the cokernel of a map of free $A_{\leq w}$-modules $V_1\otimes A_{\leq w}\to V_0\otimes A_{\leq w}$ (where $V_i$ are vector spaces), then $\beta(M)$ is the cokernel of the corresponding map $V_1\otimes\tcP_{\leq w}\to V_0\otimes\tcP_{\leq w}$. Note that $\beta$ is a right inverse of $\alpha$.

Let $\hsP_{+}=\ker(\alpha)$ and $\scP_{+}=\hsP_+\cap\scP$. By Lemma \ref{l:killnonmin}, $\scP_+\subset\scP$ is the full subcategory of objects $\calF$ whose simple subquotients in the Jordan-H\"older series are twists of $\IC_v$ for $v>e$.

\begin{lemma}
For any object $\calF\in\hsP_+$ and any $u\in W$, we have
\begin{eqnarray}\label{kertodel}
\Hom(\calF,\tDel_u)=0;\\
\label{nabtoker}
\Hom(\tnab_u,\calF)=0.
\end{eqnarray}
\end{lemma}
\begin{proof}
We first prove \eqref{kertodel}. Since $\omega\tDel_u$ is a successive extension of $\omega\pidag\Delta_u$, it suffices to show that $\Hom(\calF,\pidag\Delta_u)=0$. Write $\calF=\prolim\calF_n$. Since $\calF\in\hsM^{\leq0}$, by Lemma \ref{l:prot}, we may assume each $\calF_n\in\scM^{\leq0}$. Suppose $f_n:\omega\calF_n\to\omega\Delta_u$ is any nonzero map, we will show that this map becomes zero when composed with $\calF_m\to\calF_n$ for large $m$. In fact, since $\calF\in\hsP_+=\ker(\alpha)$, we can choose $m$ large enough so that $\alpha(\calF_m)\to\alpha(\calF_n)$ is zero. Now $f_n$ and $f_m:\omega\calF_m\to\omega\calF_n\to\omega\Delta_u$ factor through $f^0_n:\omega\pH^0\calF_n\to\omega\Delta_u$ and $f^0_m:\omega\pH^0\calF_m\to\omega\pH^0\calF_n\to\omega\Delta_u$. Let $\calG_n$ and $\calG_m$ be the image of $f^0_n$ and $f^0_m$. Then we have $\calG_m\subset\calG_n\subset\omega\Delta_u$. If both $\calG_m$ is nonzero, then by Lemma \ref{l:socle}(1), we must have $\omega\delta\subset\calG_m\subset\calG_n$, which implies that $\alpha(\calG_m)\to\alpha(\calG_n)$ is nonzero, hence $\alpha(\calF_m)\to\alpha(\calF_n)$ is nonzero, contradiction! This proves that any map $f_n$ is zero in the direct limit $\varinjlim\Hom(\calF_n,\Delta_u)$.

The proof of \eqref{nabtoker} is similar.
\end{proof}

Suppose $\tcT_1,\tcT_2$ are free-monodromic tilting sheaves in $\hsP_{\leq w}$. We will first prove that the natural map
\begin{equation}\label{eq:Emod}
\Hom_{\hsP_{\leq w}}(\tcT_1,\tcT_2)\to\Hom_{A_{\leq w}}(\alpha\tcT_1,\alpha\tcT_2)
\end{equation}
is an isomorphism of $\Frob-$modules, and then deduce the isomorphism (\ref{eq:Vff}) from (\ref{eq:Emod}).

By adjunction, we have
\begin{equation}\label{eq:adt}
\Hom_{A_{\leq w}}(\alpha\tcT_1,\alpha\tcT_2)=\Hom_{\hsP_{\leq w}}(\beta\alpha\tcT_1,\tcT_2).
\end{equation}
Consider the adjunction map $c:\beta\alpha\tcT_1\to\tcT_1$. If we apply $\alpha$ to $c$, we get an isomorphism since $\alpha\beta\cong\id$, therefore the kernel and cokernel of $c$ lie in $\hsP_{+}$. Since $\omega\tcT_1$ admits a $\tnab$-flag, $\Hom(\tcT_1,\coker(c))=0$ by the above claim, hence $\coker(c)=0$, i.e., $c$ is surjective. Therefore we have an exact sequence
\begin{equation*}
0\to\Hom_{\hsP_{\leq w}}(\tcT_1,\tcT_2)\to\Hom_{\hsP_{\leq w}}(\beta\alpha\tcT_1,\tcT_2)\to\Hom_{\hsP_{\leq w}}(\ker(c),\tcT_2)
\end{equation*}
Again, since $\omega\tcT_2$ admits a $\tDel$-flag, $\Hom(\ker(c),\tcT_2)=0$ by the above claim, hence we get an isomorphism
\begin{equation*}
\Hom_{\hsP_{\leq w}}(\tcT_1,\tcT_2)\cong\Hom_{\hsP_{\leq w}}(\beta\alpha\tcT_1,\tcT_2)
\end{equation*}
which, combined with (\ref{eq:adt}), proves (\ref{eq:Emod}).

Now we show (\ref{eq:Emod}) implies (\ref{eq:Vff}). For this we need an analog of Lemma \ref{l:eqloc}. Recall from Corollary \ref{c:tDelflagP} that $\tili^*_v\tcP_{\leq w}\cong\tcL_v(\ell(v)/2)$, a free-monodromic local system on $\tilFl_v$.

\begin{lemma}\label{l:tensorFrac}
The algebra homomorphism given by $\prod_{v\leq w}\tili^*_v$:
\begin{equation*}
\End_{\hsP_{\leq w}}(\tcP_{\leq w})
^f\to\prod_{v\leq w}\End_{\hsP_v}(\tili_v^*\tcP_{\leq w})^f\cong\prod_{v\leq w}\End_{\hsP_v}(\tcL_v)^f
\end{equation*}
is an isomorphism after tensoring by $\Frac(S)$ over the right $S$-module structures.
\end{lemma}
\begin{proof}
We do induction on $w$ (for this we need to extend the partial ordering on $W$ to a total ordering). For $w=e$ this is obvious. Suppose this is true for $\tcP_{<w}$. The exact sequence
\begin{equation}\label{eq:exforP}
0\to\tDel_w(\ell(w)/2)=\tili_{w,!}\tili^*_w\tcP_{\leq w}\to\tcP_{\leq w}\to\tili_{<w,*}\tili^*_{<w}\tcP_{\leq w}=\tcP_{<w}\to0
\end{equation}
gives a commutative diagram with exact rows
\begin{equation*}
\xymatrix{\Hom(\tcP_{\leq w},\tDel_{w}(\ell(w)/2))^f\ar[r]\ar[d]^{a} & \End(\tcP_{\leq w})^f\ar[r]\ar[d] & \End(\tcP_{<w})^f\ar[d]^{b}\\
\End(\tcL_w(\ell(w)/2))^f\ar[r] & \prod_{v\leq w}\End(\tcL_v(\ell(v)/2))^f\ar[r] & \prod_{v<w}\End(\tcL_v(\ell(v)/2))^f}
\end{equation*}
The arrow $b\otimes_S\Frac(S)$ is an isomorphism by induction hypothesis, therefore to prove the lemma, it suffices to show that $a\otimes_S\Frac(S)$ is also an isomorphism. Applying $\bR\Hom(-,\tDel_w(\ell(w)/2))$ to the exact sequence (\ref{eq:exforP}), we see
\begin{eqnarray*}
\Hom(\tcP_{<w},\tDel_w(\ell(w)/2))^f\twoheadrightarrow\ker(a);\\
\coker(a)\hookrightarrow\Ext^1(\tcP_{<w},\tDel_w(\ell(w)/2))^f.
\end{eqnarray*}
To compute the complex $\bR\Hom(\tcP_{<w},\tDel_w(\ell(w)/2))$, we write $\tcP_{<w}$ as a successive extension of $\tDel_v(\ell(v)/2)$ for $v<w$, by Corollary \ref{c:tDelflagP}. We reduce to computing $\Ext^*(\tDel_{v},\tDel_w)^f$ for $v<w$. But notice that in the second part of the proof of Lemma \ref{l:vst}, we have shown that the $S\otimes S$-action on $\tDel_w$ factors through the quotient $\calO(\Gamma^*(w))$ (see formula (\ref{eq:ssfactor}) and the discussion afterwards), therefore the $S\otimes S$-action on $\Ext^*(\tDel_{v},\tDel_w)^f$ factors through the quotient $\calO(\Gamma^*(v)\cap\Gamma^*(w))$, which is a torsion module over either copy of $S$. Therefore, $\ker(a)\otimes_S\Frac(S)$ and $\coker(a)\otimes_S\Frac(S)$ are zero, i.e., $a\otimes_S\Frac(S)$ is an isomorphism.
\end{proof}

Consider the maps
\begin{equation}\label{eq:exex}
S\otimes S\to A^f_{\leq w}\to\prod_{v\leq w}\End_{\hsP_v}(\tcL_{v}(\ell(v)/2))^f\cong\prod_{v\leq w}\calO(\Gamma^*(v)).
\end{equation}
After tensoring the maps \eqref{eq:exex} by $\Frac(S)$ over the right copy of $S$, we get
\begin{equation*}
S\otimes\Frac(S)\to A^f_{\leq w}\otimes_S\Frac(S)\isom\prod_{v\leq w}\calO(\Gamma^*(v))\otimes_S\Frac(S)\isom\prod_{v\leq w}\Frac(S).
\end{equation*}
which is obviously surjective (on the level of spectra, this corresponds to the closed embedding of the generic points of the graphs $\Gamma^*(v)$ into $V_H^*\otimes_k\Frac(S)$). Also notice that $\VV(\tcT_2)$ is free (hence torsion-free) over either copy of $S$ (writing $\tcT_2$ as a successive extension of $\tDel$'s and applying Lemma \ref{l:vst}). Therefore we can apply Lemma \ref{l:samehom} to the situation $B=S\otimes S, C=A_{\leq w}^{f}$ and $S$ the second copy of $S$ in $S\otimes S$ and conclude
\begin{equation*}
\Hom_{\hsP_{\leq w}}(\tcT_1,\tcT_2)^f\cong\Hom_{A^f_{\leq w}}(\alpha(\tcT_1)^f,\alpha(\tcT_2)^f)\cong\Hom_{S\otimes S}(\VV(\tcT_1),\VV(\tcT_2))
\end{equation*}
as $\Frob$-modules.

To conclude this section, we describe the endomorphism algebra $\End_{\hsP}(\tcP)^{f}$ explicitly in the case $G$ is finite-dimensional following Soergel and Bernstein. This result is not used in the rest of the paper.
\begin{prop}\label{p:bigproj} Assume $W$ is of finite type. Then
\begin{enumerate}
\item The algebra homomorphism $S\otimes S\to\End_{\hsP}(\tcP)^{f}$ coming from the left and right logarithmic $H$-monodromy induces an isomorphism
\begin{equation*}
S\otimes_{S^{W}}S\isom\End_{\hsP}(\tcP)^{f}.
\end{equation*}
\item Let $\hsP_{0}\subset\hsP$ be the full subcategory consisting of $\calF$ such that $\omega\calF$ is a direct sum of copies of $\tcP$. Then $\hsP_{0}$ is stable under the convolution $\conv{U}$ and the functor $\VV$ induces an equivalence of monoidal categories
\begin{equation*}
\VV_{0}:\hsP_{0}\isom\Mod^{\free}(S\otimes_{S^{W}}S,\Frob).
\end{equation*}
Here $\Mod^{\free}(S\otimes_{S^{W}}S,\Frob)$ is the full subcategory of $(S\otimes_{S^{W}}S,\Frob)$-modules which are free of finite rank as $S\otimes_{S^{W}}S$-modules, and the monoidal structure is defined as in Proposition \ref{p:vcomp}(2). 
\end{enumerate}
\end{prop}
\begin{proof}
(1) The algebra $\End_{\hsP}(\tcP)^{f}=\VV(\tcP)$ is a free $S$-module over both the left and right $S$-actions(this follows by writing $\tcP$ as a successive extension of $\tDel$'s and applying Lemma \ref{l:vst}). The sequence of maps \eqref{eq:exex} for $w=w_{0}$ (the longest element of $W$) becomes
\begin{equation*}
S\otimes S\xrightarrow{\mu}\End_{\hsP}(\tcP)^{\opp,f}\xrightarrow{\nu} \prod_{v\in W}\calO(\Gamma^{*}(v)).
\end{equation*}
By Lemma \ref{l:tensorFrac}, $\nu$ becomes an isomorphism after tensoring with $\Frac(S)$ over the right copy of $S$. Since $\End_{\hsP}(\tcP)^{\opp,f}$ is free as a right $S$-module, $\nu$ is injective. The composition $\nu\circ\mu$ factors through the quotient $S\otimes_{S^{W}}S$ followed by an injection $S\otimes_{S^{W}}S\hookrightarrow \prod_{v\in W}\calO(\Gamma^{*}(v))$, hence $\mu$ also factors as an algebra homomorphism $\mu':S\otimes_{S^{W}}S\to\End_{\hsP}(\tcP)^{\opp,f}$, which is necessarily injective. To show $\mu'$ is also surjective, by graded Nakayama lemma, we only need to show that it is so after reduction modulo the augmentation ideal of the right copy of $S$. In other words, letting $\calP=\pi_{\dagger}\tcP\in P(\Fl)$, we need to show that $S\otimes_{S^{W}}\Ql\to\End_{\Fl}(\calP)^{\opp,f}$ is surjective, which follows from Soergel's result (\cite[Endomorphismensatz 3]{Soe}, see also the footnote in \cite[\S 2.6]{BBM}). 

(2) By (1), the functor $\VV=\Hom_{\hsP}(\tcP,-)^{f}$, when restricted to $\hsP_{0}$, takes values in $\Mod^{\free}(S\otimes_{S^{W}}S,\Frob)$. The functor $\VV_{0}=\VV|_{\hsP_{0}}$ is fully faithful and essentially surjective by construction. It remains to show that $\hsP_{0}$ is stable under the convolution $\conv{U}$, and then $\VV_{0}$ is monoidal by Proposition \ref{p:vcomp}(2). Applying Lemma \ref{l:PTheta}(3) to $\Theta=\Sigma$, we have
\begin{equation*}
\tcP\conv{U}\tcP\cong\Av^{\Sigma}_{!}\Av^{\Sigma}_{\chi}\tcP.
\end{equation*}
Since $\Av^{\Sigma}_{\chi}$ is t-exact, $\omega\Av^{\Sigma}_{\chi}\tcP\in\omega\hsP_{\Sigma}$, which consists of direct sums of $\omega\tdel^{\Sigma}_{\chi}$ by Lemma \ref{l:mw}. Therefore $\omega\Av^{\Sigma}_{!}\Av^{\Sigma}_{\chi}\tcP$ is a direct sum of $\omega\Av^{\Sigma}_{!}\tdel^{\Sigma}_{\chi}\cong\omega\tcP$. This proves that $\hsP_{0}$ is stable under $\conv{U}$, and hence finishes the proof of (2).
\end{proof}


\section{Equivalences}\label{s:dual}
In this section we prove the Main Theorem, i.e., the four equivalences mentioned in \S\ref{ss:main}. The proof will rely on the construction of DG models in Appendix \ref{a:dgmodel}. We suggest reading the statement of Theorem \ref{th:CtoMod} before getting into the proofs of the four equivalences.

\subsection{Langlands duality for Kac-Moody groups}\label{ss:LD}
Throughout this section, we fix a root datum $(\xch,\Phi,\xcoch,\Phi^\vee)$ with generalized Cartan matrix $A$; the dual root datum $(\xcoch,\Phi^\vee,\xch,\Phi)$ has generalized Cartan matrix $A^t$ (the transpose of $A$). Let $G$ and $G^\vee$ be the Kac-Moody groups over $k=\FF_q$ associated to these root data. We say that the  Kac-Moody groups $G$ and $G^\vee$ are {\em Langlands dual} to each other.

\begin{remark}
When $G$ is a Kac-Moody group associated to the {\em affine} root system of a split simple group $G_0$, the group $\dG$ may  {\em not} be isogenous to a Kac-Moody group associated to the affine root system of $\dG_0$; $\dG$ is sometimes a {\em twisted} loop group.
\end{remark}

In the rest of this section, we will need to distinguish notations for $G$ and $G^\vee$. In general, the equivariant categories $\scE=\scE_G$ and $\scE_{\Theta}=\scE_{G,\Theta}$ are for the group $G$, while the monodromic categories $\hsM=\hsM_{\dG}$ and $\hsM_{\Theta}=\hsM_{\dG,\Theta}$ are for $\dG$. In \S\ref{ss:self} and \S\ref{ss:finaldual}, the notations will be further explained.

Let $H$ and $\dH$ be the Cartan subgroups of $G$ and $\dG$ respectively. We identify the Weyl groups of $G$ and $\dG$ and call it $W$. Then there is a natural $W$-equivariant and $\Frob$-equivariant isomorphism
\begin{equation}\label{VHdual}
(V_H^\vee)^\Finv\cong\xch\otimes_{\ZZ}\Ql(1)\cong V_\dH.
\end{equation}

Let $\dS_H=\Sym(V_H^\vee)$ and $S_\dH=\Sym(V_\dH)$ be (graded) algebras with $\Frob$ actions. Then \eqref{VHdual} gives a natural $W$-equivariant and $\Frob$-equivariant isomorphism $\dS_H^\Finv\cong S_{\dH}$. This isomorphism gives an equivalence of triangulated categories
\begin{eqnarray*}
(-)^\Finv:D_{\perf}(\dS_H\otimes \dS_H,\Frob)&\isom& D^b(S_\dH\otimes S_\dH,\Frob)\\
L&\mapsto&L^\Finv.
\end{eqnarray*}

\begin{defn} The {\em regrading functor} is the self-functor of the category $C^f(\Frob)$ of complexes of locally finite $\Frob$-modules with integer weights:
\begin{equation*}
\phi:C^f(\Frob)\to C^f(\Frob)
\end{equation*}
sending a complex $L=(\cdots\to L^i\to L^{i+1}\to\cdots)$ to the complex $N=(\cdots\to N^i\to N^{i+1}\to\cdots)$, where
\begin{equation}\label{eq:defN}
N^{i}_j=(L^{i-j}_{-j})^\Finv,\forall i,j\in\ZZ.
\end{equation}
Here, subscripts stand for $\Frob$-weights. Forgetting the grading, we have
\begin{equation*}
\phi(N^\bullet)=N^{\bullet,\Finv}.
\end{equation*}
\end{defn}


\subsection{Equivariant-monodromic duality}\label{ss:em}

\begin{theorem}[Equivariant-monodromic duality]\label{th:emduality} There is an equivalence of triangulated categories
\begin{equation}\label{eq:Phi}
\Phi=\Phi_{G\to\dG}: \scE=\scE_G\cong\hsM_{\dG}=\hsM.
\end{equation}
satisfying the following properties:
\begin{enumerate}
\item\label{em:monoidal} $\Phi$ has a monoidal structure which intertwines the convolutions $\conv{B}$ and $\conv{U}$;
\item\label{em:hhvv} There is an isomorphism of functors $\theta:\HH^\Finv\Rightarrow\VV\circ\Phi$ compatible with the monoidal structures of these functors;
\item\label{em:st} $\Phi(\Delta_w)\cong\tDel_w$, $\Phi(\nabla_w)\cong\tnab_w$ for all $w\in W$;
\item\label{em:puretilt} $\Phi$ sends very pure complexes of weight 0 to free-monodromic tilting sheaves;
\item\label{em:slinear} There is a functorial isomorphism of $( S_\dH\otimes S_\dH,\Frob)$-modules
\begin{equation}\label{sameSbimod}
\phi\Ext^\bullet_{\scE}(\calF_1,\calF_2)\cong\Ext^\bullet_{\hsM}(\Phi\calF_1,\Phi\calF_2)^f
\end{equation}
for any $\calF_1,\calF_2\in\scE$;
\item\label{em:twist} For any $M\in D^b(\Frob)$ and $\calF\in\scE$, there is a functorial isomorphism
\begin{equation*}
\Phi(\calF\otimes M)\cong\Phi(\calF)\otimes\phi(M).
\end{equation*}
In particular, $\Phi\circ[1](1/2)=(-1/2)\circ\Phi$.
\end{enumerate}
\end{theorem}

We have an immediate consequence of the Theorem:
\begin{cor}\label{c:t}
For each $w\in W$, let $\tcT_w:=\Phi(\IC_w)$. Then
\begin{enumerate}
\item $\tcT_w$ is a \fm tilting extension of $\tcL_w$ and $\omega\tcT_w$ is indecomposable.
\item Any free-monodromic tilting extension of $\tcL_w$ with indecomposable underlying complex is isomorphic to $\tcT_w$.
\end{enumerate}
\end{cor}
\begin{proof}
(1) The fact that $\tcT_w$ is a \fm tilting sheaf follows directly from Theorem \ref{th:emduality}(\ref{em:puretilt}). Since $\Delta_w\isom\IC_w$ (mod $\scE_{<w}$), we have $\tDel_w\isom\tcT_w$ (mod $\hsM_{<w}$). Therefore $\tcT_w$ is a \fm tilting extension of $\tcL_w$. Finally, by \eqref{sameSbimod}
\begin{equation*}
\End_{\hsM}(\tcT_w)^f\cong\oplus_{i\in\ZZ}\Ext^i_{\scE}(\IC_w,\IC_w)
\end{equation*}
is a $\ZZ_{\geq0}$-graded algebra whose degree 0 part reduces to $\Ql$, and $\End_{\hsM}(\tcT_w)$ is the completion of $\Ext^*_{\hsM}(\IC_w,\IC_w)$ with respect to the augmentation ideal. Therefore there is no nontrivial idempotent in $\End_{\hsM}(\tcT_w)$, i.e., $\omega\tcT_w$ is indecomposable.

(2) Suppose $\tcT'$ is a free-monodromic tilting extension of $\tcL_w$ with indecomposable underlying complex. Let $\calC'=\Phi^{-1}(\tcT')$. Then by Theorem \ref{th:emduality}(\ref{em:puretilt}), $\calC'$ is a very pure complex. By Lemma \ref{l:Cdecomp}, $\omega\calC'$ is a direct sum of shifted IC-sheaves. But since $\omega\tcT'$ is indecomposable, $\omega\calC'$ is also indecomposable by the same argument of (1). Therefore $\calC'$ is a (shifted and twisted) IC-sheaf. Since $\tDel_w\isom\tcT'$ (mod $\hsM_{<w}$), we have $\Delta_w\isom\calC'$ (mod $\scE_{<w}$), hence $\calC'\cong\IC_w$ and $\tcT'\cong\tcT_w$.
\end{proof}

Combining Corollary \ref{c:t}, Theorem \ref{th:emduality}(\ref{em:monoidal})(\ref{em:twist}), and Proposition \ref{p:icdecomp}, we get
\begin{cor}\label{c:tdecomp}
For $w_1,w_2\in W$, the convolution $\tcT_{w_1}\conv{U}\tcT_{w_2}$, as a {\em mixed} complex, is a direct sum of $\tcT_w(n/2)$ for $n\equiv\ell(w_1)+\ell(w_2)-\ell(w) (\textup{mod }2)$. In particular, if $\ell(w_1w_2)=\ell(w_1)+\ell(w_2)$, then $\tcT_{w_1w_2}$ is a direct summand of $\tcT_{w_1}\conv{U}\tcT_{w_2}$ with multiplicity one.
\end{cor}

The following observation will be used in establishing the parabolic-Whittaker duality.
\begin{cor}\label{c:tisp}
If $G$ is finite dimensional, let $w_0$ be the longest element in the Weyl group $W$ and recall the object $\tcP$ defined in the \S\ref{ss:tcP} (in this case it is an honest object of $\hsP$). Then
\begin{equation*}
\tcT_{w_0}\cong\tcP(-\ell(w_0)/2).
\end{equation*}
\end{cor}
\begin{proof}
By Theorem \ref{th:emduality}(\ref{em:hhvv}), we have $\Phi(\Ql)\cong\tcP$ because the functors they represent ($\HH$ and $\VV$) are intertwined under $\Phi$. Therefore by Theorem \ref{th:emduality}(\ref{em:twist}),
\begin{equation*}
\tcT_{w_0}=\Phi(\IC_{w_0})\cong\Phi(\Ql[\ell(w_0)](\ell(w_0)/2))\cong\Phi(\Ql)(-\ell(w_0)/2)\cong\tcP(-\ell(w_0)/2).
\end{equation*}
\end{proof}

The rest of this subsection is devoted to the proof of Theorem \ref{th:emduality}. First we need to pick generating objects of the categories $\scE$ and $\hsM$.

For each simple reflection $s$, by the calculation in Appendix \ref{a:SL2}, Lemma \ref{l:Hsimple} and \ref{l:Vsimple}, we have a \fm tilting sheaf $\tcT_s\in\hsP_{\leq s}$, and we have an isomorphism
\begin{equation*}
\theta_s:(\HH_s)^{\Finv}:=\HH(\IC_s)^\Finv\isom\VV_s:=\VV(\tcT_s).
\end{equation*}
We fix such an isomorphism for each $s\in\Sigma$. For each $w\in W$, {\em fix} a reduced word expression $\unw=(s_1(w),\cdots,s_m(w))$ where $m=\ell(w)$. We define
\begin{eqnarray*}
\IC_{\unw}:=\IC_{s_1(w)}\conv{B}\cdots\conv{B}\IC_{s_m(w)};&&\HH_{\unw}:=\HH(\IC_{\unw});\\
\tcT_{\unw}:=\tcT_{s_1(w)}\conv{U}\cdots\conv{U}\tcT_{s_m(w)};&&\VV_{\unw}:=\VV(\tcT_{\unw}).
\end{eqnarray*}
By Proposition \ref{p:vpure} and Proposition \ref{p:tiltconv}, $\IC_{\unw}$ is very pure of weight 0 and $\tcT_{\unw}$ is a \fm tilting sheaf. The isomorphisms $\{\theta_s|s\in\Sigma\}$ (together with Proposition \ref{p:hcomp} and \ref{p:vcomp}) induce an isomorphism
\begin{equation}\label{thw}
\theta_{\unw}:(\HH_{\unw})^{\Finv}\isom\VV_{\unw}.
\end{equation}

\subsubsection{The DG models}
We are going to define algebras and bimodules which control the categories $\scE_{\leq w}$ and $\hsM_{\leq w}$ and their respective embeddings for $w\leq w'$. For $w\leq w'$, define:
\begin{equation*}
\unE^{\leq w'}_{\leq w}:=\bigoplus_{u\leq w',v\leq w}\Ext^\bullet_{\scE}(\IC_{\unu},\IC_{\unv}).
\end{equation*}
We write $\unE_{\leq w}$ for the {\em opposite algebra} of $\unE^{\leq w}_{\leq w}$. Then $\unE^{\leq w'}_{\leq w}$ is a $(\unE_{\leq w'}, \unE_{\leq w})$-bimodule. Each $\IC_{\unv}$ ($v\leq w$) gives an $(\unE_{\leq w},\Frob)$-module $C_{\leq w,\unv}=\oplus_{u\leq w}\Ext^\bullet_{\scE}(\IC_{\unu},\IC_{\unv})$. We emphasize that we view $\unE_{\leq w}$ as a plain algebra with $\Frob$ action (placed in degree 0), not as a dg-algebra with the natural grading. Applying Theorem \ref{th:CtoMod} to the triple $(\scV_{\leq w}\subset\scE_{\leq w},\{\IC_{\unu}|u\leq w\})$, we get an equivalence of triangulated categories
\begin{equation}\label{dgeq}
\scE_{\leq w}\isom D_{\perf}(\unE_{\leq w},\Frob)
\end{equation}
where the RHS is the full triangulated subcategory of $D^b(\unE_{\leq w},\Frob)$ generated by twists of $\{C_{\leq w,\unv}|v\leq w\}$.

Similarly, we define:
\begin{equation*}
\unM^{\leq w'}_{\leq w}:=\bigoplus_{u\leq w',v\leq w}\Hom_{\hsP}(\tcT_{\unu},\tcT_{\unv})^f
\end{equation*}
and $\unM_{\leq w}=(\unM^{\leq w}_{\leq w})^{\opp}$. For each $v\leq w$, $\tcT_{\unv}$ gives an $(\unM_{\leq w},\Frob)$-module $T_{\leq w,\unv}:=\oplus_{u\leq w}\Hom(\tcT_{\unu},\tcT_{\unv})^f$. Applying Theorem \ref{th:CtoMod} and Remark \ref{r:Ff}to the triple $(\scT_{\leq w}\subset\hsM_{\leq w},\{\tcT_{\unv}|v\leq w\})$, there is an equivalence of triangulated categories
\begin{equation}\label{dgmon}
\hsM_{\leq w}\isom D_{\perf}(\unM_{\leq w},\Frob)
\end{equation}
where the RHS is the full triangulated subcategory of $D^b(\unM_{\leq w},\Frob)$ generated by twists of $\{T_{\leq w,\unv}|v\leq w\}$.

\subsubsection{Construction of $\Phi$.} We first construct an equivalence $\Phi_{\leq w}:\scE_{\leq w}\isom\hsM_{\leq w}$ for each $w\in W$. According to the equivalences \eqref{dgeq} and \eqref{dgmon}, it suffices to give an equivalence
\begin{equation*}
\Phi'_{\leq w}: D_{\perf}(\unE_{\leq w},\Frob)\isom D_{\perf}(\unM_{\leq w},\Frob)
\end{equation*}
By Proposition \ref{p:hff} and Proposition \ref{p:Vff}, we have
\begin{eqnarray}\label{eq:ecom}
\unE^{\leq w'}_{\leq w}\cong\bigoplus_{u\leq w',v\leq w}\Hom_{\dS_H\otimes\dS_H}(\HH_{\unu},\HH_{\unv})\\
\label{eq:mcom}
\unM^{\leq w'}_{\leq w}\cong\bigoplus_{u\leq w',v\leq w}\Hom_{S^\vee\otimes S^\vee}(\VV_{\unu},\VV_{\unv}).
\end{eqnarray}
The isomorphisms $\{\theta_{\unw}|w\in W\}$ in \eqref{thw} give a $\Frob$-equivariant isomorphism of algebras:
\begin{equation}\label{eq:emalg}
\unE_{\leq w}^\Finv\isom\unM_{\leq w}.
\end{equation}
For a complex of $(\unE_{\leq w},\Frob)$-module $L=(\cdots\to L^i\to L^{i+1}\to\cdots)$, we define $\Phi'_{\leq w}(L)$ to be the complex $L^{\Finv}$, which is a complex of $(\unM_{\leq w},\Frob)$-modules via the isomorphism \eqref{eq:emalg}. This gives the desired equivalence $\Phi'_{\leq w}$.

By Proposition \ref{p:Cfun}, the embedding $i_{w,w',*}:\scE_{\leq w}\hookrightarrow\scE_{\leq w'}$ corresponds to the functor
\begin{eqnarray*}
D_{\perf}(\unE_{\leq w},\Frob)&\to& D_{\perf}(\unE_{\leq w'},\Frob)\\
L&\mapsto&\unE^{\leq w'}_{\leq w}\Ltimes_{\unE_{\leq w}}L.
\end{eqnarray*}
Similarly, the embedding $\tili_{w,w',*}:\hsM_{\leq w}\hookrightarrow\hsM_{\leq w'}$ corresponds to the functor
\begin{eqnarray*}
D_{\perf}(\unM_{\leq w},\Frob)&\to& D_{\perf}(\unM_{\leq w'},\Frob)\\
N&\mapsto&\unM^{\leq w'}_{\leq w}\Ltimes_{\unM_{\leq w}}N.
\end{eqnarray*}
The isomorphisms in \eqref{eq:ecom}, \eqref{eq:mcom} and \eqref{eq:emalg} give an isomorphism $\Phi'_{\leq w'}(\unE^{\leq w'}_{\leq w})\isom\unM^{\leq w'}_{\leq w}$ as $(\unM_{\leq w'},\unM_{\leq w})$-bimodules. Therefore the embeddings $i_{w,w',*}$ and $\tili_{w,w',*}$ are naturally intertwined under $\Phi_{\leq w}$ and $\Phi_{\leq w'}$. Passing to the inductive 2-limit, we get an equivalence of triangulated categories $\Phi:\scE\to\hsM$.

\subsubsection{Verification of the properties}

Property \eqref{em:twist}. Suppose $\calF\in\scE_{\leq w}$ corresponds to the $(\unE_{\leq w},\Frob)$-complex $N$ under the equivalence \eqref{dgeq}. Then for a $\Frob$-module $M$ of weight $i$, $\calF\otimes M$ corresponds to $N\otimes M[i]$ under the equivalence \eqref{dgeq} (this follows from the construction in Theorem \ref{th:CtoMod}). Then $\Phi'_{\leq w}(N\otimes M[i])=N^\Finv\otimes M^\Finv[i]$, which corresponds to $\Phi(\calF)\otimes M^\Finv[i]$ under the equivalence \eqref{dgmon}.

Property \eqref{em:slinear}. By Lemma \ref{l:morehomeq} and Lemma \ref{l:morehommon}, we have
\begin{equation*}
\Ext^\bullet(\calF,\calF'[i])^\Finv_{\pure}\cong\Hom(\Phi(\calF),\Phi(\calF')[i])^{f}.
\end{equation*}
On the other hand,
\begin{equation*}
\oplus_j(\phi\Ext^\bullet(\calF,\calF'))^i_j=\oplus_j\Ext^{i-j}(\calF,\calF')^\Finv_{-j}=\Hom(\calF,\calF'[i])^\Finv_{\pure}.
\end{equation*}
Combining these two identities, we get \eqref{sameSbimod}.

Property \eqref{em:st}. The isomorphisms $\{\theta_{\unv}\}$ give an isomorphism of $(\unM_{\leq w},\Frob)$-modules
\begin{equation*}
C_{\leq w,\unv}^\Finv\isom T_{\leq w,\unv}, \forall v\leq w.
\end{equation*}
Therefore $\Phi(\IC_{\unw})\cong\tcT_{\unw}$.

Consider the following diagram
\begin{equation}\label{d:recoll}
\xymatrix{\scE_{<w}\ar[r]^{i_{<w,*}}\ar[d]^{\Phi_{<w}} & \scE_{\leq w}\ar[r]^{i^*_w}\ar[d]^{\Phi_{\leq w}} & \scE_w\ar@{.>}[d]^{\Phi_w}\\
\hsM_{<w}\ar[r]^{\tili_{<w,*}} & \hsM_{\leq w}\ar[r]^{\tili^*_w} & \hsM_w}
\end{equation}
By construction, the functors $\Phi_{\leq w}$ and $\Phi_{<w}$ are equivalences and there is a natural transformation making the left square commutative (we did not construct $\Phi_{<w}$ explicitly, but it is from the same construction of $\Phi_{\leq w}$ by comparing the algebras $\unE_{<w}$ and $\unM_{<w}$). Since the two rows in the diagram \eqref{d:recoll} are short exact sequences of triangulated categories, there is (an essentially unique) equivalence $\Phi_w:\scE_w\isom\hsM_{w}$ (the dotted arrow in the diagram \eqref{d:recoll}) with a natural transformation making the right square commutative. This implies that there are natural transformations intertwining the adjoints of $i^*_w$ and $\tili^*_{w}$; i.e., there is a natural transformation intertwining $i_{w,!}$ and $\tili_{w,!}$; there is another natural transformation intertwining $i_{w,*}$ and $\tili_{w,*}$. Note that
\begin{eqnarray*}
\Delta_w=i_{w,!}i^*_w\IC_{\unw};&&\tDel_w=\tili_{w,!}\tili^*_w\tcT_{\unw}.\\
\nabla_w=i_{w,*}i^*_w\IC_{\unw};&&\tnab_w=\tili_{w,*}\tili^*_w\tcT_{\unw}.
\end{eqnarray*}
Therefore we have isomorphisms $\Phi(\Delta_w)\isom\tDel_w$ and $\Phi(\nabla_w)\isom\tnab_w$ coming from the isomorphisms $\Phi(\IC_{\unw})\isom\tcT_{\unw}$.

Property (\ref{em:puretilt}). Note that the class of very pure complexes are
\begin{equation*}
\langle\Delta_w\langle0\rangle|w\in W\rangle\cap\langle\nabla_w\langle0\rangle|w\in W\rangle
\end{equation*}
while the class of free-monodromic tilting sheaves are
\begin{equation*}
\langle\tDel_w(?)|w\in W\rangle\cap\langle\tnab_w(?)|w\in W\rangle.
\end{equation*}
These two classes of objects correspond to each other under $\Phi$ by Property (\ref{em:twist}) and (\ref{em:st}).

Property (\ref{em:monoidal}). This requires the construction of a monoidal structure of $\Phi$. Fix $w,w'\in W$ such that $\ell(w)+\ell(w')=\ell(ww')$. We define an $(\unE_{\leq ww'},\unE_{\leq w}\otimes \unE_{\leq w'})$-bimodule
\begin{eqnarray*}
Q_{\leq w,\leq w'}&:=&\bigoplus_{u\leq ww',v\leq w,v'\leq w'}\Ext^\bullet_{\scE}(\IC_{\unu},\IC_{\unv}\conv{B}\IC_{\unv'})\\
&\cong&\bigoplus_{u\leq ww',v\leq w,v'\leq w'}\Hom_{\dS_H\otimes\dS_H}(\HH_{\unu},\HH_{\unv}\otimes_{\dS_H}\HH_{\unv'}).
\end{eqnarray*}
Similarly we define an $(\unM_{\leq ww'},\unM_{\leq w}\otimes \unM_{\leq w'})$-bimodule
\begin{eqnarray*}
R_{\leq w,\leq w'}&:=&\bigoplus_{u\leq ww',v\leq w,v'\leq w'}\Hom_{\hsM}(\tcT_{\unu},\tcT_{\unv}\conv{U}\tcT_{\unv'})^f\\
&\cong&\bigoplus_{u\leq ww',v\leq w,v'\leq w'}\Hom_{S_\dH\otimes S_\dH}(\VV_{\unu},\VV_{\unv}\otimes_{S_\dH}\VV_{\unv'}).
\end{eqnarray*}
They carry natural $\Frob$-actions.

By Remark \ref{r:bifunctor}, the transport of the convolution $\conv{B}$ to $D_{\perf}(\unE_{\leq w},\Frob)$ is given by the functor
\begin{eqnarray*}
D_{\perf}(\unE_{\leq w},\Frob)\times D_{\perf}(\unE_{\leq w'},\Frob)&\to &D_{\perf}(\unE_{\leq ww'},\Frob)\\
(L,L')&\mapsto& Q_{\leq w,\leq w'}\Ltimes_{(\unE_{\leq w}\otimes \unE_{\leq w'})}(L\otimes L').
\end{eqnarray*}
Again by Remark \ref{r:bifunctor}, the transport of the convolution $\conv{U}$ to $D_{\perf}(\unM_{\leq w},\Frob)$ is given by the functor
\begin{eqnarray*}
D_{\perf}(\unM_{\leq w},\Frob)\times D_{\perf}(\unM_{\leq w'},\Frob)&\to& D_{\perf}(\unM_{\leq ww'},\Frob)\\
(N,N')&\mapsto& R_{\leq w,\leq w'}\Ltimes_{(\unM_{\leq w}\otimes \unM_{\leq w'})}(N\otimes N').
\end{eqnarray*}
The isomorphisms $\{\theta_{\unw}|w\in W\}$ give a $\Frob$-equivariant isomorphism
\begin{equation*}
Q_{\leq w,\leq w'}^\Finv\cong R_{\leq w,\leq w'}
\end{equation*}
which intertwines the $(\unE_{\leq ww'},\unE_{\leq w}\otimes \unE_{\leq w'})$-bimodule structure on $Q_{\leq w,\leq w'}$ and the $(\unM_{\leq ww'},\unM_{\leq w}\otimes \unM_{\leq w'})$-bimodule structure on $R_{\leq w,\leq w'}$. This isomorphism gives a natural isomorphism making the following diagram commutative:
\begin{equation*}
\xymatrix{\scE_{\leq w}\times\scE_{\leq w'}\ar[r]^{\conv{B}}\ar@<-2ex>[d]_{\Phi_{\leq w}}\ar@<2ex>[d]^{\Phi_{\leq w'}} & \scE_{\leq ww'}\ar[d]^{\Phi_{\leq ww'}}\\
\hsM_{\leq w}\times\hsM_{\leq w'}\ar[r]^{\conv{U}} & \hsM_{\leq ww'}}
\end{equation*}
To check that these natural transformations are compatible with the associativity constraints essentially reduces to the following identification (we omit the details here)
\begin{equation*}
\bigoplus\Ext^\bullet_{\scE}(\IC_{\unu},\IC_{\unv}\conv{B}\IC_{\unv'}\conv{B}\IC_{\unv''})^\Finv\cong\bigoplus\Hom_{\hsM}(\tcT_{\unu},\tcT_{\unv}\conv{U}\tcT_{\unv'}\conv{U}\tcT_{\unv''})^f.
\end{equation*}
Passing to the inductive 2-limit as $w,w'$ run over $W$, we get the required monoidal structure of $\Phi$.

Property (\ref{em:hhvv}). Define
\begin{equation*}
H_{\leq w}:=\bigoplus_{v\leq w}\HH_{\unv}
\end{equation*}
Using (\ref{eq:ecom}), $H_{\leq w}$ can be viewed as a right $\unE_{\leq w}$-module (compatible with the $(\dS_H\otimes\dS_H,\Frob)$-module structure). Similarly, using (\ref{eq:mcom}), we define a right $\unM_{\leq w}$-bimodule
\begin{equation*}
V_{\leq w}:=\bigoplus_{v\leq w}\VV_{\unv}.
\end{equation*}
The transport of the functor $\HH$ on $D_{\perf}(\unE_{\leq w},\Frob)$ is given by $L\mapsto H_{\leq w}\Ltimes_{\unE_{\leq w}}L$; the transport of the functor $\VV$ on $D_{\perf}(\unM_{\leq w},\Frob)$ is given by $N\mapsto V_{\leq w}\Ltimes_{\unM_{\leq w}}N$. Using $\{\theta_{\unw}|w\in W\}$, we get a $\Frob$-equivariant isomorphism $H_{\leq w}^\Finv\isom V_{\leq w}$ intertwining the right $\unE_{\leq w}$-structure and right $\unM_{\leq w}$-structure (hence also intertwining the $\dS_H\otimes\dS_H$-structure and $S_\dH\otimes S_\dH$-structure). Therefore we get an isomorphism $\theta:\HH^\Finv\Rightarrow\VV\circ\Phi$ by passing to the inductive 2-limit. It is easy to check that $\theta$ is compatible with the monoidal structures by using the explicit dg-models.

\begin{remark} In the sequel, it is convenient to use two more ``compact'' algebras as dg-models of $\scE$ and $\hsM$. Let
\begin{eqnarray}\label{eq:defE}
E_{\leq w}:=\left(\bigoplus_{u,v\leq w}\Ext^\bullet_{\scE}(\IC_u,\IC_v)\right)^{\opp};\\
\label{eq:defM}
M_{\leq w}:=\left(\bigoplus_{u,v\leq w}\Hom_{\hsM}(\tcT_u,\tcT_v)^f\right)^{\opp}.
\end{eqnarray}
Then Theorem \ref{th:CtoMod} again gives equivalences
\begin{eqnarray*}
\scE_{\leq w}\isom D_{\perf}(E_{\leq w},\Frob);\\
\hsM_{\leq w}\isom D_{\perf}(M_{\leq w},\Frob).
\end{eqnarray*}
\end{remark}


\subsection{Koszul ``self-duality''}\label{ss:self}
Consider the category $\Dright=D^b_{m}(U\backslash G/B)$, the derived category of left-$U$-equivariant mixed complexes on $\Fl=G/B$. Recall $\pi:\tilFl\to\Fl$ is the projection which induces $\piddag:\hsM_G\to\Dright$. By Lemma \ref{l:fmtexist}, for each $w\in W$, $\calT_w:=\piddag\tcT_w$ is a tilting extensions of $\Ql[\ell(w)](\ell(w)/2)$ on $\Fl_{w}$ whose underlying complex is indecomposable. On the other hand, we have the forgetful functor $\Forg:\scE_G\to\Dright$ by forgetting the left-$B$-equivariant structure on objects in $\scE_G$. For $w\in W$, we still write $\IC_w\in\Dright$ for $\Forg(\IC_w)$.

Now consider the category $\Dleft:=D^b_{m}(\quot{\dB}{\dG}{\dU})$, the derived category of right-$\dU$-equivariant mixed complexes on $\dB\backslash\dG$. Now the situation is identical with $\Dright$ after interchanging left and right, $G$ and $\dG$. To distinguish objects in $\Dleft$ with objects in $\Dright$, we usually add a $(-)^\vee$ to the objects in $\Dleft$, e.g., the indecomposable tilting sheaves $\calT^\vee_w\in\Dleft$ and IC-sheaves $\IC^\vee_w\in\Dleft$, etc.

The theorem below is not really a self-duality, because the category $\Dright$ is defined in terms of $G$ while $\Dleft$ is defined in terms of $\dG$. In Remark \ref{r:whysd},  we will explain in what sense it becomes an involutive self-duality.

\begin{theorem}[``Self-duality'']\label{th:selfduality} There is an equivalence of triangulated categories
\begin{equation*}
\Psi:\Dleft:=D^b_m(\quot{\dB}{\dG}{\dU})\isom D^b_m(\quot{U}{G}{B})=:\Dright
\end{equation*}
satisfying the following properties:
\begin{enumerate}
\item\label{sd:modcat} $\Psi$ can be given a structure to intertwine the ($\scE_G$,$\hsM_G$)-bimodule category structure on $\Dleft$ (given by convolutions) and the ($\hsM_{\dG}$,$\scE_\dG$)-bimodule category structure on $\Dright$ (given by convolutions) via the equivalences $\Phi_{G\to\dG}:\scE_G\isom\hsM_\dG$ and $\Phi_{\dG\to G}:\scE_\dG\isom\hsM_G$ in Theorem \ref{th:emduality};
\item\label{sd:st} $\Psi(\Delta^\vee_w)\cong\Delta_w,\Psi(\nabla^\vee_w)\cong\nabla_w$ for all $w\in W$;
\item\label{sd:puretilt} $\Psi(\IC^\vee_w)\cong\calT_w, \Psi(\calT^\vee_w)\cong\IC_w$ for all $w\in W$. More generally, $\Psi$ {\em interchanges} very pure complexes of weight 0 and tilting sheaves;
\item\label{sd:slinear} There is a functorial isomorphism of $(S_\dH,\Frob)$-modules for any $\calF_1,\calF_2\in\Dleft$;
\begin{equation*}
\phi\Ext^\bullet_{\Dleft}(\calF_1,\calF_2)\cong\Ext^\bullet_{\Dright}(\Psi\calF_1,\Psi\calF_2)
\end{equation*}
\item\label{sd:twist} The analog of Theorem \ref{th:emduality}(\ref{em:twist}) holds for $\Psi$.
\end{enumerate}
\end{theorem}

\begin{proof}
We first build DG models for $\Dright$. Let $\scTr\subset\Dright$ be the full subcategory of mixed tilting sheaves.  The twists of $\{\calT_w|w\in W\}\subset\scT^\dagger$ generate the triangulated category $\Dright$. Define
\begin{equation*}
M^\dagger_{\leq w}:=\left(\bigoplus_{u,v\leq w}\Hom_{\Dright}(\calT_u,\calT_v)\right)^{\opp}.
\end{equation*}
Applying Theorem \ref{th:CtoMod} to the triple $(\scTr_{\leq w}\subset\Dright_{\leq w},\{\calT_u|u\leq w\})$, we get an equivalence
\begin{equation}\label{eq:dgrightM}
\Dright_{\leq w}\isom D_{\perf}(M^{\dagger}_{\leq w},\Frob)
\end{equation}
where the RHS is the full subcategory of $D^b(M^\dagger_{\leq w},\Frob)$ generated by twists of the $(M^{\dagger}_{\leq w},\Frob)$-modules $\oplus_{u\leq w}\Hom_{\Dright}(\calT_u,\calT_v)^f$ for $v\leq w$. Recall the definition of the algebra $M_{\leq w}$ in (\ref{eq:defM}). Below we use $M_{G,\leq w}$ to emphasize its dependence on $G$ rather than $\dG$. Then Lemma \ref{l:HomfreeS} gives a $\Frob$-equivariant isomorphism of algebras
\begin{equation}\label{MkillS}
M_{G,\leq w}\otimes_{S_H}\Ql\cong M^\dagger_{\leq w}.
\end{equation}

On the other hand, let $\scV^\dagger\subset\Dright$ be the full subcategory of very pure complexes of weight 0.  The twists of $\{\IC_w|w\in W\}\subset\scV^\dagger$ generate $\Dright$ as triangulated category. Define
\begin{equation*}
E^\dagger_{\leq w}:=\left(\bigoplus_{u,v\leq w}\Ext^\bullet_{\Dright}(\IC_u,\IC_v)\right)^{\opp}.
\end{equation*}
Applying Theorem \ref{th:CtoMod} to the triple $(\scV^\dagger_{\leq w}\subset\Dright_{\leq w},\{\IC_u|u\leq w\})$, we get another equivalence
\begin{equation}\label{eq:dgrightE}
\Dright_{\leq w}\isom D_{\perf}(E^\dagger_{\leq w},\Frob)
\end{equation}
where the RHS is the full subcategory of $D^b(E^\dagger_{\leq w},\Frob)$ generated by the twists of the $(E^\dagger_{\leq w},\Frob)$-modules $\oplus_{u\leq w}\Ext^\bullet_{\Dright}(\IC_u,\IC_v)$ for $v\leq w$. Recall the definition of the algebra $E_{\leq w}$ in (\ref{eq:defE}). Again we write $E_{G,\leq w}$ to emphasize its dependence on $G$. By Corollary \ref{c:eqkillS} (the isomorphism \eqref{killeqhom}) and Lemma \ref{l:purity} (which implies $E_{G,\leq w}$ is a free left $\dS_H$-module), we have a $\Frob$-equivariant isomorphism of algebras
\begin{equation}\label{EkillS}
\Ql\otimes_{\dS_H}E_{G,\leq w}\cong E^\dagger_{\leq w}.
\end{equation}

Next we define the DG models for $\Dleft$. We define
\begin{eqnarray*}
\leftexp{\dagger}{M}_{\leq w}:=\left(\bigoplus_{u,v\leq w}\Hom_{\Dleft}(\calT^\vee_u,\calT^\vee_v)\right)^{\opp};\\
\leftexp{\dagger}{E}_{\leq w}:=\left(\bigoplus_{u,v\leq w}\Ext^\bullet_{\Dleft}(\IC^\vee_u,\IC^\vee_v)\right)^{\opp}.
\end{eqnarray*}
Similar to the case of $\Dright$, we have equivalences
\begin{equation}\label{eq:dgleft}
D_{\perf}(\leftexp{\dagger}{E}_{\leq w},\Frob)\isom\Dleft_{\leq w}\isom D_{\perf}(\leftexp{\dagger}{M}_{\leq w},\Frob).
\end{equation}
We also have $\Frob$-equivariant isomorphisms of algebras
\begin{eqnarray}\label{leftMkillS}
\Ql\otimes_{S_\dH}M_{\dG,\leq w}\cong\leftexp{\dagger}{M}_{\leq w};\\
\label{leftEkillS}
E_{\dG,\leq w}\otimes_{\dS_\dH}\Ql\cong\leftexp{\dagger}{E}_{\leq w}.
\end{eqnarray}

Note that
\begin{eqnarray*}
\dS_H^\Finv=\Sym^\bullet(V^\vee_H)^\Finv\cong\Sym(V_\dH)=S_{\dH}\\
\dS_\dH^\Finv=\Sym^\bullet(V^\vee_\dH)^\Finv\cong\Sym(V_H)=S_{H}.
\end{eqnarray*}
By Theorem \ref{th:emduality}, we have isomorphisms
\begin{equation*}
E_{G,\leq w}^\Finv\cong M_{\dG,\leq w}; \hspace{1cm} E_{\dG,\leq w}^\Finv\cong M_{G,\leq w}.
\end{equation*}
By \eqref{MkillS}, \eqref{EkillS}, \eqref{leftMkillS} and \eqref{leftEkillS}, we get $\Frob$-equivariant isomorphisms of algebras
\begin{eqnarray*}
\leftexp{\dagger}{E}^\Finv_{\leq w}\cong M^{\dagger}_{\leq w},
E^{\dagger,\Finv}_{\leq w}\cong \leftexp{\dagger}{M}_{\leq w},
\end{eqnarray*}
which, together with the DG models \eqref{eq:dgrightM}, \eqref{eq:dgrightE} and \eqref{eq:dgleft}, give equivalences
\begin{equation*}
\xymatrix{D_{\perf}(\leftexp{\dagger}{E}_{\leq w},\Frob)\ar[r]^{(-)^{\Finv}}_{\sim}\ar[d]^{\wr} & D_{\perf}(M^{\dagger}_{\leq w},\Frob)\\
\Dleft_{\leq w}\ar[r]^{\Psi_{\leq w}} & \Dright_{\leq w}\ar[r]^{\Xi_{\leq w}}\ar[u]^{\wr} & \Dleft_{\leq w}\ar[d]^{\wr}\\
& D_{\perf}(E^\dagger_{\leq w},\Frob)\ar[u]^{\wr}\ar[r]^{(-)^{\Finv}}_{\sim} & D_{\perf}(\leftexp{\dagger}{M}_{\leq w},\Frob)}
\end{equation*}
Passing to the inductive 2-limit, we get equivalences
\begin{equation*}
\Dleft\xrightarrow{\Psi}\Dright\xrightarrow{\Xi}\Dleft.
\end{equation*}

We check the properties.

Properties (\ref{sd:st}) and (\ref{sd:twist}) for both $\Psi$ and $\Xi$ are verified as in Theorem \ref{th:emduality}.

\begin{claim}
There are natural isomorphism $\Xi\circ\Psi\Rightarrow\id_{\Dleft}$ and $\Psi\circ\Xi\Rightarrow\id_{\Dright}$ making $(\Psi,\Xi)$ a pair of inverse functors.
\end{claim}
\begin{proof}
We prove the first isomorphism. The argument for the second is similar. Since the Properties (\ref{sd:st}) and (\ref{sd:twist}) are satisfied by both $\Psi$ and $\Xi$, the functor $\Xi\circ\Psi$ is a t-exact self-equivalence of $\Dleft$ under the perverse t-structure (because $\Dleft^{\leq0}$ and $\Dleft^{\geq0}$ are characterized by $\langle\Delta_v[\leq0](?)\rangle$ and $\langle\nabla_v[\leq0](?)\rangle$). Therefore $\Xi\circ\Psi$ sends IC-sheaves to IC-sheaves. By Property (\ref{sd:st}), we must have $\Xi\Psi(\IC'_w)\cong\IC'_w$. In view of the first equivalence in (\ref{eq:dgleft}), the transport of $\Xi\circ\Psi$ on $D_{\perf}(E^\dagger_{\leq w},\Frob)$ is given by the identity functor by Proposition \ref{p:Cfun}. Therefore we get an isomorphism $\Xi\circ\Psi\Rightarrow\id_{\Dleft}$.
\end{proof}

Property (\ref{sd:puretilt}). It is obvious from construction that $\Psi(\IC^\vee_w)\cong\calT_w$ and $\Xi(\IC_w)\cong\calT^\vee_w$. Since $\Psi$ is an inverse of $\Xi$, therefore $\Psi(\calT^\vee_w)\cong\Psi\Xi(\IC
_w)\cong\IC_w$. The argument for Theorem \ref{th:emduality}(\ref{em:puretilt}) shows that both $\Psi$ and $\Xi$ sends very pure complexes of weight 0 to tilting sheaves. Since $\Psi$ and $\Xi$ are inverse to each other, $\Psi$ must interchange very pure complexes of weight 0 and tilting sheaves.

Finally we verify Property (\ref{sd:modcat}). The argument for Theorem \ref{th:emduality}(\ref{em:monoidal}) shows that: $\Psi$ has a structure intertwining the left-$\scE_\dG$-module category structure on $\Dleft$ and the left-$\hsM_G$-module category structure on $\Dright$; and that $\Xi$ has a structure intertwining the right-$\scE_G$-module category structure on $\Dright$ and the right-$\hsM_\dG$-module category structure on $\Dleft$. Since $\Psi$ and $\Xi$ are inverse functors, Property \eqref{sd:modcat} is proved.
\end{proof}

\begin{remark}\label{r:whysd}
When $\Lie G$ is a symmetrizable Kac-Moody algebra, we can replace $\dG$ by $G$ in Theorem \ref{th:selfduality} and get an equivalence
\begin{equation*}
\Psi_G:D^b_m(\quot{B}{G}{U})\isom D^b_m(\quot{U}{G}{B})
\end{equation*}
Let $\inv:D^b_m(\quot{U}{G}{B})\isom D^b_m(\quot{B}{G}{U})$ be the equivalence induced by the inversion map of $G$, then $\inv\circ\Psi_G$ becomes a ``self-duality'' of $D^b_m(\quot{B}{G}{U})$. Further argument shows that $\inv\circ\Psi_G$ is involutive: $(\inv\circ\Psi_G)^2$ is isomorphic to the identity functor.
\end{remark}

\begin{remark}\label{r:condW}
By Theorem \ref{th:selfduality}, the perverse t-structure on $\Dleft$ is transported by $\Psi$ to the following t-structure $(\leftexp{\textup{wt}}{\Dright}^{\leq0},\leftexp{\textup{wt}}{\Dright}^{\geq0})$ on $\Dright$:
\begin{eqnarray*}
\leftexp{\textup{wt}}{\Dright}^{\leq0}=\{\calF\in\Dright|i_w^*\calF\textup{ is a complex of weight}\geq0\};\\
\leftexp{\textup{wt}}{\Dright}^{\geq0}=\{\calF\in\Dright|i_w^!\calF\textup{ is a complex of weight}\leq0\}.
\end{eqnarray*}
The irreducible objects in the heart of this t-structure are precisely weight-0-twists of $\calT_w$. If we transport the characterizing properties of IC-sheaves to any irreducible object $\calT$ in the heart of the new t-structure on $\Dright$, we see that $\calT$ satisfies

{\em For any $w\in W$, $i^*_w\calT$ has weight$>0$ and $i^!_w\calT$ has weight$<0$.}

This is precisely the ``Condition (W)'' observed in \cite[\S1.3]{Yun2} by the second author, which served as a guiding principle in the study of weights of mixed tilting sheaves. In particular, Theorem \ref{th:selfduality} implies that the condition (W) holds for indecomposable mixed tilting sheaves on the flag variety of any Kac-Moody group. Using Theorem \ref{th:selfduality}(\ref{sd:st}), we get a simple relation between the multiplicities of standard sheaves in $\IC^\vee_w$ and in $\calT_w$, and we conclude that the ``weight polynomials'' of $\calT_w$ (cf. \cite[\S3.1]{Yun2}) are essentially given by Kazhdan-Lusztig polynomials. This gives a generalization of \cite[Theorem 1.2.1]{Yun2}.
\end{remark}


\subsection{Parabolic-Whittaker duality}\label{ss:parWh}

\begin{theorem}[Parabolic-Whittaker duality]\label{th:pwduality}
For each $\Theta\subset\Sigma$ of finite type, there is an equivalence of triangulated categories
\begin{equation*}
\Phi_{\Theta}: \scE_\Theta=\scE_{G,\Theta}\isom\hsM_{\dG,\Theta}=:\hsM_\Theta.
\end{equation*}
satisfying the following properties
\begin{enumerate}
\item\label{pw:modcat} $\Phi_\Theta$ can be given a structure to intertwine the right convolution of $\scE$ on $\scE_\Theta$ and the right convolution of $\hsM$ on $\hsM_\Theta$ (via the equivalence $\Phi:\scE\isom\hsM$ in Theorem \ref{th:emduality});
\item\label{pw:adj} There are natural isomorphisms which intertwine the adjunctions (\ref{eq:paradj}) and (\ref{d:whadj}) via the equivalences $\Phi_\Theta$ and $\Phi$;
\item\label{pw:st} $\Phi_\Theta(\Delta_{\barw})\cong\tDel_{\barw,\chi}$ and $\Phi_\Theta(\nabla_{\barw})\cong\tnab_{\barw,\chi}$ for all $\barw\in\coW$;
\item\label{pw:rest} The analogs of Theorem \ref{th:emduality}\eqref{em:puretilt}\eqref{em:slinear}\eqref{em:twist} hold.
\end{enumerate}
\end{theorem}

As in the case of Theorem \ref{th:emduality}, we have some immediate consequences.
\begin{cor}
For each $w\in[\coW]$, let $\tcT_{\barw,\chi}:=\Phi_\Theta(\IC_{\barw})$. Then
\begin{enumerate}
\item $\tcT_{\barw,\chi}$ is a \fm tilting extension of $\tcL_{\barw,\chi}$ whose underlying complex is indecomposable.
\item Any \fm tilting extension of $\tcL_{\barw,\chi}$ with indecomposable underlying complex is isomorphic to $\tcT_{\barw,\chi}$.
\end{enumerate}
\end{cor}

\begin{cor}\label{c:pullict}
For any $w\in\{\coW\}$, we have isomorphisms
\begin{eqnarray}\label{eq:pullic}
\pi^{\Theta,*}\IC_{\barw}[\lTh](\lTh/2)\cong&\IC_w&\cong\pi^{\Theta,!}\IC_{\barw}[-\lTh](-\lTh/2)\\
\label{eq:pullt}\Av^\Theta_!\tcT_{\barw,\chi}(-\lTh/2)\cong&\tcT_{w}&\cong\Av^\Theta_*\tcT_{\barw,\chi}(\lTh/2).
\end{eqnarray}
Here $\ell_{\Theta}=\ell(w_{\Theta})$ is the length of the longest element $w_{\Theta}$ in $W_{\Theta}$.
\end{cor}
\begin{proof}
The isomorphisms (\ref{eq:pullic}) follow from the fact that $\pi^\Theta$ is smooth of relative dimension $\lTh$; the isomorphisms (\ref{eq:pullt}) follow from (\ref{eq:pullic}) by Theorem \ref{th:pwduality} and Theorem \ref{th:emduality}.
\end{proof}

\begin{cor}\label{c:avdecomp}
For $w\in W$, the mixed perverse pro-sheaf $\Av^\Theta_\chi\tcT_w$ is a direct sum of $\tcT_{\barv,\chi}(n/2)$ for $n\equiv\ell(w)-\ell(\barv)(\textup{mod }2)$. In particular, for $w\in[\coW]$, $\tcT_{\barw,\chi}$ is a direct summand of $\Av^\Theta_\chi\tcT_w$ with multiplicity one.
\end{cor}

The rest of this subsection is devoted to the proof of Theorem \ref{th:pwduality}.

Recall $w_\Theta$ is the longest element in $W_\Theta$. Sometimes in a complicated symbol we write $\Theta$ for $w_\Theta$, e.g., we abbreviate $\IC_{w_\Theta}$ by $\IC_\Theta$, $\VV(\tcT_{w_\Theta})$ by $\VV_\Theta$, etc. Recall the object $\tcP_\Theta$ from (\ref{eq:Ptheta}). Applying Corollary \ref{c:tisp} to $G=L_\Theta$ we get an isomorphism $\tcP_\Theta\cong\tcT_\Theta(\lTh/2)$. By Remark \ref{rm:Ccoalg} and \ref{rm:Pcoalg}, $\calC_\Theta=\IC_\Theta[-\lTh](-\lTh/2)$ is a coalgebra object with respect to $\conv{B}$ and $\tcP_\Theta=\tcT_\Theta(\lTh/2)$ is a coalgebra object with respect to $\conv{U}$.

\begin{lemma}\label{l:canon}
For each $\Theta\subset\Sigma$ of finite type, there is a unique isomorphism $\beta_\Theta:\Phi(\calC_\Theta)\isom\tcP_\Theta$ which is a coalgebra isomorphism (Here $\Phi$ is the equivalence in Theorem \ref{th:emduality}).
\end{lemma}
\begin{proof}
Start from any isomorphism $\beta'_\Theta:\Phi(\calC_\Theta)\cong\tcP_\Theta$ (which exists because $\Phi(\IC_\Theta)=\tcT_\Theta$). Observe that
\begin{equation*}
\calC_{\Theta}\conv{B}\calC_{\Theta}=\upH^*(L_\Theta/N_\Theta)\otimes\calC_\Theta.
\end{equation*}
Therefore
\begin{equation*}
\hom_{\scE}(\calC_\Theta,\calC_\Theta\conv{B}\calC_\Theta)=\hom_{\scE}(\calC_\Theta,\upH^*(L_\Theta/N_\Theta)\otimes\calC_\Theta)=\hom_{\scE}(\calC_\Theta,\calC_\Theta)=\Ql.
\end{equation*}
Equivalently,
\begin{equation*}
\hom_{\hsP}(\Phi(\calC_\Theta),\tcP_\Theta\conv{U}\tcP_\Theta)=\Ql.
\end{equation*}
This means the diagram of co-multiplications
\begin{equation}\label{d:commadj}
\xymatrix{\Phi(\calC_\Theta)\ar[r]^(.4){\Phi(\mu)}\ar[d]^{\beta'_{\Theta}} & \Phi(\calC_\Theta\conv{B}\calC_\Theta)\ar[r]^{\sim} & \Phi(\calC_\Theta)\conv{U}\Phi(\calC_\Theta)\ar@<-3ex>[d]^{\beta'_{\Theta}}\ar@<3ex>[d]^{\beta'_{\Theta}}\\
\tcP_{\Theta}\ar[rr]^(.4){\mu'} && \tcP_{\Theta}\conv{U}\tcP_{\Theta}}
\end{equation}
is already commutative up to a nonzero scalar. Hence there is a unique nonzero scalar multiple $\beta_\Theta$ of $\beta'_\Theta$ making the diagram (\ref{d:commadj}) commutative, i.e., $\beta_\Theta$ commutes with the co-multiplications. In particular for $\Theta=\varnothing$, $\beta_\varnothing$ is the unique isomorphism $\Phi(\delta)\isom\tdel$ which preserves the structures of $\delta$ and $\tdel$ as the unit objects in the monoidal categories $\scE$ and $\hsM$.

We check that $\beta_\Theta$ also intertwines the co-unit maps $\epsilon:\calC_\Theta\to\delta$ and $\epsilon':\tcP_\Theta\to\tdel$. Again, we verify that
\begin{eqnarray*}
\hom_{\hsP}(\Phi(\calC_\Theta),\tdel)&=&\hom_{\scE}(\calC_\Theta,\delta)\\
&=&\hom_{\scE_{\leq e}}(i^*_{e}\calC_\Theta,\delta)=\hom_{\dS_H}(\Ql,\Ql)=\Ql.
\end{eqnarray*}
Therefore $\epsilon'\circ\beta_\Theta=\lambda\beta_\varnothing\circ\epsilon$ for some $\lambda\in\Ql^\times$. On the other hand, the compatibility between co-multiplication and co-unit maps give
\begin{eqnarray*}
(\epsilon\conv{B}\id)\circ\mu=\id:\calC_\Theta\to\calC_\Theta;\\
(\epsilon'\conv{U}\id)\circ\mu'=\id:\tcP_\Theta\to\tcP_\Theta.
\end{eqnarray*}
Since $\beta_\Theta$ already intertwines $\mu$ and $\mu'$, we see from the above identities that $\lambda=1$, i.e., $\beta_\Theta$ also intertwines $\epsilon$ and $\epsilon'$. This completes the proof.
\end{proof}

For $\calF_1,\calF_2\in\scV\subset\scE$ (very pure of weight 0), let $\tcT_i:=\Phi(\calF_i)\in\scT\subset\hsM$ be corresponding \fm tilting sheaves. By adjunction and Lemma \ref{l:convic}, we have
\begin{equation*}
\Ext^\bullet_{\scE_\Theta}(\pi^\Theta_*\calF_1,\pi^\Theta_*\calF_2)\cong\Ext^\bullet_{\scE}(\pi^{\Theta,*}\pi^\Theta_*\calF_1,\calF_2)\cong\Ext^\bullet_{\scE}(\calC_\Theta\conv{B}\calF_1,\calF_2).
\end{equation*}
On the other hand, by adjunction and Lemma \ref{l:PTheta}(3), we have
\begin{equation*}
\Hom_{\hsP_{\Theta}}(\Av^\Theta_\chi\tcT_1,\Av^\Theta_\chi\tcT_2)\cong\Hom_{\hsP}(\Av^\Theta_!\Av^\Theta_\chi\tcT_1,\tcT_2)\cong\Hom_{\hsP}(\tcP_\Theta\conv{U}\tcT_1,\tcT_2)
\end{equation*}
We have an isomorphism given by Theorem \ref{th:emduality}\eqref{em:slinear},
\begin{equation*}
\Ext^\bullet_{\scE}(\calC_\Theta\conv{B}\calF_1,\calF_2)^\Finv\cong\Hom_{\hsP}(\tcP_\Theta\conv{U}\tcT_1,\tcT_2)^f,
\end{equation*}
hence an isomorphism
\begin{equation*}
\psi(\calF_1,\calF_2):\Ext^\bullet_{\scE_\Theta}(\pi^\Theta_*\calF_1,\pi^\Theta_*\calF_2)^\Finv\isom\Hom_{\hsP_{\Theta}}(\Av^\Theta_\chi\tcT_1,\Av^\Theta_\chi\tcT_2)^f.
\end{equation*}

\begin{lemma}\label{l:piav}
The following diagram is commutative
\begin{equation*}
\xymatrix{\Ext^\bullet_{\scE}(\calF_1,\calF_2)^\Finv\ar[r]^(.4){\pi^\Theta_*}\ar[d]^{\Phi} & \Ext^\bullet_{\scE_\Theta}(\pi^\Theta_*\calF_1,\pi^\Theta_*\calF_2)^\Finv\ar[d]^{\psi(\calF_1,\calF_2)}\\
\Hom_{\hsP}(\tcT_1,\tcT_2)^f\ar[r]^(.4){\Av^\Theta_\chi} & \Hom_{\hsP_{\Theta}}(\Av^\Theta_\chi\tcT_1,\Av^\Theta_\chi\tcT_2)^f}
\end{equation*}
\end{lemma}
\begin{proof}
By the construction of $\psi(\calF_1,\calF_2)$, the commutativity of the above diagram is equivalent to the commutativity of
\begin{equation}\label{d:23}
\xymatrix{\Phi(\calC_\Theta)\conv{U}\Phi(\calF_1)\ar[r]^{\sim}\ar@<-2.5ex>[d]_{\beta_\Theta}\ar@<2ex>[d]^{\wr} & \Phi(\pi^{\Theta,*}\pi^\Theta_*\calF_1)\ar[r]^(.6){\textup{adj.}} & \Phi(\calF_1)\ar[d]^{\wr}\\
\tcP_\Theta\conv{U}\tcT_1\ar[r]^{\sim} & \Av^\Theta_!\Av^\Theta_\chi\tcT_1\ar[r]^(.6){\textup{adj.}} & \tcT_1}
\end{equation}
In this diagram, the composition of maps in the rows are given by $\Phi(\epsilon)\conv{U}\id$ and $\epsilon'\conv{U}\id$ (recall $\epsilon$ and $\epsilon'$ are co-unit maps for $\calC_\Theta$ and $\tcP_\Theta$). Since $\beta_\Theta$ intertwines the co-unit maps, the diagram (\ref{d:23}) is commutative.
\end{proof}

\begin{lemma}\label{l:comcom}
The isomorphisms $\psi(-,-)$ are compatible with compositions. More precisely, for $\calF_i\in\scV$, $i=1,2,3$, we have the following commutative diagram
\begin{equation}\label{d:commisom}
\xymatrix{\Ext^\bullet(\pi^\Theta_*\calF_2,\pi^\Theta_*\calF_3)^\Finv\otimes\Ext^\bullet(\pi^\Theta_*\calF_1,\pi^\Theta_*\calF_2)^\Finv\ar[r]\ar@<-8ex>[d]^{\psi(\calF_2,\calF_3)}\ar@<8ex>[d]^{\psi(\calF_1,\calF_2)} & \Ext^\bullet(\pi^\Theta_*\calF_1,\pi^\Theta_*\calF_3)^\Finv\ar[d]^{\psi(\calF_1,\calF_3)}\\
\Hom(\Av^\Theta_\chi\tcT_2,\Av^\Theta_\chi\tcT_3)^f\otimes\Hom(\Av^\Theta_\chi\tcT_1,\Av^\Theta_\chi\tcT_2)^f\ar[r] & \Hom(\Av^\Theta_\chi\tcT_1,\Av^\Theta_\chi\tcT_3)^f}
\end{equation}
\end{lemma}
\begin{proof}
We verify the case of degree 0 maps, i.e., maps in $\Hom(\pi^\Theta_*\calF_i,\pi^\Theta_*\calF_j)$. The argument for the general case is similar.

For a map $\alpha:\pi^\Theta_*\calF\to\pi^\Theta_*\calF'$, we write $\alpha^\#$ for the map $\calC_\Theta\conv{B}\calF\cong\pi^{\Theta,*}\pi^\Theta_*\calF\to\calF'$ induced by adjunction. Similarly, for a map $\gamma:\Av^\Theta_\chi\tcT\to\Av^\Theta_\chi\tcT,$, we write $\gamma^\#$ for the map $\tcP_\Theta\conv{U}\tcT\cong\Av^\Theta_!\Av^\Theta_\chi\tcT\to\tcT'$ induced by adjunction. Consider the following composition of maps
\begin{equation*}
\pi^\Theta_*\calF_1\xrightarrow{\alpha_1}\pi^\Theta_*\calF_2\xrightarrow{\alpha_2}\pi^\Theta_*\calF_3.
\end{equation*}
Then we can write $(\alpha_2\circ\alpha_1)^\#$ as the composition
\begin{equation*}
\calC_\Theta\conv{B}\calF_1\xrightarrow{\mu\conv{B}\id}\calC_\Theta\conv{B}\calC_\Theta\conv{B}\calF_1\xrightarrow{\id\conv{B}\alpha^\#_1}\calC_\Theta\conv{B}\calF_2\xrightarrow{\alpha^\#_2}\calF_3.
\end{equation*}
On the other hand, let
\begin{equation*}
\Av^\Theta_\chi\tcT_1\xrightarrow{\gamma_1}\Av^\Theta_\chi\tcT_2\xrightarrow{\gamma_2}\Av^\Theta_\chi\tcT_3
\end{equation*}
be the corresponding maps under $\psi(-,-)$. Then we can write $(\gamma_2\circ\gamma_1)^\#$ as the composition
\begin{equation*}
\tcP_\Theta\conv{U}\tcT_1\xrightarrow{\mu'\conv{U}\id}\tcP_\Theta\conv{U}\tcP_\Theta\conv{U}\tcT_1\xrightarrow{\id\conv{U}\gamma^\#_1}\tcP_\Theta\conv{U}\tcT_2\xrightarrow{\gamma^\#_2}\tcT_3.
\end{equation*}
In view of the definition of $\psi(-,-)$, the commutativity of the diagram (\ref{d:commisom}) follows from the fact that $\beta_\Theta:\Phi(\calC_\Theta)\cong\tcP_\Theta$ intertwines the co-multiplication structures $\mu$ and $\mu'$, which is proved in Lemma \ref{l:canon}.
\end{proof}

\begin{proof}[Proof of Theorem \ref{th:pwduality}]

For each $\barw\in\coW$, define an algebra with $\Frob$-action:
\begin{equation*}
E_{\Theta,\leq\barw}:=\left(\bigoplus_{u,v\in[\coW],\baru,\barv\leq\barw}\Ext^\bullet_{\scE_\Theta}(\pi^\Theta_*\IC_u,\pi^\Theta_*\IC_v)\right)^{\opp}.
\end{equation*}
Applying Theorem \ref{th:CtoMod} to the triple $(\scV_\Theta\subset\scE_\Theta,\{\pi^\Theta_*\IC_v|v\in[\coW]\})$, we get an equivalence
\begin{equation}\label{eq:dgpar}
\scE_{\Theta,\leq\barw}\isom D_{\perf}(E_{\Theta,\leq\barw},\Frob)
\end{equation}
where the RHS is by definition generated by twists of the $(E_{\Theta,\leq\barw},\Frob)$-modules
\begin{equation*}
\bigoplus_{u\in[\coW],\baru\leq\barw}\Ext^\bullet_{\scE_\Theta}(\pi^\Theta_*\IC_u,\pi^\Theta_*\IC_v)
\end{equation*}
for $v\in[\coW],\barv\leq\barw$.

Similarly, define another algebra with $\Frob$-action:
\begin{equation*}
M_{\Theta,\leq\barw}:=\left(\bigoplus_{u,v\in[\coW],\baru,\barv\leq\barw}\Hom_{\hsP_{\Theta}}(\Av^\Theta_\chi\tcT_{u},\Av^\Theta_\chi\tcT_{v})^f\right)^{\opp}.
\end{equation*}
Applying Theorem \ref{th:CtoMod} to the triple $(\scT_\Theta\subset\hsM_\Theta,\{\Av_\chi^\Theta\tcT_v|v\in[\coW]\})$, we get an equivalence
\begin{equation}\label{eq:dgwhit}
\hsM_{\Theta,\leq\barw}\isom D_{\perf}(M_{\Theta,\leq\barw},\Frob).
\end{equation}
where the RHS is by definition generated by twists of the $(M_{\Theta,\leq\barw},\Frob)$-modules
\begin{equation*}
\bigoplus_{u\in[\coW],\baru\leq\barw}\Hom_{\hsP_{\Theta}}(\Av^\Theta_\chi\tcT_{u},\Av^\Theta_\chi\tcT_{v})^f
\end{equation*} 
for $v\in[\coW],\barv\leq\barw$.

The isomorphisms $\psi(\IC_u,\IC_v)$ give an isomorphism
\begin{equation*}
\bigoplus_{u,v\in[\coW],\baru,\barv\leq\barw}\psi(\IC_u,\IC_v):E^\Finv_{\Theta,\leq\barw}\isom M_{\Theta,\leq\barw}
\end{equation*}
By Lemma \ref{l:comcom}, this is an algebra isomorphism, which induces an equivalence
\begin{equation*}
\Phi'_{\Theta,\leq\barw}:D_{\perf}(E_{\Theta,\leq\barw},\Frob)\isom D_{\perf}(M_{\Theta,\leq\barw},\Frob).
\end{equation*}
sending $L\mapsto L^\Finv$. This, together with the equivalences (\ref{eq:dgpar}) and (\ref{eq:dgwhit}), induces an equivalence
\begin{equation*}
\Phi_{\Theta,\leq\barw}:\scE_{\Theta,\leq\barw}\cong\hsM_{\Theta,\leq\barw}.
\end{equation*}
Passing to the inductive 2-limits, we get the desired equivalence $\Phi_\Theta$.

We verify the properties.

Properties \eqref{pw:st} and \eqref{pw:rest} are proved similarly as in Theorem \ref{th:emduality}.

Property (\ref{pw:adj}). We only need to construct a natural isomorphism $\Av^\Theta_\chi\circ\Phi\Rightarrow\Phi_\Theta\circ\pi^\Theta_*$, and the other natural isomorphisms $\Av^\Theta_!\circ\Phi_\Theta\Rightarrow\Phi\circ\pi^{\Theta,*}$ and $\Av^\Theta_*\circ\Phi_\Theta\Rightarrow\Phi\circ\pi^!_\Theta$ follow from the adjunctions. For each $w\in W$, we define the $(E_{\Theta,\leq\barw},E_{\leq w})$-bimodule
\begin{equation*}
\Pi_{\Theta,\leq w}:=\bigoplus_{u\in[\coW],\baru\leq\barw,v\leq w}\Ext^\bullet_{\scE_\Theta}(\pi^\Theta_*\IC_u,\pi^\Theta_*\IC_v).
\end{equation*}
Similarly, we define the $(M_{\Theta,\leq\barw},M_{\leq w})$-bimodule
\begin{equation*}
A_{\Theta,\leq w}:=\bigoplus_{u\in[\coW],\baru\leq\barw,v\leq w}\Hom_{\hsP_\Theta}(\Av^\Theta_\chi\tcT_u,\Av^\Theta_\chi\tcT_v)^f.
\end{equation*}
By Proposition \ref{p:Cfun}, the transport of $\pi^\Theta_*$ as a functor $D_{\perf}(E_{\leq w},\Frob)\to D_{\perf}(E_{\Theta,\leq\barw},\Frob)$ takes the form
\begin{equation*}
L\mapsto \Pi_{\Theta,\leq w}\Ltimes_{E_{\leq w}}L;
\end{equation*}
while the transport of $\Av^\Theta_\chi$ as a functor $D_{\perf}(M_{\leq w},\Frob)\to D_{\perf}(M_{\Theta,\leq\barw},\Frob)$ takes the form
\begin{equation*}
N\mapsto A_{\Theta,\leq w}\Ltimes_{E_{\leq w}}N.
\end{equation*}
By Lemma \ref{l:piav} and Lemma \ref{l:comcom}, the isomorphism
\begin{equation*}
\bigoplus_{u\in[\coW],\baru\leq\barw,v\leq w}\psi(\IC_u,\IC_v):\Pi_{\Theta,\leq w}^\Finv\isom A_{\Theta,\leq w}
\end{equation*}
intertwines the $(E_{\Theta,\leq\barw},E_{\leq w})$-bimodule structure on $\Pi_{\Theta,\leq w}$ and the $(M_{\Theta,\leq\barw},M_{\leq w})$-bimodule structure on $A_{\Theta,\leq w}$. Therefore this isomorphism induces a natural isomorphism $\Av^\Theta_\chi\circ\Phi\Rightarrow\Phi_\Theta\circ\pi^\Theta_*$.

Property (\ref{pw:modcat}). The verification is a combination of the argument for Theorem \ref{th:emduality}(\ref{em:monoidal}) and the Property (\ref{pw:adj}) above. The essential step is to verify that for $w\in[\coW], w'\in W$ such that $\ell(ww')=\ell(w)+\ell(w')$, the isomorphism $\bigoplus\psi(\IC_u,\IC_v\conv{B}\IC_{v'})$:
\begin{equation*}
\bigoplus\Ext^\bullet_{\scE_\Theta}(\pi^\Theta_*\IC_u,\pi^\Theta_*\IC_v\conv{B}\IC_{v'})^\Finv\isom\bigoplus\Hom_{\hsP_\Theta}(\Av^\Theta_\chi\tcT_u,\Av^\Theta_\chi\tcT_v\conv{U}\tcT_{v'})^f
\end{equation*}
(where the direct sum is over $\{u,v\in[\coW],v'\in W|\baru\leq\barw w',\barv\leq\barw,v'\leq w'\}$) intertwines the $(E_{\Theta,\leq\barw w'},E_{\Theta,\leq\barw}\otimes E_{\leq w'})$-module structure and the $(M_{\Theta,\leq\barw w'},M_{\Theta,\leq\barw}\otimes M_{\leq w'})$-module structure. Details are left to the reader.
\end{proof}


\subsection{``Paradromic-Whittavariant'' duality}\label{ss:finaldual}
Fix $\Theta\subset\Sigma$ of finite type. Let $\Dright_\Theta:=D^b_{m}((U^\Theta U^-_{\Theta},\chi)\backslash G/B)$ be the derived category of $(V_\Theta,\chi)$-equivariant mixed complexes on $\Fl_G=G/B$, which we call the ``Whittavariant'' category, taking a portmanteau of the words ``Whittaker'' and ``equivariant''.

On the other hand, let $\Dleft_\Theta:=D^b_{m}(P^\vee_\Theta\backslash \dG/\dU)$ be the derived category of right $\dU$-equivariant mixed complexes on the partial flag variety $\PFl_\dG=P^\vee_\Theta\backslash\dG$. Since objects in $\Dleft_\Theta$ are automatically monodromic under the right $H$-action, we call $\Dleft_\Theta$ the ``paradromic'' category, taking a portmanteau of the words ``parabolic'' and ``monodromic''.

Just as we deduced the self-duality from the equivariant-monodromic duality, we can also deduce the following theorem from the parabolic-Whittaker duality. We omit the proof.

\begin{theorem}[Paradromic-Whittavariant duality]\label{th:pwselfduality} Let $\Theta\subset\Sigma$ be of finite type. Then there is an equivalence of triangulated categories
\begin{equation*}
\Psi_\Theta:\Dleft_\Theta\isom\Dright_\Theta
\end{equation*}
satisfying the following properties:
\begin{enumerate}
\item\label{pwsd:modcat} $\Psi_\Theta$ can be given a structure to intertwine the right convolution of $\hsM_\dG$ on $\Dleft_\Theta$ and the right convolution of $\scE_G$ on $\Dright_\Theta$ (via the equivalence $\Phi:\scE_G\isom\hsM_\dG$ in Theorem \ref{th:emduality});
\item\label{pwsd:adj} There are natural isomorphisms which intertwine the following adjunctions (which are defined in a similar way as the adjunctions (\ref{eq:paradj}) and (\ref{d:whadj}))
\begin{equation*}
\xymatrix{\Dleft\ar[r]\ar[d]_{\Psi} & \Dleft_\Theta\ar@<1ex>[l]^{\pi^{\Theta,!}}\ar@<-2ex>[l]_{\pi^{\Theta,*}}^{\pi^\Theta_*}\ar[d]^{\Psi_\Theta}\\
\Dright\ar[r] & \Dright_\Theta\ar@<1ex>[l]^{\Av^\Theta_*}\ar@<-2ex>[l]_{\Av^\Theta_!}^{\Av^\Theta_\chi}}
\end{equation*}
\item\label{pwsd:st} $\Psi_\Theta(\Delta^\vee_{\barw})\cong\Delta_{\barw,\chi}$ and $\Psi_\Theta(\nabla^\vee_{\barw})\cong\nabla_{\barw,\chi}$ for all $\barw\in\coW$;
\item\label{pwsd:puretilt} $\Psi_\Theta(\IC^\vee_{\barw})\cong\calT_{\barw,\chi}:=\piddag\tcT_{\barw,\chi}$. More generally, $\Psi_\Theta$ interchanges very pure complexes of weight 0 and tilting sheaves;
\item\label{pwsd:rest} The analogs of parts \eqref{sd:slinear} and \eqref{sd:twist} of Theorem \ref{th:selfduality} hold.
\end{enumerate}
\end{theorem}

\begin{cor}
\begin{enumerate}
\item []
\item For $w\in\{\coW\}$, we have isomorphism
\begin{eqnarray*}
\pi^{\Theta,*}\IC_{\barw}[\lTh](\lTh/2)&\cong&\IC_w\cong\pi^{\Theta,!}\IC_{\barw}[-\lTh](-\lTh/2)\\
\Av^\Theta_!\calT_{\barw,\chi}(-\lTh/2)&\cong&\calT_{w}\cong\Av^\Theta_*\calT_{\barw,\chi}(\lTh/2).
\end{eqnarray*}
\item For $w\in[\coW]$, $\pi^{\Theta}_*\calT^\vee_w=:\calT^\vee_{\barw}$ is a tilting extension of the constant perverse sheaf $\Ql[\ell(\barw)](\ell(\barw)/2)$ on $(P^\vee_\Theta\backslash\dG)_{\barw}$ and $\omega\calT_{\barw}$ is indecomposable. For $w\notin[\coW]$, $\pi^\Theta_*\calT^\vee_w=0$.
\item For $w\in[\coW]$, $\Av^\Theta_\chi\IC_w=:\IC_{\barw,\chi}$ is the middle extension of the simple perverse $(U^-_\Theta,\chi)$-equivariant local system $\calL_{\barw,\chi}$ on $(G/B)_{w_\Theta w}$. For $w\notin[\coW]$, $\Av^\Theta_\chi\IC_w=0$.
\item $\Psi_\Theta(\calT^\vee_{\barw})\cong\IC_{\barw,\chi}$.
\end{enumerate}
\end{cor}
\begin{proof}
(1) is proved in the same way as Corollary \ref{c:pullict}.

(2) follows from \cite[Proposition 3.4.1]{Yun2}. In particular, $\calT^\vee_{\barw}$ also satisfies the condition (W) mentioned in Remark \ref{r:condW}.

(3) Note that by Theorem \ref{th:pwselfduality}(\ref{pwsd:adj}) and Theorem \ref{th:selfduality}(\ref{sd:puretilt}), \begin{equation*}
\Av^\Theta_\chi\IC_w\cong\Av^\Theta_\chi\Psi(\calT^\vee_w)\cong\Psi_\Theta(\pi^\Theta_*\calT^\vee_w).
\end{equation*}
For $w\notin[\coW]$, $\Psi_\Theta(\pi^\Theta_*\calT^\vee_w)=0$ by part (2). For $w\in[\coW]$, we have $\Psi_\Theta(\pi^\Theta_*\calT^\vee_w)\cong\Psi_\Theta(\calT^\vee_{\barw})$ by part (2). Since $\omega\calT^\vee_{\barw}$ is an indecomposable tilting sheaf, $\omega\Psi_\Theta(\calT^\vee_{\barw})$ is an indecomposable very pure complex, i.e., $\Psi_\Theta(\calT^\vee_{\barw})$ is a shifted and twisted IC-sheaf. Since $\Delta^\vee_{\barw}\isom\calT^\vee_{\barw}$ (mod $\Dleft_{\Theta,<w}$), we have $\Delta_{\barw,\chi}\isom\Psi_\Theta(\calT^\vee_{\barw})$ (mod $\Dright_{\Theta,<w}$), therefore $\Psi_\Theta(\calT^\vee_{\barw})$ is the middle extension of $\calL_{\barw,\chi}$.

(4) follows from parts (2) and (3).
\end{proof}


\appendix
\vspace{1cm}
\begin{center}

\textbf{APPENDICES}\\ by Zhiwei Yun
\end{center}

\section{Completions of monodromic categories}\label{a:compmono}
The goal of this appendix is to make rigorous the procedure of ``adding \fm objects to the category of monodromic complexes''.

\subsection{Unipotently monodromic complexes}\label{as:mono} Let $k$ be an algebraically closed field. Let $A$ be an algebraic torus over $k$ and $X$ be a right $A$-torsor over a scheme $Y$ over $k$. Let $\pi:X\to Y$ be the projection.

\begin{defn}\label{def:mono}
The $A$-{\em unipotently monodromic category} of the torsor $\pi:X\to Y$ is the full subcategory of $D^b_c(X)$ generated by the image of $\pi^!:D^b_c(Y)\to D^b_c(X)$ (or equivalently $\pi^*$). We denote this full subcategory by $D^b_c(\qw{X}{A})$, and its objects are called {\em unipotently monodromic complexes}.
\end{defn}

Note that $D^b_c(\qw{X}{A})$ inherits the perverse t-structure from $D^b_c(X)$. We denote its heart by $P(\qw{X}{A})$.

Let $r=\dim A$. We use $(\piddag,\pidag)$ to denote the adjoint pair $(\pi_![r],\pi^![-r])$. Note that under the perverse t-structures, $\pidag$ is t-exact and $\piddag$ is right t-exact.

In \cite{V}, Verdier studied the monodromic complexes in the case $A=\GG_m$ and $X$ is a cone over $Y$. His argument extends to any split torus $A$. Verdier's notion of monodromic complexes allows arbitrary tame monodromy along the fibers of $\pi$ whereas our category $D^b_c(\qw{X}{A})$ only allows unipotent monodromy. Verdier's construction of the canonical monodromy operator in \cite[\S 5]{V} applies to our situation: for each object $\calF\in D^b_c(\qw{X}{A})$ there is an action $\mu(\calF)$ of the tame Tate module of $T^t(A)=\varprojlim_{(n,p)=1} A[n]$ on the underlying complex $\omega\calF$. For $\calF\in D^b_c(\qw{X}{A})$, this action necessarily factors through the $\ell$-adic quotient $T_\ell(A)$, and gives:
\begin{equation*}
\mu(\calF):T_\ell(A)\to\Aut_{X}(\calF).
\end{equation*}
It is shown in {\em loc.cit.} that these operators commute with all morphisms in $D^b_c(\qw{X}{A})$. Since $\calF$ has unipotent monodromy along the fibers of $\pi$, the operator $\mu(\calF)$ is unipotent. Therefore it makes sense to take the logarithm of $\mu(\calF)$ and get a morphism in $D^b_c(\qw{X}{A})$:
\begin{equation*}
m(\calF):=\log(\mu(\calF)):V_A\otimes\calF\to\calF.
\end{equation*}
where $V_A=T_\ell(A)\otimes_{\ZZ_\ell}\Ql$. These logarithmic monodromy operators also commute with all morphisms in $D^b_c(\qw{X}{A})$, and $D^b_c(\qw{X}{A})$ becomes a category enriched over $S=\Sym(V_A)$-modules.

Our goal is to enlarge the category $D^b_c(\qw{X}{A})$ to a category $\hatD^{b}_{c}(\qw{X}{A})\subset\pro D^{b}_{c}(\qw{X}{A})$, by adding certain pro-objects called ``\fm'' objects. The prototypical example of such a \fm objects is the following.
\begin{exam}\label{ex:fm}
Let $Y=\pt$ and $X=A$. We will construct a pro-object in the category of unipotently monodromic local systems on $A$, called the {\em \fm local system}. Recall that a $\Ql$-local system on $A$ is given by a finite dimensional continuous $\Ql$-representation $\rho$ of $\pi_1(A,e)$. Such a local system is unipotently monodromic (i.e., being a successive extension of sheaves pulled back from $D^b_c(\pt)$) if and only if it is unipotent and hence factors through the $\ell$-adic tame quotient
\begin{equation*}
\pi_1(A,e)\twoheadrightarrow\pi^\ell_1(A,e)\cong T_\ell(A).
\end{equation*}
Let $\calL_n$ be the local system on $A$ given by the representation $\rho_{n}=\Sym(V_A)/(V_A^{n+1})$, on which an element $t\in T_\ell(A)$, viewed as an element of $V_A$, acts as multiplication by $\exp(t)$. Let
\begin{equation*}
\tcL:=\prolim\calL_n\in\pro D^{b}_{c}(\qw{A}{A}).
\end{equation*}
This pro-object is a typical example of a \fm local system.

Let $\hatS=\varprojlim_{n}\Sym(V_{A})/(V_{A}^{n+1})$. It is easy to see that we have an equivalence
\begin{equation*}
D^{b}_{c}(\qw{A}{A})\cong D^{b}(\Mod^{\nil}(\hatS)).
\end{equation*}
Here $\Mod^{\nil}(\hatS)$ stands for the abelian category of finite dimensional $\hatS$-modules. The \fm completion would be
\begin{equation*}
\hatD^{b}_{c}(\qw{A}{A})\cong D^{b}(\hatS),
\end{equation*}
of the bounded derived category of all finitely generated $\hatS$-modules. If we normalize the equivalence so it is t-exact with respect to the perverse t-structure on the LHS and the natural t-structure on the RHS, then under this equivalence, the pro-object $\tcL[r]\in\hatD^{b}_{c}(\qw{A}{A})$ corresponds to the free module $\hatS\in D^{b}(\hatS)$.
\end{exam}

\begin{remark} To better understand the situation, we consider the case $Y$ is smooth and $k=\CC$.  Then there is a parallel story for holonomic $\calD_{X}$-modules instead of constructible complexes, linked to each other via the Riemann-Hilbert correspondence. 

We recall some basic construction of \cite[\S2.5]{BB}. A weakly $A$-equivariant $\calD_{X}$-module is a quasi-coherent sheaf $\calM$ on $X$ with a $\calD_{X}$-action together with an $A$-action such that the action map $\calD_{X}\otimes_{\calO_{X}} \calM\to\calM$ is $A$-equivariant. 

Let $\Theta_{Y}$ be the tangent bundle of $Y$. Let $\Theta_{\qw{X}{A}}$ is the vector bundle on $Y$ which is the descent of the tangent bundle of $X$. It is a Lie algebroid on $Y$ and fits into an exact sequence
\begin{equation*}
0\to\Lie A\otimes\calO_{Y}\to\Theta_{\qw{X}{A}}\to\Theta_{Y}\to0.
\end{equation*}
Let $\calD_{\qw{X}{A}}\subset\calD_{X}$ be the $A$-invariant part. This is a sheaf of $\calO_{Y}$-algebras generated by the Lie algebroid $\Theta_{\qw{X}{A}}$. The functor $\calM\mapsto\calM^{A}$ gives an equivalence between weakly $A$-equivariant $\calD_{X}$-modules and $\calD_{\qw{X}{A}}$-modules which are quasi-coherent on $Y$. 

Note that $\Lie A\subset\Theta_{\qw{X}{A}}\subset\calD_{\qw{X}{A}}$ is actually central, hence $S=\Sym(\Lie A)$ is a central subalgebra of $\calD_{\qw{X}{A}}$. Localizing the category of $\calD_{\qw{X}{A}}$ at various points of $(\Lie A)^{*}=\Spec S$ corresponds to specifying the monodromy along the fibers of $X\to Y$ under the Riemann-Hilbert correspondence. Thus an $A$-unipotently monodromic $\calD_{X}$-module is the same as a quasi-coherent $\calD_{\qw{X}{A}}$-module on which $\Lie A$ acts nilpotently. 

Let $\hatS$ be the completion of $S$ with respect to the ideal $(\Lie A)$, and we define the completion
\begin{equation*}
\widehat{\calD}_{\qw{X}{A}}:=\hatS\otimes_{S}\calD_{\qw{X}{A}},
\end{equation*}
equipped with the $(\Lie A)$-adic topology. The desired completion of the derived category of $\calD_{X}$-modules can then be defined as the derived category of certain $\widehat{\calD}_{\qw{X}{A}}$-modules.
\end{remark}

The above examples suggest that $\hatD^{b}_{c}(\qw{X}{A})$ should look like a tensor product $D^{b}_{c}(\qw{X}{A})\otimes_{S}\hatS$, i.e., the category of $\hatS$-modules in $D^{b}_{c}(\qw{X}{A})$. Turning this into a rigorous construction involves two technical difficulties.

First, we need to deal with such categorical issues as: how to extend the triangulated category structure and the t-structure to the completed category; how to extend the sheaf-theoretic functors to the completed categories? The general categorical formalism for dealing with pro-completions is contained in \S\ref{progen}, and is applied to our situation in \S\ref{ss:compmono}. 

Second, to make sense of $\hatS$-modules in $D^{b}_{c}(Y)$ we need to work on the level of abelian categories (perverse sheaves). Even in the situation where $X=Y\times A$ this is not obvious, see \S\ref{ss:triv}. Extra care has to be taken in the mixed setting, see \S\ref{ss:mix}. Finally, we deal with the case where $Y$ is nicely stratified and the category in consideration is glued from simple categories coming from each stratum, see \S\ref{ss:stra}.

The ultimate goal is to define and study \fm tilting sheaves, which will be done in \S\ref{ss:fmt}. The \fm tilting sheaves will play important roles in constructing DG models for the completed categories, as we will see in Appendix \ref{a:dgmodel}.

\subsection{Pro-objects in a filtered triangulated category}\label{progen}
Let $D$ be a category and let $\pro(D)$ be the category of pro-objects in $D$. By definition, objects in $\pro(D)$ are sequences of objects (indexed by non-negative integers) $\{X_n\}_{n\geq0}$ with transition maps $\cdots\to X_2\to X_1\to X_0$. We denote such a sequence by $\prolim X_n$. The morphism sets are defined by
\begin{equation}\label{dlimit}
\Hom_{\pro(D)}(\prolim X_m,\prolim Y_n)=\varprojlim_n\varinjlim_m\Hom_D(X_m,Y_n).
\end{equation}
The Yoneda embedding of $\eta_D:D\to\Fun(D,\Set)$ extends to pro-objects to give an embedding
\begin{eqnarray}\label{Yoneda}
\widehat{\eta}_D:\pro(D)&\to&\Fun(D,\Set)\\
\notag
\prolim X_n&\mapsto&\left(Y\mapsto\varinjlim_n\Hom_D(X_n,Y)\right).
\end{eqnarray}
It is easy to check that $\widehat\eta_D$ is a full embedding.

For any partially ordered set $I$, viewed as a category, we can consider $\Fun(I,D)$ as ``diagrams of shape $I$'' in $D$. In particular, for $n\geq0$, let $[0,n]$ be the ordered set $n>\cdots>0$. Then $\Fun([0,n],D)$ is the category of chains of morphisms $X_n\to\cdots\to X_0$. In particular,  $\Fun([0,1],D)$ is the category of morphisms in $D$.

\begin{lemma}\label{l:hompro}
Let $I$ be a countable partially ordered set in which every element $i$ has only finitely many successors (a successor is an element $j<i$). Then the natural functor
\begin{equation*}
\Pi_I:\pro\Fun(I,D)\to\Fun(I,\pro(D))
\end{equation*}
is an equivalence of categories.
\end{lemma}
\begin{proof}
We first prove $\Pi_I$ is essentially surjective. We write $I=\bigcup_{N}I_N$ where $I_1\subset I_2\subset\cdots\subset I_N\subset\cdots$, each $I_N$ has cardinality $N$ and is closed under successors. We use induction on $N$ to show that each $\Pi_{I_N}$ is essentially surjective, which suffices for our purpose.

Assume $\Pi_{I_{N-1}}$ is essentially surjective. For notational simplicity, we denote $I_N$ by $I$ and $I_{N-1}$ by $J$. Let $\{i_0\}=I\backslash J$. For any diagram $X:I\ni i\mapsto\prolim X(i)_n\in\pro(D)$, apply the inductive hypothesis to its restriction to $J$, we get a projective system $\{Y_n:J\ni j\mapsto Y_n(j)\}$ and an isomorphism $\alpha:\Pi_{}(\prolim Y_n)\isom X|_J$. Since each $\alpha(j):\prolim Y_n(j)\to\prolim X(j)_n$ is an isomorphism, the maps $\prolim X(i_0)_n\to\prolim X(j)_n$ for $i_0>j$ naturally lifts to $f(i_0,j):\prolim X(i_0)_n\to \prolim Y_n(j)$. By choosing a subsequence of $\{X(i_0)_{a(n)}\}$ of $\{X(i_0)_n\}$, we can manage so that $f(i_0,j)$ comes from a projective system of maps $f(i_0,j)_n:X(i_0)_{a(n)}\to Y_n(j)$. By possibly passing to another subsequence of $\{X(i_0)_{a(n)}\}$, we can make sure that for each fixed $n$, adding $Y_n(i_0):=X(i_0)_{a(n)}$ and $\{f(i_0,j)_n:Y_n(i_0)\to Y_n(j)\}_{i_0<j}$ extends the diagram $Y_n:J\ni j\mapsto Y_n(j)$ into a diagram $\tilY_n:I\in i\mapsto Y_n(i)$. As $n$ varies, these diagrams form a projective system $\{\tilY_n\}_{n}$ in $\Fun(I,D)$. It is clear that the natural isomorphism $\prolim Y_n(i_0)=\prolim X(i_0)_{a(n)}\isom\prolim X(i_0)_n$ together with $\alpha$ extends to an isomorphism $\tilalpha:\Pi_I(\prolim\tilY_n)\isom X$. This completes the induction step.

We next prove that $\Pi_I$ is injective on morphism sets. Let $\{Y_n:I\to D\},\{Z_n:I\to D\}$ be two objects in $\pro\Fun(I,D)$. Then their Hom-set in both $\pro\Fun(I,D)$ and $\Fun(I,\pro D)$ can be naturally identified with subsets of
$$\varprojlim_n\varinjlim_m\prod_{i\in I}\Hom_D(Y_m(i),Z_n(i))=\prod_{i\in I}\varprojlim_n\varinjlim_m\Hom_D(Y_m(i),Z_n(i)).$$
From this we conclude that $\Pi_I$ is injective on $\Hom$-sets.

To prove that $\Pi_I$ is surjective on morphism sets, it suffices to show that
\begin{equation*}
\Fun([0,1],\Pi_I):\Fun([0,1],\pro\Fun(I,D))\to\Fun([0,1],\Fun(I,\pro(D)))
\end{equation*}
is essentially surjective. Consider the commutative diagram of functors
\begin{equation*}
\xymatrix{\pro\Fun([0,1],\Fun(I,D))\ar[r]^{\Pi_{[0,1]}}\ar[dd]^{\adj}_{\wr} & \Fun([0,1],\pro\Fun(I,D))\ar[d]^{\Fun([0,1],\Pi_I)}\\
 & \Fun([0,1],\Fun(I,\pro(D)))\ar[d]^{\adj}_{\wr}\\
\pro\Fun([0,1]\times I,D])\ar[r]^{\Pi_{[0,1]\times I}} & \Fun([0,1]\times I,\pro(D))}
\end{equation*}
where ``$\adj$'' is the adjunction equivalence between the Cartesian product $\times$ and $\Fun$. The essential surjectivity of $\Fun([0,1],\Pi_I)$ then follows from that of $\Pi_{[0,1]\times I}$.
\end{proof}

For a category $D$ with a shift functor $[1]$, let $\tTri(D)$ denote the category of {\em all} triangles in $D$, i.e., chains of morphisms $X\xrightarrow{f}Y\xrightarrow{f'}Z\xrightarrow{f''} X[1]\xrightarrow{f[1]}Y[1]\cdots$ such that the composition of any two consecutive arrows is zero.

Now suppose $D$ is a triangulated category equipped with a shift functor $[1]$ and a category of distinguished triangles $\Tri(D)\subset\tTri(D)$. We clearly have a functor $\gamma:\pro(\Tri(D))\to\tTri(\pro(D))$. Let $\Tri(\pro(D))$ be the essential image of $\gamma$. However, in general there is no guarantee that the distinguished triangles defined in such a way should give a triangulated structure on $\pro(D)$. In fact, the octahedral axiom does not hold automatically for situations arising from pro-objects, because no condition was imposed on the {\em morphisms} between octahedra. We will resolve this difficulty with the help of a filtered structure of $D$. For basic definitions and notations of a filtered structure on a triangulated category, see \cite[Appendix A]{B}. We will use the decreasing version of filtered categories, and use $F^{\geq n}$ to mean ``the $n$-th filtration'' and use $F^{\leq n}$ to mean ``quotient by $F^{\geq n+1}$''. Let $DF$ be a filtered triangulated category over $D$ with the ``forgetting filtration'' functor $\Omega:DF\to D$. We have the ``taking the associated graded'' functors:
\begin{equation*}
\Gr_F^n:DF\to D.
\end{equation*}
For any interval of integers $[m,n]$, let $DF^{[m,n]}$ be the full subcategory of $DF$ consisting of objects $X$ such that $\Gr_F^i(X)=0$ unless $m\leq i\leq n$. In particular, we can identify $D$ with $DF^{[0,0]}$.

For each $[m,n]$, we have a functor
\begin{equation*}
\Omega^{[m,n]}:DF^{[m,n]}\to\Fun([m,n],D)
\end{equation*}
sending $X\in DF^{[m,n]}$ to the diagram $F^{\geq n}X\to\cdots\to F^{\geq m}X=X$.

For $n=1$, $\Omega^{[0,1]}$ can be lifted to a functor
\begin{equation*}
\Omega^{\Tri}:DF^{[0,1]}\to\Tri(D)
\end{equation*}
which sends $X\in DF^{[0,1]}$ to the distinguished triangle $F^{\geq1}X\to X\to F^{\leq0}X\to F^{\geq1}X[1]$. Therefore $\Omega^{[0,1]}$ is the composition of $\tau\circ\Omega^{\Tri}$ where $\tau:\Tri(D)\to\Fun([0,1],D)$ is ``forgetting the third vertex of a triangle''.

Let $\Oct(D)$ be the category of octahedra in $D$. We recall that an octahedron is a commutative diagram of the form
\begin{equation}\label{oct}
\xymatrix{X\ar[dr]^{f}\ar@/^2pc/[rr]^{h} & & Z\ar[dr]^{h'}\ar@/^2pc/[rr]^{g'} & & W\ar[dr]^{g''}\ar@/^2pc/[rr]^{j''} & & U[1]\\
& Y\ar[dr]^{f'}\ar[ur]^{g} & & V\ar[dr]^{h''}\ar[ur]^{j'} && Y[1]\ar[ur]^{f'[1]}\\
&& U\ar[ur]^{j}\ar@/_2pc/[rr]^{f''} && X[1]\ar[ur]^{f[1]}}
\end{equation}
where $(f,f',f''), (g,g',g''),(h,h',h'')$ and $(j,j',j'')$ are distinguished triangles. There is an obvious notion of morphisms between octahedra.

The functor $\Omega^{[0,2]}$ can be lifted to a functor
\begin{equation*}
\Omega^{\Oct}:DF^{[0,2]}\to\Oct(D)
\end{equation*}
which sends $X\in DF^{[0,2]}$ to the octahedron
\begin{equation*}
\xymatrix@C=0pt{\Gr^2_FX\ar[dr]\ar@/^1pc/[rr] & & X\ar[dr]\ar@/^1pc/[rr] & & \Gr^0_FX\ar[dr]\ar@/^1pc/[rr] & & \Gr^1_FX[1]\\
& F^{\geq1}X\ar[dr]\ar[ur] & & F^{\leq1}X\ar[dr]\ar[ur] && F^{\geq1}X[1]\ar[ur]\\
&& \Gr^1_FX\ar[ur]\ar@/_2pc/[rr] && \Gr^2_FX[1]\ar[ur]}
\end{equation*}
so that $\Omega^{[0,2]}$ is the composition of $\Omega^{\Oct}$ with the functor of ``remembering the top left commutative triangle only''.


Let $D$ be a category equipped with a shift functor $[1]$ and distinguished triangles $\Tri(D)$. A strictly full subcategory $D'\subset D$ is said to be {\em triangle-complete}, if it is stable under $[1]$, and for any triangle $X\to Y\to Z\to$ in $\Tri(D)$, if two of the vertices are in $D'$, then so is the third. When $D'$ is triangle-complete, we define its distinguished triangles $\Tri(D')$ to be those in $\Tri(D)$ with all vertices in $D'$.

\begin{theorem}\label{th:proD} Let $k$ be a field and $D$ be a $k$-linear triangulated category with a filtered lifting $DF$. Let $D'\subset D$ be a full triangulated subcategory. Equip $\pro(D')$ with the shift functor $[1]$ induced from that of $D'$ and the distinguished triangles $\Tri(\pro(D'))\subset\tTri(\pro(D'))$ (recall this is the essential image of $\gamma:\pro(\Tri(D'))\to\tTri(\pro(D'))$). Let $\hatD\subset\pro(D')$ be a triangle-complete full subcategory. Assume:
\begin{enumerate}
\item[(P-1)] $\pro(\Omega^{[0,2]}):\pro(DF^{[0,2]})\to\pro\Fun([0,2],D)$ is essentially surjective.
\item[(P-2)] For any two objects $\prolim X_m\in\hatD$ and $Y_n\in D'$, $\varinjlim_m\Hom_{D'}(X_m,Y_n)$ is a finite dimensional vector space over $k$.
\end{enumerate}
Then $\hatD$ with the induced shift functor $[1]$ and distinguished triangles $\Tri(\hatD)$ is a triangulated category.
\end{theorem}

\begin{proof}
(1) For the axioms (TR1)-(TR4) of a triangulated category, we refer to Verdier's original article \cite[Chapitre I, \S 1, 1-1]{VCD}. We first check (TR1). The only thing we need to show is that any morphism $\prolim X_n\to\prolim Y_n$ in $\hatD$ extends to a distinguished triangle. We will prove this for any morphism in $\pro(D)$. Since $\hatD\subset\pro(D')\subset\pro(D)$ is triangle-complete, (TR1) then holds for $\hatD$. Consider the commutative diagram
\begin{equation}\label{funcsq}
\xymatrix{\pro\Tri(D)\ar[r]^{\pro(\tau)}\ar[d] & \pro\Fun([0,1],D)\ar[d]^(.4){\Pi_1}\\
\Tri(\pro(D))\ar[r]^(.4){T} & \Fun([0,1],\pro(D))}
\end{equation}
where $\tau,T$ are ``forgetting the third vertex'' functors. We would like to show that $T$ is essentially surjective. Axioms (TR1) for $D$ implies that $\tau$ is essentially surjective; axiom (TR2) for $D$ implies that $\tau$ is surjective on morphism sets. These two facts together imply that $\pro(\tau)$ is essentially surjective. By Lemma \ref{l:hompro}, $\Pi_1$ is an equivalence. Therefore $\Pi_1\circ\pro(\tau)$, hence $T$, is essentially surjective.

The axiom (TR2) is obvious because $\pro\Tri(D)$ is stable under rotation of triangles.

We next check (TR3). Note that by \cite[Lemma 2.2]{May}, this axiom is implied by the other axioms. We still verify it here because we will need it to check (TR4). Note that (TR3) is equivalent to saying that
\begin{equation}\label{hattau}
\hattau=T|_{\Tri(\hatD)}:\Tri(\hatD)\to\Fun([0,1],\hatD)
\end{equation}
is surjective on morphism sets.

Consider a diagram in $\hatD$
\begin{equation}\label{mortri}
\xymatrix{\prolim X_n\ar[d]^{\xi}\ar[r]^{f} & \prolim Y_n\ar[d]^{\eta}\ar[r]^{g} & \prolim Z_n\ar[r]^{h} & \prolim X_n[1]\ar[d]^{\xi[1]}\\
\prolim X'_n\ar[r]^{f'} & \prolim Y'_n\ar[r]^{g'} & \prolim Z'_n\ar[r]^{h'} & \prolim X'_n[1]}
\end{equation}
where the rows are distinguished triangles. We would like to find a morphism $\zeta:\prolim Z_n\to\prolim Z_n'$ making all the squares commutative. By the definition of $\Tri(\pro D')$, the two rows in \eqref{mortri} are pro-objects in $\Tri(D')$, i.e., $f=\prolim f_n$, $g=\prolim g_n$, etc. However, the morphisms $\xi$ and $\eta$ are morphisms in $\pro(D')$ as in \eqref{dlimit}: for fixed $n$, we have an inductive system $\xi_{m,n}:X_m\to X'_n$ compatible with transition maps $X_{m+1}\to X_m$ for large $m$. These inductive systems form a projective system as $n$ varies ,i.e., $\xi=\varprojlim_n\varinjlim_m\xi_{m,n}$. Similarly we have $\eta=\varprojlim_n\varinjlim_m\eta_{m,n}$.

Fix $n\geq0$. Then for $m$ large enough, $\xi_{m,n}$ and $\eta_{m,n}$ are defined and $f'_n\xi_{m,n}=\eta_{m,n}f_m$. By (TR3) for $D'$, the set of dotted arrows making the following diagram commutative
\begin{equation*}
\xymatrix{X_m\ar[d]^{\xi_{m,n}}\ar[r]^{f_m} & Y_m\ar[d]^{\eta_{m,n}}\ar[r]^{g_m} & Z_m\ar[r]^{h_m}\ar@{-->}[d] & X_m[1]\ar[d]^{\xi_{m,n}[1]}\\
X'_n\ar[r]^{f'_n} & Y'_n\ar[r]^{g'_n} & Z'_n\ar[r]^{h'_n} & X'_n[1]}
\end{equation*}
form a torsor $E_{m,n}$ under a subspace $H_{m,n}\subset\Hom_{D'}(Z_m,Z_n')$. The set $E$ of morphisms $\zeta:\prolim Z_n\to\prolim Z'_n$ making \eqref{mortri} commutative can thus be expressed as
\begin{equation*}
E=\varprojlim_n\varinjlim_m E_{m,n}.
\end{equation*}
Each $E_{\infty,n}=\varinjlim_m E_{m,n}$ is a torsor under $H_{\infty,n}=\varinjlim_m H_{m,n}$. By assumption (P-2), $H_{\infty,n}\subset\varinjlim_m\Hom_{D'}(Z_m,Z_n')$ is finite dimensional over $k$. Hence the projective system $\{E_{\infty,n}\}_{n\geq0}$ is Mittag-Leffler. Since each $E_{\infty,n}$ is non-empty, so is $E=\varprojlim_n E_{\infty,n}$. This proves the existence of $\zeta\in E$.

An easy consequence of (TR3) is that $\hattau$ in \eqref{hattau} is conservative.

Finally we check the octahedral axiom (TR4). For any category $C$ equipped with $[1]$ and distinguished triangles $\Tri(C)$, we can define the category $\Oct(C)$ as in \eqref{oct} and its relative $\preOct(C)$ called the category of pre-octahedra. An object in $\preOct(C)$ is a commutative diagram:
\begin{equation*}
\xymatrix{X\ar[dr]^{f}\ar@/^2pc/[rr]^{h} & & Z\ar[dr]^{h'}\ar@/^2pc/[rr]^{g'} & & W\ar[dr]^{g''} & & U[1]\\
& Y\ar[dr]^{f'}\ar[ur]^{g} & & V\ar[dr]^{h''} && Y[1]\ar[ur]^{f'[1]}\\
&& U\ar@/_2pc/[rr]^{f''} && X[1]\ar[ur]^{f[1]}}
\end{equation*}
such that $(f,f',f''),(g,g',g'')$ and $(h,h',h'')$ are in $\Tri(C)$. We have forgetful functors $\Oct(C)\xrightarrow{\alpha}\preOct(C)\xrightarrow{\beta}\Fun([0,2],C)$ whose composition only remembers the top left commutative triangle of the octahedron. In particular, we can define $\Oct(\pro(D)),\preOct(\pro(D)),\Oct(\hatD)$ and $\preOct(\hatD)$.

Consider the following diagram (which is commutative with obvious choices of natural transformations)
\begin{equation}\label{funoct}
\xymatrix{\pro(DF^{[0,2]})\ar[d]^{\pro(\Omega^{\Oct})}\ar[drr]^{\pro(\Omega^{[0,2]})}\\
\pro\Oct(D)\ar[d]\ar[r]_{\pro(\alpha)} & \pro\preOct(D)\ar[d]\ar[r]_{\pro(\beta)} & \pro\Fun([0,2],D)\ar[d]^{\Pi_2}_{\wr}\\
\Oct(\pro(D))\ar[r]^{A} & \preOct(\pro(D))\ar[r]^{B} & \Fun([0,2],\pro(D))}
\end{equation}
Axiom (TR4) for $\hatD$ is the same as saying that
\begin{equation*}
\hatalpha=A|_{\Oct(\hatD)}:\Oct(\hatD)\to\preOct(\hatD)
\end{equation*}
is essentially surjective.

By Lemma \ref{l:hompro}, $\Pi_2$ is an equivalence. By (P-1), $\pro(\Omega^{[0,2]})$ is essentially surjective. Hence the composition $B\circ A$ is essentially surjective. Let $\hatbeta$ be the restriction of $B$ to $\preOct(\hatD)$. Then the composition
\begin{equation*}
\hatbeta\circ\hatalpha:\Oct(\hatD)\xrightarrow{\hatalpha}\preOct(\hatD)\xrightarrow{\hatbeta}\Fun([0,2],\hatD)
\end{equation*}
is also essentially surjective, because if a commutative triangle in $\hatD$ can be completed into a octahedron, the vertices of the octahedron must all belong to $\hatD$ by triangle-completeness.

To recover an object in $\preOct(\pro(D))$ from its image in $\Fun([0,2],\pro(D))$, one only needs to construct distinguished triangles from the three arrows, i.e., the following is a pullback diagram:
\begin{equation*}
\xymatrix{\preOct(\hatD)\ar[d]\ar[r]^{\hatbeta} & \Fun([0,2],\hatD)\ar[d]\\
\Tri(\hatD)^3\ar[r]^{{\hattau}^3} & \Fun([0,1],\hatD)^3}
\end{equation*}
Axiom (TR3) for $\hatD$ implies that $\hattau$ is surjective on morphism sets and conservative, hence $\hatbeta$ is also surjective on morphism sets and conservative. We already proved that $\hatbeta\circ\hatalpha$ is essentially surjective, therefore $\hatalpha$ is also essentially surjective. This verifies (TR4). The proof is now complete.
\end{proof}

\begin{remark}\label{r:filD}
We also have a filtered version of Theorem \ref{th:proD}. Under the same assumption as Theorem \ref{th:proD}, let $\hatD F\subset\pro DF$ be the full subcategory consisting of objects $\prolim X_n$ such that $\prolim\Gr^i_FX_n\in\hatD$ and the filtrations of $X_n$ are uniformly bounded: i.e., there exists $N\in\ZZ_{\geq0}$ such that $\Gr^i_FX_n=0$ for all $n$ and any $i\neq[-N,N]$. Then it is easy to see that $\hatD F$ is a filtered triangulated category.
\end{remark}

\begin{exam}\label{sch}
Let $X$ be a scheme over $k$ ($k$ is a finite field or an algebraically closed field). Let $\Lambda$ be a coefficient ring, for example, $\Lambda=\FF_\ell,\ZZ/\ell^n\ZZ,\ZZ_\ell,\QQ_\ell$ or $\Ql$, with $\ell\neq\textup{char}(k)$. The derived category $D^b_c(X,\Lambda)$ is equipped with a filtered structure $D^b_cF(X,\Lambda)$ (see \cite[\S 1.1.2]{Del}).

We claim that $\pro\Omega^{[m,n]}(\Lambda):\pro(DF^{[m,n]}(X,\Lambda))\to\pro\Fun([m,n],D^b_c(X, \Lambda))$ is essentially surjective; i.e., the assumption (P-1) in Theorem \ref{th:proD} is satisfied for $D=D^b_c(X,\Lambda)$.

We first assume that $\Lambda$ is a finite ring. Let $K^b_c(X,\Lambda)$ and $K^b_cF(X,\Lambda)$ be the homotopy categories of $C^b_c(X,\Lambda)$ (constructible $\Lambda$-complexes) and $C^b_cF(X,\Lambda)$ (filtered constructible $\Lambda$-complexes). Then the forgetful functor $\Omega^{[m,n]}(\Lambda)$ admits a section, the ``telescoping functor'':
\begin{equation*}
\Tel=\Tel^{[m,n]}:\Fun([m,n],K^b_c(X,\Lambda))\to K^b_cF^{[m,n]}(X,\Lambda).
\end{equation*}
For a chain of complexes $K=[(K(n),d_n)\xrightarrow{f_n} (K(n-1), d_{n-1})\to\cdots\xrightarrow{f_{m+1}} (K(m),d_m)]$ and $m\leq i\leq n$, define
$$F^{\geq i}\Tel(K)=K(n)\oplus K(n)[1]\oplus K(n-1)\oplus K(n-1)[1]\oplus\cdots\oplus K(i+1)[1]\oplus K(i),$$ with differentials a signed sum of $d_j,d_j[1]$ and $f_j$. When $m=n-1$, $\Tel^{[m,n]}(K)$ is the mapping cylinder of $K(n)\xrightarrow{f_n}K(n-1)$.

Consider the following commutative diagram
\begin{equation}\label{KFD}
\xymatrix{\Fun([m,n],K^b_c(X,\Lambda))\ar[rr]^{\Tel^{[m,n]}}\ar[d]^{\Fun([m,n],Q(\Lambda))} && K^b_cF^{[m,n]}(X,\Lambda)\ar[d]_{QF(\Lambda)}\\
\Fun([m,n],D^b_c(X,\Lambda)) && D^b_cF^{[m,n]}(X,\Lambda)\ar[ll]_{\Omega^{[m,n]}(\Lambda)}}
\end{equation}
where $Q(\Lambda)$ and $QF(\Lambda)$ are natural quotient functors. Now take ``pro'' of the diagram \eqref{KFD}. It is easy to see that $\pro\Fun([m,n],Q(\Lambda))$ is essentially surjective, hence $\pro\Omega^{[m,n]}(\Lambda)$ is also  essentially surjective.

Now consider the case $\Lambda=R_\lambda$, a complete DVR with uniformizing parameter $\lambda$ and residue field $\FF_\lambda$, a finite field of characteristic $\ell\neq\textup{char}(k)$. Consider the projective system of diagrams \eqref{KFD} for $\Lambda=R_\lambda/(\lambda^n)$, $n=\ZZ_{\geq0}$ (by \cite{Del}, we should replace $D^b_c$ by constructible complexes with finite Tor-dimension, but we ignore this notational change). It is also easy to check that $\pro\Fun([m,n],\varprojlim_nQ(R_\lambda/(\lambda^n)))$ is essentially surjective (the argument is similar to that of Lemma \ref{l:hompro}). Therefore, taking $\pro\varprojlim_n$ of the diagram \eqref{KFD} for $R_\lambda/(\lambda^n)$, we conclude that $\pro\Omega^{[m,n]}(R_\lambda)$ is essentially surjective.

Finally we consider the case $\Lambda=E_\lambda=\Frac(R_\lambda)$. We claim that for any finite partially ordered set $I$, the natural functor
\begin{equation}\label{RtoE}
\pro\Fun(I,D^b_c(X,R_\lambda))\to\pro\Fun(I,D^b_c(X,E_\lambda))
\end{equation}
is essentially surjective. In fact, any projective system $\{K_n:I\to D^b_c(X,E_\lambda)\}_{n\geq0}$ can be viewed as a functor $K:[0,\infty)\times I\to D^b_c(X,E_\lambda)$, where $[0,\infty)\times I$ is equipped with the product partial order. For each index $\alpha\in[0,\infty)\times I$, $K(\alpha)$ is an object of $D^b_c(X,R_\lambda)$ by definition. For $\alpha>\beta$, the transition map $f^\alpha_\beta:K(\alpha)\to K(\beta)$ is a morphism in $D^b_c(X,E_\lambda)$. It is easy to see that there exists a sequence of integers $N_\alpha\in\ZZ_{\geq0}$, such that the assignment $\tilK=(K(\alpha),\tilf^\alpha_\beta=\lambda^{N_\alpha-N_\beta}f^\alpha_\beta)$ defines a functor $\tilK: [0,\infty)\times I\to D^b_c(X,R_\lambda))$, hence an object in $\pro\Fun(I,D^b_c(X,R_\lambda))$. Moreover the morphism $\tilK\to K$ defined by $\lambda^{N_\alpha}\id:K(\alpha)\to K(\alpha)$ gives an isomorphism in $\pro\Fun(I,D^b_c(X,E_\lambda))$.

The surjectivity of \eqref{RtoE} for $I=[m,n]$, together with the surjectivity of $\pro\Omega^{[m,n]}(R_\lambda)$, implies the surjectivity of $\pro\Omega^{[m,n]}(E_\lambda)$. The case $\Lambda=\Ql$ follows from the case $\Lambda=E_\lambda$ for various finite extensions $E_\lambda$ of $\QQ_\ell$.

\end{exam}

\subsection{The completion}\label{ss:compmono}
We recap the notation from \S\ref{as:mono}. We fix a full triangulated subcategory $D'(Y)\subset D^b_c(Y)$ with the induced perverse t-structure with heart $P'(Y)$. Let $\mon{X}\subset D^b_c(\qw{X}{A})$ be the full subcategory generated by $\pidag D'(Y)$, with the induced perverse t-structure with heart $\per{X}$.

\begin{defn}
\begin{enumerate}
\item []
\item An object $\prolim\calF_n\in\pro D^b_c(X)$ is called {\em $\pi$-constant} if $\prolim\piddag\calF_n\in\pro D^b_c(Y)$ is in the essential image of $D^b_c(Y)$.
\item An object $\prolim\calF_n\in\pro D^b_c(X)$ is called {\em uniformly bounded in degrees}, if it is isomorphic to $\prolim\calF'_n$ for which there exists $N\in\ZZ$ such that all $\calF'_n\in\pD^{[-N,N]}_c(X)$.
\item Let $\hmon{X}\subset\pro\mon{X}$ be the full subcategory of objects which are both $\pi$-constant and uniformly bounded in degrees.
\end{enumerate}
\end{defn}

\begin{theorem}\label{th:montri}
Let $\Tri(\hmon{X})\subset\Tri(\pro\mon{X})$ consist of those triangles whose vertices are in $\hmon{X}$. Then under the shift functor $[1]$ induced from $\pro\mon{X}$ and the distinguished triangles $\Tri(\hmon{X})$, $\hmon{X}$ becomes a triangulated category.
\end{theorem}
\begin{proof}
We would like to apply Theorem \ref{th:proD} to $D=D^b_c(X,\Ql)$, $D'=\mon{X}$ and $\hatD=\hmon{X}$. We first check that $\hmon{X}$ is a triangle-complete subcategory of $\pro\mon{X}$. For any triangle $\prolim(\calF_n\xrightarrow{f_n}\calG_n\xrightarrow{g_n}\calH_n\xrightarrow{h_n})$ in $\Tri(\pro\mon{X})$, suppose $\prolim\calF_n$ and $\prolim\calG_n$ are in $\hmon{X}$, we need to show that $\prolim\calH_n$ is also in $\hmon{X}$. Boundedness is clear, we only need to check that $\prolim\calH_n$ is also $\pi$-constant. Let $\calA,\calB\in D'(Y)$ with isomorphisms $\alpha:\calA\isom\prolim\piddag\calF_n$ and $\beta:\calB\isom\prolim\piddag\calG_n$, then we have a morphism $a:\calA\to\calB$ (which is the transport of $\prolim f_n$). Let $\calC$ be a cone of the map $a$. Consider the following diagram
\begin{equation*}
\xymatrix{\calA\ar[d]^{\alpha_n}\ar[r]^{a} & \calB\ar[d]^{\beta_n}\ar[r]^{b} & \calC\ar[r]^{c}\ar@{-->}[d] & \calA[1]\ar[d]^{\alpha_n[1]}\\
\piddag\calF_n\ar[r]^{\piddag(f_n)} & \piddag\calG_n\ar[r]^{\piddag(g_n)} & \piddag\calH_n\ar[r]^{\piddag(h_n)} & \calF_n[1]}
\end{equation*}
The choices of the dotted arrow form a torsor $E_n$ under a subgroup $H_n\subset\Hom_Y(\calC,\piddag\calH_n)$, which is a finite dimensional $\Ql$-vector space. Hence the projective system $\{E_n\}$ is Mittag-Leffler. Since each $E_n$ is nonempty, $\varprojlim E_n$ is also non-empty, i.e., we have a morphism $\gamma:\calC\to\prolim\piddag\calH_n$ making $(\alpha,\beta,\gamma)$ into a morphism of triangles. We claim that $\gamma$ is an isomorphism. In fact, we can check this by applying $\Hom_Y(-,T)$ to this morphism of triangles, for any test object $T\in D'(Y)$, using the long exact sequence of $\Hom$'s. This shows that $\prolim\calH_n$ is also $\pi$-constant, and completes the first step.

The assumption (P-1) is verified in Example \ref{sch}.

Finally, we check the assumption (P-2) for morphisms in $\hmon{X}$. Let $\prolim\calF_n,\prolim\calG_n\in\hmon{X}$, we now show that for fixed $n$, $\varinjlim_m\Hom_{X}(\calF_m,\calG_n)$ is a finite dimensional $\Ql$-vector space. This would then imply (P-2).

Since the functor $\varinjlim_m\bR\Hom_X(\calF_m,-)$ is an exact functor from $\mon{X}$ to the derived category of $\Ql$-vector spaces, it suffices to check that $\varinjlim_m\Hom_{X}(\calF_m,\calG)$ is finite dimensional for a set of generators $\calG$ of $\mon{X}$. So we may assume $\calG=\pidag\calH$ for some $\calH\in D'(Y)$. Then
\begin{eqnarray*}
\varinjlim_m\Hom_X(\calF_m,\pidag\calH)&=&\varinjlim_m\Hom_Y(\piddag\calF_m,\calH)\\
&=&\Hom_{\pro D'(Y)}(\prolim\piddag\calF_m,\calH)
\end{eqnarray*}
The $\pi$-constancy of $\prolim\calF_m$ means $\prolim\piddag\calF_m$ is isomorphic to an object in $D'(Y)$, therefore the above $\Hom$-set is a $\Hom$-set in $D'(Y)$, hence finite dimensional. This completes the proof.
\end{proof}

\begin{prop}\label{compfun}
Let $\pi_i:X_i\to Y_i$ be $A$-torsors ($i=1,2$). Let $D'(Y_i)\subset D^b_c(Y_i)$ be full triangulated subcategories. Suppose we have a commutative diagram of exact functors (i.e., we have a natural isomorphism $\alpha:\barPhi\circ\pi_{1,\dagger}\stackrel{\sim}{\Rightarrow}\pi_{2,\dagger}\circ\Phi$)
\begin{equation}\label{commpi}
\xymatrix{\mon{X_1}\ar[r]^{\Phi}\ar[d]^{\pi_{1,\dagger}} & \mon{X_2}\ar[d]^{\pi_{2,\dagger}}\\
D'(Y_1)\ar[r]^{\barPhi} & D'(Y_2)}
\end{equation}
Then $\Phi$ naturally extends to an exact functor $\hatPhi:\hmon{X_1}\to\hmon{X_2}$. Moreover, this extension is compatible with compositions, adjunctions and natural transformations.
\end{prop}
\begin{proof}
It is clear that $\pro(\Phi)$ sends distinguished triangles to distinguished triangles and commutes with $[1]$. The only thing we need to check is that $\pro(\Phi)$ sends $\pi_1$-constant objects to $\pi_2$-constant objects. But this follows from the diagram \eqref{commpi}.
\end{proof}

\begin{cor}\label{c:functor}
Let $f:X_1\to X_2$ be an $A$-equivariant morphism between $A$-torsors and $D'(Y_i)\subset D^b_c(Y_i)$ be full triangulated subcategories. Let $\barf:Y_1\to Y_2$ be the induced morphism. Let $\Phi$ be any of the exact functors $f^*,f_*,f_!$ and $f^!$, and let $\barPhi$ be the corresponding functor for $\barf$.

Suppose $\barPhi$ restricts to a functor between the $D'(Y_i)$'s, then $\Phi$ naturally extends to exact functors between the $\hmon{X_i}$'s. Moreover, these extensions are compatible with compositions, adjunctions and proper base change.
\end{cor}
\begin{proof}
(1) For $\Phi=f_!$, we have $\pi_{2,\dagger}f_!=\barf_!\pi_{1,\dagger}$, then apply Proposition \ref{compfun}.

(2) For $\Phi=f^*$, we have proper base change isomorphism $\pi_{1,\dagger}f^*\cong\barf^*\pi_{2,\dagger}$, then apply Proposition \ref{compfun}.

(3) For $\Phi=f_*$, we have a natural transformation $\pi_{2,\dagger}f_*\to\barf_*\pi_{1,\dagger}$ (apply adjunction to $f_*\to f_*\pi_1^{\dagger}\pi_{1,\dagger}\cong\pi_2^{\dagger}\barf_*\pi_{1,\dagger}$). This natural transformation is in fact an isomorphism when restricted to $\mon{X_1}$. In fact, we only need to check it on objects of the form $\pi_1^\dagger\calF$. The problem being \'etale local, we may assume that $X_{2}$ is a trivial $A$-torsor over $Y_{2}$, and fix a trivialization $X_{2}\cong Y_{2}\times A$. This induces a trivialization $X_{1}\cong Y_{1}\times A$ and $f=\barf\times\id_{A}$. By proper base change and the projection formula, we have
\begin{equation*}
\pi_{2,\dagger}f_{*}\pi_{1}^{\dagger}\calF\cong\pi_{2,!}(\barf_{*}\calF\boxtimes\Ql)\cong\barf_{*}\calF\otimes\upH_{c}^{*}(A).
\end{equation*}
On the other hand, we have
\begin{equation*}
\barf_{*}\pi_{1,\dagger}\pi_{1}^{\dagger}\calF\cong\barf_{*}(\calF\otimes\upH_{c}^{*}(A))=\barf_{*}\calF\otimes\upH_{c}^{*}(A).
\end{equation*}
Therefore $\pi_{2,\dagger}f_*\pi_{1}^{\dagger}\calF\isom\barf_*\pi_{1,\dagger}\pi_{1}^{\dagger}\calF$. Knowing $\pi_{2,\dagger}f_*\to\barf_*\pi_{1,\dagger}$ is an isomorphism, we then apply Proposition \ref{compfun} to finish the proof.

(4) For $\Phi=f^!$, we have a natural transformation $\pi_{1,\dagger}f^!\to\barf^!\pi_{2,\dagger}$, which is an isomorphism when restricted to $\mon{X_2}$ for the same reason as in (3). We then apply Proposition \ref{compfun}.
\end{proof}


The adjunction $(\piddag,\pidag)$ extends to the adjunction (the extended functors are denoted by the same letter):
\begin{equation*}
\xymatrix{\hmon{X}\ar@<.7ex>[r]^{\piddag} & D'(Y)\ar@<.7ex>[l]^{\pidag}}
\end{equation*}

\begin{lemma}\label{l:picons}
The functor $\piddag:\hmon{X}\to D'(Y)$ is conservative.
\end{lemma}
\begin{proof}
Since $\piddag$ is an exact functor, to show it is conservative we only need to show that it sends a nonzero object to a nonzero object.

Suppose $\piddag(\prolim\calF_n)=0$, then any map $\prolim\calF_n\to\pidag\calG$ is zero for any $\calG\in D'(Y)$. Since objects of the form $\pidag\calG$ generate $\mon{X}$, this means $\Hom(\prolim\calF_n,\calF)$ for any $\calF\in\mon{X}$. By the full faithfulness of the Yoneda embedding \eqref{Yoneda}, this implies that $\prolim\calF_n=0$.
\end{proof}

The following lemma is used only in the proof of Lemma \ref{l:convIst} and Lemma \ref{l:samedef}.
\begin{lemma}\label{l:actionF}
Let $a:A\times X\to X$ be the action map. Recall the \fm local system $\tcL$ on $A$ defined in Example \ref{ex:fm}. Then there is a functorial isomorphism for $\calF\in\hmon{X}$:
\begin{equation*}
a_!(\tcL\boxtimes\calF)\cong\calF[-2r](-r).
\end{equation*}
\end{lemma}
\begin{proof}
It suffices to give the isomorphism for $\calF\in\mon{X}$. We first need to construct a natural map $\tcL\boxtimes\calF\to a^*\calF=a^!\calF[-2r](-r)$. We may assume $\calF$ is a $\ZZ_\ell$-complex, and is given by a projective system $\calF_n\in D^b_{c}(X,\ZZ/\ell^n)$. For each $m\in\ZZ_{\geq1}$, write
\begin{equation*}
a_m:A\times X\xrightarrow{[m]\times\id} A\times X\xrightarrow{a}X.
\end{equation*}
Let $p:A\times X\to X$ be the projection. By \cite[Proposition 5.1]{V}, for fixed $n$, if $m$ is sufficiently divisible, we have an isomorphism
\begin{equation*}
\theta_m:p^*\calF_n\cong a_m^*\calF_n=([m]\times\id)^!(a^*\calF)
\end{equation*}
In our case $\calF$ is a successive extension of sheaves pulled back from $Y$, it is easy to see that such an isomorphism exists if $m=\ell^b$ for large $b$. By adjunction, $\theta_m$ gives
\begin{equation*}
([m]_!\ZZ/\ell^n)\boxtimes\calF_n=([m]\times\id)_!p^*\calF_n\to a^*\calF_n.
\end{equation*}
As $m=\ell^b$ and $n$ varies, $[\ell^b]_!\ZZ/\ell^n$ form a projective system indexed by two integers $b,n$. Taking projective limit, we get a map
\begin{equation}\label{prolimaction}
(\varprojlim_{b,n}[\ell^b]_!\ZZ/\ell^n)\boxtimes\calF\to a^*\calF.
\end{equation}

As a representation of $\pi_1(A,e)$, the local system $[\ell^b]_!\ZZ/\ell^n$ is $\ZZ/\ell^n[A[\ell^b]]$, the regular representation of the quotient $\pi_1(A,e)\to A[\ell^b]$.  Let $\ZZ_\ell[T_\ell(A)]^\wedge_{\aug}$ be the completion along the augmentation ideal of $\ZZ_\ell[T_\ell(A)]$. There is a natural map in $\Rep(\pi^{\ell}_1(A,e))$ (note that $\pi^\ell_1(A,e)\cong T_\ell(A)$):
\begin{equation}\label{augtodouble}
\ZZ_\ell[T_\ell(A)]^\wedge_\aug\to\varprojlim_{b,n}\ZZ/\ell^n[T_\ell(A)/\ell^b].
\end{equation}
In fact, for any $n,b$, elements of the form $(t-1)^{\ell^N}$ ($t\in T_\ell(A)$) lies in the ideal generated by $\ell^n$ and $\ell^bT_\ell(A)\subset T_\ell(A)$ in $\ZZ_\ell[T_\ell(A)]$ for large $N=N(n,b)$ (by binomial expansion).

On the other hand, we have a map in $\pro\Rep(\pi^\ell_1(A,e),\Ql)$
\begin{equation}\label{Stoaug}
\hatS=\varprojlim_n\Sym(V_A)/(V_A^n)\to\Ql[T_\ell(A)]^\wedge_\aug
\end{equation}
which sends $t\in T_\ell(A)\subset V_A\subset\hatS_A$ to $\log(t)=-\sum_{i\geq1}(1-t)^i/i\in\Ql[T_\ell(A)]^\wedge_\aug$. Combining \eqref{augtodouble} and \eqref{Stoaug}, we get a continuous map between $\pi^\ell_1(A,e)$-representations
\begin{equation*}
\hatS\to\left(\varprojlim_{b,n}\ZZ/\ell^n[T_\ell(A)/\ell^b]\right)\otimes\Ql.
\end{equation*}
Composing with the map \eqref{prolimaction}, we get the desired map $\tcL\boxtimes\calF_{\Ql}\to a^*\calF_{\Ql}$. By adjunction, we get a functorial map
\begin{equation*}
\beta(\calF):a_!(\tcL\boxtimes\calF)\to\calF[-2r](-r)
\end{equation*}
Finally we check this is an isomorphism for $\calF\in\mon{X}$. Since $\mon{X}$ is generated by $\pidag\calG$, it suffices to check $\beta(\pidag\calG)$ is an isomorphism. Applying proper base change to the Cartesian diagram
\begin{equation*}
\xymatrix{A\times X\ar[d]^{a}\ar[r]^{\id_A\times\pi} & A\times Y\ar[d]^{p_Y}\\ X\ar[r]^{\pi} & Y}
\end{equation*}
we get
\begin{equation*}
a_!(\tcL\boxtimes\pidag\calG)=\pidag p_{Y,!}(\tcL\boxtimes\calG)=\pidag(\bR\Gamma_c(A,\tcL)\otimes\calG)\cong\pidag\calG[-2r](-r).
\end{equation*}
In the last equality we used $\bR\Gamma_c(A,\tcL)=\Ql[-2r](-r)$. The above isomorphism is in fact the same as $\beta(\pidag\calG)$. This proves $\beta(-)$ is an isomorphism.
\end{proof}

\subsection{The case of a trivial $A$-torsor}\label{ss:triv}
In this section, we study the special case where $\pi:X\to Y$ is a trivial $A$-torsor. Fix a section $\epsilon:Y\to X$. Consider the t-exact functor
\begin{equation*}
\epdag=\epsilon^![r]:\mon{X}\to D'(Y).
\end{equation*}
There is a natural transformation
\begin{equation}\label{eptopi}
\epdag=\pi_!\epsilon_*\epsilon^![r]\xrightarrow{\adj}\pi_![r]=\piddag.
\end{equation}

We also consider the functor
\begin{equation*}
\Free:D'(Y)\ni\calF\mapsto\calF\boxtimes\tcL[r](r)\in\hmon{X}
\end{equation*}
where $\tcL\in\hmon{A}$ is as defined in Example \ref{ex:fm}.

\begin{defn}\label{def:fm}
Objects of the form $\Free(\calF)\in\hmon{X}$ for $\calF\in P'(Y)$ are called {\em \fm perverse local systems}.
\end{defn}

\begin{lemma}\label{l:weakadj} The functors $(\Free,\epdag)$ satisfy the following adjunction
\begin{equation}\label{wadj}
\Hom_{\hmon{X}}(\Free(\calF),\calG)\isom\Hom_{\pro(D'(Y))}(\calF,\epdag\calG)
\end{equation}
for $\calF\in D'(Y),\calG\in\mon{X}$.
\end{lemma}
\begin{proof} Note that $\epsilon^!\tcL[2r](r)\cong\epsilon^*\tcL=\hatS=\prolim\Sym(V_A)/(V_A^n)\in D^b_c(\pt)$. Let $s:\Ql\to\hatS$ be the unit map. For any map $\phi:\Free(\calF)\to\calG$, we have a map
\begin{eqnarray*}
\calF\xrightarrow{\id\otimes s}\calF\otimes\hatS\cong\calF\boxtimes\epsilon^!\tcL[2r](r)=\epdag\Free(\calF)\xrightarrow{\epdag(\phi)}\epdag\calG.
\end{eqnarray*}
This established the required map between the Hom-spaces in \eqref{wadj}. To check it is an isomorphism, it suffices to check for generating objects $\calG=\pidag\calH$, where it is obvious.
\end{proof}

For any object $\calF\in\per{X}$, $\epdag\calF\in P'(Y)$ naturally carries the nilpotent logarithmic monodromy operator
\begin{equation*}
\epdag\log(\mu_{\calF}):V_A\otimes\epdag\calF\to\epdag\calF
\end{equation*}
so that $\epdag\calF$ becomes an $\hatS$-module in $P'(Y)$, on which $V_A$ acts nilpotently.

\begin{lemma}\label{l:Smod}
The functor $\epdag$ lifts to an equivalence of abelian categories
\begin{equation*}
\sigma:\per{X}\isom\Mod^{\nil}(\hatS;P'(Y))
\end{equation*}
where $\Mod^{\nil}(\hatS;P'(Y))$ denotes the abelian category of $\hatS$-module objects in $P'(Y)$ on which $V_A$ acts nilpotently.
\end{lemma}
\begin{proof}
This is a variant of the usual Barr-Beck theorem in the following situation
\begin{equation*}
\xymatrix{\pro\per{X}\ar@{<-^{)}}[r] & \per{X}\ar[r]_{\epdag} & P'(Y)\ar@/_1pc/[ll]_{\Free}}
\end{equation*}
Here $\epdag$ is exact, faithful and conservative; the only issue is that the functors $\epdag$ and $\Free$ are only adjoint in the sense of Lemma \ref{l:weakadj}. But the argument for Barr-Beck theorem still works. We have a functor in the other direction:
\begin{equation}\label{invsig}
\Mod^{\nil}(\hatS;P'(Y))\to\per{X}
\end{equation}
which sends a nilpotent $\hatS$-module $\calF$ in $P(Y)$ to
\begin{equation}\label{delcoker}
\coker\left(V_A\otimes(\calF\boxtimes\tcL)[r](r)\xrightarrow{m(\calF)\boxtimes\id-\id\boxtimes m(\tcL)}\calF\boxtimes\tcL[r](r)\right).
\end{equation}
where $m(\calF):V_A\otimes\calF\to\calF$ and $m(\tcL):V_A\otimes\tcL\to\tcL$ are action maps given by the $\hatS$-module structures. Since $V_A^{n+1}$ acts as zero on $\calF$ for large $n$, so that \eqref{delcoker} is actually a quotient of $\calF\boxtimes\calL_n$ (for $\calL_n$, see Example \ref{ex:fm}), hence lands in $\per{X}$. It is easy to check that the functor \eqref{invsig} and $\sigma$ are inverse to each other.
\end{proof}

\begin{cor}\label{c:coinv}
Under the equivalence $\sigma$, the functors
\begin{equation*}
\xymatrix{\per{X}\ar@<.7ex>[rr]^{\pH^0\piddag} && P'(Y)\ar@<.7ex>[ll]^{\pidag}}
\end{equation*}
become
\begin{equation*}
\xymatrix{\Mod^{\nil}(\hatS;P'(Y))\ar@<.5ex>[rr]^{\otimes_{\hatS}\Ql} && P'(Y)\ar@<.5ex>[ll]^{\triv}}.
\end{equation*}
where ``$\triv$'' sends an object $\calF\in P'(Y)$ to the $\hatS$-module $\calF$ on which $V_A$ acts as 0.
\end{cor}
\begin{proof}
It is clear that $\epdag\pidag\calF=\calF$ with the trivial monodromy action, hence $\sigma\pidag$ is the same as the functor ``$\triv$''. The equality $\pH^0\piddag=(-)\otimes_{\hatS}\Ql$ follows by adjunction.
\end{proof}

\textbf{Assumption F.} Every object in $P'(Y)$ has a finite resolution by projective objects, and that the realization functor
\begin{equation*}
\rho_Y:D^b(P'(Y))\to D'(Y)
\end{equation*}
is an equivalence.

\begin{prop}\label{p:freegen}
Under Assumption F, the \fm objects $\Free(\calF),\calF\in P'(Y)$ generate $\hatD'(\qw{X}{A})$.
\end{prop}
\begin{proof}
Given an object $\prolim\calF_n\in\hmon{X}$, by Assumption F, we may resolve $\prolim\piddag\calF_n\in D'(Y)$ by projective objects in $P'(Y)$:
\begin{equation}\label{Kcomp}
[\cdots\to\calK^{-1}\to\calK^0\to\cdots]\in D^b(P'(Y))\cong D'(Y).
\end{equation}
The {\em amplitude} of the complex $\prolim\piddag\calF_n$ is the least number of nonzero terms among all such projective resolutions. If the amplitude of $\prolim\piddag\calF_n$ is $0$, i.e., $\prolim\piddag\calF_n=0$, then by Lemma \ref{l:picons}, $\prolim\calF_n=0$.

Now suppose that any $\prolim\calF_n$ such that $\prolim\piddag\calF_n$ has amplitude $<N$ can be expressed as a successive extension of \fm objects. Let $\prolim\calF_n\in\hmon{X}$ such that the amplitude of $\prolim\piddag\calF_n$ is $N$. We may assume \eqref{Kcomp} is a minimal resolution which terminals on the right at $\calK^0$ (i.e., $\calK^0\neq0$ and $\calK^i=0$ for $i>0$).

\begin{claim} The component-wise truncation $\prolim\ptau_{<0}\calF_n\to\prolim\calF_n$ is an isomorphism in $\pro\mon{X}$.
\end{claim}
\begin{proof}
By uniform boundedness, we may assume every $\calF_n\in\pD^{\leq d}_{c}(X)$ for some $d\geq0$. If $d=0$, there is nothing to argue. Suppose $d>0$, we only need to prove that $\alpha:\prolim\ptau_{<d}\calF_n\to\prolim\calF_n$ is an isomorphism, and repeat the argument.

For this, it suffices to show that $\textup{Cone}(\alpha)\cong\prolim\pH^d\calF_n$ is zero (although we do not have $\prolim\ptau_{<d}\calF_n\in\hmon{X}$ {\em a priori}, the vanishing of the cone still implies $\alpha$ is an isomorphism: we only need to apply $\Hom(-,T)$ to $\alpha$ for any test object $T\in\mon{X}$, and note that $\varinjlim$ is exact) . Let $\calP_n=\pH^d\calF_n$. Since $\piddag$ is right t-exact, $\prolim\piddag\ptau_{<d}\calF'_n\in\pD^{<d}_{c}(Y)$, the projective system of perverse sheaves $\pH^d\piddag\calP_n$ is zero. This means for fixed $n$, the transition map $\pH^d\piddag\calP_m\to\pH^d\piddag\calP_n$ is zero for large $m$. By Corollary \ref{c:coinv}, this means that the image of $\calP_m$ in $\calP_n$ falls into $V_A\calP_n$. Since each $\calP_n$ is killed by a power of $V_A$, this means the transition map $\calP_m\to\calP_n$ is zero for large $m$. This proves the claim.
\end{proof}

By this claim, we may assume that each $\calF_n\in\pD^{\leq0}_{c}(X)$. We will construct a map $\Free(\calK^0)\to\prolim\calF_n$ in $\hmon{X}$. By Lemma \ref{l:weakadj}, it suffices to give a compatible system of maps $\{\calK^0\to\epdag\calF_n\}_n$. By the assumption that $\rho_Y$ is an equivalence and $\calK^0$ is projective, such a map is equivalent to a map $\calK^0\to\pH^0\epdag\calF_n=\epdag\pH^0\calF_n$.

On one hand, we have a natural map
\begin{equation*}
\alpha_n:\calK^0\to\pH^0(\prolim\piddag\calF_n)\to\pH^0\piddag\calF_n.
\end{equation*}
On the other hand, we have a surjection by \eqref{eptopi}
\begin{equation*}
\beta_n:\epdag\pH^0\calF_n\twoheadrightarrow\pH^0\piddag\pH^0\calF_n=\pH^0\piddag\calF_n.
\end{equation*}
Since $\calK^0$ is projective, it is possible to lift $\alpha_n$ to $\tilalpha_n:\calK^0\to\epdag\pH^0\calF_n$. The set of such liftings is a torsor under $\Hom_{P'(Y)}(\calK^0,\ker(\beta_n))$, which is a Mittag-Leffler projective system. Therefore, the compatible system of maps $\{\alpha_n\}$ can be lifted to a compatible system of maps $\{\tilalpha_n\}$. According to the argument above, it gives a map $s^0:\Free(\calK^0)\to\prolim\calF_n$ such that $\piddag(s^0)$ coincides with the natural map $\calK^0\to\prolim\piddag\calF_n$.

Let $\prolim\calF'_n$ be a cone of $s^0$ in $\hmon{X}$. Then $\piddag(\prolim\calF'_n)$ is represented by the complex
\begin{equation*}
[\cdots\to\calK^{-2}\to\calK^{-1}\to0\to\cdots]\in D^b(P'(Y))\cong D'(Y).
\end{equation*}
which has amplitude $<N$. This completes the induction step.
\end{proof}

\begin{exam}
The Assumption F is essential. We give an example where $\hmon{X}$ is not generated by \fm objects. Let $X=\GG_m\times\GG_m$ be the trivial $\GG_m$-torsor over $Y=\GG_m$ via the first projection. Consider the following diagram
\begin{equation*}
\xymatrix{X=\GG_m\times\GG_m\ar[r]^(.6){\mult}\ar[d]^{\pi} & \GG_m\ar[d]\\
Y=\GG_m\ar[r] & \pt}
\end{equation*}
where ``$\mult$'' is the multiplication map. Take $D'(Y)\subset D^b_c(Y)$ to be the full triangulated subcategory generated by the constant sheaf. Let $\tcL$ denote the free-monodromic local system on $\GG_m$. Then the object $\mult^*\tcL\in\hmon{X}$ does {\em not} lie in the triangulated category generated by free-monodromic perverse local systems.

In fact we have $D'(Y)\cong D^b(\Mod^{\nil}(\Ql[[t]]))$ and $\mon{X}\cong D^b(\Mod^{\nil}(\Ql[[s,t]]))$. The object $\mult^*\tcL$ corresponds to the module $\Ql[[s,t]]/(s-t)\in\pro\mon{X}$, which lies in $\hmon{X}$ because $\pi_!\mult^*\tcL=\Ql[-2](-1)$. However, the subcategory of $\hmon{X}$ generated by free-monodromic perverse local systems can be identified with $D^b(\Mod^{t-\nil}(\Ql[[s,t]]))$ (where the superscript ``$t$-nil'' means only the action of $t$ on the module is nilpotent, and $t$ denotes the logarithmic monodromy in the $Y$-direction). Obviously $\Ql[[s,t]]/(s-t)$ does not lie in this subcategory.
\end{exam}

\begin{cor}\label{c:modES}
Under Assumption F,
\begin{enumerate}
\item The realization functor $\rho_X:D^b(\per{X})\to\mon{X}$ is an equivalence. We have a t-exact equivalence
\begin{eqnarray*}
\rho_X\circ\sigma^{-1}:D^b(\Mod^{\nil}(\hatS;P'(Y)))\xrightarrow{\sigma^{-1}}D^b(\per{X})\xrightarrow{\rho_X}\mon{X}.
\end{eqnarray*}

\item Suppose we are given a t-exact equivalence of triangulated categories
\begin{equation*}
\nu:D^b(E)\cong D'(Y)
\end{equation*}
for some finite dimensional algebra $E$ with finite cohomological dimension. Then the equivalence $\rho_X\circ\sigma^{-1}$ extends to an equivalence of triangulated categories
\begin{equation*}
\hatnu:D^b(E\otimes\hatS)\cong\hmon{X}.
\end{equation*}

\item Under $\hatnu$, the adjunctions $(\piddag,\pidag)$ becomes
\begin{equation*}
\xymatrix{D^b(E\otimes\hatS)\ar@<.5ex>[rr]^{\Ltimes_{\hatS}\Ql} && D^b(E)\ar@<.5ex>[ll]^{\triv}}.
\end{equation*}
\end{enumerate}
\end{cor}
\begin{proof}
(1) Firstly, $\rho_X$ is essentially surjective. This is because both sides are generated by objects of the form $\pidag\calF$ for $\calF\in P'(Y)$.

Next we check that $\rho_X$ induces an isomorphism between the Ext-groups for these generating objects. On one hand,
\begin{equation*}
\bR\Hom_{\mon{X}}(\pidag\calF,\pidag\calF')=\bR\Hom_{D'(Y)}(\calF,\calF')\otimes \upH^*(A).
\end{equation*}
On the other hand, by Lemma \ref{l:Smod}, we have
\begin{eqnarray*}
\bR\Hom_{\per{X}}(\pidag\calF,\pidag\calF')&=&\bR\Hom_{\Mod^{\nil}(\hatS;P'(Y))}(\calF,\calF')\\
&=&\bR\Hom_{P'(Y)}(\calF,\calF')\otimes\bR\Hom_{\hatS}(\Ql,\Ql).
\end{eqnarray*}
To see this, we need to pick a projective resolution $\calK^*$ for $\calF$ in $P'(Y)$, and use Koszul resolution of $\calK^*$ by free $\hatS$-modules.

By Assumption F, $\bR\Hom_{P'(Y)}(\calF,\calF')\isom\bR\Hom_{D'(Y)}(\calF,\calF')$ . Moreover, $\bR\Hom_{\hatS}(\Ql,\Ql)\cong\wedge^*(V_A^{\vee}[-1])\cong \upH^*(A)$, hence the two $\bR\Hom$-complexes are naturally isomorphic.

(2) The equivalence $\rho_X\circ\sigma^{-1}$ extends to pro-objects. We identify $D^b(E\otimes\hatS)$ with a full subcategory of $\pro D^b(\Mod^{\nil}(E\otimes\hatS))\cong\pro D^b(\Mod^{\nil}(\hatS;P'(Y)))$, hence getting a full embedding
\begin{equation*}
D^b(E\otimes\hatS)\hookrightarrow\pro\mon{X}.
\end{equation*}
Any complex in $D^b(E\otimes\hatS)$ is quasi-isomorphic to a complex of free objects $M\otimes\hatS$ ($M\in\Mod(E)$), hence its image lies in $\hmon{X}$. On the other hand, by Proposition \ref{p:freegen}, $\hmon{X}$ is generated by free objects $\Free(\calF)$ ($\calF\in P'(Y)$), which correspond to $\nu^{-1}(\calF)\otimes\hatS\in\Mod(E\otimes\hatS)$. Hence $\pro(\rho_X\circ\sigma^{-1})$ restricts to the desired equivalence $\hatnu$.

(3) follows from Corollary \ref{c:coinv}.
\end{proof}

\begin{remark}\label{r:simplet}
Corollary \ref{c:modES}(2) gives a t-structure on $\hmon{X}$ extending the perverse t-structure on $\mon{X}$, whose heart we denote by $\hper{X}$. {\em A priori}, it is not clear that such a t-structure exists. However, {\em a posteriori}, this t-structure can be intrinsically defined as follows: $\calF\in\leftexp{p}{\hatD}'^{[a,b]}(\qw{X}{A})$ if and only it is isomorphic to $\prolim\calF'_n$ where each $\calF'_n\in\pD'^{[a,b]}(\qw{X}{A})$. In fact, any complex $M=[0\to M^a\to\cdots\to M^b\to0]\in D^{[a,b]}(E\otimes\hatS)$, can be written as the projective limit of $M_n=[0\to M^a/V_A^{n+1}\to\cdots\to M^b/V_A^{n+1}\to0]\in D^{[a,b]}\Mod^{\nil}(E\otimes\hatS)$.
\end{remark}

\begin{remark}\label{r:projres}
The proof of Proposition \ref{p:freegen} implies a stronger result: if $\prolim\calF_n\in\hmon{X}$ and $\piddag(\prolim\calF_n)$ has a projective resolution as in \eqref{Kcomp}, then $\prolim\calF_n$ can be represented by filtered complex $\tcK\in\hatD'F(\qw{X}{A}$ (the filtered counterpart of $\hmon{X}$, see Remark \ref{r:filD}) such that $\Gr_F^i\tcK=\Free(\calK^i)[-i]$. We can identify $C^b(\hper{X})$ with a full subcategory of $\hatD'F(\qw{X}{A})$ as in \cite[\S 3.1.8]{BBD}, hence the object $\prolim\calF_n$ itself can be represented by a complex
\begin{equation*}
[\cdots\to\Free(\calK^{-1})\to\Free(\calK^0)\to\cdots]\in C^b(\hper{X}).
\end{equation*}
\end{remark}

\subsection{The mixed case}\label{ss:mix}
From now on till the end of this appendix, let $k$ be a finite field. We still consider an $A$-torsor $\pi:X\to Y$, where everything is now defined over $k$ and $A$ is a split torus over $k$. Let $D'_m(Y)\subset D^b_m(Y)$ be a full triangulated subcategory of mixed $\Ql$-complexes on $Y$. Let $\monm{X}\subset D^b_m(X)$ be the full triangulated category generated by $\pidag\calF$ for $\calF\in D'_m(Y)$, whose heart of the t-structure is denoted by $\perm{X}$.

\begin{defn}\label{def:weightbound}
\begin{enumerate}
\item []
\item A pro-object $\prolim\calF_n\in\pro D^b_m(X)$ is {\em uniformly bounded above in weights}, if it is isomorphic to a pro-object $\prolim\calF'_n$ for which there exists $N\in\ZZ$ such that each $\calF'_n$ is of weight $\leq N$.
\item Let $\hmonm{X}$ be the full subcategory of $\pro\monm{X}$ whose objects are $\pi$-constant and uniformly bounded in degrees and uniformly bounded above in weights.
\end{enumerate}
\end{defn}

The new requirement of uniform boundedness on weights does not change the arguments in the previous sections. In particular, Theorem \ref{th:montri} implies that $\hmonm{X}$ is a triangulated category, and Corollary \ref{c:functor} continues to hold in the in the mixed situation. Let $\omega:\hmonm{X}\to\hatD'(\qw{\geom{X}}{\geom{A}})$ be the functor of pull-back to $\geom{X}$ (taking the underlying complex).

\begin{lemma}\label{l:mixhom}
For objects $\calF=\prolim\calF_n,\calG=\prolim\calG_n\in\hmonm{X}$, their $\Ext$-groups fit naturally into short exact sequences:
\begin{eqnarray}\label{mixhom}
0\to\Ext^{i-1}_{\hatD'(\qw{\geom{X}}{\geom{A}})}(\calF,\calG)_{\Frob}\to\ext^i_{\hmonm{X}}(\calF,\calG)\\
\notag
\to\Ext^i_{\hatD'(\qw{\geom{X}}{\geom{A}})}(\calF,\calG)^{\Frob}\to0.
\end{eqnarray}
\end{lemma}
\begin{proof}
For fixed $m,n$, we have the exact sequence \eqref{intromix}
\begin{equation}\label{mixmn}
0\to\Ext^{i-1}_{X}(\calF_m,\calG_n)_{\Frob}\to\ext^i_{X}(\calF_m,\calG_n)\to\Ext^i_{X}(\calF_m,\calG_n)^{\Frob}\to0.
\end{equation}

For any inductive or projective system of finite dimensional vector spaces $\{H_n\}$, $\varinjlim$ and $\varprojlim$ commutes with taking $\Frob$-invariants and coinvariants. More precisely, consider the system of exact sequences
\begin{equation*}
0\to H_n^{\Frob}\to H_n\xrightarrow{\Frob-\id}H_n\to (H_n)_{\Frob}\to0.
\end{equation*}
Taking $\varinjlim$ or $\varprojlim$ preserves the exactness, hence
\begin{equation*}
(\lim H_n)^{\Frob}=\lim H_n^{\Frob};(\lim H_n)_{\Frob}=\lim (H_n)_{\Frob},
\end{equation*}
where $\lim$ means either $\varinjlim$ or $\varprojlim$.

Applying this remark to \eqref{mixmn}, taking inductive limit over $m$, we get
\begin{eqnarray}\label{mixonlyn}
0\to\Ext^{i-1}_{\hatD'(\qw{\geom{X}}{\geom{A}})}(\calF,\calG_n)_{\Frob}\to\ext^i_{\hmonm{X}}(\calF,\calG_n)\\
\notag
\to\Ext^i_{\hatD'(\qw{\geom{X}}{\geom{A}})}(\calF,\calG_n)^{\Frob}\to0.
\end{eqnarray}
Note that each $\Ext^i_{\hmon{\geom{X}}}(\calF,\calG_n)$ is still finite dimensional (because of $\pi$-constancy of $\calF$, see the proof of the Mittag-Leffler condition in Theorem \ref{th:montri}), hence we can apply the above remark to \eqref{mixonlyn}. Taking projective limit over $n$, we get the desired exact sequence \eqref{mixhom}.
\end{proof}

Now we concentrate on the case $X=Y\times A$. We have the obvious mixed analogs of Lemma \ref{l:Smod} and Corollary \ref{c:coinv}, where $\hatS$ is viewed as a $\Frob$-module ($\Frob$ acts on $V_A$ by $q^{-1}$).

We make a mixed analog of the Assumption F.

\textbf{Assumption F$_m$.} Every object in $P'_m(Y)$ has a finite resolution by objects whose images in $P'(\geom{Y})$ are projective, and the realization functor
\begin{equation*}
\rho_{Y,m}:D^b(P'_m(Y))\to D'_m(Y)
\end{equation*}
is an equivalence.

We have a mixed analog of Proposition \ref{p:freegen}:
\begin{prop}\label{p:freegenmix}
Under Assumption F$_m$, the \fm objects $\Free(\calF)$ for $\calF\in P'_m(Y)$ generate $\hmonm{X}$.
\end{prop}
\begin{proof}
We only indicate the modification in the argument comparing with the proof of Proposition \ref{p:freegen}. Instead of doing induction on the amplitude of $\prolim\piddag\calF_n$, we take into account both cohomological degrees and weights. Let $\prolim\piddag\calF_n$ be represented by a complex
\begin{equation}\label{mixK}
[\cdots\to\calK^{-1}\to\calK^0\to\cdots]\in D^b(P'_m(Y))\cong D'_m(Y).
\end{equation}
where each $\omega\calK^i\in P'(\geom{Y})$ is a projective object. For each $\calK^i$, we have a canonical finite {\em decreasing} filtration \footnote{This filtration is not be confused with the weight filtration in \cite{BBD}.}
\begin{equation*}
0\subset\cdots\subset W^{\geq v}\calK^i\subset W^{\geq v-1}\calK^i\subset\cdots\subset\calK^i
\end{equation*}
such that each $\Gr_W^v\calK^i$ is a successive extension of perverse sheaves $\calP$ whose $\omega\calP$ is an indecomposable projective object in $P'(Y)$, and the unique simple quotient of $\calP$ has weight $v$. (This canonical filtration follows from the fact that $\ext^1(\calP,\calP')=\Hom(\calP,\calP')_{\Frob}=0$ if the simple quotient of $\calP$ has larger weight than the weights of $\calP'$.)

We define the {\em width} of the $\prolim\piddag\calF_n$ to be the least number of $(v,i)$ such that $\Gr_W^v\calK^i\neq0$, among all representing complexes $\calK^*$ as in \eqref{mixK}. We do induction on the width of $\prolim\piddag\calF_n$. If its width is $0$, then $\prolim\piddag\calF_n=0$ and hence $\prolim\calF_n=0$.

Suppose for $\prolim\piddag\calF_n$ of width $<N$, $\prolim\calF_n$ is a successive extension of \fm objects. Now let $\prolim\piddag\calF_n$ be of width $N$. Let us assume that \eqref{mixK} is a representing complex which terminate at $\calK^0$, and that $W^{\geq1}\calK^0=0$ but $W^{\geq0}\calK^0\neq0$. Then $\calK^0$ has weight $\leq0$. The argument of the Claim in Proposition \ref{p:freegen} proves that we can first replace $\prolim\calF_n$ by $\prolim\ptau_{<0}\calF_n$, and then assume each $\pH^0\calF_n$ has weight $\leq0$.

We can then try to construct a map $W^{\geq 0}\calK^0\to\epdag\pH^0\calF_n$. For a mixed perverse sheaf $\calP$, let $\calP_{\geq w}$ be its quotient of weight $\geq w$ in the weight filtration of \cite[Theorem 5.3.5]{BBD}. By Corollary \ref{c:coinv}, we have an isomorphism $(\epdag\pH^0\calF_n)_{\geq0}\isom(\pH^0\piddag\calF_n)_{\geq 0}$ (because $V_A$ has weight $-2$). The projective system of maps $\alpha_n:W^{\geq0}\calK^0\to\pH^0\piddag\calF_n\to(\pH^0\piddag\calF_n)_{\geq 0}$ thus gives a projective system of maps $\overline{\alpha}_n:W^{\geq0}\calK^0\to(\epdag\pH^0\calF_n)_{\geq 0}$. Note that $\hom(W^{\geq0}\calK^0,\epdag\pH^0\calF_n)\to\hom(W^{\geq0}\calK^0,(\epdag\pH^0\calF_n)_{\geq0})$ is surjective because $\ext^1(W^{\geq0}\calK^0,\calP)=0$ for any perverse sheaf $\calP$ of weight $<0$. Hence the projective system $\overline{\alpha}_n$ can be lifted to a projective system $\tilalpha_n:W^{\geq0}\calK^0\to\epdag\pH^0\calF_n$. Note further that the canonical map $\hom(W^{\geq0}\calK^0,\calF_n)\to\hom(W^{\geq0}\calK^0,\pH^0\calF_n)$ is surjective since the next term in the long exact sequence is $\ext^1(W^{\geq0}\calK^0,\ptau_{<0}\calF_n)$, which is zero because $\ext^{\geq2}(\calK^0,\perm{X})=0$. Hence the projective system of maps $\{\tilalpha_n\}$ lifts to a map $W^{\geq w}\calK^0\to\epdag\prolim\calF_n$. Let $\prolim\calF'_n$ be the cone of this map, then $\prolim\piddag\calF'_n$ is represented by the complex.
\begin{equation*}
[\cdots\to\calK^{-1}\to\calK^0/W^{\geq0}\calK^0\to0]\in D^b(P'_m(Y))
\end{equation*}
which has width $<N$. This completes the induction step.
\end{proof}

We also have an analog of Corollary \ref{c:modES}.
\begin{cor}\label{c:modESmix}
Under Assumption F$_m$,
\begin{enumerate}
\item The realization functor $\rho_{X,m}:D^b(\perm{X})\to\monm{X}$ is an equivalence. We have a t-exact equivalence
\begin{eqnarray*}
\rho_{X,m}\circ\sigma_m^{-1}:D^b(\Mod^{\nil}(\hatS;P'_m(Y)))\xrightarrow{\sigma^{-1}_m}D^b(\perm{X})\xrightarrow{\rho_{X,m}}\monm{X}.
\end{eqnarray*}

\item Suppose we are given a t-exact equivalence of triangulated categories
\begin{equation*}
\nu_m:D^b(E,\Frob)\cong D'_m(Y)
\end{equation*}
for some finite dimensional algebra $E$ of finite cohomological dimension with a $\Frob$-action. Then $\rho_{X,m}\circ\sigma_m^{-1}$ extends to an equivalence of triangulated categories
\begin{equation*}
\hatnu_m:D^b(E\otimes\hatS,\Frob)\cong\hmonm{X}.
\end{equation*}
We define $\hperm{X}$ to be the image of $\Mod(E\otimes\hatS,\Frob)$ under $\hatnu_m$.
\item Under $\hatnu_m$, the adjunctions $(\piddag,\pidag)$ become
\begin{equation*}
\xymatrix{D^b(E\otimes\hatS,\Frob)\ar@<.5ex>[rr]^{\Ltimes_{\hatS}\Ql} && D^b(E,\Frob)\ar@<.5ex>[ll]^{\triv}}.
\end{equation*}
\end{enumerate}
\end{cor}
\begin{proof}
Most of the arguments are the same as the proof of Corollary \ref{c:modES}. We only have to notice that the ext-groups in the mixed settings (both the Yoneda ext's in $\perm{X}$ and the ext's in $\monm{X}$) are obtained by taking $\upH^*(\ZZ\Frob,-)$ on the $\bR\Hom$-complexes for the underlying complexes on $\geom{X}$. Hence the full faithfulness of $\rho_{X,m}$ follows from the calculations in the proof of Corollary \ref{c:modES}(1).
\end{proof}

\begin{remark}\label{r:projresmix} A mixed analog of Remark \ref{r:projres} holds: if $\calF\in\hmon{X}$ and $\piddag\calF$ has a resolution as in \eqref{mixK}, then $\calF$ can be represented by a filtered complex $\tcK\in\hatD'_mF(\qw{X}{A})$ such that $\tcK^i:=\Gr_F^i\tcK[i]$ satisfies $\piddag\tcK^i\cong\calK^i$, and $\omega\tcK^i\cong\Free(\omega\calK^i)$ (However, there is no guarantee that $\tcK^i$ is isomorphic to $\Free(\calK^i)$). After identifying $C^b(\hperm{X})$ with a full subcategory of $\hatD'_mF(\qw{X}{A})$, the object $\calF$ itself can be represented by a complex
\begin{equation*}
[\cdots\to\tcK^{-1}\to\tcK^0\to\cdots]\in C^b(\hsP).
\end{equation*}
which has the same length as \eqref{mixK}.

In particular, if $\piddag\calF\in P'_m(Y)$ and $\omega(\piddag\calF)$ is a projective object in $P'(Y)$, then $\omega\calF$ is itself a \fm perverse local system. If, moreover, $\Gr_W^i(\piddag\calF)$ is nonzero for at most one $i$, then $\calF\cong\Free(\piddag\calF)$.
\end{remark}

\subsection{The stratified case}\label{ss:stra}
We continue with the situation in \S\ref{ss:mix}. We further suppose that $Y$ has a finite stratification:
\begin{equation*}
Y=\bigsqcup_{\alpha\in I}Y_\alpha
\end{equation*}
such that each embedding $i_\alpha:Y_\alpha\hookrightarrow Y$ is affine and each $Y_\alpha$ is smooth of equidimension $d_\alpha$.  Let $X_\alpha:=\pi^{-1}(Y_\alpha)$. Let $i_\alpha:Y_\alpha\hookrightarrow Y$ and $\tili_\alpha:X_\alpha\hookrightarrow X$ be the inclusions. For each $\alpha\in I$, let $Y_{\leq\alpha}$ be the closure of $Y_\alpha$ and let $Y_{<\alpha}=Y_{\leq\alpha}-Y_{\alpha}$. Similarly define $X_{\leq\alpha}$ and $X_{<\alpha}$.

Let $\scD\subset D^b_m(Y)$ be a full triangulated subcategory stable under twists from sheaves on $\Spec(k)$, whose objects are constructible with respect to the given stratification. Let $\scD_{\leq\alpha}=\scD\cap D^b_m(Y_{\leq\alpha})$ and define $\scD_{<\alpha}$ similarly. Let  $\scD_\alpha:=\scD_{\leq\alpha}/\scD_{<\alpha}$, which is naturally a full subcategory of $D^b_{m}(Y_\alpha)$.

Now we take $D'_m(Y)=\scD$ and apply the constructions in Definition \ref{def:weightbound}. We denote $\monm{X}$ by $\scM$ and $\hmonm{X}$ by $\hsM$. We can also restrict the situation to any locally closed union of strata. In particular, we can define $\scM_{\leq\alpha},\hsM_{\leq\alpha},\scM_\alpha,\hsM_{\alpha}$ etc. The natural functors $\tili_{\alpha}^?,\tili_{\alpha,?},\tili_{\leq\alpha}^?,\tili_{\leq\alpha,?}$, etc. (for $?=!$ or $*$) and their adjunctions, natural transformations all extend to the completions.

We denote the non-mixed versions of the above categories by $\omega\scD,\omega\scM,\omega\hsM$, etc. These are categories of complexes on $\geom{Y},\geom{X}$, etc.

The category $\scD$ (resp. $\scM$) inherits a perverse t-structure whose heart we denote by $\scQ$ (resp. $\scP$). Similarly, let $\scQ_\alpha$ (resp. $\scP_\alpha$) be the heart of $\scD_\alpha$ (resp. $\scM_\alpha$).

We assume that each category $\scD_\alpha$ has the simplest possible type:

\textbf{Assumption S.} Each $X_\alpha$ is a trivial $A$-torsor over $Y_\alpha$, and $\upH^{*}(\geom{Y_\alpha})=\Ql$. Moreover, there is a rank one perverse local system $\calL_\alpha\in\scQ_\alpha$ such that $\omega\calL_\alpha$ is the unique irreducible object in $\omega\scQ_\alpha$.

Assumption S implies a t-exact equivalence of triangulated categories
\begin{equation*}
\nu_\alpha:D^b(\Frob)\cong\scD_\alpha.
\end{equation*}
sending the trivial $\Frob$-module $\Ql$ to $\calL_\alpha$. For each $\alpha$, Corollary \ref{c:modESmix} gives a natural equivalence
\begin{equation}\label{stratumeq}
\hatnu_\alpha:D^b(\hatS,\Frob)\cong\hsM_\alpha
\end{equation}
under which the free module $\hatS$ goes to $\tcL_\alpha=\Free(\calL_\alpha)\in\hsM_\alpha$.

Let
\begin{equation*}
\Delta_\alpha:=i_{\alpha,!}\calL_\alpha, \nabla_\alpha:=i_{\alpha,*}\calL_\alpha.
\end{equation*}
Then $\scD$ is generated as a triangulated category by either the twists of $\{\Delta_\alpha|\alpha\in I\}$ or $\{\nabla_\alpha|\alpha\in I\}$. Let
\begin{equation*}
\tDel_\alpha:=\tili_{\alpha,!}\tcL_\alpha, \tnab_\alpha:=\tili_{\alpha,*}\tcL_\alpha.
\end{equation*}

\begin{lemma}\label{l:delgen}
The triangulated category $\hsM$ is generated by either the twists of $\{\tDel_\alpha|\alpha\in I\}$ or $\{\tnab_\alpha|\alpha\in I\}$.
\end{lemma}
\begin{proof}
Any $\calF\in\hsM$ is expressed as a successive extension of $\tili_{\alpha,*}\tili^!_\alpha\calF$ (resp. $\tili_{\alpha,!}\tili^*_\alpha\calF$) for $\alpha\in I$. By Proposition \ref{p:freegenmix}, each $\tili^!_\alpha\calF\in\hsM_\alpha$ (resp. $\tili^*_\alpha\calF$) is a successive extension of shifts of \fm objects, hence a successive extension of shifts and twists of $\tcL_\alpha$ by Assumption S. Therefore $\calF$ is a successive extension of shifts and twists of $\tnab_\alpha$ (resp. $\tDel_\alpha$).
\end{proof}

\begin{lemma}\label{l:prot} The perverse t-structure on $\scM$ extends to a t-structure $(\hsM^{\leq0},\hsM^{\geq0})$ on $\hsM$, such that the natural inclusions
\begin{eqnarray}\label{Mleq}
\pro\scM^{\leq0}\cap\hsM\hookrightarrow\hsM^{\leq0};\\
\label{Mgeq}
\pro\scM^{\geq0}\cap\hsM\hookrightarrow\hsM^{\geq0}
\end{eqnarray}
are equivalences of categories.
\end{lemma}
\begin{proof}
According to Remark \ref{r:simplet}, for each $\alpha$, the equivalence $\hatnu_\alpha$ in \eqref{stratumeq} gives a t-structure $(\hsM^{\leq0}_\alpha$,$\hsM^{\geq0}_\alpha)$ on $\hsM_\alpha$. We can apply the gluing procedure in \cite[\S 1.4]{BBD} to obtain the desired t-structure on $\hsM$.

Next we prove that \eqref{Mleq} is an equivalence  (and proof for \eqref{Mgeq} is similar and will be omitted). We first prove a general result.
\begin{claim}
Let $D'$ be a triangulated category with a t-structure $(D'^{\leq0},D'^{\geq0})$. Let $\hatD\subset\pro(D')$ be a triangle-complete full subcategory satisfying the assumptions of Theorem \ref{th:proD}. Then $\hatD$ is naturally a triangulated category. Suppose $X\to Y\to Z\to X[1]$ is a distinguished triangle in $\hatD$ such that $X,Z\in\pro(D'^{\leq0})$, then $Y$ is isomorphic to an object in $\pro(D'^{\leq0})$.
\end{claim}
\begin{proof}
Let $X=\prolim X_n,Z=\prolim Z_n$ with $X_n,Z_n\in D^{\leq0}$. Then the map $f:Z\to X[1]$ is given by a projective system of maps $f_n:Z_{z(n)}\to X_n[1]$ where $\{Z_{z(n)}\}$ is a cofinal subsequence of $\{Z_n\}$. Let $Y_n[1]$ be the cone of $f_n$. It is clear then $Y_n\in D^{\leq0}$. The axiom (TR3) of triangulated categories makes $Y_n$ into a projective system $\prolim Y_n\in\pro(D^{\leq0})$. Since $\prolim Y_n[1]$ is also a cone of $f$, we have $Y\cong\prolim Y_n$.
\end{proof}
Now we prove that \eqref{Mleq} is an equivalence. If $\calF\in\hsM^{\leq0}$, we need to find a projective system $\calF'_n\in\scM^{\leq0}$ such that $\calF\cong\prolim\calF'_n$. We do this by induction on the support of $\calF$. Suppose $\calF\in\hsM^{\leq0}_{\leq\alpha}$, and by induction hypothesis we can find $\calG_n\in\scM^{\leq0}_{<\alpha}$ such that $\tili^*_{<\alpha}\calF\cong\prolim\calG_n$. Using Remark \ref{r:simplet}, we can also find $\calH_n\in\hsM^{\leq0}_\alpha$ such that $\tili^*_\alpha\calF\cong\prolim\calH_n$. Therefore $\calF$ fits into a distinguished triangle
\begin{equation*}
\tili_{\alpha,!}\tili^*_{\alpha}\prolim\calH_n\to\calF\to\tili_{<\alpha,*}\tili^*_{<\alpha}\prolim\calG_n\to\tili_{\alpha,!}\tili^*_{\alpha}\prolim\calH_n[1]
\end{equation*}
Now we apply the above claim to finish the proof.
\end{proof}

We denote the heart of the extended perverse t-structure on $\hsM$ by $\hsP$. It is clear that $\pro\scP\cap\hsM\subset\hsP$. This inclusion is in fact also an equivalence of categories, but we shall not need this fact. The objects $\tDel_\alpha,\tnab_\alpha$ belong to $\hsP$.

\subsection{Free-monodromic tilting sheaves}\label{ss:fmt}
\begin{defn}\label{def:fmt}
\begin{enumerate}
\item []
\item An object $\calT\in\omega\hsM$ is called a {\em \fm tilting sheaf}, if for each $\alpha\in I$, both complexes $\tili^*_\alpha\calT$ and $\tili^!_\alpha\calT$ (as objects in $\omega\hsM_\alpha$) are \fm perverse local systems (see Def. \ref{def:fm}; in our situation, this simply means a direct sum of $\omega\tcL_\alpha$'s).
\item An object $\calT\in\hsM$ is called a {\em (mixed) \fm tilting sheaf}, if $\omega\calT\in\hsM$ is a \fm tilting sheaf.
\end{enumerate}
\end{defn}

It is clear that $\calT\in\hsM$ is a \fm tilting sheaf if and only if it is both a successive extension of twists of $\tDel_\alpha$ (we call such an expression a {\em $\tDel$-flag}) and a successive extension of twists of $\tnab_\alpha$ (we call such an expression a {\em $\tnab$-flag}).

Let $\scT\subset\hsP$ be the additive full subcategory consisting of \fm tilting sheaves.

\begin{lemma}\label{l:pitilt}
An object $\calT\in\hsM$ is a \fm tilting sheaf if and only if $\piddag\calT\in\scD$ is a tilting sheaf.
\end{lemma}
\begin{proof}
For fixed $\alpha$ and an object $\calF_\alpha\in\hsM_\alpha$, $\omega\calF_\alpha$ is a \fm perverse local system if and only if $\piddag\calF_\alpha\in\scQ_\alpha$. In fact, this follows from the equivalence $\hatnu_\alpha$ in \eqref{stratumeq} and the well-known facts about $\hatS$-modules. This immediately implies the lemma.
\end{proof}

By \cite[\S 1.1-1.4]{BBM} and \cite[Lemma 2.2.3]{Yun2}, for each stratum $\alpha$, there is a mixed tilting sheaf $\calT_\alpha\in\scQ_{\leq\alpha}$ whose restriction to $Y_\alpha$ is $\calL_{\alpha}$ and whose underlying complex $\omega\calT_\alpha$ is indecomposable (note that {\em loc.cit} only dealt with the case when $\calL_\alpha$ is constant, however, for the argument there to work one only needs the vanishing of $\upH^{i}(\geom{X_\alpha})$ for $i=1,2$, which is ensured by Assumption S). The following lemma is an analogous existence result for mixed \fm tilting sheaves. By \cite[\S 1.4]{BBM}, $\{\omega\calT_\alpha|\alpha\in I\}$ are the only indecomposable tilting sheaves (up to isomorphism), and any tilting sheaf $T\in\omega\scQ$ is a direct sum of the $\omega\calT_\alpha$'s. A  \fm analog of this structure result will be proved in Remark \ref{rm:unique}(2).

\begin{lemma}\label{l:fmtexist}
For each $\alpha\in I$, there exists a mixed \fm tilting sheaf $\tcT_\alpha\in\hsM_{\leq\alpha}$ such that $\piddag\tcT_\alpha\cong\calT_\alpha$.
\end{lemma}
\begin{proof}
We use the pattern of the proof of \cite[\S 1.1]{BBM} and \cite[Lemma 2.2.3]{Yun2} (for the mixed case), although some new argument is required. We proceed by induction on strata. In the induction step, as in {\em loc.cit}, we may assume that on $X$ has a minimal stratum $Z$, and the required mixed \fm tilting sheaf has been constructed on $U=X-Z$ (call it $\tcT_U$, such that $\piddag\tcT_U\cong\calT_{\alpha}|_U$). Let $\tilj:U\hookrightarrow X$ and $\tili:Z\hookrightarrow X$ be the inclusions. Since $\tcT_U$ is a successive extension of twists of the $\tDel_{\beta,U}$'s, $\tilj_!\tcT_U$ is still a successive extension of twists of the $\tDel_\beta$'s, hence belongs to $\hsP$. Same remark applies to $\tilj_*\tcT_U$.

The complex $[\tilj_!\tcT_U\to\tilj_*\tcT_U]\in D^b(\hsP)\cong\hsM$, after applying $\piddag$, becomes the complex $[j_!\calT_{\alpha,U}\to j_*\calT_{\alpha,U}]\in D^b(\scQ)\cong\scD$, which can be represented by $[i_*\calA\xrightarrow{0}i_*\calB]$ for some $\calA,\calB\in\scQ_Z$ (cf. the argument in {\em loc.cit}).

By Remark \ref{r:projresmix}, the complex $[\tilj_!\tcT_U\to\tilj_*\tcT_U]$ itself can therefore be represented by $[\tili_*\tcK^{-1}\xrightarrow{d}\tili_*\tcK^{0}]$, where $\tcK^{-1},\tcK^0\in\hsP_Z$ satisfy $\piddag\tcK^{-1}\cong\calA$ and $\piddag\tcK^0\cong\calB$. We therefore get an extension class
\begin{equation*}
\tili_*\tcK^0\to[\tilj_!\tcT_U\to\tilj_*\tcT_U]\to\tilj_!\tcT_U[1].
\end{equation*}
Let $\tcT\in\hsP$ be an object realizing the above extension class. Then $\piddag\tcT$ realizes a similar extension class $i_*\calB\to j_!\calT_{\alpha,U}[1]$, which is known to be realized by $\calT_\alpha$ (cf. the argument in {\em loc.cit}). Therefore $\piddag\tcT\cong\calT_\alpha$. By Lemma \ref{l:pitilt}, this implies that $\tcT$ is a \fm tilting sheaf.
\end{proof}


\begin{lemma}\label{l:HomfreeS}
Let $\tcT_1,\tcT_2\in\scT$. Then $\Hom_{\hsM}(\tcT_1,\tcT_2)$ is a free $\hatS$-module, and there is a $\Frob$-equivariant isomorphism
\begin{equation*}
\Hom_{\hsM}(\tcT_1,\tcT_2)\otimes_{\hatS}\Ql\cong\Hom_{\scD}(\piddag\tcT_1,\piddag\tcT_2).
\end{equation*}
\end{lemma}
\begin{proof}
We prove a stronger version when $\tcT_{1}$ is only assumed to have a $\tDel$-flag and $\tcT_{2}$ is only assumed to have a $\tnab$-flag. The functorial map $\Hom_{\hsM}(\tcT_1,\tcT_2)\to\Hom_{\scD}(\piddag\tcT_1,\piddag\tcT_2)$ necessarily factors through the quotient $\Hom_{\hsM}(\tcT_1,\tcT_2)\otimes_{\hatS}\Ql$ because the monodromy operator acts trivially after taking $\piddag$.

For $X=X_\alpha=A\times Y_\alpha$ a single stratum, we simply apply Corollary \ref{c:modESmix}(3). In general, we proceed by induction on strata. In the induction step, let $X_\alpha$ be an open stratum and assume the lemma holds for $X_{<\alpha}=X-X_\alpha$ (extend the partial order on strata to a total order). Then we have an exact sequence
\begin{equation*}
0\to\Hom_{\hsM_{<\alpha}}(\tili_{<\alpha}^*\tcT_1,\tili^!_{<\alpha}\tcT_2)\to\Hom_{\hsM}(\tcT_1,\tcT_2)\to\Hom_{\hsM_{\alpha}}(\tili^*_\alpha\tcT_1,\tili^*_\alpha\tcT_2)\to0.
\end{equation*}
Since the two ends are free $\hatS$-modules, so is the middle one. Moreover, letting $\calT_i=\piddag\tcT_i$, we have a commutative diagram of short exact sequences
\begin{equation*}
\xymatrix{\Hom_{\hsM_{<\alpha}}(\tili_{<\alpha}^*\tcT_1,\tili^!_{<\alpha}\tcT_2)\otimes_{\hatS}\Ql\ar[d]\ar[r] &\Hom_{\hsM}(\tcT_1,\tcT_2)\otimes_{\hatS}\Ql\ar[d]\ar[r] & \Hom_{\hsM_{\alpha}}(\tili^*_\alpha\tcT_1,\tili^*_\alpha\tcT_2)\otimes_{\hatS}\Ql\ar[d]\\
\Hom_{\scD_{<\alpha}}(i_{<\alpha}^*\calT_1,i^!_{<\alpha}\calT_2)\ar[r] &\Hom_{\scD}(\calT_1,\calT_2)\ar[r] & \Hom_{\scD_{\alpha}}(i^*_\alpha\calT_1,i^*_\alpha\calT_2)}
\end{equation*}
The two vertical arrows on the left and right ends are isomorphisms by induction hypothesis, therefore the middle vertical arrow is also an isomorphism.
\end{proof}

\section{Construction of DG models}\label{a:dgmodel}
In this appendix, we construct differential-graded (DG) models for certain triangulated categories of complexes of sheaves on schemes or stacks. These DG models are known to exist in greater generality; however, we need explicit constructions for the purpose of proving the various equivalences in \S\ref{s:dual}. The basic strategy is to single out certain distinguished generators of the category in question (such as IC-sheaves or \fm tilting sheaves) and show that the endomorphism algebra of their direct sum is a {\em formal} DG algebra. We then identify the original category with the derived category of DG modules over this formal DG algebra. We remark that this strategy is standard in geometric representation theory, see \cite[\S9.5-9.7]{ABG} and \cite[\S6.5]{BF}; see also \cite{Sch} for an approach in the setting of complex algebraic geometry and mixed Hodge modules. Our contribution here is to give a unified way of treating diverse situations (such as equivariant and monodromic categories that appear in the main body of the article).

\subsection{A simple subcategory}\label{ss:setting}
We will consider one of the following two situations.

(i) Let $X$ be a global quotient stack (see \S\ref{not:geom}) over a finite field $k$ with a finite stratification $X=\bigsqcup_{\alpha\in I}X_\alpha$ such that each embedding $i_\alpha:X_\alpha\hookrightarrow X$ is affine. Let $\scD\subset D^b_m(X)$ be a full triangulated subcategory stable under twists (tensoring by $\Frob$-modules), and all of whose objects are constructible along the given stratification.

(ii) Consider the situation of \S\ref{ss:stra}. Let $Y$ be a {\em scheme} as in (i) and let $\pi:X\to Y$ be an $A$-torsor, where $A$ is a split torus over $k$. Let $X=\sqcup_{\alpha\in I}X_\alpha$ be the induced stratification: $X_\alpha=\pi^{-1}(Y_\alpha)$. Let $D'(Y)\subset D^b_m(Y)$ be a full triangulated subcategory stable under twists (tensoring by $\Frob$-modules), and all of whose objects are constructible along the given stratification. Let $\scD=\hmonm{X}$.

In either of the two situations, we denote by $X_{\leq\alpha},X_{<\alpha}$ the closure and boundary of $X_\alpha$. We therefore get full subcategories $\scD_{\leq\alpha},\scD_{<\alpha}\subset\scD$ by considering $X_{\leq\alpha}$ and $X_{<\alpha}$ instead of $X$. Let $\scD_\alpha=\scD_{\leq\alpha}/\scD_{<\alpha}$. We use $\omega\scD,\omega\scD_{\alpha}$ etc. to denote the non-mixed versions of $\scD,\scD_\alpha$ etc. For example, $\omega\scD$ is the image of $\scD$ in $D^b_m(\geom{X})$ or $\hatD'_m(\qw{\geom{X}}{\geom{A}})$.

\textbf{Assumption C$_1$.} Suppose we are given, for each $\alpha\in I$, a full additive subcategory $\scC_\alpha\subset\scD_\alpha$ stable under tensoring with unipotent $\Frob$-modules, such that for any objects $\calC_1,\calC_2\in\scC_\alpha$,
\begin{equation*}
\Ext^i_{\scD_\alpha}(\calC_1,\calC_2)^{\Funi}=0, \textup{ for }i\neq0.
\end{equation*}

Let $\scC\subset\scD$ be the full additive subcategory consisting of objects $\calF$ such that $i^*_\alpha\calF,i^!_\alpha\calF\in\scC_\alpha$ for all $\alpha\in I$. Then $\scC$ is also stable under tensoring with unipotent $\Frob$-modules. An immediate consequence of Assumption C$_1$ is:
\begin{lemma}\label{l:hominD}  For $\calC_1,\calC_2\in\scC$, we have
\begin{eqnarray}\label{omhom}
\Ext^i_{\scD}(\calC_1,\calC_2)^{\Funi}=0,\textup{ for }i\neq0.\\
\label{C12mixhom}
\ext^i_{\scD}(\calC_1,\calC_2)=
\begin{cases}\Hom_{\scC}(\calC_1,\calC_2)^{\Frob} & i=0;\\ \Hom_{\scC}(\calC_1,\calC_2)_{\Frob} & i=1;\\ 0 & \textup{otherwise}\end{cases}
\end{eqnarray}
\end{lemma}
\begin{proof}
\eqref{omhom}. We do induction by strata. For a single stratum this follows from Assumption C$_1$. Suppose \eqref{omhom} holds for objects in $\scC_{<\alpha}$. Then for $\calC_1,\calC_2\in\scC_{\leq\alpha}$, we have a long exact sequence
\begin{equation}\label{stratass}
\cdots\to\Ext^i(i^*_{<\alpha}\calC_1,i^!_{<\alpha}\calC_2)^{\Funi}\to \Ext^i(\calC_1,\calC_2)^{\Funi}\to\Ext^i(i^*_\alpha\calC_1,i^*_\alpha\calC_2)^{\Funi}\to\cdots
\end{equation}
where the Ext-groups are taken in $\omega\scD_{<\alpha},\omega\scD_{\leq\alpha}$ and $\omega\scD_{\alpha}$ respectively. We then use the induction hypothesis and Assumption C$_1$ for $\scC_\alpha$ to finish the induction step.

\eqref{C12mixhom} follows from \eqref{omhom} and \eqref{intromix}. In situation (ii), we refer to Lemma \ref{l:mixhom} for the calculation of $\Ext$-groups in $\scD$.
\end{proof}

\begin{exam}\label{ex:ic}
In situation (i), we assume $\upH^i(\geom{X_\alpha})$ is pure of  weight $i$. Let $\scC_\alpha$ consist of mixed complexes $\calC$ on $X_\alpha$ which are pure of weight 0 and constant over $\geom{X_\alpha}$. The purity of $\upH^i(X_\alpha)$ ensures that $\Ext^i_{\scD_\alpha}(\calC_1,\calC_2)$ is pure of weight $i$ for $\calC_1,\calC_2\in\scC_\alpha$, which, in particular, implies Assumption C$_1$.

In this case, $\scC$ consists of {\em very pure} complexes (compare Def.\ref{def:vpure}). A typical example in applications is that of $X_\alpha=\BB A$, the classifying space of a torus $A$.
\end{exam}

\begin{exam}\label{ex:fmt}
In situation (ii), we suppose Assumption S in \S\ref{ss:stra} holds. Recall $\tcL_\alpha$ a \fm perverse local system on $X_\alpha$. Let $\scC_\alpha$ consist of objects $\tcL_\alpha\otimes M$, for any $\Frob$-modules $M$. The vanishing of $\upH^{>0}(\geom{Y_\alpha})$ (see Assumption S) ensures a stronger vanishing than Assumption C$_1$:$\Ext^i_{\scD_\alpha}(\calC_1,\calC_2)=0$ for $i\neq0$ and $\calC_1,\calC_2\in\scC_\alpha$.

In this case, $\scC$ consists of {\em \fm tilting sheaves} (see Def.\ref{def:fmt}). Note that we may take $A$ to be the trivial torus, then no completion is needed, and $\scC$ consists of tilting sheaves.
\end{exam}

Let $\scD F$ be the filtered version of $\scD$ (see Remark \ref{r:filD} for situation (ii)). Let $\Omega:\scD F\to\scD$ be the ``forgetting the filtration'' functor. Let $\scD F(\scC)$ be the full subcategory consisting of filtered complexes $\calK$ such that $\Gr^i_F\calK\in\scC[-i]$ for each $i\in\ZZ$. We have a natural functor
\begin{equation*}
\Gr^*_F:\scD F(\scC)\to C^b(\scC)
\end{equation*}
which sends $\calK$ to the complex
\begin{equation*}
\cdots\to\Gr^{i}_F\calK[i]\xrightarrow{d_i}\Gr^{i+1}_F\calK[i+1]\to\cdots
\end{equation*}
where $d_i$ comes from the third arrow of the distinguished triangle $\Gr^{i+1}_F\calK\to F^{\leq i+1}F^{\geq i}\calK\to \Gr^i_F\calK\to\Gr^{i+1}_F[1]$. The argument of \cite[Proposition  3.1.8]{BBD} shows that

\begin{lemma}
The functor $\Gr^*_F$ is an equivalence of categories.
\end{lemma}
Here, the key point that makes the argument in {\em loc.cit.} work is the vanishing of $\ext^{<0}_{\scD}$ between objects in $\scC$.

Let $\tilrho(\scC)$ be the composition $C^b(\scC)\xrightarrow{(\Gr^*_F)^{-1}}\scD F(\scC)\xrightarrow{\Omega}\scD$. Then $\tilrho(\scC)$ factors through an exact functor
\begin{equation*}
\rho(\scC):K^b(\scC)/K^b_{\acyc}(\scC)\to\scD.
\end{equation*}
where $K^b_{\acyc}(\scC)\subset K^b(\scC)$ is the thick subcategory consisting of complexes in $\scC$ whose image in $\scD$ is 0. We call such complexes {\em acyclic complexes}.

\begin{lemma}\label{l:Cexact}
For any $\calC\in\scC$, $\Hom_{\scC}(\calC,-)^{\Funi}$ and $\Hom_{\scC}(-,\calC)^{\Funi}$ transform acyclic complexes in $K^b_{\acyc}(\scC)$ into long exact sequences of vector spaces.
\end{lemma}
\begin{proof}
We prove the statement about $\Hom_{\scC}(\calC,-)^{\Funi}$ and the other one is similar. Let $\calK\in\scD F(\scC)$ be a filtered object which gives an acyclic complex in $\scC$ by taking $\Gr^{*}_{F}$, i.e., $\calK$ is isomorphic to the zero object in $\scD$. There is a spectral sequence $\{E_{r}\}_{r}$ with
\begin{equation*}
E_{1}^{p,q}=\Ext^{p+q}_{\scD}(\calC,\Gr^{p}_{F}\calK)=\Ext^{q}_{\scD}(\calC,\Gr^{p}_{F}\calK[p])
\end{equation*}
abutting to $\Ext^{p+q}_{\scD}(\calC,\calK)=0$. Taking $(-)^{\Funi}$, we get a spectral sequence $\{E^{\Funi}_{r}\}_{r}$ with $E_{1}^{\Funi}$ concentrated on the row $q=0$ by Lemma \ref{l:hominD}. The differentials on $E^{\Funi}_{1}$ make it the complex of vector spaces $[\cdots\to\Hom_{\scC}(\calC,\Gr^{p}_{F}\calK[p])\to\cdots]$ obtained by applying $\Hom_{\scC}(\calC,-)$ to $\Gr^{*}_{F}\calK$. This complex is necessarily exact because $\{E^{\Funi}_{r}\}_{r}$ abuts to zero.
\end{proof}

Let $\scC^{\Funi}$ be the category with the same objects as $\scC$, but the Hom sets are defined by
\begin{equation*}
\Hom_{\scC^{\Funi}}(\calC_1,\calC_2):=\Hom_{\scC}(\calC_1,\calC_2)^{\Funi}.
\end{equation*}

\begin{lemma}\label{l:nullhomo}
If a complex $\calK\in K^b(\scC)$ has zero image in $\scD$, then it is zero in $K^b(\scC^{\Funi})$, i.e., $\id_{\calK}$ is homotopic to the zero map in $C^b(\scC^{\Funi})$.
\end{lemma}
\begin{proof}
Suppose $[\cdots\to\calC_1\xrightarrow{\partial_1}\calC_0\to0]$ is a complex in $\scC$ that terminates at degree 0. We construct inductively a homotopy $h_i\in\Hom(\calC_i,\calC_{i+1})^{\Funi}$  such that
\begin{equation}\label{homot}
h_{i-1}\partial_{i}+\partial_{i+1}h_{i}=\id_{\calC_i}, \textup{ for }i=0,1,...
\end{equation}

Starting with $i=0$. By Lemma \ref{l:Cexact}, we have a long exact sequence
\begin{equation*}
\cdots\to\Hom(\calC_0,\calC_1)^{\Funi}\xrightarrow{\partial_1}\Hom(\calC_0,\calC_0)^{\Funi}\to0\to\cdots
\end{equation*}
Therefore $\id_{\calC_0}$ lifts to a map $h_0\in\Hom(\calC_0,\calC_1)^{\Funi}$.

Suppose we have found $h_0,\cdots,h_{i-1}$ satisfying \eqref{homot}. Then by Lemma \ref{l:Cexact}, we again have a long exact sequence
\begin{equation*}
\cdots\to\Hom(\calC_i,\calC_{i+1})^{\Funi}\xrightarrow{\partial_{i+1}}\Hom(\calC_i,\calC_i)^{\Funi}\xrightarrow{\partial_i}\Hom(\calC_i,\calC_{i-1})^{\Funi}\to\cdots
\end{equation*}
Since $\id_{\calC_i}-h_{i-1}\partial_i$ has zero image under $\partial_i$, it lifts to the desired map $h_i\in\Hom(\calC_i,\calC_{i+1})^{\Funi}$. This completes the induction.
\end{proof}

\begin{prop}\label{p:rhoff} The functor $\rho(\scC):K^b(\scC)/K^b_{\acyc}(\scC)\to\scD$ is fully faithful. It is an equivalence of categories if each $\scC_\alpha$ generates $\scD_\alpha$ as a triangulated category.
\end{prop}
\begin{proof}
By definition, the $\ext$-groups in $K^b(\scC)/K^b_{\acyc}(\scC)$ are computed by
\begin{eqnarray}\label{Tlim}
\ext^i_{K^b(\scC)/K^b_{\acyc}(\scC)}(\calC,\calC')&=&\varinjlim_{\calK\to\calC\textup{ with acyclic cone}}\hom_{K^b(\scC)}(\calK,\calC'[i])
\end{eqnarray}
We will exhibit a cofinal set of maps to $\calC$ with acyclic cones. Consider $\calC[t]/t^n=\calC\otimes\Ql[t]/t^n\in\scD$ where $\Ql[t]/t^{n+1}$ is viewed as a $\Frob$-module where $\Frob$-acts as multiplication by $\exp(t)$. Since $\Ql[t]/t^n$ is a unipotent $\Frob$-module, and $\scC$ is stable under tensoring with unipotent $\Frob$-modules, $\calC[t]/t^n\in\scC$. Let $\calC[[t]]:=\prolim\calC[t]/t^n\in\pro\scC$.

Recall we have a forgetful functor $\scC\to\scC^{\Funi}$. It admits a left adjoint $\scC^{\Funi}\to\pro\scC$ sending $\calC\mapsto\calC[[t]]$. The adjunction means
\begin{equation}\label{funi}
\hom_{\pro\scC}(\calC[[t]],\calC')=\varinjlim\Hom(\calC[t]/t^{n+1},\calC')^{\Frob}\cong\Hom(\calC,\calC')^{\Funi}.
\end{equation}
In fact, the bijection is given by restricting $\phi:\calC[[t]]\to\calC'$ to $\omega\calC=\omega\calC\otimes1\subset\omega\calC[[t]]$; its inverse is given by sending $\psi:\omega\calC\to\omega\calC'$ to $\phi$ where $\phi|(\omega\calC\otimes t^n)=\log(\Frob)^n\psi$ (this makes sense because $(\Frob-1)^N\psi=0$ for large $N$).

Now suppose we are given a complex $\calK=[\cdots\to\calK^{-1}\xrightarrow{d^{-1}}\calK^0\to\cdots]$ which maps to $\calC$ (i.e., $f:\calK^0\to\calC$) with acyclic cone
\begin{equation*}
\cdots\to\calK^{-1}\xrightarrow{-d^{-1}}\calK^{0}\xrightarrow{(-d^0,f)}\calK^{1}\oplus\calC\xrightarrow{(-d^{1},0)}\calK^2\to\cdots.
\end{equation*}
Taking $\Hom(\calC,-)^{\Funi}$ on this sequence still yields a long exact sequence by Lemma \ref{l:Cexact}:
\begin{equation*}
\cdots\to\Hom(\calC,\calK^0)^{\Funi}\to\Hom(\calC,\calK^1\oplus\calC)^{\Funi}\to\cdots.
\end{equation*}
By \eqref{funi}, we get an exact sequence
\begin{equation*}
\cdots\to\hom(\calC[[t]],\calK^0)\to\hom(\calC[[t]],\calK^1\oplus\calC)\to\hom(\calC[[t]],\calK^2)\to\cdots.
\end{equation*}
Let $pr:\calC[[t]]\to\calC$ be the natural projection. Then $(0,pr)\in\hom(\calC[[t]],\calK^1\oplus\calC)$ has zero image in $\hom(\calC[[t]],\calK^2)$, hence lifts to a map $\phi^0:\calC[[t]]\to\calK^0$. Similarly, $t\phi^0\in\hom(\calC[[t]],\calK^0)$ has zero image in $\hom(\calC[[t]],\calK^1\oplus\calC)$, hence lifts to a map $\phi^{-1}:\calC[[t]]\to\calK^{-1}$. The maps $(\phi^{-1},\phi^0):\left[\calC[[t]]\xrightarrow{t}\calC[[t]]\right]\to\calK$ give a map between complexes which factors through
\begin{equation*}
\xymatrix{\cdots\to0\ar@<3ex>[d]\ar[r] & \calC[t]/t^n\ar[r]^{t}\ar[d]^{\bar{\phi}^{-1}} & \calC[t]/t^{n+1}\ar[r]\ar[d]^{\bar{\phi^0}} & 0\ar@<-3ex>[d]\to\cdots\\
\cdots\to\calK^{-2}\ar[r]^{d^{-2}} & \calK^{-1}\ar[r]^{d^{-1}} & \calK^0\ar[r]^{d^0} & \calK^1\to\cdots}
\end{equation*}
for some $n$. This proves the cofinality of the maps $\left[\calC[t]/t^{n}\xrightarrow{t}\calC[t]/t^{n+1}\right]\to\calC$.

Finally, when computing $\ext^*_{K^b(\scC)/K^b_{\acyc}(\scC)}(\calC,\calC')$ using the cofinal set of maps $\left[\calC[[t]]\xrightarrow{t}\calC[[t]]\right]\to\calC$, we see it is the cohomology of the following two-step complex
\begin{equation*}
0\to\hom(\calC[[t]],\calC')\xrightarrow{t}\hom(\calC[[t]],\calC')\to0.
\end{equation*}
By \eqref{funi}, this is the same as
\begin{equation*}
0\to\Hom(\calC,\calC')^{\Funi}\xrightarrow{\log(\Frob)}\Hom(\calC,\calC')^{\Funi}\to0.
\end{equation*}
But this is quasi-isomorphic to the complex calculating $\ext^*_{\scD}(\calC,\calC')$ (see Lemma \ref{l:hominD}). This shows that $\rho(\scC)$ is fully faithful.
\end{proof}

\subsection{The DG model}\label{ss:dgmodel}
To get nice DG models of $\scC$ and $\scD$, we make two more assumptions.

\textbf{Assumption C$_2$.} For every $\alpha$, there is an object $\tcL_\alpha\in\scC_\alpha$ such that every object in $\scC_\alpha$ has the form $\tcL_\alpha\otimes M$ for some complex $M$ of $\Frob$-modules. Moreover, $\scC_\alpha$ generates $\scD_\alpha$ as a triangulated category.

Note that this assumption is {\em not} saying that $\tcL_\alpha\otimes M\in\scC_\alpha$ for {\em any} $M\in D^b(\Frob)$. For example, in Example \ref{ex:ic}, $\scC_{\alpha}$ is stable under tensoring with $\Frob$-modules of weight $0$ only.

It is clear that the twists of either the objects $\{i_{\alpha,!}\calL_\alpha|\alpha\in I\}$ or the objects $\{i_{\alpha,*}\calL_\alpha|\alpha\in I\}$ generate $\scD$ as a triangulated category.

\textbf{Assumption C$_3$.} For every $\alpha$, there exists an object $\calC_\alpha\in\scC_{\leq\alpha}$ such that $i^*_\alpha\calC_\alpha\cong\tcL_\alpha$. Moreover, the kernel of the ring homomorphism $i^*_\alpha:\End_{\scC}(\calC_\alpha)^{\Funi}\to\End_{\scC_\alpha}(\tcL_\alpha)^{\Funi}$ is nilpotent.

\begin{exam}\label{ex:icc} In Example \ref{ex:ic}, assume the stratification on $X$ is given by orbits of a group $G$ acting on $X$. Assumption C$_2$ is satisfied with $\tcL_\alpha$ being the constant sheaf. Here it is crucial that we work with complexes with {\em integer} weights: otherwise $\scC_\alpha$ would not generate $\scD_\alpha$ as a triangulated category.

Let $\IC_\alpha=i_{\alpha,!*}\Ql$ (we may need to shift $\Ql$ to make sense of $i_{\alpha,!*}$, then shift back). Obviously $\IC_\alpha$ is $G$-equivariant, hence geometrically constant along each orbit. Then $\IC_\alpha\in\scC$ if and only if it is {\em very pure}, i.e., both $i^*_\alpha\IC_\alpha$ and $i^!_\alpha\IC_\alpha$ are pure of weight zero as complexes. In this case, Assumption C$_3$ is satisfied with $\calC_\alpha=\IC_\alpha$. In fact, the restriction map $\End_{\scD}(\IC_\alpha)\to\End_{\scC_\alpha}(\Ql)$ is an isomorphism.
\end{exam}

\begin{exam}\label{ex:fmtt} The example in \ref{ex:fmt} satisfies Assumptions C$_2$ and C$_3$ if we take $\tcL_\alpha$ to be the \fm perverse local system as before. In fact, we take $\calC_\alpha=\tcT_\alpha$ constructed in Lemma \ref{l:fmtexist}. Note that $\piddag\tcT_\alpha=\calT_\alpha$ is a tilting sheaf on $Y_{\leq\alpha}$ such that the natural map $i^!_\alpha\calT_\alpha\to i^*_\alpha\calT_\alpha$ is zero (see the construction in \cite[\S 1.1]{BBM}), therefore the natural map $\tili^!_\alpha\tcT_\alpha\to\tili^*_\alpha\tcT_\alpha$ is topologically nilpotent (in the $V_A$-adic topology, since both are $\hatS$-modules of finite type). This ensures Assumption C$_3$.
\end{exam}

\begin{lemma}\label{l:Cdecomp}
For every object $\calC\in\scC$, there are complexes $M_\alpha\in D^b(\Frob)$ such that $\calC$ is isomorphic to $\oplus_\alpha M_\alpha\otimes\calC_\alpha$ up to Frobenius semisimplification, i.e., there is an isomorphism $\oplus_\alpha M_\alpha\otimes\calC_\alpha\isom\calC$ in $\scC^{\Funi}$. In particular, $\alpha\leftrightarrow\omega\calC_\alpha$ sets up a bijection between the strata set $I$ and the isomorphism classes of indecomposable objects in $\omega\scC$.
\end{lemma}
\begin{proof}
We will make use of the simple observation that for $\calC_1,\calC_2\in\scC_{\leq\alpha}$, the restriction map
\begin{equation*}
\Hom_{\scC}(\calC_1,\calC_2)^{\Funi}\to\Hom_{\scC_\alpha}(i^*_\alpha\calC_1,i^*_\alpha\calC_2)^{\Funi}
\end{equation*}
is surjective. In fact, this follows from the long exact sequence \eqref{stratass} and the vanishing of $\Ext^1_{\scC_{<\alpha}}(i^*_{<\alpha}\calC_1,i^!_{<\alpha}\calC_2)^{\Funi}$.

We do induction on the support of $\calC$. Suppose $\calC\in\scC_{\leq\alpha}$. By Assumption C$_2$, $i^*_\alpha\calC=M_\alpha\otimes\omega\tcL_\alpha$ for some $M_\alpha\in D^b(\Frob)$. The above observation gives maps in both directions in $\scC^{\Funi}$
\begin{equation*}
\omega(M_\alpha\otimes\calC_\alpha)\xrightarrow{\phi}\omega\calC\xrightarrow{\psi}\omega(M_\alpha\otimes\calC_\alpha).
\end{equation*}
whose restrictions on $X_\alpha$ are identity. The composition $\psi\phi\in\End(M_\alpha\otimes\calC_\alpha)^{\Funi}$ is an isomorphism because its restriction to $X_\alpha$ is (here we use the nilpotency Assumption C$_3$). This implies that $\calC\isom M_\alpha\otimes\calC_\alpha\oplus\calC'$ in $\scC^{\Funi}$ for some $\calC'\in\scC_{<\alpha}$. We then apply induction hypothesis to $\calC'$.
\end{proof}

\begin{remark}\label{rm:unique}
(i) In Example \ref{ex:icc}, the objects $M_\alpha$ that appear in the decomposition above are necessarily pure of weight 0. The above statement can be rephrased as ``every very pure complex is a direct sum of shifted simple perverse sheaves up to Frobenius semisimplification'', which is a special case of the decomposition theorem \cite[Th\'eor\`eme 6.2.5]{BBD}.

(ii) In Example \ref{ex:fmtt}, the objects $M_\alpha$ that appear in the decomposition above are necessarily in degree 0. The above statement can be rephrased as ``every mixed \fm tilting sheaf is a direct sum of indecomposable mixed \fm tilting sheaves $\tcT_\alpha$ up to Frobenius semisimplification''. Note, however, this statement does {\em not} imply that any mixed \fm tilting sheaf with indecomposable underlying complex is isomorphic to the twist of some $\tcT_\alpha$.
\end{remark}

\begin{cor}\label{c:Cdecomp}
Let $\calC\in\scC$ be such that $\End_{\scC}(\calC)$ is $\Frob$-semisimple. Then $\calC\cong\bigoplus_{\alpha\in I}M_\alpha\otimes\calC_\alpha$ for $\Frob$-semisimple complexes $M_\alpha$ (i.e., complexes $M_\alpha=[\cdots\to M^0_\alpha\to M^1_\alpha\to\cdots]$ where each $M^i_\alpha$ is $\Frob$-semisimple).
\end{cor}
\begin{proof}
By Lemma \ref{l:Cdecomp}, the idempotents $\iota_\alpha$ corresponding to the direct summand $\omega(M_\alpha\otimes\calC_\alpha)$ of $\omega\calC_\alpha$ belong to $\End(\calC)^{\Funi}$. Since $\Frob$ acts semisimply on $\End(\calC)$, these idempotents are $\Frob$-invariant, hence the $M_\alpha\otimes\calC_\alpha$ are direct summands in $\scC$. Since $\id_{\calC_\alpha}\otimes\End(M_\alpha)\subset\End(\calC)$, $\End(M_\alpha)$ is also $\Frob$-semisimple. This implies that $M_\alpha$ is itself $\Frob$-semisimple.
\end{proof}

Suppose we are given another set of objects $\naC_\alpha\in\scC_{\leq\alpha}$, one for each $\alpha\in I$, such that $i^*_\alpha\naC_\alpha\cong\tcL_\alpha$. For example, we could take $\naC_\alpha$ to be $\calC_\alpha$.

\begin{lemma}\label{l:Cgen}
The triangulated category $\scD$ is generated by the twists of the objects $\{\naC_\alpha|\alpha\in I\}$.
\end{lemma}
\begin{proof}
We do induction on the strata\footnote{For this, we need to extend the partial order on the set of strata to a total order, and redefine $\scD_{\leq\alpha}$, etc. Suppose we have done this modification in the notations.}. Suppose $\scD_{<\alpha}$ is generated by twists of $\{\naC_\beta|\beta<\alpha\}$. We want to show that $\scD_{\leq\alpha}$ is generated by twists of $\{\naC_\beta|\beta\leq\alpha\}$. We have a canonical map $i_{\alpha,!}\tcL_\alpha\to\naC_\alpha$ whose cone lies in $\scD_{<\alpha}$, hence $i_{\alpha,!}\tcL_\alpha$ is generated by twists of $\{\naC_\beta|\beta\leq\alpha\}$. By Assumption C$_2$, $\scD_{\leq\alpha}$ is generated by $i_{\alpha,!}\tcL_\alpha$ and $\scD_{<\alpha}$, we are done.
\end{proof}

Let $\naC=\oplus_{\alpha}\naC_\alpha$ and
\begin{equation}\label{algE}
E=\bigoplus_{i\in\ZZ}\Ext^i_{\scD}(\naC,\naC)^{\opp}
\end{equation}
be a $\Ql$-algebra with Frobenius action (forgetting the cohomological grading).

\begin{theorem}\label{th:CtoMod} Fix a triple $(\scC\subset\scD,\{\naC_\alpha|\alpha\in I\})$ satisfying Assumptions C$_1$, C$_2$ and C$_3$. Then
\begin{enumerate}
\item The functor
\begin{equation*}
h_{\naC}=\bigoplus_{i\in\ZZ}\Ext^i_{\scD}(\naC,-):\scC\to\Mod(E,\Frob)
\end{equation*}
is a fully faithful embedding.
\item The functor $C^b(\scC)\to C^b(E,\Frob)$ induces a fully faithful embedding
\begin{equation*}
h_{\naC}:K^b(\scC)/K^b_{\acyc}(\scC)\hookrightarrow D^b(E,\Frob).
\end{equation*}
\item The composition of functors (note that $\rho(\scC)$ is an equivalence by Proposition \ref{p:rhoff} and Lemma \ref{l:Cgen})
\begin{equation*}
\MM=\MM(\scC\subset\scD;\{\naC_\alpha\}):\scD\xrightarrow{\rho(\scC)^{-1}}K^b(\scC)/K^b_{\acyc}(\scC)\xrightarrow{h_{\naC}}D^b(E,\Frob)
\end{equation*}
is fully faithful, and the essential image is the full triangulated subcategory generated by the twists of $(E,\Frob)$-modules $\{\Hom(\naC,\naC_\alpha)|\alpha\in I\}$.
\end{enumerate}
\end{theorem}
\begin{proof}
(1) Let $\calC_1,\calC_2\in\scC$. By Lemma \ref{l:hominD},
\begin{equation*}
\hom_{\scC}(\calC_1,\calC_2)=\Hom_{\scC}(\calC_1,\calC_2)^{\Frob}.
\end{equation*}
On the other hand,
\begin{equation*}
\hom_{\Mod(E,\Frob)}(h_{\naC}(\calC_1),h_{\naC}(\calC_2))=\Hom_E(h_{\naC}(\calC_1),h_{\naC}(\calC_2))^{\Frob}.
\end{equation*}
Therefore it suffices to show that the natural map
\begin{equation*}
H(\calC_1,\calC_2):\Hom_{\scC}(\calC_1,\calC_2)^{\Funi}\to\Hom_E(h_{\naC}(\calC_1),h_{\naC}(\calC_2))^{\Funi}
\end{equation*}
is an isomorphism for any $\calC_1,\calC_2\in\scC$. If $\calC_1=\naC$, we have
\begin{eqnarray*}
\Hom_E(h_{\naC}(\naC),h_{\naC}(\calC_2))^{\Funi}&=&\Hom_E(E,h_{\naC}(\calC_2))^{\Funi}\\
&=&h_{\naC}(\calC_2)^{\Funi}\\
&=&\Hom_{\scC}(\naC,\calC_2)^{\Funi}
\end{eqnarray*}
The last equality follows from Lemma \ref{l:hominD}. Hence $H(\calC_1,\calC_2)$ is an isomorphism for $\calC_1=\naC$. Therefore it is an isomorphism for $\calC_1=\naC_\alpha$, for any $\alpha$. By Lemma \ref{l:Cdecomp}, $\omega\naC_\alpha$ is a direct sum of the $\omega\calC_\beta$'s, which in particular contains $\omega\calC_\alpha$ as a direct summand in $\scC^{\Funi}$, hence $H(\calC_\alpha,\calC_2)$ is also an isomorphism. By Lemma \ref{l:Cdecomp} again, this means $H(\calC_1,\calC_2)$ is an isomorphism for all $\calC_1,\calC_2\in\scC$.

(2) By Lemma \ref{l:nullhomo}, objects in $K^b_{\acyc}(\scC)$ are null-homotopic in $K^b(\scC^{\Funi})$, hence they get mapped to acyclic complexes in $K^b(E,\Frob)$. Therefore we have the factorization $h:K^b(\scC)/K^b_{\acyc}(\scC)\to D^b(E,\Frob)$. By Proposition \ref{p:rhoff}, the $\ext$-groups in $K^b(\scC)/K^b_{\acyc}(\scC)$ are computed in the same way as in $\scD$, i.e., as in Lemma \ref{l:hominD}. Notice that the $\ext$-groups in $D^b(E,\Frob)$ are computed similarly by $\Frob$-invariants and coinvariants. Therefore $h$ is fully faithful.

(3) Obvious.
\end{proof}

\subsection{Functoriality of the DG model}

We study the functorial properties of the equivalences in Theorem \ref{th:CtoMod}. Let $\scD$ and $\scD'$ be two categories as in \S\ref{ss:setting}. Let $\Phi:\scD\to\scD'$ be an exact functor which admits a filtered lifting $\Phi F:\scD F\to\scD'F$. Let $\scC$ (resp. $\scC'$) be the subcategory of $\scD$ (resp. $\scD'$) satisfying all Assumptions C$_i$ ($i=1,2$ and $3$). Let $\naC\in\scC$ and $\naC'\in\scC'$ be the sum of generating objects $\{\naC_\alpha\}$ and $\{\naC'_\alpha\}$ as in Lemma \ref{l:Cgen}, and let $E,E'$ be the algebras as in \eqref{algE}.

\begin{prop}\label{p:Cfun}
Suppose $\Phi$ sends $\scC$ to $\scC'$. Let $\Phi|\scC:\scC\to\scC'$ be the restriction of $\Phi$. Then there are canonical natural isomorphisms making the following diagram commutative
\begin{equation*}
\xymatrix{\scD\ar[d]^{\Phi} & K^b(\scC)/K^b_{\acyc}(\scC)\ar[l]_(.7){\rho(\scC)}\ar[r]^(.6){h_{\naC}}\ar[d]^{K(\Phi|\scC)} & D^b(E,\Frob)\ar[d]^{B_{\Phi}\Ltimes_E(-)}\\
\scD' & K^b(\scC')/K^b_{\acyc}(\scC')\ar[l]_(.7){\rho(\scC')}\ar[r]^(.6){h_{\naC'}} & D^b(E',\Frob)}
\end{equation*}
where $B_\Phi$ is the $(E',E)$-bimodule (with $\Frob$-action)
\begin{equation*}
B_\Phi=\Hom_{\scC'}(\naC',\Phi(\naC)).
\end{equation*}
\end{prop}
\begin{proof}
The commutativity of the left side square is obvious. To give the natural transformation for the right side square, we only need to give a natural isomorphism making the following diagram commutative
\begin{equation*}
\xymatrix{\scC\ar[r]^(.4){h_{\naC}}\ar[d]^{\Phi|\scC} & \Mod(E,\Frob)\ar[d]^{B_\Phi\otimes_E(-)}\\
\scC'\ar[r]^(.4){h_{\naC'}} & \Mod(E',\Frob)}
\end{equation*}
There is a natural transformation
\begin{eqnarray*}
\beta(-)&:&B_{\Phi}\otimes_{E}\Hom_{\scC}(\naC,-)\\
&=&\Hom_{\scC'}(\naC',\Phi(\naC))\otimes_{E}\Hom_{\scC}(\naC,-)\\
&\to&\Hom_{\scC'}(\naC',\Phi(-))
\end{eqnarray*}
sending $f\otimes g$ to $\Phi(g)\circ f$. Since $\beta(-)$ is $\Frob$-equivariant, it suffices to show that $\beta(\calC)$ is an isomorphism for any $\calC\in\scC$.  This is obvious if $\calC_1=\naC$, hence also for $\calC=\naC_\alpha$ for any $\alpha$. By Lemma \ref{l:Cdecomp}, $\omega\naC_\alpha$ contains $\omega\calC_\alpha$ as a direct summand in $\scC^{\Funi}$, hence $\beta(\calC_\alpha)$ is also an isomorphism. By Lemma \ref{l:Cdecomp} again, this means $\beta(\calC)$ is an isomorphism for all $\calC\in\scC$.
\end{proof}

\begin{remark}\label{r:bifunctor}
The above Proposition has obvious versions for functors of the form $\Phi:\scD_1\times\cdots\times\scD_r\to\scD$ where $\scD_i$ and $\scD$ fit into the setting of Theorem \ref{th:CtoMod}. In particular, suppose $\scD$ carries a monoidal structure $\ast:\scD\times\scD\to\scD$ which restricts to a monoidal structure on $\scC$. Let $C$ be the $(E,E\otimes E)$-bimodule $\Hom_{\scC}(\naC,\naC\ast\naC)$. Then we have a natural commutative diagram of monoidal structures (i.e., together with compatibility among the associativity constraints):
\begin{equation*}
\xymatrix{\scD^2\ar[d]^{\ast} & (K^b(\scC)/K^b_{\acyc}(\scC))^2\ar[l]_(.7){\rho(\scC)}\ar[r]^(.6){h_{\naC}}\ar[d]^{K(\ast|\scC)} & D^b(E,\Frob)^2\ar[d]^{C\Ltimes_{E\otimes E}(-)}\\
\scD & K^b(\scC)/K^b_{\acyc}(\scC)\ar[l]_(.7){\rho(\scC)}\ar[r]^(.6){h_{\naC}} & D^b(E,\Frob)}
\end{equation*}
The compatibility among the associativity constraints follows from the canonicity of the natural isomorphisms in Proposition \ref{p:Cfun}.
\end{remark}

\subsection{Application to equivariant categories}
We first apply Theorem \ref{th:CtoMod} to the special case $X=\BB A$ of Example \ref{ex:icc}. This yields
\begin{cor}\label{c:ba}
\begin{enumerate}
\item []
\item There is an equivalence of triangulated categories
\begin{equation*}
D^b_m(\BB A)\cong D_{\perf}(\dS_A,\Frob)
\end{equation*}
sending the constant sheaf $\Ql$ to $\dS_A$. Here $\dS_A=\Sym(V_A^\vee)$ is viewed as a $\Ql$-algebra with $\Frob$-action (placed in degree 0) and $D_{\perf}(\dS_A,\Frob)\subset D^b(\dS_A,\Frob)$ is the full triangulated subcategory generated by twists of $\dS_A$.
\item The pull-back functor $D^b_m(\BB A)\to D^b_m(\pt)\cong D^b(\Frob)$ corresponds to the functor
\begin{equation*}
(-)\Ltimes_{\dS_A}\Ql:D_{\perf}(\dS_A,\Frob)\to D^b(\Frob).
\end{equation*}
\end{enumerate}
\end{cor}
In fact, Corollary \ref{c:ba}(2) above follows from the functoriality of the DG model in Proposition \ref{p:Cfun}.

\begin{cor}\label{c:eqkillS}
Let $X$ be a scheme with a left action of a torus $A$. Let $\pi:X\to[A\backslash X]$ be the projection. For any $\calF_1,\calF_2\in D^b_m([A\backslash X])$, we view $\RHom_{[A\backslash X]}(\calF_1,\calF_2)$ as an object in $D^b_m(\BB A)\cong D^b_m(\dS_A,\Frob)$ via Corollary \ref{c:ba}. Then we have a functorial isomorphism for $\calF_1,\calF_2\in D^b_m([A\backslash X])$:
\begin{equation}\label{killeqhom}
\RHom_{[A\backslash X]}(\calF_1,\calF_2)\Ltimes_{\dS_A}\Ql\cong\RHom_X(\pi^*\calF_1,\pi^*\calF_2)
\end{equation}
In particular, taking $\calF_1$ to be the constant sheaf, we get
\begin{equation}\label{killeqcoho}
\bR\Gamma([A\backslash X],\calF)\Ltimes_{\dS_A}\Ql\cong\bR\Gamma(X,\pi^*\calF)
\end{equation}
\end{cor}
\begin{proof}
Applying smooth base change to the Cartesian diagram
\begin{equation*}
\xymatrix{X\ar[r]^{\pi}\ar[d] & [A\backslash X]\ar[d]\\\textup{pt}\ar[r]^{p} & \BB A}
\end{equation*}
and the complex $\bR\underline{\Hom}(\calF_1,\calF_2)\in D^b_m([A\backslash X])$, we get
\begin{equation*}
p^*\RHom_{[A\backslash X]}(\calF_1,\calF_2)\cong\RHom_X(\pi^*\calF_1,\pi^*\calF_2).
\end{equation*}
It remains to apply Corollary \ref{c:ba}(2).
\end{proof}

More generally, in the situation of Example \ref{ex:icc}, we have a fully faithful embedding:
\begin{equation*}
\MM=\MM(\scC\subset\scD;\{\naC_\alpha\}):\scD\to D^b(E,\Frob).
\end{equation*}
where $\scC\subset\scD$ is the category of very pure complexes.

We can say more about the Hom-sets under $\MM$. For any locally finite $\Frob$-module $M$, let $M_i$ be the submodule of weight $i$ (i.e., the sum of generalized eigenspaces with eigenvalues of weight $i$). For any graded $\Frob$-module $N^\bullet=\oplus N^i$, we define
\begin{equation*}
N^\bullet_{\pure}:=\oplus_{i\in\ZZ} N^i_i.
\end{equation*}

\begin{lemma}\label{l:morehomeq}
There is a functorial isomorphism
\begin{equation}\label{purehom}
\Ext^\bullet_{\scD}(\calF_1,\calF_2)_{\pure}\cong\Hom_{E}(\MM\calF_1,\MM\calF_2)
\end{equation}
for $\calF_1,\calF_2\in\scD$.
\end{lemma}
\begin{proof}
The argument of Theorem \ref{th:CtoMod}(1) shows that
\begin{equation*}
\Ext^\bullet_{\scD}(\calC_1,\calC_2)\cong\Hom_{E}(h_{\naC}(\calC_1),h_{\naC}(\calC_2))
\end{equation*}
for $\calC_1,\calC_2\in\scC$. Note that for $\calC_1,\calC_2\in\scC$, $\Ext^i(\calC_1,\calC_2)$ is pure of weight $i$, hence \eqref{purehom} holds for $\calF_1,\calF_2\in\scC$. In general, we represent objects $\calF_i\in\scD$ by complexes of objects in $\scC$ and use a spectral sequence argument to deduce \eqref{purehom}.
\end{proof}

\subsection{Application to monodromic categories}
Applying Theorem \ref{th:CtoMod} to Example \ref{ex:fmtt}, we get a fully faithful embedding:
\begin{equation*}
\MM=\MM(\scT\subset\hsM;\{\naC_\alpha\}):\hsM\to D^b(E,\Frob)
\end{equation*}
where $\scT\subset\hsM$ is the category of \fm tilting sheaves. Again, we can say more about the Hom-sets under $\MM$.

\begin{lemma}\label{l:morehommon}
There is a functorial isomorphism
\begin{equation*}
\Hom_{\hsM}(\calF_1,\calF_2)\cong\Hom_{E}(\MM\calF_1,\MM\calF_2)
\end{equation*}
for $\calF_1,\calF_2\in\hsM$.
\end{lemma}
\begin{proof}
We only need to note that there is a functorial isomorphism
\begin{equation*}
\Hom_{\scT}(\tcT_1,\tcT_2)\cong\Hom_{E}(h_{\naC}(\tcT_1),h_{\naC}(\tcT_2))
\end{equation*}
for any $\tcT_1,\tcT_2\in\scT$.
\end{proof}

\begin{remark}\label{r:Ff}
In Example \ref{ex:fmtt}, the algebra $E$ will be an $\hatS_{A}=\varprojlim\Sym(V_A)/V_A^n$-module of finite type. Recall the functor $(-)^f$ in \eqref{Frobfin}. For any $(\hatS_{A},\Frob)$-module $M$ of finite type, $M^f$ is an $(S_{A},\Frob)$-module of finite type, where $S_{A}=\Sym(V_A)=\hatS_{A}^f$. It is easy to see that the adjoint functors
\begin{equation*}
\xymatrix{D^b(E,\Frob)\ar@<-.7ex>[r]_{(-)^f} & D^b(E^f,\Frob)\ar@<-.7ex>[l]_{\hatS_{A}\otimes_{S_{A}}(-)}}
\end{equation*}
are actually equivalences of categories. Therefore, in Theorem \ref{th:CtoMod}, we may also use $D^b(E^f,\Frob)$ as a DG model for the completed monodromic category $\scD=\hsM$.
\end{remark}

\section{Calculations for $SL(2)$}\label{a:SL2}
In this section, we specialize to the case $G=\SL(2)$ and the other notations (e.g., $\scE$, $\hsM$) are understood to be associated to $\SL(2)$. Let $B=UH$ be a Borel subgroup with unipotent radical $U$. The flag variety $\Fl=\PP^1$ and the enhanced flag variety is $\tilFl=\AA^2-\{0\}$ with the projection $\pi:\tilFl\to\Fl$ identified with the usual $\GG_m$-quotient $\AA^2-\{0\}\to\PP^1$. We denote the inclusion of the open and closed $B$-orbit into $\Fl$ (resp. $\tilFl$) by $j$ and $i$ (resp. $\tilj$ and $\tili$). Let $s$ be the nontrivial element in the Weyl group $W$. Let $\IC$ be the IC-sheaf of $\PP^1$.

A well-known computation of $\upH^*_B(\PP^1)$ gives the following
\begin{alemma}\label{l:Hsimple} There is a $\Frob$-equivariant isomorphism of $\dS$-bimodules:
\begin{equation*}
\HH(\IC)\cong\calO(\Gamma(e)\cup\Gamma(s))[1](1/2).
\end{equation*}
\end{alemma}

\noindent{\bf The \fm tilting sheaf.} We will construct a free-monodromic tilting object $\tcT\in\hsM$ whose underlying complex is indecomposable. For each $n\geq 1$, we have a local system $\calL_n$ on the open stratum of $\tilFl$ corresponding to the representation $S/V_H^{n+1}S$ of $\pi_1(H,e)$. Let $\Delta_n=\tilj_!\calL_n[2](3/2), \nabla=\tilj_*\calL_n[2](3/2)$. We have an exact sequence in $\scP$:
\begin{equation*}
0\to\pidag\delta(1/2+n)\to\Delta_n\to\nabla_n\to\pidag\delta(-1/2)\to0.
\end{equation*}
Passing to the projective limit, we get an exact sequence in $\hsP$:
\begin{equation}\label{eq:tiltst}
0\to\tDel\to\tnab\to\pidag\delta(-1/2)\to0.
\end{equation}
Now we define $\tcT$ by the fibered product of $\tnab$ and $\tdel(-1/2)$ over $\pidag\delta(-1/2)$. Therefore it fits into two exact sequences
\begin{equation}\label{sl2fmt}
0\to\tDel\to\tcT\to\tdel(-1/2)\to0
\end{equation}
\begin{equation*}
0\to\tdel(1/2)\to\tcT\to\tnab\to0.
\end{equation*}
where $\tdel(1/2)$ is identified with the kernel of $\tdel(-1/2)\twoheadrightarrow\pidag\delta(-1/2)$. This shows that $\tcT$ is a \fm tilting sheaf.

\begin{alemma}\label{l:Vsimple}
\begin{enumerate}
\item []
\item There is an isomorphism of $(S\times S,\Frob)$-algebras:
\begin{equation*}
\End_{\hsP}(\tcT)\cong\calO(\Gamma^*(e)\cup\Gamma^*(s)).
\end{equation*}
\item There is a $\Frob$-equivariant isomorphism of $S$-bimodules:
\begin{equation*}
\VV(\tcT)\cong\calO(\Gamma^*(e)\cup\Gamma^*(s))(-1/2).
\end{equation*}
\end{enumerate}
\end{alemma}
\begin{proof}
(2) Recall the object $\tcP\in\hsP$ which represents $\VV$ (see Lemma \ref{l:PTheta}). Since $\calT=\piddag\tcT$ is an indecomposable tilting sheaf on $\PP^1$, it is easy to see that $\omega\calT$ is also a projective cover of $\omega\delta$. Therefore $\omega\tcT$ is a projective cover of $\omega\pidag\delta$ in $\hsP$. Since $\tcT\twoheadrightarrow\pidag\delta(-1/2)$ is the highest weight quotient, $\Hom(\tcP,\tcT(1/2))^{\Funi}=\Ql$, hence $\hom(\tcP,\tcT(1/2))=\Ql$. Any nonzero homomorphism $\tcP\to\tcT(1/2)$ is in fact an isomorphism because after taking $\piddag$ it is. Therefore $\VV(\tcT)=\Hom(\tcP,\tcT)=\End(\tcT)(-1/2)$, and the statement follows from (1).

(1) We have maps
\begin{equation}\label{eq:ssrho}
S\otimes S\to\End(\tcT)\xrightarrow{(\tilj^{*},\tili^{*})}\End_{\hsP_s}(\tcL)\times\End_{\hsP_e}(\tdel)=\calO(\Gamma^*(s))\times\calO(\Gamma^*(e))
\end{equation}
where the first arrow is given by the left and right logarithmic $H$-monodromy and the second given by restrictions to two strata. The exact sequence (\ref{eq:tiltst}) gives an exact sequence
\begin{equation*}
0\to\Hom(\tcT,\tDel)\to\End(\tcT)\xrightarrow{\tili^{*}}\End(\tdel(-1/2))=\Hom(\tcT,\tdel(-1/2))\to0
\end{equation*}
On the other hand by Lemma \ref{l:vst} and the isomorphism $\tcP\cong\tcT(1/2)$, we have
\begin{equation*}
\Hom(\tcT,\tDel)\cong\Hom(\tcP,\tDel)(1/2)\cong\VV(\tDel)(1/2)\cong\calO(\Gamma^*(s))(1)
\end{equation*}
and the natural homomorphism $\tili^*:\Hom(\tcT,\tDel)\to\End(\tDel)$ is the inclusion $\calO(\Gamma^*(s))(1)\hookrightarrow\calO(\Gamma^*(s))$. Therefore $(\tilj^{*},\tili^{*})$ in (\ref{eq:ssrho}) is injective. The composition of the maps in (\ref{eq:ssrho}) has image $\calO(\Gamma^*(s)\cup\Gamma^*(e))$, hence the map $S\otimes S\to\End(\tcT)$ factors through
\begin{equation*}
S\otimes S\twoheadrightarrow\calO(\Gamma^*(s)\cup\Gamma^*(e))\to\End(\tcT).
\end{equation*}
Therefore we have a commutative diagram of exact sequences
\begin{equation*}
\xymatrix{\calO(\Gamma^*(s))(1)\ar[r]\ar[d]^{\wr} & \calO(\Gamma^*(s)\cup\Gamma^*(e))\ar[r]\ar[d] & \calO(\Gamma^*(e))\ar[d]^{\wr}\\
\Hom(\tcT,\tDel)\ar[r] & \End(\tcT)\ar[r] & \End(\tdel(-1/2))}
\end{equation*}
Since the first and third vertical maps are already shown to be isomorphisms, the middle one must also be an isomorphism.
\end{proof}

Finally we compute the convolutions in $\scE$. Observe that for any $\calF\in\scE$, we have
\begin{eqnarray*}
\calF\conv{B}\IC&\cong& \upH^*(\PP^1,\calF)\otimes\IC;\\
\calF\conv{B}\delta&\cong&\calF.
\end{eqnarray*}

\begin{alemma}\label{l:convSL2} We have
\begin{eqnarray*}
\Delta\conv{B}\Delta&\in&\langle\Delta(1/2),\Delta[-1](-1/2),\delta\rangle\\
\nabla\conv{B}\nabla&\in&\langle\nabla(-1/2),\nabla[1](1/2),\delta\rangle
\end{eqnarray*}
\end{alemma}

\begin{proof}
We prove the first relation; the second can be proved similarly. Applying $\Delta\conv{B}$ to the distinguished triangle
\begin{equation}\label{eq:exfordel}
\delta(1/2)\to\Delta\to\IC\to,
\end{equation}
we get another distinguished triangle
\begin{equation*}
\Delta(1/2)\to\Delta\conv{B}\Delta\to\IC[-1](-1/2)\to.
\end{equation*}
In other words, $\Delta\conv{B}\Delta\in\langle\Delta(1/2),\IC[-1](-1/2)\rangle$. The triangle (\ref{eq:exfordel}) also implies that $\IC[-1](-1/2)\in\langle\Delta[-1](-1/2),\delta\rangle$. Therefore $\Delta\conv{B}\Delta\in\langle\Delta(1/2),\Delta[-1](-1/2),\delta\rangle$.
\end{proof}


\section*{List of symbols}

\noindent
\begin{footnotesize}
\begin{tabular}{ll}
$G$ & A Kac-Moody group\\
$G^\vee$ & A Kac-Moody group whose root system is dual to that of $G$\\
$H,H^\vee$ & A fixed Cartan subgroup of $G$ and its dual in $G^\vee$\\
$V_H$ & $\Ql$-Tate module of $H$\\
$B,U$ & The standard Borel in $G$ and its unipotent radical\\
$B^\vee, U^\vee$ & The standard Borel in $G^\vee$ and its unipotent radical\\
$W$ & The Weyl group of $(G,H)$\\
$W_\Theta$ & The Weyl group of $(L_\Theta,H)$\\
$[\coW]$ & The shortest representatives in the coset $\coW$\\
$\{\coW\}$ & The longest representatives in the coset $\coW$\\
$w_\Theta, \lTh$ & The longest element in $W_\Theta$ and its length\\
$P_\Theta=U^\Theta L_\Theta$ & The standard parabolic subgroup and its Levi decomposition\\
$U_\Theta,U^-_\Theta$ & $L_\Theta\cap U$ and its opposite maximal unipotent subgroup of $L_\Theta$\\
$\Fl,\tilFl$ & The flag variety $G/B$ and its enhancement $G/U$\\
$\pi$ & The projection $G/U\to G/B$\\
$\PFl_G$ & The partial flag variety $P_\Theta\backslash G$\\
$\pi^\Theta$ & The projection $B\backslash G\to P_\Theta\backslash G$\\
$\chi$ & A nondegenerate additive character of $U^-_\Theta$\\
$\scE_G$ & The equivariant category $D^b_{m}(\quot{B}{G}{B})$\\
$\scE_{G,\Theta}$ & The parabolic category $D^b_{m}(\quot{P_\Theta}{G}{B})$\\
$\scM_G,\hsM_G$ & The monodromic category $D^b_{m}(\wqw{B}{G}{B})$ and its completion\\
$\scM_{G,\Theta},\hsM_{G,\Theta}$ & The Whittaker category $D^b_{m}((U^\Theta U^-_\Theta,\chi)\backslash\qw{G}{B})$ and its completion \\
$\Dright,\Dleft$ & $D^b_{m}(\quot{U}{G}{B})$ and $D^b_{m}(\quot{\dB}{\dG}{\dU})$\\
$\Dleft_\Theta$ & The paradromic category $D^b_{m}(\quot{P^\vee_\Theta}{\dG}{\dU})$\\
$\Dright_\Theta$ & The ``Whittavariant'' category $D^b_{m}(\quot{(U^\Theta U^-_\Theta,\chi)}{G}{B})$\\
$\Av^\Theta_\chi$ & The averaging functor $\hsM\to\hsM_\Theta$\\
$\IC_w$ & The intersection cohomology complex of $\Fl_{\leq w}$ in various categories.\\
$\Delta_w,\Delta_{\barw}$ & The standard sheaves in $\scE$ and $\scE_\Theta$\\
$\nabla_w,\nabla_{\barw}$ & The costandard sheaves in $\scE$ and $\scE_\Theta$\\
$\tcL$ & The \fm local system on a torus $H$\\
$\tDel_w,\tDel_{\barw,\chi}$ & The \fm standard sheaves in $\hsM$ and $\hsM_\Theta$\\
$\tnab_w,\tnab_{\barw,\chi}$ & The \fm costandard sheaves in $\hsM$ and$\hsM_\Theta$\\
$\tcT_w,\tcT_{\barw,\chi}$ & The indecomposable \fm tilting sheaves in $\hsM$ and $\hsM_\Theta$\\
$\calT_w,\calT^\vee_w$ & The indecomposable tilting sheaves in $\Dright$ and $\Dleft$\\
$\calC_\Theta$ & The constant sheaf on $\Fl_{\leq w_\Theta}$\\
$\tcP_\Theta$ & Its underlying complex is a projective cover of $\omega\delta$ in $\omega\hsP_{\leq w_\Theta}$\\
$\dS_H$ & $\Sym(V_H^\vee)$\\
$S_H$ & $\Sym(V_H)$; logarithmic monodromy operators by $H$\\
$\HH$ & The global section functor of $\scE$ and its cohomology\\
$\VV$ & The averaging functor: $\hsM\to D^b(S\otimes S,\Frob)$\\
$\omega$ & Forgetting the mixed structure\\
\end{tabular}
\end{footnotesize}


\end{document}